\documentclass[11pt]{amsart}

\usepackage[centertags]{amsmath}
\usepackage{amsfonts}
\usepackage{amssymb}
\usepackage{amsthm}
\usepackage[english]{babel}
\usepackage[active]{srcltx}
\usepackage{enumerate}
\usepackage{tikz}
\usepackage{mathrsfs}

\numberwithin{equation}{section}

\newtheorem{thm}{Theorem}[section]
\newtheorem{cor}[thm]{Corollary}
\newtheorem{lem}[thm]{Lemma}
\newtheorem{prop}[thm]{Proposition}

\newtheorem{defn}[thm]{Definition}

\def\N{{{\Bbb N}}}

\def\T{{{\Bbb T}}}
\def\B{{{ \textbf{B}}}}

\theoremstyle{remark}
\newtheorem{rem}[thm]{Remark}

\begin{document}
\title[Function spaces of logarithmic smoothness]{Function spaces of logarithmic smoothness: embeddings and characterizations}
\author
{\'Oscar Dom\'inguez}\address{O. Dom\'inguez, Departamento de An\'alisis Matem\'atico y Matem\'atica Aplicada, Facultad de Matem\'aticas, Universidad Complutense de Madrid\\
Plaza de Ciencias 3, 28040 Madrid, Spain.}
\email{oscar.dominguez@ucm.es}

\author{Sergey  Tikhonov} \address{S. Tikhonov, Centre de Recerca Matem\`{a}tica\\
Campus de Bellaterra, Edifici C
08193 Bellaterra (Barcelona), Spain;
ICREA, Pg. Llu\'{i}s Companys 23, 08010 Barcelona, Spain,
 and Universitat Aut\`{o}noma de Barcelona.}
\email{ stikhonov@crm.cat}

\date{\today}


\maketitle



\bigskip
\begin{abstract}
In this paper we present a comprehensive treatment of function spaces with logarithmic smoothness (Besov, Sobolev, Trie\-bel-Lizorkin). We establish the following results:
\begin{enumerate}
	\item Sharp embeddings between the Besov spaces defined by differences and by Fourier-analytical decompositions as well as between Besov and Sobolev/Trie\-bel-Lizorkin spaces;
	\item Various new characterizations for Besov norms in terms of different K-functionals. For instance, we derive characterizations via ball averages, approximation methods, heat kernels, and Bianchini-type norms;
	\item Sharp estimates for Besov norms of derivatives and potential operators (Riesz and Bessel potentials) in terms of norms of functions themselves. We also obtain quantitative estimates of regularity properties of the fractional Laplacian.
\end{enumerate}
The key tools behind our results are limiting interpolation techniques and new characterizations of Besov and Sobolev norms in terms of the behavior of the Fourier transforms for functions such that their Fourier transforms are of monotone type or lacunary series.


\end{abstract}
\tableofcontents
%
%
%
%

\newpage

\section{Introduction}

Function spaces of generalized smoothness play a crucial role in obtaining  the complete solution of several important  questions, for which the classical approach turns out to be limited. Let us mention just a few of them: the celebrated Br\'ezis-Wainger inequality \cite{BrezisWainger} on the logarithmic Lipschitz continuity of functions from the Sobolev space $H^{1+d/p}_p(\mathbb{R}^d), 1 < p < \infty$; the investigation of compactness of limiting embeddings, which  requires the finer tuning given by the logarithmic smoothness \cite{Leopold}; applications to probability theory and the theory of stochastic processes \cite{FarkasLeopold}; smoothness spaces with power-logarithmic majorants, which  are used extensively in functional analysis (see, e.g.,
\cite{robert})
and differential equations (see, e.g., \cite{Colombini, zuazua}); or function spaces defined on fractals and a related spectral theory (see the monographs by Triebel \cite{Triebel01, Triebel3}).

\subsection{Besov, Sobolev, and Triebel--Lizorkin spaces.}
There are several ways to introduce function spaces of generalized smoothness. Among them, the approaches based on differences and Fourier-analytical decompositions are the most popular. For $1 \leq p
\leq \infty, 0 < q \leq \infty,$ and $- \infty < b < \infty$, the Besov space of generalized smoothness can be described via the quasi-norms
\begin{equation}\label{Intro1}
    \|f\|_{\textbf{B}^{s,b}_{p,q}(\mathbb{R}^d)} = \|f\|_{L_p(\mathbb{R}^d)}+ \left(\int_0^1 (t^{-s} (1 - \log t)^b \omega_k(f,t)_p)^q
    \frac{dt}{t}\right)^{1/q}, \quad s \geq
0,
\end{equation}
 and
\begin{equation}\label{Intro2}
    \|f\|_{B^{s,b}_{p,q}(\mathbb{R}^d)} = \Bigg(\sum_{j=0}^\infty \Big(2^{j s} (1 + j)^b \|(\varphi_j
    \widehat{f})^\vee\|_{L_p(\mathbb{R}^d)}\Big)^q\Bigg)^{1/q},\quad s \in  \mathbb{R}.
\end{equation}
See Section \ref{Section 2.2} below for precise definitions. Note that if $b = 0$ we recover the classical Besov spaces
$\mathbf{B}^s_{p,q}(\mathbb{R}^d)$ and ${B}^s_{p,q}(\mathbb{R}^d)$. In the limiting case where $s = 0$ we obtain
the spaces $\mathbf{B}^{0,b}_{p,q}(\mathbb{R}^d)$ and ${B}^{0,b}_{p,q}(\mathbb{R}^d)$ having zero classical smoothness and an additional
logarithmic smoothness with exponent $b$. Even though these spaces  are close to $L_p(\mathbb{R}^d)$, they have several special properties due to 
  their structure as Besov spaces.

Working with positive classical smoothness (i.e., $s > 0$), it is well known that Besov spaces can be equivalently characterized by (\ref{Intro1}) or (\ref{Intro2}). More precisely, we have
\begin{equation}\label{BesovComparison}
	\mathbf{B}^{s,b}_{p,q}(\mathbb{R}^d) =
B^{s,b}_{p,q}(\mathbb{R}^d) \quad \text{ for } \quad s > 0,
\end{equation}
with equivalence of quasi-norms (see \cite[2.5.12]{Triebel1} and \cite[Theorem
2.5]{HaroskeMoura}). However, this formula is no longer true in the
limiting case when $s=0$. Indeed, the space $\mathbf{B}^{0,b}_{p,q}(\mathbb{R}^d)$ is a subspace of $L_p(\mathbb{R}^d)$, but $B^{0,b}_{p,q}(\mathbb{R}^d)$ is not necessarily formed by locally integrable functions (see \cite{CaetanoLeopold} with the forerunner \cite{SickelTriebel}). However, it is still possible to compare these two scales of Besov spaces with the help of some shifts in the logarithmic smoothness. Namely, the following result was obtained in
\cite[Theorem 3.3]{CobosDominguez3}: If $1 < p < \infty, 0 < q \leq \infty$ and $b > -1/q$, then
\begin{equation}\label{1}
    B^{0,b + 1/\min\{2,p,q\}}_{p,q}(\mathbb{R}^d) \hookrightarrow
    \textbf{B}^{0,b}_{p,q}(\mathbb{R}^d) \hookrightarrow
    B^{0,b+1/\max\{2,p,q\}}_{p,q}(\mathbb{R}^d).
\end{equation}
In particular, we have
\begin{equation}\label{BesovZero}
 \mathbf{B}^{0,b}_{2,2}(\mathbb{R}^d) =
B^{0,b+1/2}_{2,2}(\mathbb{R}^d) \quad \text{for} \quad b > -1/2,
\end{equation}
and putting $b=0$, we derive that the classical Besov space $\mathbf{B}^{0}_{2,2}(\mathbb{R}^d)$ can be described through Fourier decompositions with an additional logarithmic smoothness of exponent $1/2$. Consequently, logarithmic smoothness arises in a natural way when we investigate relations between different scales of classical Besov spaces in limiting situations. Little is known about the sharpness of (\ref{1}), except for the case $p=\min\{2,p,q\}$ (respectively, $p=\max\{2,p,q\}$) where it was proved in \cite{CobosDominguezTriebel} that $b + 1/p$ is the best possible logarithmic smoothness for which the left-hand side embedding (respectively, right-hand side embedding) in (\ref{1}) holds. On the other hand, Winfried Sickel \cite{Sickel} conjectured that (\ref{BesovZero}) is the only possible case when the $\mathbf{B}$- and $B$-spaces with classical smoothness zero coincide.


Another scale formed by smoothness spaces, which is intensively investigated in the literature, is given by Triebel-Lizorkin spaces. For $1 < p < \infty, 0 < q \leq \infty,$ and $-\infty < s, b < \infty$,
 the Triebel-Lizorkin space $F^{s,b}_{p,q}(\mathbb{R}^d)$ is formed by all tempered distributions $f \in
\mathcal{S}'(\mathbb{R}^d)$ for which
\begin{equation*}
    \|f\|_{F^{s,b}_{p,q}(\mathbb{R}^d)} = \Big\|\Big(\sum_{j=0}^\infty (2^{js} (1 + j)^b |(\varphi_j \widehat{f})^\vee
    (\cdot)|)^q\Big)^{1/q}\Big\|_{L_p(\mathbb{R}^d)} < \infty
\end{equation*}
(with the usual modification if $q=\infty$).
It was known for a long time that the
 Sobolev spaces
$H^{s,b}_p(\mathbb{R}^d)$
 are particular cases of the spaces $F^{s,b}_{p,q}(\mathbb{R}^d)$. Namely (see \cite[2.3.3]{Triebel} and \cite[Theorem 3.4]{CobosFernandez}),
\begin{equation}\label{LPgeneral}
	F^{s,b}_{p,2}(\mathbb{R}^d) = H^{s,b}_p(\mathbb{R}^d)
\end{equation}
with equivalence of norms. In particular, if $b=0$ we recover the classical (fractional) Sobolev spaces $H^s_p(\mathbb{R}^d)$. See Section \ref{Section 2.2} below for further details 
 on Triebel-Lizorkin and Sobolev spaces.

In order to state our main results and explain the new ideas and phenomena, we first resume briefly several crucial embeddings between $B$, $\mathbf{B}$, $H$, and $F$-spaces.

We start with embeddings with constant smoothness and integrability, that are, in the form $X^{s,b}_{p,q} \hookrightarrow Y^{s,b}_{p,r}$, where $X, Y \in \{B, \mathbf{B}, H, F\}$.  The proofs can be found, for instance, in \cite[Proposition 2.3.2/2]{Triebel1}, \cite[Theorem 3.1.1]{SickelTriebel} and, for the case $b \neq 0$, we refer to \cite[Proposition 3.4]{CaetanoMoura}.

\begin{prop}\label{RecallEmb}
Let $1 < p < \infty, 0 <  q, r \leq \infty$, and $-\infty < s, b < \infty$. Then, we have
\begin{enumerate}[\upshape(i)]
\item
    $F^{s,b}_{p,r}(\mathbb{R}^d) \hookrightarrow B^{s,b}_{p,q}(\mathbb{R}^d) \quad \text{if and only if} \quad q \geq
    \max\{p,r\}$,

\item $B^{s,b}_{p,q}(\mathbb{R}^d) \hookrightarrow F^{s,b}_{p,r}(\mathbb{R}^d) \quad \text{if and only if} \quad  q \leq
    \min\{p,r\}$.
    \end{enumerate}
In particular, $F^{s,b}_{p,p}(\mathbb{R}^d) = B^{s,b}_{p,p}(\mathbb{R}^d)$.

If, in addition, $s > 0$ then
    \begin{enumerate}[\upshape(i)]
    \setcounter{enumi}{2}
    \item      $F^{s,b}_{p,r}(\mathbb{R}^d) \hookrightarrow \mathbf{B}^{s,b}_{p,q}(\mathbb{R}^d) \quad \text{if and only if} \quad  q \geq
    \max\{p,r\}$,
      \item  $H^{s,b}_p(\mathbb{R}^d) \hookrightarrow \mathbf{B}^{s,b}_{p,q}(\mathbb{R}^d) \quad \text{if and only if} \quad  q \geq
    \max\{p,2\}$,
    \item $\mathbf{B}^{s,b}_{p,q}(\mathbb{R}^d) \hookrightarrow F^{s,b}_{p,r}(\mathbb{R}^d) \quad \text{if and only if} \quad  q \leq
    \min\{p,r\}$,
\item
    $\mathbf{B}^{s,b}_{p,q}(\mathbb{R}^d) \hookrightarrow H^{s,b}_p(\mathbb{R}^d) \quad \text{if and only if} \quad  q \leq
    \min\{p,2\}$.
\end{enumerate}
\end{prop}

Note that by (\ref{BesovComparison}) and (\ref{LPgeneral}), the embeddings given in (iii) and (iv) (respectively, (v) and (vi)) are special cases of (i) (respectively, (ii)).

Now we deal with the simple embeddings for Besov and Triebel--Lizorkin spaces of the form $X^{s_0,b_0}_{p,q_0} \hookrightarrow X^{s_1,b_1}_{p,q_1}$,  which are  immediate consequences of H\"older's inequality.
\begin{prop}\label{RecallEmb*}
	Let $1 < p < \infty, 0 < q_0, q_1 \leq \infty$, and $-\infty < s_0, s_1, b_0, b_1 < \infty$. Then, we have
	\begin{equation*}
		B^{s_0, b_0}_{p,q_0}(\mathbb{R}^d) \hookrightarrow B^{s_1, b_1}_{p,q_1}(\mathbb{R}^d)
	\end{equation*}
	if one of the following conditions is valid
	\begin{enumerate}[\upshape(i)]
		\item $s_0 > s_1$,
		\item $s_0 = s_1, q_0 \leq q_1$, and $b_0 \geq b_1$,
		\item $s_0 = s_1, q_0 > q_1$, and $b_0 + \frac{1}{q_0} > b_1 + \frac{1}{q_1}$.
	\end{enumerate}
	The corresponding assertion for $F$-spaces is also true.
\end{prop}

The counterpart of Proposition \ref{RecallEmb*} for $\mathbf{B}$-spaces reads as follows.
\begin{prop}\label{RecallEmb**}
	Let $1 < p < \infty, 0 < q_0, q_1 \leq \infty, s_0, s_1 \geq 0,$ and $-\infty< b_0, b_1 < \infty \, (b_i \geq -1/q_i \quad \text{if} \quad s_i = 0, i = 0, 1)$. Then, we have
	\begin{equation*}
		\mathbf{B}^{s_0, b_0}_{p,q_0}(\mathbb{R}^d) \hookrightarrow \mathbf{B}^{s_1, b_1}_{p,q_1}(\mathbb{R}^d)
	\end{equation*}
	if one of the following conditions is valid
	\begin{enumerate}[\upshape(i)]
		\item $s_0 > s_1 \geq 0$,
		\item $s_0 = s_1 > 0, q_0 \leq q_1$, and $b_0 \geq b_1$,
		\item $s_0 = s_1 > 0, q_0 > q_1$, and $b_0 + \frac{1}{q_0} > b_1 + \frac{1}{q_1}$,
		\item $s_0 = s_1 = 0, b_0 + \frac{1}{q_0} > b_1 + \frac{1}{q_1}$,
		\item $s_0 = s_1 = 0, b_0 + \frac{1}{q_0} = b_1 + \frac{1}{q_1}$, and $q_0 \leq q_1$.
	\end{enumerate}
\end{prop}

To the best of our knowledge, the parts ``only if'' in Propositions \ref{RecallEmb*} and \ref{RecallEmb**} remain as open questions.

Below we give the key embeddings with different metrics, that are, in the form $X_{p_0} \hookrightarrow Y_{p_1}$.
\begin{prop}\label{RecallEmb2}
	Let  $-\infty < b < \infty$.
	\begin{enumerate}[\upshape(i)]
		\item Sobolev embeddings: Let $1 \leq p_0 < p_1 \leq \infty, -\infty < s_1 < s_0 < \infty$ with
		\begin{equation*}
			s_0 -\frac{d}{p_0} = s_1 -\frac{d}{p_1},
		\end{equation*}
		and $0 < q \leq \infty$. Then,
		\begin{equation}\label{e0}
		B^{s_0,b}_{p_0,q}(\mathbb{R}^d) \hookrightarrow B^{s_1,b}_{p_1,q}(\mathbb{R}^d).
		\end{equation}
		\item Franke-Jawerth embeddings: Let $1 \leq p_0 < p < p_1 \leq \infty$ and $-\infty < s_1 < s < s_0 < \infty$ with
\begin{equation*}
	s_0 - \frac{d}{p_0} = s -\frac{d}{p} = s_1 -\frac{d}{p_1}.
\end{equation*}
Then
\begin{equation}\label{e0*}
	B^{s_0,b}_{p_0,p} (\mathbb{R}^d) \hookrightarrow H^{s,b}_p(\mathbb{R}^d) \hookrightarrow B^{s_1,b}_{p_1,p}(\mathbb{R}^d).
\end{equation}
	\end{enumerate}
\end{prop}

The embedding given in Proposition \ref{RecallEmb2}(i) is known as the Sobolev embedding theorem and  has a long history in the theory of function spaces.
Details, explanations and references may be found in \cite[2.7.1]{Triebel1}. See also \cite[Proposition 1.9]{Moura}. On the other hand, Proposition \ref{RecallEmb2}(ii) is the Franke-Jawerth embeddings  \cite{Jawerth, Franke} (see also \cite{Marschall}, \cite{CaetanoMoura}, \cite{Vybiral*}).  It is worth mentioning that the embeddings (\ref{e0*}) also hold when we replace the space $H^{s.b}_p(\mathbb{R}^d)$ by $F^{s,b}_{p,q}(\mathbb{R}^d)$. We stress that the latter holds  for any $0 < q \leq \infty$, that is, the parameter $q$ does not play a role in Franke-Jawerth embeddings unlike in the case of Sobolev embeddings (\ref{e0}).

\subsection{Applications of function spaces with smoothness zero.}
In this paper we will pay special attention to function spaces with smoothness zero. There are several important reasons to go deeply into the theory of these spaces. Our analysis goes back to Nikolski$\breve{\text{\i}}$, Lions and Lizorkin \cite{NikolskiiLionsLizorkin}, who  pointed out that lack of classical smoothness yields different results on function spaces as compared to positive smoothness.
In particular, DeVore, Riemenschneider and Sharpley \cite{DeVoreRiemenschneiderSharpley} applied the usefulness of Besov spaces of smoothness close to zero on weak-type interpolation theory to study the smoothness properties of the Hilbert transform. They proved \cite[Corollary 6.3]{DeVoreRiemenschneiderSharpley} that there is a loss of logarithmic smoothness in order to observe that the Hilbert transform of a function belongs to $\mathbf{B}^{0,b}_{1,q}(\mathbb{R}^d)$. This result is in sharp contrast with the setting of the Besov spaces $\mathbf{B}^{s,b}_{1,q}(\mathbb{R}^d), s > 0,$ where it is known that the Hilbert transform acts boundedly. We also refer to the paper by Gol'dman \cite{Goldman} where a rather general method to construct Besov-type spaces and, in particular Besov spaces of smoothness zero, was investigated.


Discussing applications, let us first deal with spaces with zero smoothness ($s=b=0$).
Note that a function $f\in \mathbf{B}^{0,0}_{\infty,1}(\mathbb{R}^d)$
  is called Dini continuous. This condition is widely used in different topics in mathematics.
In particular, for periodic functions  this condition guaranties that a conjugate function is continuous. The Fourier series of a Dini continuous  function converges uniformly (Dini's test). There are many different applications in functional analysis \cite{Garnett}, in PDEs \cite{Gilbarg}, and in probability theory
\cite{Barnsley}. On the other hand, the $\mathbf{B}^{0,0}_{\infty,2}(\mathbb{R}^d)$ condition plays a major role in the paper by Fabes, Jerison and Kenig \cite{FabesJerisonKenig} in order to characterize the absolute continuity of harmonic measure.


The $L_1$-Dini condition as well as the $L_p$-Dini condition, that is, when $f\in \mathbf{B}^{0,0}_{1,1}(\mathbb{R}^d)$ or $f\in \mathbf{B}^{0,0}_{p,1}(\mathbb{R}^d)$, respectively,  is used in various questions of harmonic analysis (in particular, related to  a study of singular integrals  \cite{Lu}) and probability theory \cite{Khoshnevisan}. For instance, Calder\'on, Weiss and Zygmund \cite{CalderonWeissZygmund} showed that if a kernel satisfies the $L_1$-Dini condition,
 then the corresponding singular integral operator exists almost everywhere for integrable functions and that is of strong type $(p,p), 1 < p < \infty$, and of weak type $(1,1)$. Kurtz and Wheeden \cite{KurtzWheeden}  extended this result to weighted Lebesgue spaces assuming that the kernel satisfies the $L_p$-Dini condition with $p>1$.

Furthermore, for periodic functions defined on the unit circle $\T$, the condition
$f\in \textbf{B}^{0,0}_{2,2}(\T)$ is equivalent to the famous  Kolmogorov--Seliverstov--Plessner condition $\sum (a_n^2(f)+b_n^2(f))\log n<\infty$, where $a_n(f)$, $b_n(f)$ are Fourier coefficients of $f$. The latter is the crucial condition for convergence problems
of trigonometric/orthogonal series, see, e.g., \cite[Chapter V]{Bari} and \cite{Lukashenko}.

 Recently there has been a renewed interest in analyzing the structural properties of Besov spaces of smoothness close to zero, not only
 in 
  the theory of function spaces (see, e.g.,  \cite{SeegerTrebels, Vybiral, CaetanoGogatishviliOpic, CaetanoGogatishviliOpic2, Besov, TriebelReport, besov}), but also in applications to approximation theory \cite{CobosDominguezKuhn, KashinTemlyakov, Li, Lin}, trace and extension theorems in the setting of metric measure spaces \cite{MalyShanmugalingamSnipes}, PDE problems under low regularity assumptions on the coefficients \cite{Colombini, DongZhang, MazyaMcOwen}, theory of maximal regularity in parabolic equations in Banach spaces \cite{Milman}, sharp estimates for singular integrals \cite{DingLai, DingLai2}, or problems on mixing flows (Bressan's conjecture) \cite{Bianchini, HadzicSeegerSmartStreet}.

\subsection{Goals and results}
The first goal of this paper is to find
 the precise interrelation between the function spaces with smoothness near zero, i.e., when the classical smoothness parameter $s = 0$.
 This question was also posed recently by Andreas Seeger in \cite{Seeger}.
 Secondly, motivated by applications to PDE and calculus of variations, we intend to find various equivalent characterizations of Besov spaces, in particular,  in terms of ball averages and Bianchini-type norms. As a by-product, we derive smoothness properties of differential operators (partial derivatives and fractional Laplace operator) and potential operators (Bessel and Riesz potentials).


Let us discuss our first goal in detail.
 We will obtain new embedding results for Triebel-Lizorkin, Sobolev, and several Besov spaces with logarithmic smoothness (defined Fourier analytically, in terms of differences, in terms of averages over the ball, Bianchini-type) to fill the gap between these spaces in the literature. Among other results, we prove that
 \begin{equation}\label{Goal1}
     	H^{0,b+1/q}_{p} (\mathbb{R}^d) \hookrightarrow \mathbf{B}^{0,b}_{p,q}(\mathbb{R}^d) \quad \text{ if and only if}\quad q \geq \max\{p,2\},
\end{equation}
\begin{equation}\label{Goal2}
   	\mathbf{B}^{0,b}_{p,q}(\mathbb{R}^d) \hookrightarrow H^{0,b+1/q}_{p}(\mathbb{R}^d)\quad \text{ if and only if} \quad q \leq \min\{p,2\},
\end{equation}
and, in particular,
\begin{equation}\label{Goal3}
	\mathbf{B}^{0,b}_{2,2}(\mathbb{R}^d) = H^{0,b+1/2}_2(\mathbb{R}^d), \quad b > -1/2,
\end{equation}
cf. Proposition \ref{RecallEmb}(iv), (vi) and (\ref{BesovZero}). The basic idea to show (\ref{Goal1})-(\ref{Goal3}) is to apply the Fourier-analytical decomposition of the space $\mathbf{B}^{0,b}_{p,q}(\mathbb{R}^d)$ that has been recently obtained in 
\cite{CobosDominguezTriebel}. Namely, we have
    \begin{equation*}
        \|f\|_{\mathbf{B}^{0,b}_{p,q}(\mathbb{R}^d)} \asymp \Bigg(\sum_{j=0}^\infty \Big[(1 + j)^b \Big\|\Big(\sum_{\nu=j}^\infty |(\varphi_\nu \widehat{f})^\vee
        (\cdot)|^2\Big)^{1/2}\Big\|_{L_p(\mathbb{R}^d)}\Big]^q\Bigg)^{1/q}.
    \end{equation*}

Similarly to (\ref{1}), we observe that an additional logarithmic smoothness with exponent $1/q$ arises both in (\ref{Goal1}) and (\ref{Goal2}). Moreover, we also show that this shift is the best possible. For example, if $q \geq \max\{p,2\}$, then
 \begin{equation}\label{Goal3new}
     	H^{0,\xi}_{p} (\mathbb{R}^d) \hookrightarrow \mathbf{B}^{0,b}_{p,q}(\mathbb{R}^d) \quad \text{ if and only if}\quad \xi \geq b+1/q.
\end{equation}
The corresponding result for (\ref{Goal2}) is also obtained. Regarding (\ref{Goal3}), we obtain the following sharpness assertion
	 \begin{equation}\label{Goal3new2}
	 	\mathbf{B}^{0,b}_{p,q}(\mathbb{R}^d) = H^{0,\xi}_p (\mathbb{R}^d) \quad \text{if and only if} \quad p = q = 2 \quad \text{and} \quad \xi = b+1/2.
	 \end{equation}

Our approach can also be applied to derive optimality in the full range of parameters of the embeddings (\ref{1}). Therefore, we are able to cover the cases left open in \cite{CobosDominguezTriebel}. We now show that
 \begin{equation}\label{Goal3new3}
    B^{0,\xi}_{p,q}(\mathbb{R}^d) \hookrightarrow
    \mathbf{B}^{0,b}_{p,q}(\mathbb{R}^d) \quad \text{if and only if} \quad \xi \geq b + 1/\min\{2,p,q\}
\end{equation}
and
\begin{equation}\label{Goal3new4}
	 \mathbf{B}^{0,b}_{p,q}(\mathbb{R}^d) \hookrightarrow
    B^{0,\xi}_{p,q}(\mathbb{R}^d)  \quad \text{if and only if} \quad \xi \leq b + 1/\max\{2,p,q\}.
\end{equation}
Furthermore, we provide a positive answer to Sickel's Conjecture \cite{Sickel}, that is,
\begin{equation}\label{Goal3new5}
		\mathbf{B}^{0,b}_{p,q}(\mathbb{R}^d) = B^{0,\xi}_{p,q}(\mathbb{R}^d) \quad \text{if and only if} \quad p=q=2 \quad \text{and} \quad \xi= b+1/2.
	\end{equation}
	
	Let us fix the integrability parameter $p$ and the fine index $q$. It turns out that the relationships between function spaces with smoothness zero are determined by adequate shifts of the logarithmic smoothness $b$. Assume first that $q \geq \max\{p,2\}$. Then, we will show that (see Figure 1)
	\begin{itemize}
		\item the space $\mathbf{B}^{0,b}_{p,q}(\mathbb{R}^d)$ is embedded into $B^{0,b+1/q}_{p,q}(\mathbb{R}^d)$. In addition, we have $\mathbf{B}^{0,b}_{p,q}(\mathbb{R}^d) \neq B^{0,b+1/q}_{p,q}(\mathbb{R}^d)$,
		\item the spaces $H^{0,b+1/q}_p(\mathbb{R}^d)$ and $B^{0,b+1/\min\{p,2\}}_{p,q}(\mathbb{R}^d)$ are embedded into $\mathbf{B}^{0,b}_{p,q}(\mathbb{R}^d)$,
		\item the spaces $H^{0,b+1/q}_p(\mathbb{R}^d)$ and $B^{0,b+1/\min\{p,2\}}_{p,q}(\mathbb{R}^d)$ are not comparable.
	\end{itemize}

	\bigskip

\begin{center}
\begin{tikzpicture}[fill opacity=0.05, ,xscale=0.8,yscale=0.8]

\fill (0,0) ellipse (5.0 and 3.0);
\draw[fill opacity=0.2] (0,0) ellipse (5.0 and 3.0);

\fill (0,0) ellipse (4.0 and 2.2);
\draw[fill opacity=0.2] (0,0) ellipse (4.0 and 2.2);

\fill[white] (-1.8,0) ellipse (2.5 and 1.1);
\draw[fill opacity=0.2](-1.4,0) ellipse (2.5 and 1.1);

\fill[white] (1.8,0) circle (2.55 and 1.1);
\draw[fill opacity=0.2] (1.4,0) circle (2.55 and 1.1);

\node[fill opacity=2,xscale=0.8,yscale=0.8] at (0,1.6) {$\mathbf{B}^{0,b}_{p,q}(\mathbb{R}^d)$};
\node[fill opacity=2,xscale=0.8,yscale=0.8] at (-2.5,0) {$H^{0,b+\frac{1}{q}}_p(\mathbb{R}^d)$};
\node[fill opacity=2,xscale=0.8,yscale=0.8] at (2.5,0) {$B^{0,b + \frac{1}{\min\{2,p\}}}_{p,q}(\mathbb{R}^d)$};
\node[fill opacity=2,xscale=0.8,yscale=0.8] at (0,-2.6) {$B^{0,b + \frac{1}{q}}_{p,q}(\mathbb{R}^d)$};

\end{tikzpicture}
\end{center}

\phantom{qqq}

\begin{center}
{\small \textbf{Fig. 1:} Relationships between the Besov and Sobolev  spaces involving only logarithmic smoothness in the case $q \geq \max\{p,2\}$.}
\end{center}

\bigskip

On the other hand, if $q \leq \min\{p,2\}$ then we will establish the following results (see Figure 2)
	\begin{itemize}
		\item the space $\mathbf{B}^{0,b}_{p,q}(\mathbb{R}^d)$ contains the space $B^{0,b+1/q}_{p,q}(\mathbb{R}^d)$. In addition, we have $\mathbf{B}^{0,b}_{p,q}(\mathbb{R}^d) \neq B^{0,b+1/q}_{p,q}(\mathbb{R}^d)$,
		\item both $H^{0,b+1/q}_p(\mathbb{R}^d)$ and $B^{0,b+1/\max\{2,p\}}_{p,q}(\mathbb{R}^d)$ contain the space $\mathbf{B}^{0,b}_{p,q}(\mathbb{R}^d)$,
		\item the spaces $H^{0,b+1/q}_p(\mathbb{R}^d)$ and $B^{0,b+1/\max\{2,p\}}_{p,q}(\mathbb{R}^d)$ are not comparable.
	\end{itemize}
	
	\bigskip

\begin{center}
\begin{tikzpicture}[fill opacity=0.05, ,xscale=0.8,yscale=0.8]

\fill (-1,0) ellipse (3.9 and 2.1);
\draw[fill opacity=0.2] (-1,0) ellipse (3.9 and 2.1);

\fill (1.55,0) ellipse (4.2 and 2.1);
\draw[fill opacity=0.2] (1.55,0) ellipse (4.2 and 2.1);

\fill[white] (-1.4,0) ellipse (2.6 and 1.1);
\draw[fill opacity=0.2]
(0,0) ellipse (2.6 and 1.1);

\fill[white] (1.4,0) circle (1.0 and 0.8);
\draw[fill opacity=0.2] (0,0) circle (1.0 and 0.8);

\node[fill opacity=2,xscale=0.8,yscale=0.8] at (0,0) {$B^{0,b + \frac{1}{q}}_{p,q}(\mathbb{R}^d)$};
\node[fill opacity=2,xscale=0.8,yscale=0.8] at (-1.8,0) {$\mathbf{B}^{0,b}_{p,q}(\mathbb{R}^d)$};
\node[fill opacity=2,xscale=0.8,yscale=0.8] at (4.3,0) {$B^{0,b + \frac{1}{\max\{2,p\}}}_{p,q}(\mathbb{R}^d)$};
\node[fill opacity=2,xscale=0.8,yscale=0.8] at (-3.7,0) {$H^{0,b + \frac{1}{q}}_p(\mathbb{R}^d)$};

\end{tikzpicture}
\end{center}

\phantom{qqq}

\begin{center}
{\small \textbf{Fig. 2:} Relationships between  the Besov and Sobolev spaces involving only logarithmic smoothness 
  in the case  $q \leq \min \{p,2\}$.}
\end{center}

An interesting remark is in order here. Sharp embeddings between function spaces with classical smoothness zero involve certain shifts in their logarithmic smoothness. This is in sharp contrast with the case of positive smoothness (see Proposition \ref{RecallEmb}).
 This phenomenon should be linked
  to the Littlewood-Paley theorem for $L_p(\mathbb{R}^d)$, that is,
\begin{equation}\label{intro:aux*}
	\|f\|_{L_p(\mathbb{R}^d)} \asymp \Big\|\Big(\sum_{j=0}^\infty |(\varphi_j \widehat{f})^\vee (\cdot)|^2\Big)^{1/2}\Big\|_{L_p(\mathbb{R}^d)}, \quad 1 < p < \infty,
\end{equation}
and the fact that the space $\mathbf{B}^{0,b}_{p,q}(\mathbb{R}^d)$ is very close to $L_p(\mathbb{R}^d)$, in the sense that $\mathbf{B}^{0,b}_{p,q}(\mathbb{R}^d)$ can be characterized as a limiting interpolation space between $L_p(\mathbb{R}^d)$ and any Sobolev space $H^s_p(\mathbb{R}^d)$ or Besov space $\mathbf{B}^s_{p,q}(\mathbb{R}^d)$ with $s > 0$. Further details will be given in Section \ref{Section:Interpolation methods}.

 We also investigate the Sobolev embeddings for the spaces $\mathbf{B}^{0,b}_{p,q}(\mathbb{R}^d)$, that is, the $\mathbf{B}^{0,b}_{p,q}(\mathbb{R}^d)$-counterparts of (\ref{e0}). Let $1 \leq p_0 < p < p_1 \leq \infty$ and $-\infty < s_1 < 0 < s_0 < \infty$ with
	\begin{equation*}
		s_0 - \frac{d}{p_0} = -\frac{d}{p} = s_1 - \frac{d}{p_1}.
	\end{equation*}
	Let also $0 < q \leq \infty$ and $b > -1/q$. Using (\ref{1}) and (\ref{e0}), we derive that
\begin{equation}\label{Intro:SobolevZeroSmoothness}
	B^{s_0,b+1/\min\{2,p,q\}}_{p_0,q}(\mathbb{R}^d) \hookrightarrow \mathbf{B}^{0,b}_{p,q}(\mathbb{R}^d) \hookrightarrow B^{s_1,b + 1/\max\{2,p,q\}}_{p_1,q}(\mathbb{R}^d).
\end{equation}
Notice that these embeddings are obtained as compositions of the embeddings (\ref{1}) and (\ref{e0}) which are both optimal (see (\ref{Goal3new3}) and (\ref{Goal3new4})).
However, we will show that the embeddings (\ref{Intro:SobolevZeroSmoothness}) are not the best possible and they can be improved applying Holmstedt's reiteration formulas for limiting interpolation spaces. Namely, we show that
 	\begin{equation}\label{Goal3new6}
		B^{s_0, b + 1/\min\{p,q\}}_{p_0,q}(\mathbb{R}^d) \hookrightarrow \mathbf{B}^{0,b}_{p,q}(\mathbb{R}^d) \hookrightarrow B^{s_1, b + 1/\max\{p,q\}}_{p_1,q}(\mathbb{R}^d).
	\end{equation}
	Moreover, we prove that these embeddings are sharp, that is,
	\begin{equation}\label{Goal3new7}
		B^{s_0,\xi}_{p_0,q}(\mathbb{R}^d)  \hookrightarrow \mathbf{B}^{0,b}_{p,q}(\mathbb{R}^d) \quad \text{if and only if}
 \quad \xi \geq b + 1/\min\{p,q\},
 \end{equation}
 and
  	\begin{equation}\label{Goal3new8}
		 \mathbf{B}^{0,b}_{p,q}(\mathbb{R}^d) \hookrightarrow B^{s_1,\xi}_{p_1,q}(\mathbb{R}^d) \quad \text{if and only if}
 \quad \xi \leq b + 1/\max\{p,q\}.
 \end{equation}
 The comparison between embeddings (\ref{1}) and (\ref{Goal3new6}) will also be addressed.


Concerning Propositions \ref{RecallEmb*} and \ref{RecallEmb**}, we will show the optimality of the conditions given there. For instance, the sharpness assertion for Proposition \ref{RecallEmb**} reads as follows: The embedding
	\begin{equation*}
		\mathbf{B}^{s_0, b_0}_{p,q_0}(\mathbb{R}^d) \hookrightarrow \mathbf{B}^{s_1, b_1}_{p,q_1}(\mathbb{R}^d)
	\end{equation*}
	holds if and only if one of the following conditions is valid
	\begin{enumerate}[\upshape(i)]
		\item $s_0 > s_1 \geq 0$,
		\item $s_0 = s_1 > 0, q_0 \leq q_1$, and $b_0 \geq b_1$,
		\item $s_0 = s_1 > 0, q_0 > q_1$, and $b_0 + \frac{1}{q_0} > b_1 + \frac{1}{q_1}$,
		\item $s_0 = s_1 = 0, b_0 + \frac{1}{q_0} > b_1 + \frac{1}{q_1}$,
		\item $s_0 = s_1 = 0, b_0 + \frac{1}{q_0} = b_1 + \frac{1}{q_1}$, and $q_0 \leq q_1$.
	\end{enumerate}
	
	Next we discuss the novelty of our approach
  to construct counterexamples which lead to sharpness of the embedding results. The usual procedure applied nowadays to establish sharp embeddings between function spaces relies on the decomposition methods developed by Frazier and Jawerth \cite{FrazierJawerth}. Using this approach, one can reduce complicated problems in function spaces to simpler problems in sequence spaces such as the mixed sequence space $\ell_q(\ell_p)$. See, e.g., the books by Triebel \cite{Triebel01} and Haroske \cite{HaroskeBook}. However, this strategy is not adequate to work with functions from the space $\mathbf{B}^{0,b}_{p,q}(\mathbb{R}^d)$. Indeed, the wavelet description of $\mathbf{B}^{0,b}_{p,q}(\mathbb{R}^d)$ was recently obtained in \cite{CobosDominguezTriebel}, but it involves truncated Littlewood-Paley-type constructions (cf. (\ref{intro:aux*})) which makes it rather difficult to construct counterexamples via the wavelet transform. Hence, it would be desirable to introduce some alternative and simple but effective criteria which allow us to decide whether a given function belongs to $\mathbf{B}, B, H$ and $F$ spaces. This was the starting point and main motivation to develop the approach proposed in this paper. It is related to the following general problem: to find sufficiently rich function classes such that their elements have certain limiting smoothness properties. Dealing with such function classes allows us not only to
completely characterize the function norms
 but also to refine known embeddings with respect to parameters.


Let us give two examples  to illustrate this approach. Firstly, regarding characterizations of function spaces, we recall
that $f\in H^{s}_p(\mathbb{R}^d)$ with $p=2$ if and only if $(1+|\xi|^{2})^{\frac s2} \widehat{f}(\xi)\in L_2(\mathbb{R}^d)$.
Such a characterization is possible only when $p=2$. To extend it for $p\ne 2$ requires additional conditions of  structural characteristics of functions. We will obtain a full characterization of the Sobolev norm as well as Besov $B$ and $\mathbf{B}$ norms in terms of behavior of the Fourier transforms
for $p\ne 2$ for functions such that their Fourier transforms are of monotone type or that can be represented as a lacunary series.

Secondly, we recall  the classical embedding between Sobolev and Besov  spaces (cf. Proposition \ref{RecallEmb}(i) with $r=2$): 
 \begin{equation}\label{Goal4}
 H^{s,b}_p(\mathbb{R}^d) \hookrightarrow B^{s,b}_{p,q}(\mathbb{R}^d) \quad \text{if and only if} \quad q \geq
    \max\{p,2\}.
  \end{equation}
      In many questions,  one needs to refine conditions on $q$ in this embedding, that is, to define classes $X$ and $Y$ such that
 \begin{equation*}
    X\cap H^{s,b}_p(\mathbb{R}^d) \hookrightarrow B^{s,b}_{p,q}(\mathbb{R}^d) \quad \text{if and only if} \quad q \geq p
    \end{equation*}
    and
    \begin{equation*}
     Y\cap H^{s,b}_p(\mathbb{R}^d) \hookrightarrow B^{s,b}_{p,q}(\mathbb{R}^d) \quad \text{if and only if} \quad q \geq 2.
     \end{equation*}
      A similar question can be asked for all the embeddings given in Proposition \ref{RecallEmb}, as well as those obtained in (\ref{Goal1}) and (\ref{Goal2}).
As a consequence, this approach gives us an opportunity to  prove necessary and sufficient conditions  for embeddings to hold in terms of involved smoothness and integrability parameters.

The second main goal of the paper is to obtain new characterizations of Besov spaces in terms of oscillations and Bianchini-type norms. The motivation comes from recent applications to the PDE. Bressan posed a question on mixing flows, which at the time of writing of this paper still remains open;
for a precise statement of this problem see \cite{Bressan}.
  An approach to Bressan's problem was proposed by Bianchini \cite{Bianchini} via the following norm of the characteristic functions $f = \chi_A$,
 \begin{equation}\label{introduction:Bianchini1}
 	 \|f\|_{L_1(\mathbb{R}^d)} + \int_0^1 \left\|f - B_t f \right\|_{L_1(\mathbb{R}^d)} \frac{dt}{t}.
 \end{equation}
 Here $B_t f(x)$ denotes the average of $f$ over the ball $B_t(x)$ (see precise definitions in Section \ref{balls1}). The key observation in \cite{Bianchini} is that (\ref{introduction:Bianchini1}) is a measure of mixing of $A$ satisfying
\begin{equation}\label{introduction:Bianchini2}
	\|\chi_A\|_{\mathbf{B}^{0,0}_{1,1}(\mathbb{R}^d)} \asymp \|\chi_A\|_{L_1(\mathbb{R}^d)} + \int_0^1 \left\|\chi_A - B_t \chi_A \right\|_{L_1(\mathbb{R}^d)} \frac{dt}{t}.
\end{equation}
Very recently, a remarkable contribution approaching Bressan's conjecture was obtained by Had$\check{z}$i\'c, Seeger, Smart and Street \cite{HadzicSeegerSmartStreet}. In particular, their approach strongly relies  on the fact that we deal with the space of zero smoothness.
   An alternative approach, which is based on Sobolev-type functionals involving only logarithmic smoothness, can be found in \cite{Leger}.

Let $s \geq 0, 1 \le p \le \infty,$ $0 < q \leq \infty$, and $-\infty < b < \infty$. Assume that $l \in \mathbb{N}$ such that $l > s/2$. Our goal is to show that
\begin{equation}\label{introduction:Bianchini3}
	\|f\|_{\mathbf{B}^{s,b}_{p,q}(\mathbb{R}^d)} \asymp  \|f\|_{L_p(\mathbb{R}^d)} + \left(\int_0^1 t^{-s q} (1 - \log t)^{b q}  \|f - B_{l,t} f\|_{L_p(\mathbb{R}^d)}^q \frac{dt}{t}\right)^{1/q}.
\end{equation}
Here, $B_{l,t}f$ is the $l$-th order average of $f$, with $B_{1,t}f = B_t f$ (see Section \ref{balls1}).

Some remarks are in order here: firstly, the characterization (\ref{introduction:Bianchini3}) extends (\ref{introduction:Bianchini2}) to any locally integrable function, to the full range of parameters $s,b,p,q$, and to higher order averages. Secondly, working with positive smoothness (i.e., $s > 0$), characterizations of smoothness spaces in terms of averages on balls have their roots in the works by  Wheeden \cite{Wheeden, Wheeden2}; see also \cite{AlabernMateuVerdera}. In particular, if $s > 0, b=0$ and $p > 1$ in (\ref{introduction:Bianchini3}) then we recover recent results by Dai, Gogatishvili, Yang, and Yuan \cite{DaiGogatishviliYangYuan}.

Another important step in Bianchini's paper \cite{Bianchini} is the following characterization of $\mathbf{B}^{0,0}_{1,1}(\mathbb{R}^d)$ in terms of
functions of bounded variation. Namely,
\begin{equation}\label{introduction:Bianchini4}
\|f\|_{\mathbf{B}^{0,0}_{1,1}(\mathbb{R}^d)} \asymp \int_0^1  \inf_{g \in \text{BV}(\mathbb{R}^d)} \Big\{\|f-g\|_{L_1(\mathbb{R}^d)}+t \|g\|_{\text{BV}(\mathbb{R}^d)} \Big\} \frac{dt}{t}.
\end{equation}
We aim to extend (\ref{introduction:Bianchini4}) for the general setting and give necessary and sufficient conditions on
 parameters so that the Besov norm is equivalent to various Bianchini type norms.
For the Besov space on $\mathbb{R}$ we obtain similar results involving the space of functions of bounded $p$-variation.

To address both types  of characterizations of Besov spaces (see (\ref{introduction:Bianchini3}) and (\ref{introduction:Bianchini4})), we shall apply (limiting) interpolation techniques. In particular, (\ref{introduction:Bianchini3}) is a special case of the following characterization of $\mathbf{B}^{s,b}_{p,q}(\mathbb{R}^d)$ in terms of the $K$-functional for the couple $(L_p(\mathbb{R}^d), \dot{\mathscr{L}}^{\alpha}_p(\mathbb{R}^d))$, where $\dot{\mathscr{L}}^{\alpha}_p(\mathbb{R}^d)$ is the Riesz potential space. Namely, we show that
\begin{equation}\label{introduction:Bianchini5}
\|f\|_{\mathbf{B}^{s,b}_{p,q}(\mathbb{R}^d)}
\asymp \|f\|_{L_p(\mathbb{R}^d)} + \left(\int_0^1 t^{-s q} (1-\log t)^{b q}
K(t^{\alpha},f;L_p(\mathbb{R}^d), \dot{\mathscr{L}}^{\alpha}_p(\mathbb{R}^d))^q
 \frac{dt}{t}\right)^{1/q}.
 \end{equation}
Here $K(t^\alpha,f;L_p(\mathbb{R}^d), \dot{\mathscr{L}}^\alpha_p(\mathbb{R}^d)) =
\inf_{\Delta^{\alpha/2} g \in L_p(\mathbb{R}^d)} \Big(\|f-g\|_{L_p(\mathbb{R}^d)} + t^\alpha \|\Delta^{\alpha/2} g\|_{L_p(\mathbb{R}^d)}\Big)$, $\alpha > s$, and
 $\Delta^{\alpha/2}$ is the $\frac{\alpha}{2}$-th power of Laplacian. Furthermore, applying realization results for $K(t^\alpha,f;L_p(\mathbb{R}^d), \dot{\mathscr{L}}^\alpha_p(\mathbb{R}^d))$, we will see that the characterization (\ref{introduction:Bianchini5}) in fact comprises several descriptions of Besov spaces in terms of approximation processes (Bochner-Riesz and Weierstrass means) and semi-groups (heat kernels and harmonic extensions).

Concerning the Bianchini-type characterization (\ref{introduction:Bianchini4}), we notice that it can be expressed as the limiting interpolation formula $\mathbf{B}^{0,0}_{1,1}(\mathbb{R}^d) = (L_1(\mathbb{R}^d), \text{BV}(\mathbb{R}^d))_{(0,0),1}$ and $\text{BV}(\mathbb{R}^d) = \text{Lip}^{(1,0)}_{1,\infty}(\mathbb{R}^d)$ (cf. Section \ref{Preliminaries} for precise definitions). Thus, we are able to compute 
 explicitly the $K$-functionals for the following couples:
\begin{itemize}
 \item[$\cdot$]
 $(L_p(\mathbb{R}^d), \mathbf{B}^{s,b}_{p,q}(\mathbb{R}^d)), \quad s \geq 0$,
  \item[$\cdot$] $(\mathbf{B}^{s,b}_{p,q}(\mathbb{R}^d), W^k_p(\mathbb{R}^d)),  \quad s \geq 0$,
 \item[$\cdot$] $(L_p(\mathbb{R}^d), \text{Lip}^{(k,-\alpha)}_{p,r}(\mathbb{R}^d))$,
  \item[$\cdot$] $(\mathbf{B}^{0,b}_{p,q}(\mathbb{R}^d), W^k_p(\mathbb{R}^d))$,
   \item[$\cdot$] $(\text{Lip}^{(k,-\alpha)}_{p,r}(\mathbb{R}^d), W^k_p(\mathbb{R}^d))$,
    \item[$\cdot$] $(\mathbf{B}^{0,b}_{p,q}(\mathbb{R}^d), \text{Lip}^{(k,-\alpha)}_{p,r}(\mathbb{R}^d))$,
    \end{itemize}
   in terms of the moduli of smoothness.
    This  complements the well-known result \cite{BennettSharpley}
    \begin{equation*}
		K(t^k,f;L_p(\mathbb{R}^d), W^k_p(\mathbb{R}^d)) \asymp  t^k  \|f\|_{L_p(\mathbb{R}^d)} + \omega_k(f,t)_p, \quad 0 < t < 1,
\end{equation*}
and allows us to completely characterize Besov spaces in terms of Bianchini-type norms.


Finally, we discuss  some applications  of our techniques based on sharp embeddings and limiting interpolation methods
to differential operators.
Recently there has been a great interest in the study of the fractional Laplace operator $(-\Delta)^s$ (or more generally, nonlocal differential operators) which is motivated by an increasing number of models involving $(-\Delta)^s$ in the description of physical phenomena where the scope of classical local operators (e.g., the classical Laplace operator $-\Delta$) is limited. For instance, it is used in elasticity \cite{DipierroPalatucciValdinoci}, mathematical finance \cite{Silvestre1, Silvestre, CaffarelliSalsaSilvestre}, fluid dynamics \cite{CaffarelliVasseur, KiselevNazarovVolberg}, image processing \cite{GilboaOsher}. In particular, the issue of regularity of $(-\Delta)^s$ is a central question in PDE as can be seen in the papers by Silvestre \cite{Silvestre1, Silvestre}, Caffarelli, Salsa and Silvestre \cite{CaffarelliSalsaSilvestre}, Ros-Oton and Serra \cite{RosOtonSerra}, and Grubb \cite{Grubb}.


In Section \ref{section-fr-lap} we obtain the following result.
Let $u$ be the unique solution to the Poisson-type equation
\begin{equation}\label{introduction:app}
	(-\Delta)^s u = f \quad \text{in} \quad \mathbb{R}^d.
\end{equation}
Let $\lambda > 0$. Then, we prove that
\begin{equation}\label{introduction:app1}
		\left(\int_0^t (\xi^{-2 s} \omega_{ \lambda + 2 s} (u, \xi)_p)^q \frac{d \xi}{\xi}\right)^{1/q} \lesssim \omega_\lambda(f,t)_p \iff q \geq \max\{p,2\}
	\end{equation}
	and
\begin{equation}\label{introduction:app2}
		\left(\int_0^t (\xi^{-2 s} \omega_{ \lambda + 2 s} (u, \xi)_p)^q \frac{d \xi}{\xi}\right)^{1/q} \gtrsim \omega_\lambda(f,t)_p \iff q \leq \min\{p,2\}.
	\end{equation}
	In particular, we have
\begin{equation}\label{introduction:app3}
		\left(\int_0^t (\xi^{-2 s} \omega_{ \lambda + 2 s} (u, \xi)_2)^2 \frac{d \xi}{\xi}\right)^{1/2} \asymp \omega_\lambda(f,t)_2.
	\end{equation}
These formulas are quantitative assertions for the smoothness properties of the solution $u$ and the datum $f$ of (\ref{introduction:app}). As a consequence, we derive the boundedness properties of $(-\Delta)^s$ on the Besov spaces $\mathbf{B}^{\lambda,b}_{p,q}(\mathbb{R}^d)$ and a quantitative version of the $L_p$-regularity result for the Dirichlet boundary problem
\begin{equation}\label{introduction:app4}
		 \left\{\begin{array}{cl}  (-\Delta)^s u = f & \text{in} \quad  \Omega, \\
		u= 0  & \text{on} \quad \mathbb{R}^d \setminus \Omega,
		       \end{array}
                        \right.
	\end{equation}
	where $\Omega \subset \mathbb{R}^d$ is a bounded open set, which has been recently investigated by Biccari, Warma and Zuazua \cite{BiccariWarmaZuazua, BiccariWarmaZuazua2}.

\subsection{Structure of the paper.}
The paper is organized as follows.

In Section \ref{Preliminaries}, we collect basic definitions and known  assertions for function spaces that we will use in the paper. We also provide
needed (limiting) interpolation results and Hardy-type inequalities.

 In Section \ref{section3}, 
 we obtain new embedding theorems for function spaces with smoothness close to zero, which complement Propositions \ref{RecallEmb} and \ref{RecallEmb2}.

 In Sections \ref{section4} and \ref{section5} we develop the core of our approach to establish sharp embeddings.
 More specifically, in Section  \ref{section4} we investigate  a special class of functions, $\widehat{GM}^d$, such that their Fourier transforms satisfy the general monotone condition. 
 This study has its routes in analysis of periodic functions having Fourier series with monotonically  decreasing coefficients. It is known that such functions have many important integrability and smoothness properties \cite{Bari, Zygmund}. Investigations of such questions for the non-periodic case  allow us, in particular,
 to fully characterize the weighted Lebesgue norm of a function in terms of the weighted  norm of its Fourier transform. This question is known as the Boas problem; for more detail  see \cite{Boas, GorbachevLiflyandTikhonov, GorbachevTikhonov, Sagher}.
 Let us stress  that in order to characterize smoothness spaces it is essential to deal with the general monotone condition rather than  just monotone. This is related to multiplier properties of $GM$ class, see Lemma \ref{Lemma 3.1} below.

 In  particular, we show in Section \ref{section4} that the Besov and Sobolev norms
of $\widehat{GM}^d$ functions can be fully described
  in terms of behavior of their Fourier transforms.   In its turn,  these descriptions allow us to prove that, for $X=\widehat{GM}^d$,
\begin{equation}\label{Goal4new}
    	X  \cap H^{s,b}_p(\mathbb{R}^d) \hookrightarrow
        B^{s,b}_{p,q}(\mathbb{R}^d) \quad \text{ if and only if } \quad q \geq p
\end{equation}
(cf. (\ref{Goal4})) and
   \begin{equation}\label{gm-emb}     X\cap H^{0,b+1/q}_p(\mathbb{R}^d) \hookrightarrow
        \mathbf{B}^{0,b}_{p,q}(\mathbb{R}^d) \quad \text{ if and only if } \quad q \geq p
  \end{equation}
 (cf. (\ref{Goal1}))
  and the inverse embeddings for $ q \leq p$.
    Moreover, our method allows us to refine the
embeddings (\ref{1})
\begin{equation}\label{Goal4new2}
X\cap  B^{0,b + 1/\min\{p,q\}}_{p,q}(\mathbb{R}^d) \hookrightarrow
    X\cap  \mathbf{B}^{0,b}_{p,q}(\mathbb{R}^d) \hookrightarrow
    X\cap  B^{0,b+1/\max\{p,q\}}_{p,q}(\mathbb{R}^d),
    \end{equation}
as well as embeddings from
   $B^{0,b}_{p,q}(\mathbb{R}^d)$ into Lebesgue spaces given in \cite{CaetanoLeopold, SickelTriebel}.

In this section we also investigate the diversity of Sobolev and Besov spaces. As can be seen in the book by Triebel \cite[2.3.9]{Triebel1}, this question was a central issue in the early development of the theory of function spaces. Let us recall that, in general, we have
\begin{equation}\label{Goal5}
	H^{s_0}_{p_0}(\mathbb{R}^d) =  B^{s_1}_{p_1,q_1}(\mathbb{R}^d) \quad \text{if and only if} \quad s_0=s_1 \quad \text{and} \quad p_0=p_1=q_1=2,
\end{equation}
and
\begin{equation}\label{Goal6}
	B^{s_0}_{p_0,q_0}(\mathbb{R}^d) = B^{s_1}_{p_1,q_1}(\mathbb{R}^d) \quad \text{if and only if} \quad s_0=s_1, p_0=p_1 \quad \text{and} \quad q_0=q_1.
\end{equation}
We prove here rather surprising results. Namely, we show that the class $X=\widehat{GM}^d$ has a sufficiently rich structure so that
 Besov and Sobolev norms of its elements are equivalent for a much wider range  of parameters than in the trivial cases given by (\ref{Goal5}) and (\ref{Goal6}). Among other results, we show that
   \begin{equation*}
     	X \cap H^{s,b}_p(\mathbb{R}^d) = X \cap B^{s,b}_{p,q}(\mathbb{R}^d) \quad \text{ if and only if } \quad q=p,
    \end{equation*}
      \begin{equation*}
     	 X \cap H^{0,b+1/q}_p(\mathbb{R}^d)  =
    X \cap \mathbf{B}^{0,b}_{p,q}(\mathbb{R}^d)  \quad \text{ if and only if } \quad q=p,
    \end{equation*}
          \begin{equation*}
     	 X \cap B^{0,b+1/q}_{p,q}(\mathbb{R}^d)  =
    X \cap \mathbf{B}^{0,b}_{p,q}(\mathbb{R}^d)  \quad \text{ if and only if } \quad q=p,
    \end{equation*}
 and, under a mild integrability condition near the origin,
		\begin{equation}\label{Goal9}
		X\cap B^{s_0,b}_{p_0,q}(\mathbb{R}^d) = X\cap  B^{s_1,b}_{p_1,q}(\mathbb{R}^d)  \quad \text{if and only if} 		\quad	s_0 -\frac{d}{p_0} = s_1 -\frac{d}{p_1},
		\end{equation}
		\begin{equation}\label{Goal10}
X\cap H^{s,b}_p(\mathbb{R}^d) = X\cap  B^{s_1,b}_{p_1,p}(\mathbb{R}^d) \quad \text{if and only if} \quad s -\frac{d}{p} = s_1 -\frac{d}{p_1},
\end{equation}
and
\begin{equation}\label{Goal11}
	X\cap  B^{s_0,b}_{p_0,p} (\mathbb{R}^d) =
X\cap H^{s,b}_p(\mathbb{R}^d)  \quad \text{if and only if} \quad s_0 - \frac{d}{p_0} = s -\frac{d}{p}.
\end{equation}
Here,
$1 < p_0 < p < p_1 < \infty$, $0 < q \leq \infty, -\infty < b < \infty$, and $-\infty < s_1 < s < s_0 < \infty$.
These formulas consist of refinements of (\ref{Goal5}) and (\ref{Goal6}). In particular, (\ref{Goal9}) sharpens the Sobolev embeddings (\ref{e0}), while (\ref{Goal10}) and (\ref{Goal11}) strengthen the Franke-Jawerth embeddings (\ref{e0*}).



In Section \ref{section5}, we consider the class $\mathfrak{L}$ of functions represented as  the product of an entire function and the lacunary trigonometric series. We prove similar results to those obtained in Section \ref{section4}. In particular, we prove that, for $Y =\mathfrak{L}$,

$$    	Y  \cap H^{s,b}_p(\mathbb{R}^d) \hookrightarrow
        {B}^{s,b}_{p,q}(\mathbb{R}^d) \quad \text{ if and only if } \quad q \geq 2
$$   and
  \begin{equation}\label{lac-emb}     Y \cap H^{0,b+1/q}_p(\mathbb{R}^d) \hookrightarrow
        \mathbf{B}^{0,b}_{p,q}(\mathbb{R}^d) \quad \text{ if and only if } \quad q \geq 2
\end{equation}
(cf. (\ref{Goal4new}) and (\ref{gm-emb})), and the corresponding inverses for $ q \leq 2$. Moreover, we obtain more general results on embeddings between Triebel-Lizorkin and Besov spaces both with classical and logarithmic smoothness.
 We also refine embedding (\ref{1}) as follows:
$$Y \cap  B^{0,b + 1/\min\{2,q\}}_{p,q}(\mathbb{R}^d) \hookrightarrow
    Y \cap  \mathbf{B}^{0,b}_{p,q}(\mathbb{R}^d) \hookrightarrow
    Y \cap  B^{0,b+1/\max\{2,q\}}_{p,q}(\mathbb{R}^d)$$
(cf. (\ref{Goal4new2})).

It is worth mentioning that our technique given in Sections \ref{section4} and \ref{section5} provides a relatively simple approach to construct functions satisfying limiting smoothness properties
with respect to
  behavior of their Besov and Sobolev norms. This method will be implemented in the subsequent Sections \ref{section-propositions}--\ref{section8}. Let us describe the contents of these sections in more detail.

In Section \ref{section-propositions}, we obtain necessary and sufficient conditions for the embeddings
$B^{s_0, b_0}_{p,q_0}(\mathbb{R}^d) \hookrightarrow B^{s_1, b_1}_{p,q_1}(\mathbb{R}^d)$
and
$\mathbf{B}^{s_0, b_0}_{p,q_0}(\mathbb{R}^d) \hookrightarrow \mathbf{B}^{s_1, b_1}_{p,q_1}(\mathbb{R}^d)$, cf. Propositions \ref{RecallEmb*} and \ref{RecallEmb**}.

Section \ref{SectionOptimalityWB} provides optimality of logarithmic smoothness for embeddings (\ref{gm-emb})
and (\ref{lac-emb}). We find necessary and sufficient conditions for
$H^{0,\xi}_p(\mathbb{R}^d) \hookrightarrow \mathbf{B}^{0,b}_{p,q}(\mathbb{R}^d)$,
$\max\{p,2\} \leq q$, (see (\ref{Goal3new})), as well as the corresponding reverse and
$H^{0,\xi}_p(\mathbb{R}^d) = \mathbf{B}^{0,b}_{p,q}(\mathbb{R}^d)$ (see (\ref{Goal3new2})).

In Section \ref{section7}, we find optimal embeddings between the following spaces involving only logarithmic smoothness: $H^{0,b+1/q}_p(\mathbb{R}^d)$, $\mathbf{B}^{0,b}_{p,q}(\mathbb{R}^d)$, and ${B}^{0,b+\tau}_{p,q}(\mathbb{R}^d)$ for various values of $\tau.$

Section \ref{section8} concerns with  optimality of embeddings between Besov spaces with logarithmic smoothness. We will show the sharpness of parameters in (\ref{1}) as well as in Sobolev embedding (\ref{Goal3new6}). More precisely, we prove (\ref{Goal3new3}), (\ref{Goal3new4}), (\ref{Goal3new7}) and (\ref{Goal3new8}). We also solve the question posed by Sickel, see (\ref{Goal3new5}). As already mentioned above, the proofs are partly based on results proved in Sections
\ref{section4} and \ref{section5}. Furthermore, the optimality of the fine index $q$ in embedding results creates additional obstacles. We are able to overcome these obstacles with the help of limiting interpolation techniques.

Let us also emphasize that our approach works in the periodic setting. To avoid  unnecessary complications in the presentation of results for periodic functions, we restrict our attention to the one-dimensional case.


Sections \ref{kfunctional} and \ref{section3.5} deal with characterizations of Besov spaces. More precisely, in Section \ref{kfunctional} we establish characterizations in terms of ball averages, approximation processes, and semi-groups; whereas the main concern in Section \ref{section3.5} is Bianchini-type characterizations.

Section \ref{section9} studies  the smoothness properties of functions  and their derivatives.
It is known  \cite[p. 343]{BennettSharpley} that a function $f
 \in\mathbf{B}^{s,b}_{p,q}(\mathbb{R}^d)$ if and only if its derivatives $D^\beta f, |\beta| \leq k,$ belong to $\mathbf{B}^{s-k,b}_{p,q}(\mathbb{R}^d)$ provided that $s>k$.
 We deal with limiting case $s=k$.
Moreover, we obtain new sharp inequalities for moduli of smoothness of derivatives in terms of moduli of smoothness of functions themselves.

Section \ref{Lifting Estimates} deals with  boundedness properties of the Bessel potential operator $I_\sigma = (\text{id} - \Delta)^{\sigma/2}$.
We establish sharp inequalities for moduli of smoothness of $I_\sigma f$ in terms of the moduli of smoothness of $f$. Further, we obtain the following sharp results:
$
I_\sigma: X^{s,b_0}_{p,q}(\mathbb{R}^d) \longrightarrow Y^{s-\sigma, b_1}_{p,q}(\mathbb{R}^d),
$
where $X, Y$ is one of the Besov spaces $\mathbf{B}$ or $B$ and $s = 0, \sigma$. Quite surprisingly, $I_\sigma$ does not act as a lift operator for $\mathbf{B}^{0,b}_{p,q}(\mathbb{R}^d)$. This is in sharp contrast with the case of $\mathbf{B}^{s,b}_{p,q}(\mathbb{R}^d), s > 0,$ and $B^{s,b}_{p,q}(\mathbb{R}^d), -\infty < s < \infty$, where it is well known that $I_\sigma$ is an isomorphism.

In Section \ref{section-fr-lap} we first clarify the relationship between Besov and Lipschitz spaces introduced in Section \ref{Preliminaries} and
those that are widely used in the study of regularity properties  of differential operators such as H\"older, Zygmund and $\Lambda$-spaces.
Second, we study regularity estimates of the fractional Laplace operator, see
(\ref{introduction:app})--(\ref{introduction:app4}).

Finally, in Appendix \ref{A} we collect the notations used in this paper.

\subsection*{Acknowledgement}
The authors would like to thank Hans Triebel for his comments on a preliminary version of the paper.

The first author has been supported in part by MTM2017-84058-P (AEI/FEDER, UE) and by FCT Grant SFRH/BPD/121838/2016. The second
author was partially supported by MTM 2017-87409-P, 2017 SGR 358, and by the
CERCA Programme of the Generalitat de Catalunya.

\newpage
\section{Preliminaries}\label{Preliminaries}

\subsection{General notation}
As usual, $\mathbb{R}^d$ \index{\bigskip\textbf{Sets}!$\mathbb{R}^d$}\label{SETR} denotes the $d$-dimensional real Euclidean space, $\mathbb{T}$\index{\bigskip\textbf{Sets}!$\mathbb{T}$}\label{SETT}  is an unit circle, $\mathbb{Z}$ is the collection of all integers, $\mathbb{Z}^d$ \index{\bigskip\textbf{Sets}!$\mathbb{Z}^d$}\label{SETZ} is the lattice of all points in $\mathbb{R}^d$ with integer-valued components, $\mathbb{N}$\index{\bigskip\textbf{Sets}!$\mathbb{N}$}\label{SETN} is the collection of all natural numbers, and $\mathbb{N}_0 = \mathbb{N} \cup \{0\}$\index{\bigskip\textbf{Sets}!$\mathbb{N}_0$}\label{SETN0}.

Given two (quasi-) Banach spaces $X$ and $Y$, we write $X \hookrightarrow Y$\index{\bigskip\textbf{Numbers, relations}!$X \hookrightarrow Y$}\label{XY}  if $X \subset Y$ and the natural embedding from $X$ into $Y$ is continuous. By $\text{id}_X$\index{\bigskip\textbf{Operators}!$\text{id}_X$}\label{ID} we mean the identity map in $X$.

Let $|\cdot|_d$ stand for the $d$-dimensional Lebesgue measure. For $E \subset \mathbb{R}^d$, we denote by $\chi_E$\index{\bigskip\textbf{Functionals and functions}!$\chi_E$}\label{INDICATOR} the characteristic function of $E$.

For $a \in \mathbb{R}$, let $[a] = \max \{k \in \mathbb{Z}: k \leq a\}$\index{\bigskip\textbf{Numbers, relations}!$[a]$}\label{[a]}. For $1 \leq p \leq \infty$, we denote by $p'$\index{\bigskip\textbf{Numbers, relations}!$p'$}\label{p'} the dual exponent of $p$ given by $\frac{1}{p} + \frac{1}{p'} = 1$.

We will assume that $A\lesssim B$\index{\bigskip\textbf{Numbers, relations}!$A\lesssim B$}\label{AB} means that $A\leq C B$ with a positive constant $C$ depending only on nonessential parameters.
If $A\lesssim B\lesssim A$, then $A\asymp B$\index{\bigskip\textbf{Numbers, relations}!$A\asymp B$}\label{ASYMP}.

\subsection{Function spaces}\label{Section 2.2}
Let $1 \leq p
\leq \infty, 0 < q \leq \infty, - \infty < b < \infty$ and $s \geq
0$. The Besov space $\mathbf{B}^{s,b}_{p,q}(\mathbb{R}^d)$\index{\bigskip\textbf{Spaces}!$\mathbf{B}^{s,b}_{p,q}(\mathbb{R}^d)$}\label{BESOVDIFF} is formed by all $f \in L_p(\mathbb{R}^d)$
for which
\begin{equation}\label{norm1}
	|f|_{\mathbf{B}^{s,b}_{p,q}(\mathbb{R}^d)} =  \left(\int_0^1 (t^{-s} (1 - \log t)^b \omega_k(f,t)_p)^q
    \frac{dt}{t}\right)^{1/q} < \infty
\end{equation}
(with the usual modification if $q=\infty$).
Here $k \in \mathbb{N}$ with $k > s$ and $\omega_k(f,t)_p$\index{\bigskip\textbf{Functionals and functions}!$\omega_k(f,t)_p$}\label{MODK} is the $k$-th order modulus of smoothness of $f,$ given by
\begin{equation}\label{modulus}
	\omega_k(f,t)_p = \sup_{|h| \leq t} \|\Delta^k_h f\|_{L_p(\mathbb{R}^d)},\quad  t > 0,
\end{equation}
where \index{\bigskip\textbf{Functionals and functions}!$\Delta^k_h$}\label{DELTA}
\begin{equation}\label{differences}
 \Delta^1_h f (x) = f(x + h)-f(x), \quad \Delta^{k+1}_h = \Delta^1_h (\Delta^k_h),\qquad x, h \in \mathbb{R}^d.
 \end{equation}
The Besov space $\mathbf{B}^{s,b}_{p,q}(\mathbb{R}^d)$ becomes a quasi-Banach space equipped with
\begin{equation}\label{norm}
	\|f\|_{\mathbf{B}^{s,b}_{p,q}(\mathbb{R}^d)} =
\|f\|_{L_p(\mathbb{R}^d)} + |f|_{\mathbf{B}^{s,b}_{p,q}(\mathbb{R}^d)}.
\end{equation}
If $b=0$ we recover the classical Besov spaces $\mathbf{B}^s_{p,q}(\mathbb{R}^d)$. Note that if $s=0$ in $\mathbf{B}^{s,b}_{p,q}(\mathbb{R}^d)$ the case of interest is when $b \geq -1/q \, (b > 0 \text{ if } q=\infty)$. Otherwise, it is easy to see that $\mathbf{B}^{0,b}_{p,q}(\mathbb{R}^d)=L_p(\mathbb{R}^d)$ and $\|f\|_{\mathbf{B}^{0,b}_{p,q}(\mathbb{R}^d)} \asymp \|f\|_{L_p(\mathbb{R}^d)}$.

For $k \in \mathbb{N}$, the logarithmic Lipschitz space $\text{Lip}^{(k,-b)}_{p,q}(\mathbb{R}^d)$\index{\bigskip\textbf{Spaces}!$\text{Lip}^{(k,-b)}_{p,q}(\mathbb{R}^d)$}\label{LOGLIPSCHITZ}  consists of all $f \in L_p(\mathbb{R}^d)$ for which
\begin{equation}\label{Lipschitz}
	\|f\|_{\text{Lip}^{(k,-b)}_{p,q}(\mathbb{R}^d)} = \|f\|_{L_p(\mathbb{R}^d)} + \left(\int_0^1(t^{-k} (1 - \log t)^{-b} \omega_k(f,t)_p)^q \frac{dt}{t}\right)^{1/q} < \infty
\end{equation}
(with the usual modification if $q=\infty$). In order to avoid trivial spaces, we  assume that $b > 1/q$ if $q < \infty$ ($b \geq 0$ if $q=\infty$). If $k=1$ we obtain the spaces $\text{Lip}^{(1,-b)}_{p,q}(\mathbb{R}^d)$ studied in detail in \cite{Haroske, EdmundsEvans, HaroskeBook}. 

Some distinguished elements within the class $\text{Lip}^{(k,-b)}_{p,q}(\mathbb{R}^d)$ are the following: Let $q= \infty$ and $b=0$. If $p = \infty$ and $k=1$ then the space $\text{Lip}^{(1,0)}_{\infty,\infty}(\mathbb{R}^d)$ coincides with the classical Lipschitz space $\text{Lip}(\mathbb{R}^d)$\index{\bigskip\textbf{Spaces}!$\text{Lip}(\mathbb{R}^d)$}\label{LIPSCHITZ}, which is endowed with the norm
 	\begin{equation*}
		\|f\|_{\text{Lip}(\mathbb{R}^d)} = \|f\|_{L_\infty(\mathbb{R}^d)} + \sup_{\substack{x,y \in \mathbb{R}^d \\ 0 < |x-y| < 1/2}} \frac{|f(x)-f(y)|}{|x-y|}.
	\end{equation*}
If $p=1$ and $k=1$ in $\text{Lip}^{(k,0)}_{p,\infty}(\mathbb{R}^d)$, then we recover the well-known space $\text{BV}(\mathbb{R}^d)$\index{\bigskip\textbf{Spaces}!$\text{BV}(\mathbb{R}^d)$}\label{BV} formed by all integrable functions of bounded variation (see, e.g., \cite[p. 245]{CohenDahmenDaubechiesDeVore}), i.e.,
\begin{equation}\label{def:BV}
\text{BV}(\mathbb{R}^d)=\{f\in L_1(\mathbb{R}^d):\; \omega_1(f,t)_1=O(t)\}.
\end{equation}
If $1 < p < \infty$ and $k \in \mathbb{N}$ in $\text{Lip}^{(k,0)}_{p,\infty}(\mathbb{R}^d)$, then we arrive at the classical Sobolev spaces $W^k_p(\mathbb{R}^d)$ (see (\ref{SobolevSpacesDefinition}) below). Further information may be found in \cite{T10, T11, HaroskeTriebel11, HaroskeTriebel13}; see also \cite[Section 4.2.1]{Triebel13}.

By $\mathbf{B}^{s,b}_{p,q}(\mathbb{T})$\index{\bigskip\textbf{Spaces}!$\mathbf{B}^{s,b}_{p,q}(\mathbb{T})$}\label{BESOVDIFFPER} and $\text{Lip}^{(k,-b)}_{p,q}(\mathbb{T})$\index{\bigskip\textbf{Spaces}!$\text{Lip}^{(k,-b)}_{p,q}(\mathbb{T})$}\label{LOGLIPSCHITZPER} we mean the periodic counterparts which are defined by replacing the $L_p(\mathbb{R}^d)$-norm in (\ref{norm1}), (\ref{modulus}), (\ref{norm}), and (\ref{Lipschitz}) by the $L_p(\mathbb{T})$-norm.

As we mentioned above, Besov spaces of generalized smoothness can be also introduced via the
Fourier-analytical approach; see, e.g., \cite{Triebel1, SchmeisserTriebel, CobosFernandez, Goldman, Moura, FarkasLeopold, CaetanoLeopold}.
Let $\mathcal{D}(\mathbb{T})$ be the set of all complex-valued infinitely differentiable functions on $\mathbb{T}$. So, if $f \in \mathcal{D}(\mathbb{T})$ and $x,y \in \mathbb{T}$ with $x - y = 2 k \pi, k \in \mathbb{Z}$, then $f(x)=f(y)$. The topology in $\mathcal{D}(\mathbb{T})$\index{\bigskip\textbf{Spaces}!$\mathcal{D}(\mathbb{T})$}\label{D} is defined by the family of semi-norms $\|f\|_\alpha = \sup \{|D^\alpha f (x)| : x \in \mathbb{T}\}$ where $\alpha$ is any non-negative integer. We write $\mathcal{D}'(\mathbb{T})$\index{\bigskip\textbf{Spaces}!$\mathcal{D}'(\mathbb{T})$}\label{D'} for the topological dual of $\mathcal{D}(\mathbb{T})$. Let $\mathcal{S}(\mathbb{R}^d)$\index{\bigskip\textbf{Spaces}!$\mathcal{S}(\mathbb{R}^d)$}\label{S} and $\mathcal{S}'(\mathbb{R}^d)$\index{\bigskip\textbf{Spaces}!$\mathcal{S}'(\mathbb{R}^d)$}\label{S'} be the Schwartz space of all complex-valued rapidly decreasing infinitely differentiable functions on $\mathbb{R}^d$, and the space of tempered distributions on $\mathbb{R}^d$, respectively. The Fourier transform of $f \in \mathcal{S}(\mathbb{R}^d)$ is defined by \index{\bigskip\textbf{Operators}!$\widehat{f}$}\label{FT}
\begin{equation*}
	\widehat{f}(\xi) = (2 \pi)^{-d/2} \int_{\mathbb{R}^d}  f(x) e^{-i x \cdot \xi} d x, \quad \xi \in \mathbb{R}^d,
\end{equation*}
where $x \cdot \xi = \sum_{j=1}^d x_j \xi_j$, and the inverse Fourier transform is given by \index{\bigskip\textbf{Operators}!$f^\vee$}\label{IFT}
\begin{equation*}
	f^\vee(\xi) = (2 \pi)^{-d/2} \int_{\mathbb{R}^d} f(x)  e^{i x \cdot \xi} d x.
\end{equation*}
These operators are extended to $\mathcal{S}'(\mathbb{R}^d)$ in the usual way.

Take $\varphi_0 \in \mathcal{S}(\mathbb{R}^d)$ such that
\begin{equation}\label{SmoothFunction}
	\text{supp } \varphi_0 \subset \{x \in \mathbb{R}^d : |x| \leq 2\} \text{ and } \varphi_0(x) = 1 \text{ if } |x| \leq 1.
\end{equation}
For $j \in \mathbb{N}$ and $x \in \mathbb{R}^d$ let
\begin{equation}\label{resolution}
	\varphi_j(x) = \varphi_0(2^{-j}x) - \varphi_0(2^{-j+1}x).
\end{equation}
Then the sequence $\{\varphi_j\}_{j \in \mathbb{N}_0}$ forms a smooth dyadic resolution of unity in $\mathbb{R}^d$, $\sum_{j=0}^\infty \varphi_j (x) = 1$ for all $x \in \mathbb{R}^d$.

Let $1 \leq p \leq \infty, 0 < q \leq \infty$, and $-\infty < s, b < \infty$. The Besov space $B^{s,b}_{p,q}(\mathbb{R}^d)$\index{\bigskip\textbf{Spaces}!$B^{s,b}_{p,q}(\mathbb{R}^d)$}\label{BESOVF} consists
of all $f \in \mathcal{S}'(\mathbb{R}^d)$ such that
\begin{equation*}
    \|f\|_{B^{s,b}_{p,q}(\mathbb{R}^d)} = \Bigg(\sum_{j=0}^\infty \Big(2^{j s} (1 + j)^b \|(\varphi_j
    \widehat{f})^\vee\|_{L_p(\mathbb{R}^d)}\Big)^q\Bigg)^{1/q} < \infty
\end{equation*}
(the sum should be replaced by the supremum if $q=\infty$). We should remark that the spaces $B^{s,b}_{p,q}(\mathbb{R}^d)$ are non-trivial spaces formed by distributions for any $s \in \mathbb{R}$, which are independent of the choice of $\varphi_0$ satisfying (\ref{SmoothFunction}), in the sense of equivalent quasi-norms.

Analogously, one introduces the periodic Besov space $B^{s,b}_{p,q}(\mathbb{T})$\index{\bigskip\textbf{Spaces}!$B^{s,b}_{p,q}(\mathbb{T})$}\label{BESOVFPER} as the set of all $f \in \mathcal{D}'(\mathbb{T})$ which have a finite quasi-norm
\begin{equation}\label{BesovPerDef}
    \|f\|_{B^{s,b}_{p,q}(\mathbb{T})} = \Bigg(\sum_{j=0}^\infty \Big(2^{j s} (1 + j)^b \|(\varphi_j
    \widehat{f})^\vee\|_{L_p(\mathbb{T})}\Big)^q\Bigg)^{1/q}.
\end{equation}
In this setting, it is worthwhile mentioning that $(\varphi_j \widehat{f})^\vee$ is equal to the trigonometric polynomial  of degree at most $2^{j+1}$   given by $\sum_{k = -\infty}^\infty \varphi_j(k) \widehat{f}(k) e^{i k x}$, where $\widehat{f}(k)$ are the Fourier coefficients of $f \in \mathcal{D}'(\mathbb{T})$. In particular, if $f \in L_p(\mathbb{T})$ then
\begin{equation}\label{FourierCoef}
	\widehat{f}(k) = \frac{1}{2 \pi} \int_{-\pi}^\pi f(x) e^{-i k x} d x.
\end{equation}
Hence, we can rewrite (\ref{BesovPerDef}) as
\begin{equation}\label{Section2:new}
	\|f\|_{B^{s,b}_{p,q}(\mathbb{T})} = \left(\sum_{j=0}^\infty 2^{j s q} (1 + j)^{b q} \Big\|\sum_{k=-\infty}^\infty \varphi_j(k) \widehat{f}(k) e^{i k x}\Big\|_{L_p(\mathbb{T})}^q\right)^{1/q}.
\end{equation}

Note that if $b=0$ in $B^{s,b}_{p,q}(\mathbb{R}^d)$ or $B^{s,b}_{p,q}(\mathbb{T})$ we get the classical spaces $B^{s}_{p,q}(\mathbb{R}^d)$ or $B^{s}_{p,q}(\mathbb{T})$, respectively. These spaces have been extensively studied in the literature. For further details and properties, we refer to the books \cite{Nikolskii}, \cite{Peetre}, \cite{Triebel1}, and \cite{DeVoreLorentz}.

 For $1 < p < \infty, 0 < q \leq \infty$, and $-\infty < s, b < \infty$,
 the Triebel-Lizorkin space $F^{s,b}_{p,q}(\mathbb{R}^d)$\index{\bigskip\textbf{Spaces}!$F^{s,b}_{p,q}(\mathbb{R}^d)$}\label{TL} is formed by all $f \in
\mathcal{S}'(\mathbb{R}^d)$ for which
\begin{equation*}
    \|f\|_{F^{s,b}_{p,q}(\mathbb{R}^d)} = \Big\|\Big(\sum_{j=0}^\infty (2^{js} (1 + j)^b |(\varphi_j \widehat{f})^\vee
    (\cdot)|)^q\Big)^{1/q}\Big\|_{L_p(\mathbb{R}^d)} < \infty
\end{equation*}
(with the usual modification if $q=\infty$). This definition is independent of the choice of $\varphi_0$ satisfying (\ref{SmoothFunction}).

The Sobolev space
$H^{s,b}_p(\mathbb{R}^d)$\index{\bigskip\textbf{Spaces}!$H^{s,b}_p(\mathbb{R}^d)$}\label{SOB} consists of all $f \in \mathcal{S}'(\mathbb{R}^d)$ such that
\begin{equation}\label{def:SobolevSpace}
    \|f\|_{H^{s,b}_p(\mathbb{R}^d)} = \|((1 + |x|^2)^{s/2} (1 + \log (1 + |x|^2))^b \widehat{f})^\vee\|_{L_p(\mathbb{R}^d)}
\end{equation}
is finite. These spaces are also called generalized Bessel potential spaces. See \cite{KalyabinLizorkin, OpicTrebels, FarkasLeopold}.

The periodic counterparts $F^{s,b}_{p,q}(\mathbb{T})$\index{\bigskip\textbf{Spaces}!$F^{s,b}_{p,q}(\mathbb{T})$}\label{TLPER} and $H^{s,b}_p(\mathbb{T})$\index{\bigskip\textbf{Spaces}!$H^{s,b}_p(\mathbb{T})$}\label{SOBPER}  are formed by all $f \in \mathcal{D}'(\mathbb{T})$ with the Fourier series
\begin{equation*}
	f (x) = \sum_{k=-\infty}^\infty \widehat{f}(k) e^{i k x} \text{ (convergence in $\mathcal{D}'(\mathbb{T})$)}
\end{equation*}
such that
\begin{equation}\label{Section 2.1: New2}
	\|f\|_{F^{s,b}_{p,q}(\mathbb{T})} = \Big\|\Bigg(\sum_{j=0}^\infty \Big(2^{j s} (1 + j)^{b} \Big|\sum_{k=-\infty}^\infty \varphi_j(k) \widehat{f}(k) e^{i k x} \Big|\Big)^q \Bigg)^{1/q}\Big\|_{L_p(\mathbb{T})} < \infty
\end{equation}
and
\begin{equation*}
	\|f\|_{H^{s,b}_p(\mathbb{T})} = \Big\|\sum_{k=-\infty}^\infty (1 + |k|^2)^{s/2} (1 + \log (1 + |k|^2))^b \widehat{f}(k) e^{i k x} \Big\|_{L_p(\mathbb{T})} < \infty,
\end{equation*}
respectively.

In the case  $b=0$ in $F^{s,b}_{p,q}(\mathbb{R}^d)$ (respectively, $F^{s,b}_{p,q}(\mathbb{T})$) we obtain the spaces $F^s_{p,q}(\mathbb{R}^d)$ (respectively, $F^s_{p,q}(\mathbb{T})$) which are studied systematically in \cite{Triebel, Triebel1, SchmeisserTriebel}.
If $b=0$ in $H^{s,b}_p(\mathbb{R}^d)$ (respectively, $H^{s,b}_p(\mathbb{T})$) we get the well-known (fractional) Sobolev spaces $H^s_p(\mathbb{R}^d)$ (respectively, periodic (fractional) Sobolev spaces $H^s_p(\mathbb{T})$). In particular, if $s=k \in \mathbb{N}_0$ and $1 < p < \infty$, we obtain the classical Sobolev spaces $W^k_p(\mathbb{R}^d)$ (respectively, periodic Sobolev spaces $W^k_p(\mathbb{T})$\index{\bigskip\textbf{Spaces}!$W^k_p(\mathbb{T})$}\label{SOBPERCLAS}) formed by
$L_p$-integrable (respectively, periodic) functions whose (weak) derivatives of order less
than or equal to $k$ are also $L_p$-integrable. See \cite[2.3.3]{Triebel} and \cite[Theorem 3.5.4, p. 169]{SchmeisserTriebel}. We equip the space $W^k_p(\mathbb{R}^d), 1 \leq p \leq \infty$\index{\bigskip\textbf{Spaces}!$W^k_p(\mathbb{R}^d)$}\label{SOBCLAS}, with the norm
\begin{equation}\label{SobolevSpacesDefinition}
    \|f\|_{W^{k}_p(\mathbb{R}^d)} = \sum_{|\alpha| \leq k} \|D^\alpha f\|_{L_p(\mathbb{R}^d)}.
\end{equation}
Here, $D^\alpha$ \index{\bigskip\textbf{Operators}!$D^\alpha$}\label{DERIVATIVE} denotes the differential operator $\frac{\partial^{|\alpha|}}{\partial x_1^{\alpha_1} \cdots \partial x_d^{\alpha_d}}$ with multiindex $\alpha$, $|\alpha| = \sum_{i=1}^d \alpha_i$.

Analogously, we introduce the norm $\|f\|_{W^k_p(\mathbb{T})}$ for $f \in W^k_p(\mathbb{T})$.

We also mention a partial result of (\ref{LPgeneral}), the Littlewood-Paley theorem,  which claims that
\begin{equation}\label{LP}
	\|f\|_{L_p(\mathbb{R}^d)} \asymp \Big\|\Big(\sum_{j=0}^\infty |(\varphi_j \widehat{f})^\vee (\cdot)|^2\Big)^{1/2}\Big\|_{L_p(\mathbb{R}^d)}, \quad 1 < p < \infty.
\end{equation}

The periodic counterpart of (\ref{LPgeneral}) also holds true. Namely, for
$1 < p < \infty$ and $s,b \in \mathbb{R}$, we have
	\begin{equation*}
	F^{s,b}_{p,2}(\mathbb{T}) = H^{s,b}_p(\mathbb{T}).
\end{equation*}
 	In the case $s=b=0$ this can be found in \cite[3.5.4]{SchmeisserTriebel}.
 In the general case one can use the lifting operator $\mathfrak{J}_{s,b}$\index{\bigskip\textbf{Operators}!$\mathfrak{J}_{s,b}$}\label{LIFTPER}.
Set, for any $s, b \in \mathbb{R}$,
	\begin{equation*}
		\mathfrak{J}_{s,b} : f \rightarrow \sum_{k=-\infty}^\infty (1 + |k|^2)^{s/2} (1 + \log (1 + |k|^2))^b \widehat{f}(k) e^{i k x}.
	\end{equation*}
	It is clear that $\mathfrak{J}_{s,b}$ is an isometry from $H^{s,b}_p(\mathbb{T})$ onto $L_p(\mathbb{T})$. Moreover,  $\mathfrak{J}_{s,b}$ maps $F^{s,b}_{p,2}(\mathbb{T})$ isomorphically onto $F^0_{p,2}(\mathbb{T})$. To show this, one can follow
 the same argument as in the proof of
  \cite[Theorem 2.3.8]{Triebel1}.

\subsection{Interpolation methods}\label{Section:Interpolation methods}
Let $A_0$ and  $A_1$ be quasi-Banach spaces with $A_1 \hookrightarrow A_0$. For every $a \in A_0$, the Peetre's $K$-functional \index{\bigskip\textbf{Functionals and functions}!$K(t,a)$}\label{K} is defined by
\begin{equation}\label{Peetre}
	K(t,a) = K(t, a; A_0, A_1) = \inf \{\|a_0\|_{A_0} + t \|a_1\|_{A_1} : a =a_0 + a_1, a_j \in A_j \}, \quad  t > 0.
\end{equation}
In some cases, the $K$-functional is defined by using a quasi-semi-norm for $A_1$. In this paper, we will deal with the following three interpolation methods.

First, for $0 < \theta < 1$ and $0 < q \leq \infty$, the real interpolation space $(A_0,A_1)_{\theta,q}$\index{\bigskip\textbf{Spaces}!$(A_0,A_1)_{\theta,q}$}\label{REAL} consists of all those $a \in A_0$ having a finite quasi-norm
\begin{equation}\label{def:IntSpace}
	\|a\|_{(A_0,A_1)_{\theta,q}} = \left(\int_0^\infty (t^{-\theta}  K(t,a))^q \frac{dt}{t}\right)^{1/q}
\end{equation}
(with the usual modification if $q=\infty$). See \cite{BennettSharpley, BerghLofstrom, Triebel}.

Second, for $\mathbb{A}=(\alpha_0,\alpha_\infty) \in \mathbb{R}^2$, let $\ell (t) = 1 + |\log t|$\index{\bigskip\textbf{Functionals and functions}!$\ell (t)$}\label{LOG} and
\begin{equation}\label{brokenlog}
    \ell^{\mathbb{A}}(t) = \left\{\begin{array}{lcl}
                                \ell^{\alpha_0}(t) & \text{ for } & t \in (0,1], \\
                                             &              &              \\
                                \ell^{\alpha_\infty}(t) & \text{ for } & t \in
                                (1,\infty).
                                \end{array}
                        \right.
\end{equation}
 The logarithmic interpolation space $(A_0,A_1)_{\theta,q;\mathbb{A}}$\index{\bigskip\textbf{Spaces}!$(A_0,A_1)_{\theta,q;\mathbb{A}}$}\label{REALLOG} is formed by all $a \in A_0$ such that
 \begin{equation*}
 	\|a\|_{(A_0,A_1)_{\theta,q;\mathbb{A}}} = \left(\int_0^\infty (t^{-\theta} \ell^{\mathbb{A}}(t)  K(t,a))^q \frac{dt}{t}\right)^{1/q} < \infty.
 \end{equation*}
See \cite{EvansOpic, EvansOpicPick, GogatishviliOpicTrebels}. Under suitable assumptions on $\mathbb{A}$ and $q$, these spaces are well-defined even if $\theta=0$ or $\theta=1$. If $\alpha_0=\alpha_\infty = \alpha$, we use the notation $(A_0,A_1)_{\theta,q;\alpha}$, that is,
\begin{equation}\label{inter*}
	\|a\|_{(A_0,A_1)_{\theta,q;\alpha}} =  \left(\int_0^\infty (t^{-\theta} \ell^\alpha(t) K(t,a))^q \frac{dt}{t}\right)^{1/q}.
\end{equation}
Note that if $\alpha = 0$ in $(A_0,A_1)_{\theta,q;\alpha}$ we recover the classical spaces $(A_0,A_1)_{\theta,q}$ (see (\ref{def:IntSpace})).

Before going further, let us make an easy observation. Assume that $0 < \theta < 1, 0 < q \leq \infty$, and $\alpha \in \mathbb{R}$. Since $A_1 \hookrightarrow A_0$, we have
\begin{equation}\label{inter}
	\|a\|_{(A_0,A_1)_{\theta,q;\alpha}} \asymp  \left(\int_0^1 (t^{-\theta} \ell^\alpha(t) K(t,a))^q \frac{dt}{t}\right)^{1/q}.
\end{equation}

Third, we introduce the limiting cases of the right-hand side space of (\ref{inter}). For $\theta=1$ or $\theta=0$ and $\alpha \in \mathbb{R}$, the space $(A_0,A_1)_{(\theta,\alpha),q}$\index{\bigskip\textbf{Spaces}!$(A_0,A_1)_{(\theta,\alpha),q}$}\label{REALLIM} consists of all those $a \in A_0$ with
\begin{equation}\label{limitinginterpolation}
	\|a\|_{(A_0,A_1)_{(\theta,\alpha),q}} = \left(\int_0^1 (t^{-\theta} \ell^\alpha(t) K(t,a))^q \frac{dt}{t}\right)^{1/q} < \infty
\end{equation}
(see \cite{CobosFernandezCabreraKuhnUllrich, CobosDominguez2}). To avoid trivial spaces we shall assume that $\alpha < -1/q$ if $q < \infty$ ($\alpha \leq 0$ if $q=\infty$) when $\theta = 1$. In the case $\theta = 0$ we are interested in $\alpha \geq -1/q$ if $q < \infty$ ($\alpha > 0$ if $q=\infty$). Otherwise, it is plain to check that $(A_0,A_1)_{(0,\alpha),q} = A_0$ (with equivalence of quasi-norms). Under some additional assumptions, it turns out that the spaces $(A_0,A_1)_{(\theta,\alpha),q}, \theta =0, 1,$ are special cases of the spaces $(A_0, A_1)_{\theta,q;\mathbb{A}}$ (see \cite[Proposition 1]{EdmundsOpic}).

Let $\mathcal{F}$ be any of these three interpolation methods. Clearly, the following embeddings hold
\begin{equation}\label{PrelimInterpolationEmb}
	A_1 \hookrightarrow \mathcal{F}(A_0,A_1) \hookrightarrow A_0.
\end{equation}
Furthermore, $\mathcal{F}$ has the interpolation property for bounded linear operators: If $A_0, A_1, B_0, B_1$ are quasi-Banach spaces with $A_1 \hookrightarrow A_0, B_1 \hookrightarrow B_0$ and $T$ is a linear operator such that the restrictions $T: A_0 \longrightarrow B_0$ and $T:A_1 \longrightarrow B_1$ are bounded, then the restriction $T: \mathcal{F}(A_0,A_1) \longrightarrow \mathcal{F}(B_0,B_1)$ is also bounded.

Moreover, working with some particular interpolation methods
one can show
 that the left-hand side embedding in (\ref{PrelimInterpolationEmb}) is dense. For example, if $0 < \theta < 1, 0 < q < \infty$, and $-\infty < \alpha < \infty$, then
\begin{equation}\label{PrelimInterpolationEmbDense}
	A_1 \quad \text{is densely embedded into} \quad (A_0,A_1)_{\theta,q;\alpha}
\end{equation}
(see \cite[Theorem 2.2]{Gustavsson}).

For later use we write down some Holmstedt's and reiteration formulae for interpolation spaces. Recall the shorthand notation $K(t,a) = K(t, a; A_0, A_1)$ for $t > 0$.

\begin{lem}[cf. \cite{BennettSharpley}, \cite{EvansOpicPick}, \cite{Fernandez-MartinezSignes}]\label{PrelimHolmstedt}
	Let $A_0, A_1$ be quasi-Banach spaces with $A_1 \hookrightarrow A_0$.
	\begin{enumerate}[\upshape(i)]
		\item Let $0 < \theta < 1, \alpha \in \mathbb{R}$, and $0 < q \leq \infty$. Then
		\begin{equation*}
			K(t^\theta \ell^{-\alpha}(t),a; A_0, (A_0,A_1)_{\theta,q;\alpha}) \asymp t^\theta \ell^{-\alpha}(t) \left(\int_t^\infty (u^{-\theta} \ell^\alpha(u) K(u,a))^q \frac{du}{u}\right)^{1/q}.
		\end{equation*}
		\item Let $0 < \theta < 1, \alpha \in \mathbb{R}$, and $0 < q \leq \infty$. Then
		\begin{equation*}
			K(t^{1-\theta} \ell^\alpha(t),a; (A_0,A_1)_{\theta,q;\alpha}, A_1) \asymp \left(\int_0^t (u^{-\theta} \ell^\alpha(u) K(u,a))^q\frac{du}{u}\right)^{1/q}.
		\end{equation*}
		\item Let $0 < q \leq \infty$ and $\alpha < -1/q \, (\alpha \leq 0 \text{ if } q = \infty)$. Then
		\begin{align*}
			K(t (1-\log t)^{-\alpha -\frac{1}{q}}, a; A_0, (A_0, A_1)_{(1,\alpha),q}) &  \\
			& \hspace{-5cm} \asymp K(t, a)  + t (1-\log t)^{-\alpha -\frac{1}{q}} \left(\int_t^1 (u^{-1} (1-\log u)^\alpha K(u, a))^q \frac{du}{u} \right)^{\frac{1}{q}}
			\end{align*}
			for all $0 < t< 1$.
			\item Let $0 < q \leq \infty$ and $\alpha > -1/q$. Then
			\begin{align*}
				K(t (1-\log t)^{\alpha + \frac{1}{q}}, a; (A_0, A_1)_{(0,\alpha),q}, A_1) & \\
				& \hspace{-5cm} \asymp (1-\log t)^{\alpha + \frac{1}{q}} K(t,a) + \left(\int_0^t ((1-\log u)^\alpha K(u,a))^q \frac{du}{u}\right)^{\frac{1}{q}}
			\end{align*}
			for all $0 < t< 1$.
			\item Let $0 < q \leq \infty$ and $\alpha > -1/q$. Then
			\begin{align*}
				K((1-\log t)^{-\alpha -1/q}, a; A_0, (A_0, A_1)_{(0,\alpha),q}) & \\
				& \hspace{-5cm} \asymp (1-\log t)^{-\alpha -1/q} \left(\int_t^1 ((1-\log u)^\alpha K(u,a))^q \frac{du}{u}\right)^{1/q}
			\end{align*}
			for all $0 < t< 1$.
			\item Let $0 < q \leq \infty$ and $\alpha < -1/q \,  (\alpha \leq 0 \text{ if } q = \infty)$. Then
			\begin{align*}
				K((1-\log t)^{\alpha +1/q}, a; (A_0, A_1)_{(1,\alpha),q}, A_1) & \\
				& \hspace{-5cm} \asymp \left(\int^t_0 (u^{-1}(1-\log u)^\alpha K(u,a))^q \frac{du}{u}\right)^{1/q}
			\end{align*}
			for all $0 < t< 1$.
			\item Let $0 < p, q \leq \infty, \alpha > -1/q$ and $\beta < -1/p$. Then
			\begin{align*}
				K(t (1-\log t)^{\alpha + \frac{1}{q} -\beta -\frac{1}{p}}, a; (A_0, A_1)_{(0,\alpha), q}, (A_0,A_1)_{(1,\beta),p}) & \\
				& \hspace{-7cm} \asymp (1- \log t)^{\alpha + \frac{1}{q}} K(t,a) + \left(\int_0^t ((1-\log u)^{\alpha} K(u,a))^q \frac{du}{u}\right)^{\frac{1}{q}}  \\
				& \hspace{-6.5cm}+ t (1-\log t)^{\alpha + \frac{1}{q} - \beta - \frac{1}{p}} \left(\int_t^1 (u^{-1} (1-\log u)^\beta K(u,a))^p \frac{du}{u}\right)^{\frac{1}{p}}
			\end{align*}
			for all $0 < t< 1$.
	\end{enumerate}
\end{lem}

\begin{lem}\label{PrelimLemma7.2}
	Let $A_0, A_1$ be quasi-Banach spaces with $A_1 \hookrightarrow A_0$. Assume that $0 < \theta, \eta < 1, 0 < p, q \leq \infty, -\infty < \alpha, \beta < \infty$, and $b > -1/q >c$. Then
	\begin{enumerate}[\upshape(i)]
		\item $((A_0,A_1)_{(0,b),q} , A_1)_{\theta, p;\alpha} = (A_0,A_1)_{\theta,p; (1-\theta)(b+1/q) + \alpha},$
		\item $(A_0, (A_0, A_1)_{(1,c),q})_{\theta,p;\alpha} = (A_0, A_1)_{\theta, p ; \theta (c + 1/q) + \alpha},$
		\item $((A_0,A_1)_{\theta,p;\alpha}, A_1)_{(1,c),q} = (A_0, A_1)_{(1,c),q}$,
			\item $(A_0, (A_0,A_1)_{\theta,p;\alpha})_{\eta,q;\beta} = (A_0,A_1)_{\eta \theta,q; \beta + \eta \alpha}$,
		\item $(A_0, (A_0,A_1)_{\theta,p;\alpha})_{(0,b),q} = (A_0,A_1)_{(0,b),q}$,
		\item $(A_0, A_1)_{\theta,q; b + 1/\min\{p,q\}} \hookrightarrow ((A_0,A_1)_{\theta,p}, A_1)_{(0,b),q} \hookrightarrow (A_0, A_1)_{\theta,q; b + 1/\max\{p,q\}}$.
	\end{enumerate}
\end{lem}
See \cite[Theorems 7.1, 7.2 and 7.4*]{EvansOpicPick} and \cite[Lemmas 2.2 and 2.5(b)]{CobosDominguez2}. 

It is well known that Besov and Lipschitz spaces can be characterized as interpolation spaces. Let $1 \leq p \leq \infty$. For $t > 0$ and $k \in \mathbb{N}$, we have
	\begin{equation}\label{KFunctional}
		K(t^k,f;L_p(\mathbb{R}^d), W^k_p(\mathbb{R}^d)) \asymp  \min\{1,t^k\}  \|f\|_{L_p(\mathbb{R}^d)} + \omega_k(f,t)_p
	\end{equation}
	and
	\begin{equation*}
		K(t^k,f;L_p(\mathbb{T}), W^k_p(\mathbb{T})) \asymp \min\{1,t^k\}  \|f\|_{L_p(\mathbb{T})} + \omega_k(f,t)_p
	\end{equation*}
 	(see \cite[Theorem 5.4.12, p. 339]{BennettSharpley} and \cite[Theorem 1]{JohnenScherer}). Consequently, given $0 < s < k \in \mathbb{N}, 0 < q \leq \infty$, and $-\infty < b < \infty$, one easily gets that
	\begin{equation}\label{PrelimInterpolation}
	 (L_p(\mathbb{R}^d), W^k_p(\mathbb{R}^d))_{s/k,q;b} = \mathbf{B}^{s, b}_{p,q}(\mathbb{R}^d) \quad \text{and} \quad  (L_p(\mathbb{T}), W^k_p(\mathbb{T}))_{s/k,q;b} = \mathbf{B}^{s, b}_{p,q}(\mathbb{T}).
	\end{equation}
	If $s=0$ and $k \in \mathbb{N}$, we have
	\begin{equation}\label{PrelimInterpolationnew}
	 (L_p(\mathbb{R}^d), W^k_p(\mathbb{R}^d))_{(0,b),q} = \mathbf{B}^{0, b}_{p,q}(\mathbb{R}^d) \quad \text{and} \quad (L_p(\mathbb{T}), W^k_p(\mathbb{T}))_{(0,b),q} = \mathbf{B}^{0, b}_{p,q}(\mathbb{T}).
	\end{equation}
	
	As a consequence of (\ref{PrelimInterpolation}) and (\ref{PrelimInterpolationnew}) and reiteration formulas given in Lemma \ref{PrelimLemma7.2} it is possible to derive further interpolation formulas. We shall illustrate below this approach obtaining some interpolation formulas that we shall use in the paper.

Combining Lemma \ref{PrelimLemma7.2}(iv), (v), (\ref{PrelimInterpolation}) and (\ref{PrelimInterpolationnew}) we get
	\begin{equation}\label{PrelimInterpolationnew2.2}
	 (L_p(\mathbb{R}^d), \mathbf{B}^{s,\xi}_{p,r}(\mathbb{R}^d))_{\frac{s_0}{s},q;b} = \mathbf{B}^{s_0, b + \frac{s_0 \xi}{s}}_{p,q}(\mathbb{R}^d) \quad \text{and} \quad  (L_p(\mathbb{T}), \mathbf{B}^{s,\xi}_{p,r}(\mathbb{T}))_{\frac{s_0}{s},q;b} = \mathbf{B}^{s_0, b + \frac{s_0 \xi}{s}}_{p,q}(\mathbb{T})
	\end{equation}
	 and (see also \cite[Theorem 3.4]{CobosDominguezTriebel} and \cite[(4.2), p. 60]{CobosDominguez2})
	\begin{equation}\label{PrelimInterpolationnew2}
	 (L_p(\mathbb{R}^d), \mathbf{B}^{s,\xi}_{p,r}(\mathbb{R}^d))_{(0,b),q} = \mathbf{B}^{0, b}_{p,q}(\mathbb{R}^d) \quad \text{and} \quad (L_p(\mathbb{T}),  \mathbf{B}^{s,\xi}_{p,r}(\mathbb{T}))_{(0,b),q} = \mathbf{B}^{0, b}_{p,q}(\mathbb{T}).
	\end{equation}
	Here, $s > s_0 >0, -\infty < b, \xi < \infty, 1 \leq p \leq \infty$, and $0 < q, r \leq \infty$.

	Let $1 \leq p \leq \infty, 0 < q, r, u \leq \infty, b > -1/r > d, k \in \mathbb{N}, s > 0, -\infty < c, \alpha < \infty,$ and $0 < \theta < 1$. Applying (\ref{PrelimInterpolationnew}), Lemma \ref{PrelimLemma7.2}(i) and (\ref{PrelimInterpolation}), we obtain
	\begin{equation}\label{PrelimInterpolationnew2.3}
	\begin{aligned}	(\mathbf{B}^{0,b}_{p,r}(\mathbb{R}^d), W^k_p(\mathbb{R}^d))_{\theta,q;\alpha} &= \mathbf{B}^{k \theta, (1-\theta)(b + \frac{1}{r}) + \alpha}_{p,q} (\mathbb{R}^d), \\ (\mathbf{B}^{0,b}_{p,r}(\mathbb{T}), W^k_p(\mathbb{T}))_{\theta,q;\alpha} &= \mathbf{B}^{k \theta, (1-\theta)(b + \frac{1}{r}) + \alpha}_{p,q} (\mathbb{T}).
\end{aligned}	\end{equation}
	On the other hand, by (\ref{PrelimInterpolationnew2}), Lemma \ref{PrelimLemma7.2}(i) and (\ref{PrelimInterpolationnew2.2}), we have
	\begin{equation}\label{PrelimInterpolationnew3}
	\begin{aligned}	(\mathbf{B}^{0,b}_{p,r}(\mathbb{R}^d), \mathbf{B}^{s,c}_{p,u}(\mathbb{R}^d))_{\theta,q;\alpha} &= \mathbf{B}^{s \theta, (1-\theta)(b + \frac{1}{r}) + \theta c + \alpha}_{p,q} (\mathbb{R}^d),
\\  	(\mathbf{B}^{0,b}_{p,r}(\mathbb{T}), \mathbf{B}^{s,c}_{p,u}(\mathbb{T}))_{\theta,q;\alpha} &= \mathbf{B}^{s \theta, (1-\theta)(b + \frac{1}{r}) + \theta c + \alpha}_{p,q} (\mathbb{T}).
\end{aligned}
\end{equation}

	In case that $s=k \in \mathbb{N}$ in (\ref{PrelimInterpolation}) we have
	\begin{equation}\label{interpolationLipschitz}
		(L_p(\mathbb{R}^d), W^k_p(\mathbb{R}^d))_{(1,-b),q} = \text{Lip}^{(k,-b)}_{p,q}(\mathbb{R}^d) \quad \text{and} \quad (L_p(\mathbb{T}), W^k_p(\mathbb{T}))_{(1,-b),q} = \text{Lip}^{(k,-b)}_{p,q}(\mathbb{T}) .
	\end{equation}
	In particular,
		\begin{equation}\label{interpolationLipschitz1}
		(L_\infty(\mathbb{R}^d), W^1_\infty(\mathbb{R}^d))_{(1,0),\infty} = \text{Lip}(\mathbb{R}^d) \quad \text{and} \quad (L_\infty(\mathbb{T}), W^1_\infty(\mathbb{T}))_{(1,0),\infty} = \text{Lip}(\mathbb{T})
	\end{equation}
	and
	\begin{equation}\label{0}
	 (L_1(\mathbb{R}^d), W^1_1(\mathbb{R}^d))_{(1,0),\infty} = \text{BV}(\mathbb{R}^d) \quad \text{and} \quad  (L_1(\mathbb{T}), W^1_1(\mathbb{T}))_{(1,0),\infty} = \text{BV}(\mathbb{T}).
\end{equation}

Let $1 \leq p \leq \infty, 0 < r, u \leq \infty, b >  1/r, 0 < s < k \in \mathbb{N}$, and $-\infty < c < \infty$. By (\ref{PrelimInterpolation}), Lemma \ref{PrelimLemma7.2}(iii), and (\ref{interpolationLipschitz}), we have
	\begin{equation}\label{PrelimInterpolationnew4}
		(\mathbf{B}^{s,c}_{p,u}(\mathbb{R}^d), W^k_p(\mathbb{R}^d))_{(1,-b),r} = \text{Lip}^{(k,-b)}_{p,r}(\mathbb{R}^d)
	\end{equation}
	and
	\begin{equation*}
	(\mathbf{B}^{s,c}_{p,u}(\mathbb{T}), W^k_p(\mathbb{T}))_{(1,-b),r} = \text{Lip}^{(k,-b)}_{p,r}(\mathbb{T}).
	\end{equation*}
	
The interpolation properties of the Fourier-analytically defined Besov spaces of generalized smoothness on $\mathbb{R}^d$ were investigated by Cobos and Fernandez \cite{CobosFernandez} with the help of the retraction method.\footnote{ Note that this method also works in the periodic case.}
In particular,
the following interpolation formula is valid
 \cite[Theorem 5.3]{CobosFernandez}.

\begin{lem}\label{PrelimLemmaCF}
	Let $1 \leq p \leq \infty, 0 < q_0, q_1, q \leq \infty, -\infty < s_0 \neq s_1 < \infty, - \infty < b_0, b_1, \alpha < \infty$ and $0 < \theta < 1$. Put $s = (1 - \theta) s_0 + \theta s_1$, and $b = (1 - \theta) b_0 + \theta b_1$. Then, we have
\begin{equation*}
(B^{s_0, b_0}_{p,q_0}(\mathbb{R}^d), B^{s_1, b_1}_{p,q_1}(\mathbb{R}^d))_{\theta,q;\alpha} = B^{s,b + \alpha}_{p,q}(\mathbb{R}^d).
\end{equation*}
	The corresponding formula for periodic Besov spaces also holds true.
\end{lem}

Putting $r=2$ in the embeddings given in Proposition \ref{RecallEmb}(i) and (ii) and applying (\ref{LPgeneral}) we derive, in particular, that
\begin{equation}\label{Prelim7.18}
	B^{s,b}_{p,\min\{p,2\}}(\mathbb{R}^d) \hookrightarrow H^{s,b}_{p}(\mathbb{R}^d) \hookrightarrow B^{s,b}_{p, \max\{p,2\}}(\mathbb{R}^d), \quad 1 < p < \infty.
\end{equation}
Therefore, it is an immediate consequence of (\ref{Prelim7.18}) and Lemma \ref{PrelimLemmaCF} that
\begin{equation}\label{PrelimInterpolationBWnew}
	(B^{s_0,b_0}_{p,q_0}(\mathbb{R}^d), H^{s_1, b_1}_p(\mathbb{R}^d))_{\theta,q;\alpha} = B^{s,(1-\theta)b_0 + \theta b_1 + \alpha}_{p,q}(\mathbb{R}^d)
\end{equation}
and
\begin{equation}\label{PrelimInterpolationW}
	(H^{s_0,b_0}_p(\mathbb{R}^d), H^{s_1, b_1}_p)_{\theta,q} = B^{(1-\theta)s_0 + \theta s_1, (1-\theta)b_0 + \theta b_1}_{p,q}(\mathbb{R}^d),
\end{equation}
where $1 < p < \infty$ and the remaining parameters satisfy the conditions given in Lemma \ref{PrelimLemmaCF}. The periodic counterparts of (\ref{PrelimInterpolationBWnew}) and (\ref{PrelimInterpolationW}) also hold.

Note that Lemma \ref{PrelimLemmaCF} shows that interpolation between Besov spaces $B^{s,b}_{p,q}(\mathbb{R}^d)$ with fixed $p$ produces spaces of the same type. In general, this is not true if we change the parameter $p$ (see \cite[2.4.1]{Triebel}). However, under some restrictions on the involved parameters, the following result still holds \cite[(7), p. 182]{Triebel}.

\begin{lem}\label{PrelimLemmaT}
		Let $1 < p_0, p_1 < \infty, p_0 \neq p_1, 1 \leq q_0, q_1 < \infty, -\infty < s_0, s_1 < \infty,$ and $0 < \theta < 1$. Assume that
		\begin{equation*}
			s = (1-\theta) s_0 + \theta s_1 \quad \text{and} \quad \frac{1-\theta}{p_0} + \frac{\theta}{p_1} = \frac{1}{p} = \frac{1-\theta}{q_0} + \frac{\theta}{q_1}.
		\end{equation*}
		Then, we have
\begin{equation*}
(B^{s_0}_{p_0,q_0}(\mathbb{R}^d), B^{s_1}_{p_1,q_1}(\mathbb{R}^d))_{\theta,p} = B^{s}_{p,p}(\mathbb{R}^d).
\end{equation*}
\end{lem}

We will show in Lemma \ref{PrelimLemmaT2} below that Lemma \ref{PrelimLemmaT} also holds for Besov spaces of generalized smoothness.

\subsection{Hardy-type inequalities}

We collect some Hardy-type inequalities that we shall use on several occasions throughout the course of the paper. For more information,
 see the monographs \cite{OpicKufner} and \cite{KufnerPersson}. Regarding Hardy inequalities for monotone functions see also \cite{Zeltser}.

\begin{lem}
	Let $\lambda > 0, 1 \leq q \leq \infty$ and $-\infty < b < \infty$. Let $\psi$ be a non-negative measurable function on $(0,1)$. Then
	\begin{equation}\label{HardyInequal1}
		\left(\int_0^1 \left(t^{-\lambda} (1 - \log t)^b \int_0^t \psi(s) ds \right)^q \frac{dt}{t}\right)^{1/q} \lesssim \left(\int_0^1 (t^{1-\lambda} (1-\log t)^b \psi(t))^q \frac{dt}{t}\right)^{1/q}
	\end{equation}
	and
	\begin{equation}\label{HardyInequal2}
		\left(\int_0^1 \left(t^{\lambda} (1 - \log t)^b \int_t^1 \psi(s) ds \right)^q \frac{dt}{t}\right)^{1/q} \lesssim \left(\int_0^1 (t^{1+\lambda} (1-\log t)^b \psi(t))^q \frac{dt}{t}\right)^{1/q}.
	\end{equation}
	Furthermore, if $\psi (t) = t^{\alpha -1} \varphi(t)$ with $\alpha > 0$ and $\varphi$ is a decreasing function then the inequalities (\ref{HardyInequal1}) and (\ref{HardyInequal2}) still hold true when $0 < q < 1$.
\end{lem}

 \begin{lem}
	Let $1 \leq q \leq \infty$ and $b + 1/q \neq 0$. Let $\psi$ be a non-negative measurable function on $(0,1)$. Then, if $b + 1/q > 0$,
	\begin{equation}\label{HardyInequal3}
		\left(\int_0^1 \left((1 - \log t)^b \int_0^t \psi(s) ds \right)^q \frac{dt}{t}\right)^{1/q} \lesssim \left(\int_0^1 (t (1-\log t)^{b + 1} \psi(t))^q \frac{dt}{t}\right)^{1/q}
	\end{equation}
	and, if $b + 1/q < 0$,
	\begin{equation}\label{HardyInequal4}
		\left(\int_0^1 \left((1 - \log t)^b \int_t^1 \psi(s) ds \right)^q \frac{dt}{t}\right)^{1/q} \lesssim \left(\int_0^1 (t (1-\log t)^{b + 1} \psi(t))^q \frac{dt}{t}\right)^{1/q}.
	\end{equation}
\end{lem}

\begin{rem}\label{RemarkHardy}
 (i) It is worthwhile mentioning that the inequalities (\ref{HardyInequal1}) and (\ref{HardyInequal2}) hold true when we replace the integral $\int_0^1$ by $\int_a^b$ with $0 \leq a < b \leq \infty$ (with appropriate modifications).

 (ii) By simple change of variables, it follows from (\ref{HardyInequal3}) that if $q \geq 1$ and $b + 1/q > 0$ then
 	\begin{equation}\label{HardyInequal5}
		\left(\int_1^\infty \left((1 + \log t)^b \int_t^\infty \psi(s) ds \right)^q \frac{dt}{t}\right)^{1/q} \lesssim \left(\int_1^\infty (t (1+\log t)^{b + 1} \psi(t))^q \frac{dt}{t}\right)^{1/q}.
	\end{equation}
	Analogously, we derive the counterpart of (\ref{HardyInequal4}), that is, assume $q \geq 1$ and $b+1/q < 0$ then
	\begin{equation*}
		\left(\int_1^\infty \left((1 + \log t)^b \int_1^t \psi(s) ds \right)^q \frac{dt}{t}\right)^{1/q} \lesssim \left(\int_1^\infty (t (1+\log t)^{b + 1} \psi(t))^q \frac{dt}{t}\right)^{1/q}.
	\end{equation*}

\end{rem}

We also give some Hardy's inequalities for sequences.

\begin{lem}
	Let $\lambda > 0, 1 \leq q \leq \infty$ and $-\infty < b < \infty$. Let $b_j \geq 0, j \in \mathbb{N},$ and $j_0 \in \mathbb{N}$. Then
	\begin{equation*}
		\left(\sum_{j=j_0}^\infty \left(j^{-\lambda} (1 + \log j)^b \sum_{k=j_0}^j b_k \right)^q \frac{1}{j}\right)^{1/q} \lesssim \left(\sum_{j=j_0}^\infty (j^{1-\lambda} (1+\log j)^b b_j)^q \frac{1}{j}\right)^{1/q},
	\end{equation*}
	\begin{equation*}
		\left(\sum_{j=j_0}^\infty \left(2^{-j \lambda} (1 + j)^b \sum_{k=j_0}^j b_k \right)^q \right)^{1/q} \lesssim \left(\sum_{j=j_0}^\infty (2^{-j \lambda} (1+ j)^b b_j)^q \right)^{1/q},
	\end{equation*}
	\begin{equation}\label{HardyInequal2*}
		\left(\sum_{j=j_0}^\infty \left(j^{\lambda} (1 + \log j)^b \sum_{k=j}^\infty b_k \right)^q \frac{1}{j}\right)^{1/q} \lesssim \left(\sum_{j=j_0}^\infty (j^{1+\lambda} (1+\log j)^b b_j)^q \frac{1}{j}\right)^{1/q},
	\end{equation}
	and
	\begin{equation}\label{HardyInequal2**}
		\left(\sum_{j=j_0}^\infty \left(2^{j \lambda} (1 + j)^b \sum_{k=j}^\infty b_k \right)^q \right)^{1/q} \lesssim \left(\sum_{j=j_0}^\infty (2^{j \lambda} (1+ j)^b b_j)^q \right)^{1/q}.
	\end{equation}
\end{lem}
%
%

\newpage
\section{Embeddings between Besov, Sobolev and Triebel-Lizorkin spaces with logarithmic smoothness}\label{section3}

\subsection{Embeddings with smoothness zero}
In this subsection, we first obtain embeddings between Triebel-Lizorkin and Besov spaces similar to those given in   Proposition \ref{RecallEmb}(iii) and (v) for the limiting case $s = 0$.
As a corollary we extend embeddings  between Sobolev and Besov spaces involving only logarithmic smoothness. The latter complements Proposition \ref{RecallEmb}(iv) and (vi).

\begin{thm}\label{Theorem 2.1}
    Let $1 < p < \infty, 0 < q,r \leq \infty$, and $b > -1/q$. Then
    \begin{equation}\label{newF}
    	F^{0,b+1/q+1/2-1/\max\{r,2\}}_{p,r} (\mathbb{R}^d) \hookrightarrow \mathbf{B}^{0,b}_{p,q}(\mathbb{R}^d) \quad \text{ if }\quad q \geq \max\{p,r\}
    \end{equation}
    and
    \begin{equation}\label{newF2}
    	\mathbf{B}^{0,b}_{p,q}(\mathbb{R}^d) \hookrightarrow F^{0,b+1/q + 1/2 - 1/\min\{r,2\}}_{p,r}(\mathbb{R}^d)\quad \text{ if } \quad q \leq \min\{p,r\}.
    \end{equation}
\end{thm}

The proof of Theorem \ref{Theorem 2.1} relies on the description of $\mathbf{B}^{0,b}_{p,q}(\mathbb{R}^d)$
    in terms of Fourier-analytical decompositions recently obtained in
    \cite[Theorem 4.3]{CobosDominguezTriebel}.

    \begin{lem}\label{lem: CDT}
    Let $1 < p < \infty, 0 < q \leq \infty,$ and $b > -1/q$. Then, we have
    \begin{equation}\label{2.4}
        \|f\|_{\mathbf{B}^{0,b}_{p,q}(\mathbb{R}^d)} \asymp \Bigg(\sum_{j=0}^\infty \Big[(1 + j)^b \Big\|\Big(\sum_{\nu=j}^\infty |(\varphi_\nu \widehat{f})^\vee
        (\cdot)|^2\Big)^{1/2}\Big\|_{L_p(\mathbb{R}^d)}\Big]^q\Bigg)^{1/q}.
    \end{equation}
    \end{lem}

\begin{proof}[Proof of Theorem \ref{Theorem 2.1}]
To show (\ref{newF}), we  assume that $q \geq  \max\{p,r\}$. If $r \leq 2$, applying (\ref{2.4}) and Minkowski's
    inequality twice,
    \begin{align*}
       \|f\|_{\mathbf{B}^{0,b}_{p,q}(\mathbb{R}^d)} &\lesssim  \left(\sum_{j=0}^\infty (1 + j)^{bq} \Big\|\Big(\sum_{\nu=j}^\infty |(\varphi_\nu
        \widehat{f})^\vee(\cdot)|^r\Big)^{1/r}\Big\|_{L_p(\mathbb{R}^d)}^q\right)^{1/q}
        \\
        & \leq \Bigg(\int_{\mathbb{R}^d} \Bigg(\sum_{j=0}^\infty \Big(\sum_{\nu=j}^\infty |(\varphi_\nu \widehat{f})^\vee (x)|^r\Big)^{q/r} (1 + j)^{b q}\Bigg)^{p/q}
        dx\Bigg)^{1/p} \\
        & \leq \Bigg(\int_{\mathbb{R}^d} \Bigg(\sum_{\nu = 0}^\infty |(\varphi_\nu \widehat{f})^\vee (x)|^r \Big(\sum_{j=0}^\nu (1 +
        j)^{bq}\Big)^{r/q}\Bigg)^{p/r} dx\Bigg)^{1/p} \\
        & \lesssim \Big\|\Bigg(\sum_{\nu=0}^\infty \Big((1 + \nu)^{b + 1/q} |(\varphi_\nu \widehat{f})^\vee
        (\cdot)|\Big)^r\Bigg)^{1/r}\Big\|_{L_p(\mathbb{R}^d)}  = \|f\|_{F^{0,b+1/q}_{p,r}(\mathbb{R}^d)}
    \end{align*}
    where we have used that $b + 1/q > 0$.

    If $r > 2$, we put $\rho=r/2$ and choose $\varepsilon$ such that $\frac{1}{2 \rho'} < \varepsilon < \frac{1}{2 \rho'} + b + \frac{1}{q}$. For every $x \in \mathbb{R}^d$, by H\"older's inequality we derive
    \begin{align*}
    	\left(\sum_{\nu=j}^\infty |(\varphi_\nu \widehat{f})^\vee (x)|^2\right)^{1/2}
     \lesssim j^{-\varepsilon + 1/2\rho'} \left(\sum_{\nu=j}^\infty (\nu^\varepsilon |(\varphi_\nu \widehat{f})^\vee (x)|)^r\right)^{1/r}.
    \end{align*}
    Therefore, (\ref{2.4}) and Minkowski's inequality yield that
    \begin{align*}
    	\|f\|_{\mathbf{B}^{0,b}_{p,q}(\mathbb{R}^d)} & \lesssim \left(\sum_{j=0}^\infty (1 + j)^{b q - \varepsilon q + q/2 \rho'} \Big\|\Big(\sum_{\nu=j}^\infty (\nu^\varepsilon |(\varphi_\nu \widehat{f})^\vee (\cdot)|)^r\Big)^{1/r} \Big\|_{L_p(\mathbb{R}^d)}^q\right)^{1/q} \\
	& \leq \Bigg(\int_{\mathbb{R}^d} \Bigg(\sum_{j=0}^\infty (1 + j)^{(b-\varepsilon + 1/2 \rho') q} \Big(\sum_{\nu = j}^\infty (\nu^\varepsilon |(\varphi_\nu \widehat{f})^\vee(x)|)^r\Big)^{q/r}\Bigg)^{p/q} d x\Bigg)^{1/p} \\
	& \leq \Bigg(\int_{\mathbb{R}^d} \Bigg(\sum_{\nu =0}^\infty \nu^{\varepsilon r} |(\varphi_\nu \widehat{f})^\vee (x)|^r \Big(\sum_{j=0}^\nu (1 + j)^{(b - \varepsilon + 1/2 \rho') q}\Big)^{r/q}\Bigg)^{p/r} d x\Bigg)^{1/p} \\
	& \lesssim \Bigg\| \Bigg(\sum_{\nu = 0}^\infty \nu^{(b + 1/q + 1/2 \rho') r} |(\varphi_\nu \widehat{f})^\vee(\cdot)|^r\Bigg)^{1/r} \Bigg\|_{L_p(\mathbb{R}^d)} = \|f\|_{F^{0,b+1/q+1/2-1/r}_{p,r}(\mathbb{R}^d)}.
    \end{align*}

	Next we proceed with the converse embedding (\ref{newF2}). Firstly, let us assume that $r \geq 2$. Then, using (\ref{2.4}) and Minkowski's inequality twice,
	\begin{align*}
		\|f\|_{\mathbf{B}^{0,b}_{p,q}(\mathbb{R}^d)} & \gtrsim \Bigg(\sum_{j=0}^\infty \Bigg[(1 + j)^b\Big\|\Big(\sum_{\nu=j}^\infty |(\varphi_\nu \widehat{f})^\vee
        (\cdot)|^r\Big)^{1/r}\Big\|_{L_p(\mathbb{R}^d)}\Bigg]^q\Bigg)^{1/q}
        \\
        & \geq \Bigg(\int_{\mathbb{R}^d} \Bigg[\sum_{j=0}^\infty \Bigg(\sum_{\nu=j}^\infty |(\varphi_\nu \widehat{f})^\vee (x)|^r\Bigg)^{q/r} (1 + j)^{b q}\Bigg]^{p/q} d x\Bigg)^{1/p} \\
        & \geq \Bigg(\int_{\mathbb{R}^d} \Bigg[\sum_{\nu=0}^\infty |(\varphi_\nu \widehat{f})^\vee(x)|^r \Bigg(\sum_{j=0}^\nu (1 + j)^{b q}\Bigg)^{r/q}\Bigg]^{p/r} d x\Bigg)^{1/p} \\
        & \asymp \Bigg(\int_{\mathbb{R}^d} \Bigg[\sum_{\nu=0}^\infty |(\varphi_\nu \widehat{f})^\vee(x)|^r (1 + \nu)^{b r + r/q}\Bigg]^{p/r} d x\Bigg)^{1/p} = \|f\|_{F^{0,b+1/q}_{p,r}(\mathbb{R}^d)}.
	\end{align*}
	
	In the case $r < 2$, we set $\rho = 2/r$. Take any $\varepsilon$ for which $\varepsilon < 1/2 - 1/r$. For $j \in \mathbb{N}_0$, by H\"older's inequality we obtain
	\begin{equation}\label{New2.4}
		\left(\sum_{\nu=j}^\infty (1 + \nu)^{\varepsilon r} |(\varphi_\nu \widehat{f})^\vee(x)|^r\right)^{1/r}
 \lesssim (1 + j)^{\varepsilon + 1/r -1/2} \left(\sum_{\nu=j}^\infty  |(\varphi_\nu \widehat{f})^\vee(x)|^2\right)^{1/2}.
	\end{equation}
	Then, applying Minkowski's inequality and (\ref{New2.4}),
	\begin{align*}
		\|f\|_{F^{0,b+1/q+1/2-1/r}_{p,r}(\mathbb{R}^d)} & \\
		& \hspace{-3cm}= \Big\|\Bigg(\sum_{\nu=0}^\infty \Big[(1 + \nu)^{b - \varepsilon + 1/2 - 1/r + 1/q} (1 + \nu)^\varepsilon |(\varphi_\nu \widehat{f})^\vee (\cdot)|\Big]^r\Bigg)^{1/r}\Big\|_{L_p(\mathbb{R}^d)} \\
		& \hspace{-3cm} \asymp \Big\|\Bigg(\sum_{\nu=0}^\infty (1 + \nu)^{\varepsilon r} |(\varphi_\nu \widehat{f})^\vee(\cdot)|^r \Big(\sum_{j=0}^\nu (1 + j)^{(b - \varepsilon - 1/r \rho') q} \Big)^{r/q} \Bigg)^{1/r}\Big\|_{L_p(\mathbb{R}^d)} \\
		& \hspace{-3cm} \leq \Big\|\Bigg(\sum_{j=0}^\infty (1 + j)^{(b- \varepsilon - 1/r + 1/2) q} \Big(\sum_{\nu=j}^\infty (1 + \nu)^{\varepsilon r} |(\varphi_\nu \widehat{f})^\vee(\cdot)|^r \Big)^{q/r}\Bigg)^{1/q}\Big\|_{L_p(\mathbb{R}^d)}  \\
		&\hspace{-3cm}   \leq \Bigg(\sum_{j=0}^\infty (1 + j)^{(b - \varepsilon -1/r + 1/2) q} \Big\|\Big(\sum_{\nu=j}^\infty (1 + \nu)^{\varepsilon r} |(\varphi_\nu \widehat{f})^\vee (\cdot)|^r\Big)^{1/r}\Big\|_{L_p(\mathbb{R}^d)}^q\Bigg)^{1/q} \\
		& \hspace{-3cm}  \lesssim \Bigg(\sum_{j=0}^\infty (1 + j)^{(b - \varepsilon - 1/r + 1/2)q} (1 + j)^{(\varepsilon + 1/r-1/2) q} \Big\|\Big(\sum_{\nu=j}^\infty |(\varphi_\nu \widehat{f})^\vee (\cdot)|^2\Big)^{1/2}\Big\|_{L_p(\mathbb{R}^d)}^q\Bigg)^{1/q} \\
		& \hspace{-3cm}  \asymp \|f\|_{\mathbf{B}^{0,b}_{p,q}(\mathbb{R}^d)}
	\end{align*}
	where the last estimate follows from (\ref{2.4}).
\end{proof}

\begin{rem} Taking $r=q$ in
Theorem \ref{Theorem 2.1} we arrive at
    \begin{equation}
    	F^{0,b+1/\min\{q,2\}}_{p,q} (\mathbb{R}^d) \hookrightarrow \mathbf{B}^{0,b}_{p,q}(\mathbb{R}^d) \quad \text{ if }\quad q \geq p \quad \text{and} \quad b > -1/q,
    \end{equation}
      \begin{equation}
    	\mathbf{B}^{0,b}_{p,q}(\mathbb{R}^d) \hookrightarrow F^{0,b+1/\max\{q,2\}}_{p,q}(\mathbb{R}^d)\quad \text{ if } \quad q \leq p \quad \text{and} \quad b > -1/q,
    \end{equation}
    \begin{equation}\label{newF3}
    	F^{0,b+1/\min\{p,2\}}_{p,p}(\mathbb{R}^d) \hookrightarrow \mathbf{B}^{0,b}_{p,p}(\mathbb{R}^d) \hookrightarrow F^{0,b+1/\max\{p,2\}}_{p,p}(\mathbb{R}^d), \qquad b > -1/p,
    \end{equation}
    and
    \begin{equation}\label{newF4}
     \mathbf{B}^{0,b}_{2,2}(\mathbb{R}^d) = F^{0,b+1/2}_{2,2}(\mathbb{R}^d), \qquad b > -1/2.
     \end{equation}
     	Note that the formulas (\ref{newF3}) and (\ref{newF4}) can also be derived from (\ref{1}) and (\ref{BesovZero}), respectively, because $F^{0,b+1/\min\{p,2\}}_{p,p}(\mathbb{R}^d) = B^{0,b+1/\min\{p,2\}}_{p,p}(\mathbb{R}^d)$, $F^{0,b+1/\max\{p,2\}}_{p,p}(\mathbb{R}^d) = B^{0,b+1/\max\{p,2\}}_{p,p}(\mathbb{R}^d)$, and $F^{0,b+1/2}_{2,2}(\mathbb{R}^d) = B^{0,b+1/2}_{2,2}(\mathbb{R}^d)$ (see Proposition \ref{RecallEmb}).
	
	\end{rem}

Letting  $r=2$ in Theorem \ref{Theorem 2.1} we obtain the following

\begin{cor}\label{Corollary 2.3}
	Let $1 < p < \infty, 0 < q \leq \infty$, and $b > -1/q$. Then
	    \begin{equation}\label{2.2}
        H^{0,b+1/q}_p(\mathbb{R}^d) \hookrightarrow \emph{\textbf{B}}^{0,b}_{p,q}(\mathbb{R}^d) \quad \text{ if
        } \quad q \geq \max\{p,2\}
    \end{equation}
    and
    \begin{equation}\label{2.3}
        \emph{\textbf{B}}^{0,b}_{p,q}(\mathbb{R}^d) \hookrightarrow H^{0,b+1/q}_p(\mathbb{R}^d) \quad \text{ if
        } \quad q \leq \min\{p,2\}.
    \end{equation}
    In particular, for $b > -1/2$ we obtain with equivalence of
    norms
    \begin{equation}\label{SobolevZero}
        \mathbf{B}^{0,b}_{2,2}(\mathbb{R}^d) = H^{0,b+1/2}_2(\mathbb{R}^d).
    \end{equation}
	
\end{cor}

\begin{rem}
	(i) The formula (\ref{SobolevZero}) can also be derived from (\ref{BesovZero}) or (\ref{newF4}) since
	\begin{equation*}
	 B^{0,b+1/2}_{2,2}(\mathbb{R}^d) = F^{0,b+1/2}_{2,2}(\mathbb{R}^d) = H^{0,b+1/2}_2(\mathbb{R}^d)
	 \end{equation*}
	 (see Proposition \ref{RecallEmb} and (\ref{LPgeneral})).
	However, as we will see in Section \ref{section7}, outside of the case when $p=q=2$, the embeddings  given in Corollary \ref{Corollary 2.3} and (\ref{1}) are not comparable. Moreover, we will prove in Section \ref{SectionOptimalityWB} (see Theorem \ref{Theorem 6.6} below) that (\ref{2.2}), (\ref{2.3}) and (\ref{SobolevZero}) are sharp in the following sense
	     \begin{equation}\label{Sharp2.2}
        H^{0,b+1/q}_p(\mathbb{R}^d) \hookrightarrow \mathbf{B}^{0,b}_{p,q}(\mathbb{R}^d) \iff q \geq \max\{p,2\},
    \end{equation}
        \begin{equation}\label{Sharp2.3}
        \mathbf{B}^{0,b}_{p,q}(\mathbb{R}^d) \hookrightarrow H^{0,b+1/q}_p(\mathbb{R}^d) \iff q \leq \min\{p,2\},
    \end{equation}
    and
	 \begin{equation}\label{SharpBesovSobolev}
	 	\mathbf{B}^{0,b}_{p,q}(\mathbb{R}^d) = H^{0,\xi}_p (\mathbb{R}^d) \iff p = q = 2 \quad \text{and} \quad \xi = b+1/2.
	 \end{equation}

	 (ii) It will be shown in Theorem \ref{Proposition 5.3} below that the losses of logarithmic smoothness in the Sobolev norms obtained in (\ref{2.2}) and (\ref{2.3}) are the best possible. More precisely, we will prove that if $q \geq \max\{p,2\}$ then
	 	\begin{equation*}
	H^{0,\xi}_p(\mathbb{R}^d) \hookrightarrow \mathbf{B}^{0,b}_{p,q}(\mathbb{R}^d) \text{ \qquad if and only if  \qquad } \xi \geq b + \frac{1}{q},
\end{equation*}
and if $q \leq \min\{p,2\}$ then
	\begin{equation*}
	\mathbf{B}^{0,b}_{p,q}(\mathbb{R}^d) \hookrightarrow H^{0,\xi}_p(\mathbb{R}^d) \text{ \qquad  if and only if  \qquad } \xi \leq b + \frac{1}{q}.
\end{equation*}
\end{rem}

\subsection{Sobolev embeddings for $\mathbf{B}^{0,b}_{p,q}(\mathbb{R}^d)$}

The aim of this section is to derive the Sobolev embedding theorem for Besov spaces $\mathbf{B}^{0,b}_{p,q}(\mathbb{R}^d)$, that is, the $\mathbf{B}$-counterpart of Proposition \ref{RecallEmb2}(i).


\begin{thm}\label{TheoremSobolevEmb}
	Let $1 \leq p_0 < p < p_1 \leq \infty, -\infty < s_1 < 0 < s_0 < \infty, 0 < q \leq \infty$, and $b > -1/q$ with
		\begin{equation}\label{differentialdimension}
		s_0 - \frac{d}{p_0} = -\frac{d}{p} = s_1 - \frac{d}{p_1}.
	\end{equation}
	Then
	\begin{equation}\label{EmbSobolev}
		B^{s_0, b + 1/\min\{p,q\}}_{p_0,q}(\mathbb{R}^d) \hookrightarrow \mathbf{B}^{0,b}_{p,q}(\mathbb{R}^d) \hookrightarrow B^{s_1, b + 1/\max\{p,q\}}_{p_1,q}(\mathbb{R}^d).
	\end{equation}
\end{thm}
\begin{rem}\label{Remark 3.5}
	(i) The left-hand side embedding in (\ref{EmbSobolev}) was already proved in \cite[Corollary 2.8]{GogatishviliOpicTikhonovTrebels} and \cite[Theorem 3.7]{CobosDominguez3}. The method given in \cite{GogatishviliOpicTikhonovTrebels} relies on Ulyanov inequalities, that is, inequalities for the moduli of smoothness in different metrics. Alternatively, we give here an approach based on limiting interpolation, which is a streamlined version of that given in \cite{CobosDominguez3}. Furthermore, this method allows us to obtain the right-hand side embedding in (\ref{EmbSobolev}), which can not be derived by using Ulyanov inequalities because $s_1 < 0$.
	
	(ii) Note that the embeddings (\ref{EmbSobolev}) do not hold in the limiting cases  $s_0=0$ or $s_1=0$. More precisely, if $s_1=0$ (respectively, $s_0 =0$) and so, $p_1=p$ (respectively, $p_0=p$), then the embedding $\mathbf{B}^{0,b}_{p,q}(\mathbb{R}^d) \hookrightarrow B^{0,b+1/\max\{p,q\}}_{p,q}(\mathbb{R}^d)$ (respectively, $B^{0,b+1/\min\{p,q\}}_{p,q}(\mathbb{R}^d) \hookrightarrow \mathbf{B}^{0,b}_{p,q}(\mathbb{R}^d)$) is not true in general. In fact, we will show in Section \ref{Optimality of embeddings with smoothness near zero} below that the embeddings given in (\ref{1}) are the best possible with respect to logarithmic smoothness parameters.
	
	(iii) It will be shown in Section \ref{Optimality of Sobolev embeddings} that the embeddings (\ref{EmbSobolev}) are optimal with respect to logarithmic smoothness parameters. To be more precise, given any $\varepsilon > 0$ we have that
	\begin{equation*}
		B^{s_0, \xi}_{p_0,q}(\mathbb{R}^d)  \hookrightarrow \mathbf{B}^{0,b}_{p,q}(\mathbb{R}^d) \iff \xi \geq b + \frac{1}{\min\{p,q\}}
	\end{equation*}
	and
	\begin{equation*}
		 \mathbf{B}^{0,b}_{p,q}(\mathbb{R}^d) \hookrightarrow B^{s_1, \xi}_{p_1,q}(\mathbb{R}^d) \iff \xi \leq b +  \frac{1}{\max\{p,q\}}.
	\end{equation*}
	
	(iv) Both (\ref{EmbSobolev}) and (\ref{1}) establish embeddings involving the Besov spaces $\mathbf{B}^{0,b}_{p,q}(\mathbb{R}^d)$. Then, it is a natural question to compare (\ref{EmbSobolev}) with (\ref{1}). Note that if $2 \geq \min\{p,q\}$ then (see Proposition \ref{RecallEmb2}(i))
	\begin{equation*}
		B^{s_0,b+1/\min\{p,q\}}_{p_0,q}(\mathbb{R}^d) \hookrightarrow B^{0,b+1/\min\{2,p,q\}}_{p,q}(\mathbb{R}^d)
	\end{equation*}
	and so, the left-hand side embedding in (\ref{EmbSobolev}) follows directly from (\ref{1}). Similarly, if $2 \leq \max\{p,q\}$ then the right-hand side embedding in (\ref{EmbSobolev}) is an immediate consequence of (\ref{1}). However, we will see in Section \ref{Comparison between embeddings with smoothness near zero and Sobolev embeddings} that the embeddings given in (\ref{EmbSobolev}) and (\ref{1}) are not comparable in general.
\end{rem}
\begin{proof}[Proof of Theorem \ref{TheoremSobolevEmb}]
	We shall only prove the right-hand side embedding in (\ref{EmbSobolev}). The proof of the left-hand side embedding can be carried out similarly. By our assumptions on the parameters, accordingly to Proposition \ref{RecallEmb2}(ii) we have
	\begin{equation}\label{NewRemark}
		L_p(\mathbb{R}^d) \hookrightarrow B^{s_1}_{p_1,p}(\mathbb{R}^d) \quad \text{and} \quad W^1_p(\mathbb{R}^d) \hookrightarrow B^{s_1 + 1}_{p_1,p}(\mathbb{R}^d).
	\end{equation}
	If we interpolate these embeddings by the $((0,b),q)$-method (cf. (\ref{limitinginterpolation})), we derive
	\begin{equation}\label{InterpolationEmb1}
		(L_p(\mathbb{R}^d), W^1_p(\mathbb{R}^d))_{(0,b),q} \hookrightarrow (B^{s_1}_{p_1,p}(\mathbb{R}^d) , B^{s_1 + 1}_{p_1,p}(\mathbb{R}^d))_{(0,b),q}.
	\end{equation}
	
	By (\ref{PrelimInterpolationnew}), we have
	\begin{equation}\label{NewRemark2}
	 \mathbf{B}^{0,b}_{p,q}(\mathbb{R}^d) = (L_p(\mathbb{R}^d), W^1_p(\mathbb{R}^d))_{(0,b),q}.
	 \end{equation}
	
	  Concerning the right-hand side space in (\ref{InterpolationEmb1}), we can apply Lemmas \ref{PrelimLemmaCF} and \ref{PrelimLemma7.2}(vi) to get
	\begin{align*}
		\big(B^{s_1}_{p_1,p}(\mathbb{R}^d) , B^{s_1 + 1}_{p_1,p}(\mathbb{R}^d)\big)_{(0,b),q} & =\big((B^{s_1-1}_{p_1,p}(\mathbb{R}^d) , B^{s_1 + 1}_{p_1,p}(\mathbb{R}^d))_{1/2,p}, B^{s_1 + 1}_{p_1,p}(\mathbb{R}^d)\big)_{(0,b),q} \\
		&
 \hookrightarrow \big(B^{s_1-1}_{p_1,p}(\mathbb{R}^d) , B^{s_1 + 1}_{p_1,p}(\mathbb{R}^d)\big)_{1/2,q;b+1/\max\{p,q\}}
\\
&= B^{s_1,b+1/\max\{p,q\}}_{p_1,q}(\mathbb{R}^d).
	\end{align*}
	Whence, we obtain that $\mathbf{B}^{0,b}_{p,q}(\mathbb{R}^d) \hookrightarrow B^{s_1, b + 1/\max\{p,q\}}_{p_1,q}(\mathbb{R}^d)$.
\end{proof}

\begin{rem}
(i) The periodic counterpart of Theorem \ref{TheoremSobolevEmb} also holds true. More specifically: 	Let $1 < p_0 < p < p_1 < \infty, -\infty < s_1 < 0 < s_0 < \infty, 0 < q \leq \infty$, and $b > -1/q$ with (\ref{differentialdimension}). Then
	\begin{equation*}
		B^{s_0, b + 1/\min\{p,q\}}_{p_0,q}(\mathbb{T}) \hookrightarrow \mathbf{B}^{0,b}_{p,q}(\mathbb{T}) \hookrightarrow B^{s_1, b + 1/\max\{p,q\}}_{p_1,q}(\mathbb{T}).
	\end{equation*}
	This can be shown by following the same approach as in the proof of Theorem \ref{TheoremSobolevEmb}
 using the fact that
 the embeddings (\ref{NewRemark}) also hold for periodic functions (see \cite[(8), page 170]{SchmeisserTriebel}), as well as the interpolation formula (\ref{NewRemark2}) (see (\ref{PrelimInterpolationnew})).
	
	(ii) If $b=-1/q$ the corresponding result to Theorem \ref{TheoremSobolevEmb} reads as follows: Let $1 \leq p_0 < p < p_1 \leq \infty, -\infty < s_1 < 0 < s_0 < \infty$ and $0 < q \leq \infty$ with (\ref{differentialdimension}). Then
	\begin{equation*}
		B^{s_0, -1/q + 1/\min\{p,q\}, 1/\min\{p,q\}}_{p_0,q}(\mathbb{R}^d) \hookrightarrow \mathbf{B}^{0,-1/q}_{p,q}(\mathbb{R}^d) \hookrightarrow B^{s_1, -1/q + 1/\max\{p,q\}, 1/\max\{p,q\}}_{p_1,q}(\mathbb{R}^d).
	\end{equation*}
	Here, $B^{s,b,\xi}_{p,q}(\mathbb{R}^d), -\infty < \xi < \infty,$\index{\bigskip\textbf{Spaces}!$B^{s,b,\xi}_{p,q}(\mathbb{R}^d)$}\label{BESOVFLOG} is the Besov space of iterated logarithmic smoothness endowed with the quasi-norm
	\begin{equation}\label{BesovIteratedSmoothness}
		 \|f\|_{B^{s,b,\xi}_{p,q}(\mathbb{R}^d)} = \left(\sum_{j=0}^\infty \Big(2^{j s} (1 + j)^b (1 + \log (1+j))^\xi \|(\varphi_j
    \widehat{f})^\vee\|_{L_p(\mathbb{R}^d)}\Big)^q\right)^{1/q}.
	\end{equation}
	The proof is similar to that of Theorem \ref{TheoremSobolevEmb}  but now using the limiting case of Lemma \ref{PrelimLemma7.2}(vi) when $b = -1/q$ which has been obtained in \cite[Lemma 3.1(b)]{CobosDominguez4}. Namely, if $0 < \theta < 1, 0 < p \leq \infty$ and $0 < q < \infty$, then
	\begin{align}
	(A_0, A_1)_{\theta,q; -1/q + 1/\min\{p,q\}, 1/\min\{p,q\}} & \hookrightarrow ((A_0,A_1)_{\theta,p}, A_1)_{(0,-1/q),q}  \nonumber \\
	& \hspace{-4cm} \hookrightarrow (A_0, A_1)_{\theta,q; -1/q + 1/\max\{p,q\}, 1/\max\{p,q\}} \label{limiting interpolation reiteration formula}.
	\end{align}
	Here, \index{\bigskip\textbf{Spaces}!$(A_0,A_1)_{\theta,q;\alpha,\xi}$}\label{REALLOGIT}
	\begin{equation*}
	\|a\|_{(A_0,A_1)_{\theta,q;\alpha,\xi}} =  \left(\int_0^\infty (t^{-\theta} \ell^\alpha(t) \ell^\xi (\ell (t)) K(t,a))^q \frac{dt}{t}\right)^{1/q}, \quad -\infty < \xi < \infty.
\end{equation*}
In particular, if $\xi = 0$ we recover the spaces $(A_0,A_1)_{\theta,q;\alpha}$ introduced in (\ref{inter*}).
\end{rem}

\newpage
\section{Characterizations and embedding theorems for general monotone functions}\label{section4}

\subsection{Definition and basic properties of general monotone functions}

We recall the definition of the general monotone functions given in \cite{LiflyandTikhonov, Tikhonov}. A complex-valued function $\varphi (z), z
>0,$ is called \emph{general monotone}  if it is locally of bounded variation and for some  constant $C > 1$ the following is true
\begin{equation}\label{3.1}
    \int_z^{2z}  |d \varphi (u)| \leq C |\varphi (z)|
\end{equation}
for all $z > 0$. Constant $C$ in (\ref{3.1}) is independent of $z$.
The set of all general monotone functions is denoted by $GM$\index{\bigskip\textbf{Sets}!$GM$}\label{GM}. Examples of general monotone functions consist of:
decreasing functions,
 quasi-monotone
functions $\varphi$ (i.e, $\varphi(t) t^{-\alpha}$ is non-increasing for some $\alpha \geq 0$), and increasing functions $\varphi$ such that
$ \varphi(2z) \lesssim \varphi(z)$. Note also that (\ref{3.1}) implies
\begin{equation}\label{3.2}
    |\varphi (u)| \lesssim |\varphi (z)|\quad \text{ for any } \quad z \leq u \leq
    2z,
\end{equation}
which subsequently gives
\begin{equation}\label{3.3}
    |\varphi (z)| \lesssim \int_{z/c}^\infty \frac{|\varphi(u)|}{u}
    du, \qquad c>1.
\end{equation}

For later use we recall the following lemma on multipliers of general monotone functions (see \cite[Remark 5.5]{LiflyandTikhonov}).

\begin{lem}\label{Lemma 3.1}
    Let $\varphi, \alpha \in GM$, then $\varphi \alpha \in GM$.
\end{lem}
Now we are in a position to give the main definition in this section.
First, we recall that
the Fourier transform of a radial function $f(x) = f_0(|x|)$ is also radial, $\widehat{f}(\xi)=F_{0}(|\xi|)$
(see \cite[Ch.~4]{SteinWeiss}) and it can be written as the Fourier--Hankel transform
\begin{equation}\label{FourierHankel}
F_{0}(s)= \frac{2 \pi^{d/2}}{\Gamma\left(\frac{d}{2}\right)} \int_{0}^{\infty}f_{0}(t)j_{d/2-1}(st)t^{d-1}\,dt,
\end{equation}
where $j_{\alpha}(t)=\Gamma(\alpha+1)(t/2)^{-\alpha}J_{\alpha}(t)$ is the
normalized Bessel function ($j_{\alpha}(0)=1$), $\alpha\ge -1/2$, with $J_\alpha(t)$ the classical Bessel function of the first kind of order $\alpha$.

Let $\widehat{GM}^d$ \index{\bigskip\textbf{Sets}!$\widehat{GM}^d$}\label{GMF} be the collection of all radial functions such that the radial component $F_0$ of the Fourier transform of $f$, that is, $\widehat{f}(\xi)= F_0(|\xi|)$, belongs to the class $GM$, is positive and satisfies the condition
\begin{equation}\label{3.4new}
    \int_0^1 u^{d-1} F_0(u) du + \int_1^\infty u^{(d-1)/2} |d
    F_0(u)| < \infty;
\end{equation}
see \cite{GorbachevTikhonov}.
In other words,
 $\widehat{GM}{}^{d}$ consists of radial functions
$f(x)=f_{0}(|x|)$, $x\in \mathbb{R}^{d}$, which are defined in terms of the inverse Fourier--Hankel transform
\begin{equation}\label{3.4new+}
f_{0}(z)=\frac{2}{\Gamma\left(\frac{d}{2}\right) (2 \sqrt{\pi})^d}\int_{0}^{\infty}F_{0}(s)j_{d/2-1}(zs)s^{d-1}\,ds,
\end{equation}
where the function $F_{0}\in GM$ and satisfies  condition (\ref{3.4new}).
We note that, by  \cite[Lemma 1]{GorbachevLiflyandTikhonov}, the integral in \eqref{3.4new}
converges in the improper sense and therefore $f_{0}(z)$ is continuous for $z>0$.
 If $d=1$ we simply write $\widehat{GM}$.

In the discrete case, $\widehat{GM}$ coincides with the  well investigated class of periodic functions $f(x) \sim \sum_{n=1}^\infty (a_n \cos n x + b_n \sin nx)$ such that the sequences of their Fourier coefficients $\{a_n\}_{n \in \mathbb{N}}, \{b_n\}_{n \in \mathbb{N}}$
sa\-tis\-fy the discrete general monotone condition, that is,
$ \sum_{k=n}^{2n-1} |\Delta d_k| \leq C |d_n|$
 for all $n\in \mathbb{N}$ ($\Delta d_k := d_k - d_{k+1}$) (cf. (\ref{3.1})); see \cite{Tikhonov, LiflyandTikhonov} and the references therein.
 In particular, the important class (see, e.g., \cite[Chapters V, XII]{Zygmund}) of Fourier series with monotonic coefficients belongs to $\widehat{GM}$.
In Subsection \ref{Section 3.8},
we address some important properties of
 the subclass of the Besov space, consisting of the periodic functions from $\widehat{GM}$.

\subsection{Characterization of spaces $\mathbf{B}^{s,b}_{p,q}(\mathbb{R}^d)$}\label{SubsectionContinuous}
Working with the $\widehat{GM}^d$ class, we are able to characterize
functions from the Besov space $\textbf{B}^{s,b}_{p,q}(\mathbb{R}^d)$ in terms of
the growth properties of their Fourier transform.

\begin{thm}\label{Theorem 3.2}
    Let $\frac{2d}{d+1} < p < \infty, 0 < q \leq \infty$, and $- \infty < b < \infty$. Let $f \in \widehat{GM}^d$.
    In the case  $s>0$ we have
    \begin{equation}\label{3.3.1}
        \|f\|_{{\mathbf{B}}^{s,b}_{p,q}(\mathbb{R}^d)} \asymp \left(\int_0^1 t^{dp - d -1} F_0^p(t) dt\right)^{1/p} +  \left(\int_1^\infty t^{sq + d q -dq/p-1} (1 + \log t)^{b q} F_0^q(t) dt \right)^{1/q}
    \end{equation}
    whenever $q < \infty$ and
    \begin{equation}\label{3.3.2}
         \|f\|_{{\mathbf{B}}^{s,b}_{p,\infty}(\mathbb{R}^d)} \asymp \left(\int_0^1 t^{dp - d -1} F_0^p(t) dt\right)^{1/p} + \sup_{t > 1} t^{s + d -d/p} (1 + \log t)^b F_0(t).
    \end{equation}
    In the case  $s=0$ we have
    \begin{multline}\label{3.3.3}\qquad
        \|f\|_{\mathbf{B}^{0,b}_{p,q}(\mathbb{R}^d)} \asymp \left(\int_0^1 t^{dp - d -1} F_0^p(t) dt\right)^{1/p} \\ \hspace{5cm}+ \left(\int_1^\infty \left[(1 + \log t)^b \left(\int_t^\infty u^{d p -d-1} F_0^p(u) du\right)^{1/p}\right]^q
        \frac{dt}{t} \right)^{1/q}
    \end{multline}
      whenever $q < \infty$ and
    \begin{equation}\label{3.3.4}
              \|f\|_{\mathbf{B}^{0,b}_{p,\infty}(\mathbb{R}^d)} \asymp \left(\int_0^1 t^{dp - d -1} F_0^p(t) dt\right)^{1/p}  + \sup_{t > 1} \,(1 + \log t)^b \left(\int_t^\infty u^{d p -d-1} F_0^p(u)
        du\right)^{1/p}.
    \end{equation}
\end{thm}
\begin{rem}\label{RemarkContinuous}
(i) If the right-hand sides in (\ref{3.3.1}), (\ref{3.3.2}), (\ref{3.3.3}), and (\ref{3.3.4}) are finite, then
\begin{equation*}
	\left(\int_0^\infty t^{d p - d - 1} F_0^p(t) dt\right)^{1/p} < \infty,
\end{equation*}
and, subsequently, $f \in L_p(\mathbb{R}^d)$.
This fact follows from the Hardy-Littlewood-type estimate for functions $f \in \widehat{GM}^d$ obtained in \cite[Theorem 1]{GorbachevLiflyandTikhonov} (see also \cite[(4.10)]{GorbachevTikhonov}) which asserts that
\begin{equation}\label{HL}
	\|f\|_{L_p(\mathbb{R}^d)} \asymp \left(\int_0^\infty t^{d p - d - 1} F_0^p(t) dt\right)^{1/p},\qquad  \frac{2d}{d+1} < p < \infty.
\end{equation}
 For the one-dimensional case and monotone functions, see \cite{Tit, Boas, Sagher}.

(ii) The interesting cases in (\ref{3.3.3}) and (\ref{3.3.4}) are $b \geq -1/q$ and $b > 0$, respectively. Otherwise, it is readily seen that both sides in (\ref{3.3.3}) and (\ref{3.3.4}) are equivalent to $\|f\|_{L_p(\mathbb{R}^d)}$ (see Section \ref{Section 2.2}).

(iii) Note that if $q=p$ then the right-hand side of (\ref{3.3.3}) is equivalent to
  \begin{equation}\label{3.3.3+}
 \left(\int_0^1 t^{dp - d -1} F_0^p(t) dt\right)^{1/p} + \left(\int_1^\infty t^{ d q -dq/p-1} (1 + \log t)^{(b+1/q) q} F_0^q(t) dt\right)^{1/q}, \end{equation}
 showing the natural jump for the logarithmic exponent when $s=0$, cf. (\ref{3.3.1}) and (\ref{BesovZero}).
 This is not the case when $q \neq p$.
   Indeed, if $q < p$ the example $F_0 (t) = t^{-d + d/p} (1 + |\log t|)^{-\beta}$ with $b + 1/p + 1/q < \beta < b + 2/q$ shows that (\ref{3.3.3}) holds but the second integral in (\ref{3.3.3+}) diverges.
    In the case $p < q$, one can take the function $F_0(t) = t^{-d + d/p} (1 + |\log t|)^{-\beta}$ with $\max\{b + 2/q, 1/p\} < \beta < b + 1/p + 1/q$.
\end{rem}

\begin{proof}[Proof of Theorem \ref{Theorem 3.2}]
        We shall use the following description of the modulus of smoothness in terms of the Fourier transform for functions in the $\widehat{GM}^d$ class (see \cite[Corollary 4.1 and
    (7.6)]{GorbachevTikhonov})
    \begin{equation*}
        \omega^p_{k}(f,t)_p \asymp t^{ k p} \int_0^{1/t} u^{ k p + dp - d
        -1} F_0^p(u) du + \int_{1/t}^\infty u^{dp - d -1} F^p_0(u)
        du
    \end{equation*}
    where $k \in \mathbb{N}$ if $d=1$ and $k$ is even if $d \geq 2$. Assume further that $k > s$. Then, we have
    \begin{align*}
        |f|_{\textbf{B}^{s,b}_{p,q}(\mathbb{R}^d)} 
        & \asymp \left(\int_0^1 \left[t^{k-s}(1 - \log t)^b  \left(\int_0^{1/t} u^{k p + dp -d-1} F_0^p(u) du\right)^{1/p}\right]^q
        \frac{dt}{t}\right)^{1/q} \\
        & \hspace{1cm}+ \left(\int_0^1 \left[t^{-s} (1 - \log t)^b \left(\int_{1/t}^\infty u^{d p -d-1} F_0^p(u) du\right)^{1/p}\right]^q
        \frac{dt}{t}\right)^{1/q} \\
        & \asymp \left(\int_0^1 t^{k p + dp -d-1} F_0^p(t)
        dt\right)^{1/p}
        \\& \hspace{1cm}
        + \left(\int_1^\infty \left[t^{s-k}(1 + \log t)^b \left(\int_1^t u^{k p + dp -d-1} F_0^p(u) du\right)^{1/p}\right]^q
        \frac{dt}{t}\right)^{1/q} \\
        & \hspace{1cm}+ \left(\int_1^\infty \left[t^s (1 + \log t)^b \left(\int_t^\infty u^{d p -d-1} F_0^p(u) du\right)^{1/p}\right]^q
        \frac{dt}{t}\right)^{1/q}\\
        & =: I_1 + I_2 + I_3.
    \end{align*}
    First, we point out that
$\displaystyle
        I_1 \leq \left(\int_0^1 t^{dp -d-1} F_0^p(t)
        dt\right)^{1/p}.$
    Next we show that
    \begin{equation}\label{3.7new}
     \quad   I_2 \lesssim \left(\int_{1/2}^1 t^{dp-d-1} F^p_0(t)
        dt\right)^{1/p} + \left(\int_1^\infty (t^{s+d-d/p} (1 + \log t)^b F_0(t))^q
        \frac{dt}{t}\right)^{1/q}.
    \end{equation}
    Indeed, if $q \geq p$ then we can apply Hardy's inequality (see (\ref{HardyInequal1}) and Remark \ref{RemarkHardy}(i)) to derive
    \begin{equation*}
        I_2  \lesssim \left(\int_1^\infty (t^{s+d-d/p} (1 + \log t)^b F_0(t))^q
        \frac{dt}{t}\right)^{1/q}.
    \end{equation*}
    If $q < p$, for a given $ t > 1$, we find  $j_0 \in \mathbb{N}_0$ such that $2^{j_0} \leq t < 2^{j_0 + 1}$ and then, by using the monotonicity property (\ref{3.2}) of $GM$ functions, we obtain
    \begin{equation}
\begin{aligned}
        \left(\int_1^t u^{ k p + d p - d -1} F^p_0(u)
        du\right)^{1/p} &\lesssim \left(\sum_{j=0}^{j_0} 2^{j( k + d -d/p) p}
        F^p_0(2^j)\right)^{1/p}  \\
        &\hspace{-4.5cm}
         \leq \left(\sum_{j=0}^{j_0} 2^{j( k + d -d/p) q}
        F^q_0(2^j)\right)^{1/q} 
        \lesssim \left(\int_{1/2}^t u^{ k q + dq -dq/p-1} F^q_0(u)
        du\right)^{1/q}. \label{newestimate}
   \end{aligned}
\end{equation}
    Applying this latter estimate, H\"older's inequality  and changing the order of
    integration, we have
    \begin{align*}
        I_2 & \lesssim \left(\int_1^\infty t^{sq- k q} (1 + \log t)^{b q} \int_{1/2}^t u^{ k q + dq -dq/p} F^q_0(u)
        \frac{du}{u} \frac{dt}{t} \right)^{1/q} \\
        & \asymp \left(\int_{1/2}^1 u^{ k q + dq -dq/p} F^q_0(u)
        \frac{du}{u} \right)^{1/q} \\
        & \hspace{1cm}+ \left(\int_1^\infty t^{sq- k q} (1 + \log t)^{bq} \int_1^t u^{  k q+ dq -dq/p} F^q_0(u) \frac{du}{u}
        \frac{dt}{t}\right)^{1/q} \\
        & \lesssim \left(\int_{1/2}^1 u^{dp-d-1} F^p_0(u)
        du\right)^{1/p} + \left(\int_1^\infty (u^{s+ d-d/p} (1 + \log u)^b
        F_0(u))^q
        \frac{du}{u}\right)^{1/q}
    \end{align*}
    and so, (\ref{3.7new}) follows.

   Further,  H\"older's inequality yields
    \begin{align}
        \int_t^\infty \frac{F_0(u)}{u} du
         \lesssim t^{-d + d/p} \left(\int_t^\infty F_0^p(u) u^{dp-d-1}
        du\right)^{1/p}. \label{3.7}
    \end{align}
   This and (\ref{3.3}) imply
    \begin{align}
        \left(\int_1^\infty (t^{s+d-d/p} (1 + \log t)^b F_0(t))^q
        \frac{dt}{t}\right)^{1/q} \nonumber \\
        & \hspace{-5cm}\lesssim \left(\int_1^\infty \left(t^{s+d-d/p} (1 + \log t)^b \int_{t/c}^\infty \frac{F_0(u)}{u} du\right)^q
        \frac{dt}{t}\right)^{1/q} \nonumber \\
        & \hspace{-5cm}\lesssim \left(\int_1^\infty \left[t^s (1 + \log t)^b \left(\int_{t/c}^\infty F_0^p(u) u^{dp-d-1}
        du\right)^{1/p}\right]^q \frac{dt}{t}\right)^{1/q} \nonumber \\
        & \hspace{-5cm}\lesssim \left(\int_0^\infty t^{dp -d-1} F_0^p(t)
        dt\right)^{1/p} + I_3. \label{Section6:new}
    \end{align}
    Thus, taking into account  (\ref{3.7new}), we arrive at
    \begin{equation*}
        I_2 \lesssim \left(\int_0^\infty t^{dp -d-1} F_0^p(t)
        dt\right)^{1/p} + I_3.
    \end{equation*}
In light of (\ref{HL}), we have that, for any $s\ge 0$,
    \begin{align}
        \|f\|_{\textbf{B}^{s,b}_{p,q}(\mathbb{R}^d)} & = \|f\|_{L_p(\mathbb{R}^d)} +
        |f|_{\textbf{B}^{s,b}_{p,q}(\mathbb{R}^d)} \nonumber\\
        & \asymp \left(\int_0^\infty t^{d p -d-1} F_0^p(t)
        dt\right)^{1/p} + I_3 \nonumber\\
        & \asymp \left(\int_0^1 t^{d p -d-1} F_0^p(t)
        dt\right)^{1/p} + I_3, \label{newestimate2}
    \end{align}
which is the  desired characterization for  $s=0$ (cf. (\ref{3.3.3})).

    Suppose that $s > 0$. Let us verify the following estimate
    \begin{equation}\label{3.9new}\;
        I_3 \lesssim  \left(\int_0^1 t^{d p -d-1} F_0^p(t)
        dt\right)^{1/p} + \left(\int_1^\infty t^{sq + d q -dq/p-1} (1 + \log t)^{b q} F_0^q(t)
        dt\right)^{1/q}.
    \end{equation}
    If $q \geq p$, it follows from Hardy's inequality (cf. (\ref{HardyInequal2}) and Remark \ref{RemarkHardy}(i)) that
    \begin{equation*}
        I_3 \lesssim \left(\int_1^\infty t^{sq+dq-dq/p-1} (1 + \log t)^{bq} F^q_0(t)
        dt\right)^{1/q}
    \end{equation*}
    which obviously implies (\ref{3.9new}).
    If $q < p$, there exists $c>1$ such that
    \begin{equation}\label{ProofBesovGM}
        \left(\int_t^\infty u^{dp-d-1} F^p_0(u) du\right)^{1/p}
        \lesssim \left(\int_{t/c}^\infty u^{dq -dq/p-1} F^q_0(u)
        du\right)^{1/q}.
    \end{equation}
    This can be shown
    applying a similar argument we used to obtain (\ref{newestimate}).
    This implies, together with Fubini's theorem and H\"older's inequality, that
    \begin{align*}
        I_3 & \lesssim \left(\int_1^\infty t^{sq} (1 + \log t)^{bq} \int_{t/c}^\infty u^{dq-dq/p-1} F^q_0(u) du
        \frac{dt}{t}\right)^{1/q} \\
        & \lesssim \left(\int_{1/c}^\infty u^{sq+dq-dq/p-1} (1 + \log cu)^{bq} F^q_0(u)
        du\right)^{1/q} \\
        & \lesssim \left(\int_0^1 u^{sq+dq-dq/p-1} 
        F^q_0(u)
        du\right)^{1/q} \\
        & \hspace{1cm}+ \left(\int_1^\infty u^{sq+dq-dq/p-1} (1 + \log u)^{bq} F^q_0(u)
        du\right)^{1/q} \\
        & \lesssim \left(\int_0^1 u^{d p- d-1} F^p_0(u) du\right)^{1/p} \\
        & \hspace{1cm}+ \left(\int_1^\infty u^{sq+dq-dq/p-1} (1 + \log u)^{bq} F^q_0(u)
        du\right)^{1/q} .
    \end{align*}
    Hence, (\ref{3.9new}) also holds in this case.

Combining (\ref{newestimate2}) and (\ref{3.9new}), we obtain
   \begin{equation*}
   	  \|f\|_{\textbf{B}^{s,b}_{p,q}(\mathbb{R}^d)} \lesssim \left(\int_0^1 t^{d p -d-1} F_0^p(t)
        dt\right)^{1/p} + \left(\int_1^\infty t^{sq+dq-dq/p-1} (1 + \log t)^{bq} F^q_0(t)
        dt\right)^{1/q}.
   \end{equation*}

   Finally, the converse estimate follows from (\ref{Section6:new}) and (\ref{newestimate2}).
\end{proof}

\begin{rem}\label{Remark 3.3}
    Let us consider a slightly wider  class (\cite{GorbachevLiflyandTikhonov, LiflyandTikhonov}) of general monotone functions $\varphi(z)$  which are locally of bounded variation, vanishes at
    infinity, and there is a constant $c > 1$ depending on $\varphi$
    such that
    \begin{equation}\label{3.12}
        \int_z^\infty |d \varphi(u)| \lesssim \int_{z/c}^\infty
        \frac{|\varphi(u)|}{u} du < \infty,\qquad  z >0.
    \end{equation}
Noting that such  functions  satisfy the monotonicity condition (\ref{3.3}), we claim that  Theorem \ref{Theorem 3.2}  holds in the case $q \geq p$
 if we consider the  class formed
    by all
    radial functions $f(x) = f_0(|x|), x \in \mathbb{R}^d,$ with the Fourier transform $F_0 = \widehat{f} \geq 0$
    satisfying
    conditions (\ref{3.12}) and (\ref{3.4new}).
    In the particular case  $s > 0$ and $b=0$, this was
    proved in \cite[Theorem 7.3]{GorbachevTikhonov}.
\end{rem}

\subsection{Characterization of spaces $B^{s,b}_{p,q}(\mathbb{R}^d)$}

The goal of this section is to obtain the counterpart of Theorem \ref{Theorem 3.2} for spaces $B^{s,b}_{p,q}(\mathbb{R}^d)$, that is, the description of functions from $B^{s,b}_{p,q}(\mathbb{R}^d)$ in terms of the growth properties of their Fourier transforms. If $s > 0$ this follows from Theorem \ref{Theorem 3.2} and the fact that  Besov spaces $\mathbf{B}^{s,b}_{p,q}(\mathbb{R}^d)$ and $B^{s,b}_{p,q}(\mathbb{R}^d)$ coincide (cf. (\ref{BesovComparison})). However, one can not apply Theorem \ref{Theorem 3.2} to characterize $B^{s,b}_{p,q}(\mathbb{R}^d)$ with $s \leq 0$. We are able to overcome this obstacle  with the help of the lifting property of the Besov spaces $B^{s,b}_{p,q}(\mathbb{R}^d)$. Let us  recall this property. For further details, we refer to \cite[Theorem 2.3.8]{Triebel1} for the case $b=0$ and to \cite[Proposition 1.8]{Moura} for the general case  $b \in \mathbb{R}$.

For $\sigma \in \mathbb{R}$, the lifting operator is defined by\index{\bigskip\textbf{Operators}!$I_\sigma$}\label{LIFT}
\begin{equation}\label{liftdef}
	I_\sigma f = ((1 + |\xi|^2)^{\sigma/2} \widehat{f})^\vee, \qquad f \in \mathcal{S}'(\mathbb{R}^d).
\end{equation}

\begin{lem}\label{LemmaLift}
	Let $1 \leq p \leq \infty, 0 < q \leq \infty$, and $s, b, \sigma \in \mathbb{R}$. Then, $I_\sigma$ maps $B^{s,b}_{p,q}(\mathbb{R}^d)$ isomorphically onto $B^{s-\sigma,b}_{p,q}(\mathbb{R}^d)$, and
	\begin{equation}\label{lift}
		\|I_\sigma f\|_{B^{s-\sigma,b}_{p,q}(\mathbb{R}^d)} \asymp \|f\|_{B^{s,b}_{p,q}(\mathbb{R}^d)}.
	\end{equation}
\end{lem}

\begin{thm}\label{Theorem 3.6}
	Let $\frac{2d}{d+1} < p < \infty, 0 < q \leq \infty$, and $s,b \in \mathbb{R}$. Assume that $f \in \widehat{GM}^d$. Then
	\begin{equation}\label{BesovGM}
		\|f\|_{B^{s,b}_{p,q}(\mathbb{R}^d)} \asymp \left(\int_0^1 t^{d p-d - 1} F_0^p(t) dt\right)^{1/p} + \left(\int_1^\infty t^{s q +  dq -d q/p - 1} (1 + \log t)^{b q} F_0^q(t) dt\right)^{1/q}.
	\end{equation}
\end{thm}

\begin{rem}
	The characterizations of $\mathbf{B}^{0,b}_{p,q}(\mathbb{R}^d)$ and $B^{0,b}_{p,q}(\mathbb{R}^d)$ given by (\ref{3.3.3}) and (\ref{BesovGM}), respectively, are of a different kind. In particular, for the function $F_0(t) = t^{-d + d/p} (1 + |\log t|)^{\beta}$
with $\max\{1/p,b+1/q\} < \beta \leq b + 1/q + 1/p$,  the right-hand sides of  (\ref{3.3.3}) and (\ref{BesovGM}) are finite and infinite, respectively.
 Further relationships between these two kinds of Besov spaces of smoothness near zero for $GM$ functions will be given in Subsection \ref{Embeddings BB} below.
\end{rem}

\begin{proof}[Proof of Theorem \ref{Theorem 3.6}]
	If $s > 0$ the characterization (\ref{BesovGM}) follows from (\ref{3.3.1}). Then, it remains to show (\ref{BesovGM}) with $s \leq 0$. Let $\sigma < s$. Applying the lifting property (\ref{lift}) and (\ref{BesovComparison}), we get
	\begin{equation*}
		\|f\|_{B^{s,b}_{p,q}(\mathbb{R}^d)} \asymp \|I_{\sigma} f\|_{B^{s-\sigma,b}_{p,q}(\mathbb{R}^d)} \asymp \|I_{\sigma} f\|_{\mathbf{B}^{s-\sigma,b}_{p,q}(\mathbb{R}^d)}
	\end{equation*}
	and then, the result can be obtained by using Lemma \ref{Lemma 3.1} and (\ref{3.3.1}).
\end{proof}

\subsection{Characterization of the space $H^{s,b}_p(\mathbb{R}^d)$}
In this subsection, we give a complete  description of
$\widehat{GM}^d$-functions from the Sobolev  space
$H^{s,b}_p(\mathbb{R}^d)$ in terms of decay properties of their Fourier transforms.

\begin{thm}\label{Theorem 3.9}
    Let $\frac{2d}{d+1} < p < \infty$, and $s,b \in \mathbb{R}$. For $f \in
    \widehat{GM}^d$, we have
    \begin{equation}\label{W}
        \|f\|_{H^{s,b}_p(\mathbb{R}^d)} \asymp \left(\int_0^1 t^{dp - d -1} F_0^p(t) dt\right)^{1/p} + \left(\int_1^\infty t^{s p + d p - d - 1} (1 + \log  t)^{b p}
        F_0^p(t) dt \right)^{1/p}.
    \end{equation}
\end{thm}

\begin{proof}
Taking into account Lemma \ref{Lemma 3.1} and using (\ref{HL}), we obtain
      \begin{align*}
        \|f\|^p_{H^{s,b}_p(\mathbb{R}^d)} & \asymp \int_0^\infty t^{d p -
        d-1} (1 + t^2)^{s p/2} (1 + \log (1 + t^2))^{b p} F^p_0(t)
        dt \\
        & \asymp \int_0^1 t^{d p - d-1} F^p_0(t) dt
        + \int_1^\infty t^{s p + d p -d-1} (1 + \log t)^{b p}
        F^p_0(t) dt.
    \end{align*}

\end{proof}

\begin{rem}
	According to \cite[Theorem 4.4 and Corollary 4.6(iii)]{CaetanoLeopold}, we have
	\begin{equation}\label{CL}
		  H^{s,b}_p(\mathbb{R}^d) \hookrightarrow L_p(\mathbb{R}^d) \iff
    \left\{\begin{array}{lcl}
                            s > 0 & \text{ and }  & b \in \mathbb{R}, \\
                            & & \\
                            s=0 & \text{ and } & b \geq 0.
            \end{array}
            \right.
	\end{equation}
In fact,
 using Theorem \ref{Theorem 3.9},
we can easily obtain, to a certain extent, a sharper result
\begin{equation}\label{CL++}	\widehat{GM}^d \cap H^{s,b}_p(\mathbb{R}^d) \hookrightarrow L_p(\mathbb{R}^d) \iff
    \left\{\begin{array}{lcl}
                            s > 0 & \text{ and }  & b \in \mathbb{R}, \\
                            & & \\
                            s=0 & \text{ and } & b \geq 0,
            \end{array}
            \right.
	\end{equation}
which, in particular, implies the only-if part in (\ref{CL}).
To show (\ref{CL++}), assume that $s=0$ and $b < 0$. Set
		\begin{equation*}
    g(t) = \left\{\begin{array}{lcl}
                            (1 - \log t)^{-\beta_0} & ,  & 0 < t  < 1, \\
                            & & \\
                            (1 + \log t)^{-\beta_1} & , & t \geq 1,
            \end{array}
            \right.
	\end{equation*}
	where $\beta_0 > 1/p$ and $b+1/p < \beta_1 \leq 1/p$. Then, by Theorem \ref{Theorem 3.9} and (\ref{HL}) we have that the function $f(x) = f_0(|x|)$ given by $F_0(t) = t^{-d + d/p} g(t)$ via (\ref{3.4new+}) satisfies that $f \in H^{0,b}_p(\mathbb{R}^d)$ but $f \not \in L_p(\mathbb{R}^d)$.	 
	If $s < 0$, then the function $F_0(t) = t^{-d + d/p + \varepsilon}$ with $0 < \varepsilon < -s$ allows us to construct the desired counterexample showing that $H^{s,b}_p(\mathbb{R}^d) \not \hookrightarrow L_p(\mathbb{R}^d)$.
		
Conversely, we have that $\widehat{GM}^d \cap H^{s,b}_p(\mathbb{R}^d) \hookrightarrow L_p(\mathbb{R}^d)$ under one of the conditions given in the right-hand side of (\ref{CL}). Indeed, since for any $f\in \widehat{GM}^d \cap H^{s,b}_p(\mathbb{R}^d)$  the right-hand side of (\ref{W}) is finite, we have that
$\displaystyle
		\left(\int_0^\infty t^{dp - d -1} F_0^p(t) dt\right)^{1/p} < \infty,
	$ 
	or, equivalently, $f \in L_p(\mathbb{R}^d)$ (see (\ref{HL})).
\end{rem}

\subsection{Embeddings of $B^{0,b}_{p,q}(\mathbb{R}^d)$ in $L_p(\mathbb{R}^d)$}

If $s > 0$ then it is clear that $B^{s,b}_{p,q}(\mathbb{R}^d)$ is continuously embedded into $L_p(\mathbb{R}^d)$. In the limiting case when $s=0$, such an embedding holds under certain assumptions on the parameters. Namely, it was shown in \cite[Theorem 4.3 and Corollary 4.6(i)]{CaetanoLeopold} (with the forerunner  \cite[Theorem 3.3.2 and Corollary 3.3.1]{SickelTriebel}) that
\begin{equation}\label{Loc}
  B^{0,b}_{p,q}(\mathbb{R}^d) \hookrightarrow L_p(\mathbb{R}^d) \iff  \left\{\begin{array}{lcl}
                            b \geq 0 & \text{ if }  & 0 < q \leq
                            \min\{2,p\}, \\
                            & & \\
                            b > 1/p - 1/q & \text{ if } & 1 < p \leq
                            2 \text{ and } p < q \leq \infty, \\
                            & & \\
                            b > 1/2 -1/q & \text{ if } & 2 < p <
                            \infty \text{ and } 2 < q \leq \infty.
            \end{array}
            \right.
\end{equation}

Our following result improves the range of parameters given in (\ref{Loc}) for which $B^{0,b}_{p,q}(\mathbb{R}^d)$ is embedded into $L_p(\mathbb{R}^d)$ when we work with $GM$ functions.

\begin{thm}\label{TheoremEmbGM}
	Let $\frac{2d}{d+1} < p < \infty$, and $0 < q \leq \infty$. Then
	\begin{equation}\label{Lp}
	\widehat{GM}^d \cap B^{0,b}_{p,q}(\mathbb{R}^d) \hookrightarrow L_p(\mathbb{R}^d) \iff
    \left\{\begin{array}{lcl}
    				 b \geq 0 & \text{ if } &q \leq
                            p , \\

                            & & \\
                              b > 1/p - 1/q & \text{ if }  & p < q  .
            \end{array}
            \right.
	\end{equation}
	
\end{thm}
\begin{rem} In contrast to (\ref{Loc}), necessary and sufficient conditions given by (\ref{Lp})
  do not depend on the conditions $p \leq 2$ or $p>2$.
  In particular, if $2 = \min\{2,p,q\}$ then, by (\ref{Lp}), the embedding $\widehat{GM}^d \cap B^{0,b}_{p,q}(\mathbb{R}^d) \hookrightarrow L_p(\mathbb{R}^d)$ holds for a wider range of the logarithmic smoothness parameter than for the embedding
  $B^{0,b}_{p,q}(\mathbb{R}^d) \hookrightarrow L_p(\mathbb{R}^d)$. 
    More specifically, we have
	\begin{equation*}
		B^{0,b}_{p,q}(\mathbb{R}^d) \hookrightarrow L_p(\mathbb{R}^d) \quad \text{if} \quad b > 1/2-1/q = 1/\min\{2,p,q\}-1/q,
	\end{equation*}
	whereas
		\begin{equation*}
		\widehat{GM}^d \cap B^{0,b}_{p,q}(\mathbb{R}^d) \hookrightarrow L_p(\mathbb{R}^d) \quad \text{if} \quad b > 1/\min\{p,q\}-1/q .
	\end{equation*}
\end{rem}
\begin{proof}[Proof of Theorem \ref{TheoremEmbGM}]
	It is not hard to check that if any of two conditions given in the right-hand side of (\ref{Lp}) holds, then
	\begin{align*}
		\left(\int_1^\infty t^{dp-d-1} F^p_0(t) dt\right)^{1/p} &\lesssim \left(\int_{\frac{1}{2}}^1 t^{d p-d -1} F^p_0(t) dt\right)^{1/p}  \\
		& \hspace{1cm}+ \left(\int_1^\infty t^{d q -d q/p - 1} (1 + \log t)^{b q} F_0^q(t) dt\right)^{1/q}.
	\end{align*}
	Then, the embedding $\widehat{GM}^d \cap B^{0,b}_{p,q}(\mathbb{R}^d) \hookrightarrow L_p(\mathbb{R}^d)$ follows from Theorem \ref{Theorem 3.6} and (\ref{HL}).
	
	Conversely, we prove that the embedding $\widehat{GM}^d \cap B^{0,b}_{p,q}(\mathbb{R}^d) \hookrightarrow L_p(\mathbb{R}^d)$ does not hold if the involved parameters do not fulfilled one of the conditions stated in the right-hand side of (\ref{Lp}). Indeed, assume first that $p < q$ and $b = 1/p - 1/q$.  Let
	\begin{equation*}
		 F_0(t) = \left\{\begin{array}{lcl}
                            t^{-d + d/p + \varepsilon} & ,  & 0 < t  < 1, \\
                            & & \\
                            t^{-d + d/p} (1 + \log t)^{-1/p} (1+\log (1 + \log t))^{-\delta} & , & t \geq 1,
            \end{array}
            \right.
	\end{equation*}
	where $0 < \varepsilon < d - d/p$ and $1/q < \delta < 1/p$. It is easy to check that
the function
$f(x)=f_{0}(|x|)$, where $f_{0}$ is given by (\ref{3.4new+}),
 belongs to $B^{0,1/p-1/q}_{p,q}(\mathbb{R}^d)$ (see (\ref{BesovGM})) but not to $L_p(\mathbb{R}^d)$ (see (\ref{HL})).
	
	Moreover, the embedding given in (\ref{Lp}) is not valid  for $b < 1/p -1/q$ and $p < q$ too, because $B^{0,1/p-1/q}_{p,q}(\mathbb{R}^d) \hookrightarrow B^{0,b}_{p,q}(\mathbb{R}^d)$ and so, we can apply the preceding result.
	
	Assume now that the embedding $\widehat{GM}^d \cap B^{0,b}_{p,q}(\mathbb{R}^d) \hookrightarrow L_p(\mathbb{R}^d)$ holds for $p \geq q$. Then, we shall prove that $b \geq 0$. We define the sequences of functions given by $(F_0)_n  = \chi_{(0,n)}$ and $f_n (x)= (f_0)_n(|x|)$ satisfying (\ref{3.4new+}) for $n \in \mathbb{N}$. By our assumption and using (\ref{BesovGM}) and (\ref{HL}), we have
	\begin{equation*}
		 n^{d-d/p} 	\asymp \|f_n\|_{L_p(\mathbb{R}^d)} \lesssim \|f_n\|_{B^{0,b}_{p,q}(\mathbb{R}^d)} \asymp n^{d-d/p} (1 + \log n)^b.
	\end{equation*}
	Consequently, we conclude that $b \geq 0$.
\end{proof}

\subsection{Embeddings between $H^{s,b}_p(\mathbb{R}^d)$ and $B^{s,b}_{p,q}(\mathbb{R}^d)$}
Putting $r=2$ in Proposition \ref{RecallEmb}(i)-(ii), we derive the classical embeddings between Besov and Sobolev spaces. Namely, let $1 < p < \infty, 0 < q \leq \infty$, and $-\infty < s, b < \infty$, then
\begin{equation}\label{4.28new}
	H^{s,b}_p(\mathbb{R}^d) \hookrightarrow B^{s,b}_{p,q}(\mathbb{R}^d) \quad \text{if and only if} \quad q \geq \max\{p,2\}
\end{equation}
and
\begin{equation}\label{4.29new}
	B^{s,b}_{p,q}(\mathbb{R}^d)  \hookrightarrow H^{s,b}_p(\mathbb{R}^d) \quad \text{if and only if} \quad q \leq \min\{p,2\}.
\end{equation}

Theorems \ref{Theorem 3.6} and \ref{Theorem 3.9} allow us to supplement (\ref{4.28new}) and (\ref{4.29new}).
 In more detail,
   the range of parameters $p$ and $q$ in the following result is wider than the
corresponding range in (\ref{4.28new}) and (\ref{4.29new}).

\begin{thm}\label{Theorem 3.10new}
    Let $\frac{2d}{d+1} < p < \infty, 0 < q \leq \infty$, and $-\infty < s , b < \infty$. Then
    \begin{equation}\label{emb1new}
    	 \widehat{GM}^d \cap H^{s,b}_p(\mathbb{R}^d) \hookrightarrow
        B^{s,b}_{p,q}(\mathbb{R}^d) \quad \text{ if and only if } \quad q \geq p
    \end{equation}
    and
     \begin{equation}\label{emb3new}
        \widehat{GM}^d \cap B^{s,b}_{p,q}(\mathbb{R}^d) \hookrightarrow
        H^{s,b}_p(\mathbb{R}^d) \quad \text{ if and only if }  \quad q \leq p.
    \end{equation}
    In particular, we have
    \begin{equation*}
     	\widehat{GM}^d \cap H^{s,b}_p(\mathbb{R}^d) = \widehat{GM}^d \cap B^{s,b}_{p,q}(\mathbb{R}^d) \quad \text{ if and only if } \quad q=p.
    \end{equation*}
\end{thm}
\begin{proof} We will prove only  (\ref{emb1new}). To verify (\ref{emb3new}), we can proceed similarly. Using the monotonicity properties of $GM$-functions (\ref{3.2}) and the assumption $q \geq p$ we have
	\begin{align*}
		F_0^q(1)+\int_1^\infty t^{s q + d q - d q/p -1} (1+ \log t)^{b q} F_0^q(t) dt & \asymp \sum_{j=0}^\infty  2^{j(s q + d q - d q/p)} (1+ j)^{b q} F_0^q(2^j) \\
		& \hspace{-4cm}\leq \left(\sum_{j=0}^\infty  2^{j(s + d - d /p)p} (1+ j)^{b p} F_0^p(2^j)\right)^{q/p} \\
		& \hspace{-4cm}\asymp F_0^q(1)+ \left(\int_1^\infty t^{s p + d p -d -1} (1 + \log t)^{b p} F_0^p(t) dt\right)^{q/p}.
	\end{align*}
	Then  embedding given in (\ref{emb1new}) follows from the characterizations (\ref{BesovGM}) and (\ref{W}).
	
	Next we treat the only-if-part, that is, if the embedding $ \widehat{GM}^d \cap H^{s,b}_p(\mathbb{R}^d) \hookrightarrow
        B^{s,b}_{p,q}(\mathbb{R}^d)$ holds then $q \geq p$. Indeed, assume that $q < p$. We define
        		\begin{equation*}
    F_0(t) = \left\{\begin{array}{lcl}
                            t^{\varepsilon -d+d/p} & ,  & 0 < t  < 1, \\
                            & & \\
                            t^{-s-d+d/p} (1 + \log t)^{-\beta} & , & t \geq 1,
            \end{array}
            \right.
	\end{equation*}
        with $\varepsilon > 0$ and $b+1/p < \beta \leq b + 1/q$. It is readily seen that the function $f(x) = f_0(|x|)$ where $f_0$ is given by (\ref{3.4new+}), satisfies that $f \in H^{s,b}_p(\mathbb{R}^d)$ and $f \not \in B^{s,b}_{p,q}(\mathbb{R}^d)$. This leads to a contradiction and so, $q \geq p$.
\end{proof}

\subsection{Embeddings between $H^{0,b}_p(\mathbb{R}^d)$ and $\mathbf{B}^{0,b}_{p,q}(\mathbb{R}^d)$}

Applying the obtained characterizations in Theorems \ref{Theorem 3.2} and \ref{Theorem 3.9}, here we  establish new embeddings between Besov and Sobolev spaces, which improve those obtained in Corollary \ref{Corollary 2.3}.
 More specifically, 
   the range of parameters $p$ and $q$ in the following result is wider than the
corresponding range in Corollary \ref{Corollary 2.3}.

\begin{thm}\label{Theorem 3.10}
    Let $\frac{2d}{d+1} < p < \infty, 0 < q \leq \infty,$ and $b > -1/q$. Then
    \begin{equation}\label{emb2}
        \widehat{GM}^d \cap H^{0,b+1/q}_p(\mathbb{R}^d) \hookrightarrow
        \emph{\textbf{B}}^{0,b}_{p,q}(\mathbb{R}^d) \quad \text{ if and only if } \quad q \geq p
    \end{equation}
    and
    \begin{equation}\label{emb4}
        \widehat{GM}^d \cap \emph{\textbf{B}}^{0,b}_{p,q}(\mathbb{R}^d) \hookrightarrow
        H^{0,b+1/q}_p(\mathbb{R}^d) \quad \text{ if and only if } \quad q \leq p.
    \end{equation}
    In particular, we have
    \begin{equation}\label{sharpemb4}
     	 \widehat{GM}^d \cap H^{0,b+1/q}_p(\mathbb{R}^d)  =
    \widehat{GM}^d \cap \mathbf{B}^{0,b}_{p,q}(\mathbb{R}^d)  \quad \text{ if and only if } \quad q=p.
    \end{equation}
\end{thm}

\begin{rem}
	In Propositions \ref{Proposition 5.1} and \ref{Proposition 5.4} below we will show that $\xi=b+1/q$ in $H^{0,\xi}_p(\mathbb{R}^d)$ is the best possible exponent of the logarithmic smoothness for which the embeddings (\ref{emb2}) and (\ref{emb4}) hold, respectively.
\end{rem}

\begin{proof}[Proof of Theorem \ref{Theorem 3.10}]
    We proceed by showing (\ref{emb2}). Assume that $q \geq p$. Applying Minkowski's inequality for integrals, we get
    \begin{align*}
         \left(\int_1^\infty \left[(1 + \log t)^b \left(\int_t^\infty u^{d p -d-1} F_0^p(u) du\right)^{1/p}\right]^q
        \frac{dt}{t}\right)^{1/q} \\
        & \hspace{-7cm}\leq \left(\int_1^\infty u^{dp - d- 1} F^p_0(u) \left(\int_1^u (1 + \log t)^{b q} \frac{dt}{t}\right)^{p/q}
        du\right)^{1/p} \\
        & \hspace{-7cm}\lesssim \left(\int_1^\infty u^{dp - d - 1} (1 + \log u)^{(b+1/q)p} F^p_0(u)
        du\right)^{1/p}.
    \end{align*}
	Hence, the embedding $\widehat{GM}^d \cap H^{0,b+1/q}_p(\mathbb{R}^d) \hookrightarrow
        \mathbf{B}^{0,b}_{p,q}(\mathbb{R}^d)$ follows from (\ref{3.3.3}) and (\ref{W}).

        Conversely, we will show that if $\widehat{GM}^d \cap H^{0,b+1/q}_p(\mathbb{R}^d) \hookrightarrow
        \mathbf{B}^{0,b}_{p,q}(\mathbb{R}^d)$ then $q \geq p$. This will be done by contradiction. Suppose that $q < p$. Let
        \begin{equation*}
        		F_0(t) = t^{-d+d/p} (1 + |\log t|)^{-b-1/p-1/q} (1 + \log (1 + |\log t|))^{-\delta}, \quad t >0,
        \end{equation*}
        with $1/p < \delta < 1/q$, and $f(x) = f_0(|x|)$ where $f_0$ is defined by (\ref{3.4new+}). Accordingly to (\ref{W}) and (\ref{3.3.3}), we have $f \in H^{0,b+1/q}_p(\mathbb{R}^d)$ and $f \not \in  \mathbf{B}^{0,b}_{p,q}(\mathbb{R}^d)$, respectively. This leads to a contradiction.

        Analogously, one can prove (\ref{emb4}).
\end{proof}

\subsection{Embeddings between $B^{0,b}_{p,q}(\mathbb{R}^d)$ and $\mathbf{B}^{0,b}_{p,q}(\mathbb{R}^d)$}\label{Embeddings BB}

Let us recall the relationships between the spaces $B^{0,b}_{p,q}(\mathbb{R}^d)$ and $\mathbf{B}^{0,b}_{p,q}(\mathbb{R}^d)$ given by (\ref{1}).
Let $1 < p < \infty, 0 < q \leq \infty$, and $b > -1/q$. Then, we have
\begin{equation*}
    B^{0,b + 1/\min\{2,p,q\}}_{p,q}(\mathbb{R}^d) \hookrightarrow
    \textbf{B}^{0,b}_{p,q}(\mathbb{R}^d) \hookrightarrow
    B^{0,b+1/\max\{2,p,q\}}_{p,q}(\mathbb{R}^d).
\end{equation*}
Our next result shows that working with $GM$ functions we are able to improve the loss of logarithmic smoothness in the previous embeddings.

\begin{thm}\label{Section 6.8: Theorem}
	Let $\frac{2d}{d+1} < p < \infty, 0 < q \leq \infty$, and $b > -1/q$. Then,
	\begin{equation}\label{EmbBesovGM}
  \widehat{GM}^d\cap  B^{0,b + 1/\min\{p,q\}}_{p,q}(\mathbb{R}^d) \hookrightarrow
    \widehat{GM}^d\cap  \mathbf{B}^{0,b}_{p,q}(\mathbb{R}^d) \hookrightarrow
    \widehat{GM}^d\cap  B^{0,b+1/\max\{p,q\}}_{p,q}(\mathbb{R}^d).
\end{equation}
In particular, we have
\begin{equation}\label{EmbBesovGMNew}
	\widehat{GM}^d\cap  B^{0,b + 1/q}_{p,q}(\mathbb{R}^d) =
    \widehat{GM}^d\cap  \mathbf{B}^{0,b}_{p,q}(\mathbb{R}^d) \quad \text{if and only if} \quad q=p.
\end{equation}
\end{thm}
\begin{proof}
	We start by showing the left-hand side embedding of (\ref{EmbBesovGM}). Let $f \in \widehat{GM}^d$. If $q \geq p$, then an application of Hardy's inequality (see (\ref{HardyInequal5})) together with (\ref{3.3.3}) and (\ref{BesovGM}) leads to
	\begin{align*}
		\|f\|_{\mathbf{B}^{0,b}_{p,q}(\mathbb{R}^d)}  &\lesssim  \left(\int_0^1 t^{d p-d - 1} F_0^p(t) dt\right)^{1/p} + \left(\int_1^\infty t^{dq -d q/p - 1} (1 + \log t)^{(b + 1/p) q} F_0^q(t) dt\right)^{1/q}  \\
		&\asymp \|f\|_{B^{0,b+1/p}_{p,q}(\mathbb{R}^d)}.
	\end{align*}
	
	On the other hand, if $q < p$ we can apply (\ref{ProofBesovGM}) and change the order of integration to get
	\begin{align*}
		\left(\int_1^\infty (1 + \log t)^{b q} \left(\int_t^\infty u^{d p -d-1} F_0^p(u) du\right)^{q/p}
        \frac{dt}{t} \right)^{1/q}\\
        & \hspace{-8cm} \lesssim \left(\int_1^\infty (1 + \log t)^{b q} \int_{t/c}^\infty u^{d q - d q/p - 1} F_0^q(u) du \frac{dt}{t}\right)^{1/q} \\
        & \hspace{-8cm} \lesssim \left(\int_0^1 u^{d p - d -1} F_0^p(u) du \right)^{1/p} + \left(\int_1^\infty u^{d q - d q/p - 1} (1 + \log u)^{(b + 1/q)q} F_0^q(u) du\right)^{1/q},
	\end{align*}
	where we have used the H\"older's inequality in the last estimate. Therefore, by (\ref{3.3.3}) and (\ref{BesovGM}), we have $\|f\|_{\mathbf{B}^{0,b}_{p,q}(\mathbb{R}^d)} \lesssim \|f\|_{B^{0,b+1/q}_{p,q}(\mathbb{R}^d)}$.
	
	Next we proceed with the proof of the right-hand side embedding in (\ref{EmbBesovGM}). Assume first that $p \leq q$. A similar argument we applied to show (\ref{newestimate}) yields that
	\begin{equation*}
		\left(\int_t^\infty u^{d p - d- 1} F_0^p(u) du\right)^{1/p} \gtrsim \left(\int_{c \,t}^\infty u^{d q - d q/p - 1} F_0^q(u) du\right)^{1/q}
	\end{equation*}
	for some $c > 1$. Then a combination of this latter estimate, the monotonicity property (\ref{3.2}), a change of the order of integration, formula (\ref{HL}), and the characterizations (\ref{3.3.3}) and (\ref{BesovGM}) imply
	\begin{align*}
		\|f\|_{B^{0,b+1/q}_{p,q}(\mathbb{R}^d)} & \asymp \left(\int_0^1 t^{d p-d - 1} F_0^p(t) dt\right)^{1/p}
\\ &+ \left(\int_c^\infty t^{ dq -d q/p - 1} (1 + \log \big(\frac{t}{c}\big))^{(b+1/q) q} F_0^q(t) dt\right)^{1/q} \\
		& \hspace{-1cm}\lesssim \left(\int_0^\infty t^{d p-d - 1} F_0^p(t) dt\right)^{1/p} + \left(\int_1^\infty (1 + \log t)^{b q} \int_{c \, t}^\infty u^{d q - d q/p -1} F_0^q(u) du \frac{dt}{t} \right)^{1/q}  \\
		& \hspace{-1cm}\lesssim \|f\|_{L_p(\mathbb{R}^d)} + \left(\int_1^\infty (1 + \log t)^{b q} \left(\int_t^\infty u^{d p - d - 1} F_0^p(u) du\right)^{q/p} \frac{dt}{t}\right)^{1/q} \\
		& \hspace{-1cm}\lesssim \|f\|_{\mathbf{B}^{0,b}_{p,q}(\mathbb{R}^d)}
	\end{align*}
	because $\mathbf{B}^{0,b}_{p,q}(\mathbb{R}^d) \hookrightarrow L_p(\mathbb{R}^d)$. So we have shown the desired estimate if $p \leq q$.
	
	Suppose now that $p > q$. Let $\beta$ be such that $\beta < 1/p - 1/q$. For $u \geq 1$, by H\"older's inequality, we have
	\begin{align*}
		\left(\int_u^\infty t^{d q - d q/p - 1} (1 + \log t)^{\beta q} F_0^q(t) dt \right)^{1/q}  \\		& \hspace{-5cm}\lesssim (1 + \log u)^{\beta + 1/q -1/p} \left(\int_u^\infty t^{d p - d - 1} F_0^p(t) dt\right)^{1/p}.
	\end{align*}
	Applying this estimate together with (\ref{BesovGM}) and (\ref{3.3.3}), one obtains
	\begin{align*}
		\|f\|_{B^{0,b+1/p}_{p,q}(\mathbb{R}^d)}  & \lesssim \left(\int_0^\infty t^{d p-d - 1} F_0^p(t) dt\right)^{1/p} \\
		 &\hspace{-1cm}+ \left(\int_1^\infty t^{dq -d q/p - 1} (1 + \log t)^{\beta q} F_0^q(t) \int_1^t (1 + \log u)^{(b-\beta -1/q + 1/p)q} \frac{du}{u} dt\right)^{1/q} \\
		 &\hspace{-1cm} \asymp \|f\|_{L_p(\mathbb{R}^d)} \\
&\hspace{-0.5cm}+ \left(\int_1^\infty (1 + \log u)^{(b-\beta -1/q + 1/p) q} \int_u^\infty t^{d q - dq/p-1} (1 + \log t)^{\beta q} F_0^q(t) dt \frac{du}{u}\right)^{1/q}
\\
		  &
\hspace{-1cm}\lesssim \|f\|_{L_p(\mathbb{R}^d)} 
+ \left(\int_1^\infty (1 + \log u)^{b q} \left(\int_u^\infty t^{d p - d - 1} F_0^p(t) dt\right)^{q/p} \frac{du}{u}\right)^{1/q} \\
		  &
\hspace{-1cm}\lesssim \|f\|_{\mathbf{B}^{0,b}_{p,q}(\mathbb{R}^d)}
	\end{align*}
	where we used (\ref{HL}) again.
	
	Finally we prove (\ref{EmbBesovGMNew}). It is clear from (\ref{EmbBesovGM}) that $\widehat{GM}^d\cap  B^{0,b + 1/p}_{p,p}(\mathbb{R}^d) =
    \widehat{GM}^d\cap  \mathbf{B}^{0,b}_{p,p}(\mathbb{R}^d)$. It remains to show the only-if-part in (\ref{EmbBesovGMNew}). This will be done by contradiction. Assume, e.g., that $q < p$. Let
	\begin{equation*}
		F_0(t) = t^{-d + d/p} (1 + |\log t|)^{-\beta},\quad t > 0,
	\end{equation*}
	where $b + 1/q + 1/p < \beta < b + 2/q$.
 Define
the function
$f(x)=f_{0}(|x|)$, where $f_{0}$ is given by (\ref{3.4new+}). Accordingly to Theorems \ref{Theorem 3.2} and \ref{Theorem 3.6} we derive that $f \in \mathbf{B}^{0,b}_{p,q}(\mathbb{R}^d)$ but $f \not \in B^{0,b+1/q}_{p,q}(\mathbb{R}^d)$, respectively. Then we arrive at a contradiction. The case $q > p$ can be treated similarly.

\end{proof}

\subsection{Sobolev embeddings}

Recall the Sobolev embedding theorem for Besov spaces stated in Proposition \ref{RecallEmb2}(i): Let $1 < p_0 < p_1 < \infty, -\infty < s_1 < s_0 < \infty$ with
		\begin{equation*}
			s_0 -\frac{d}{p_0} = s_1 -\frac{d}{p_1},
		\end{equation*}
		$0 < q \leq \infty$ and $-\infty < b < \infty$. Then,
		\begin{equation}\label{SobolevEmbedding}
		B^{s_0,b}_{p_0,q}(\mathbb{R}^d) \hookrightarrow B^{s_1,b}_{p_1,q}(\mathbb{R}^d).
		\end{equation}
The goal of this subsection is to show that both spaces coincide
    when we work with general monotone functions. We will use the following notation. For $f \in \widehat{GM}^d, f(x) = f_0(|x|),$ and  any $\tau>0$, set $$J_f(\tau)=\left(\int_{0}^1 t^{d \tau-d -1} F^\tau_0(t) dt\right)^{1/\tau},$$
where $F_0$ is given by (\ref{FourierHankel}).
Note that we have
\begin{equation}\label{Jf}
	J_f(\tau)\lesssim J_f(\varkappa) \quad \text{if} \quad \varkappa\le \tau.
\end{equation}
	 This follows by using a similar discretization we used to show estimate (\ref{newestimate}).

\begin{thm}\label{SobolevEmbeddingImprovement}
	Let $\frac{2d}{d+1} < p_0  < p_1 <\infty, -\infty < s_1 < s_0 < \infty, -\infty < b < \infty$, and $0 < q \leq \infty$.
	\begin{enumerate}[\upshape(i)]
	\item Assume that
	\begin{equation}\label{DiffDim}
	s_0 - \frac{d}{p_0}  = s_1 -\frac{d}{p_1}.
\end{equation}
 Let $f \in \widehat{GM}^d$. If $J_f(p_0)<\infty$, then
	\begin{equation}\label{DiffDim1}
		f \in B^{s_0,b}_{p_0,q}(\mathbb{R}^d) \iff f \in B^{s_1,b}_{p_1,q}(\mathbb{R}^d).
	\end{equation}
	\item Conversely, if
	\begin{equation*}
		\widehat{GM}^d \cap B^{s_0,b}_{p_0,q}(\mathbb{R}^d) = \widehat{GM}^d \cap B^{s_1,b}_{p_1,q}(\mathbb{R}^d),
	\end{equation*}
	then (\ref{DiffDim}) holds.
	\end{enumerate}
\end{thm}
\begin{proof}
	By (\ref{DiffDim}) we have
	\begin{equation*}
		\int_1^\infty t^{s_0 q + d q - d q/p_0 - 1} (1 + \log t)^{b q} F^q_0(t) dt = \int_1^\infty t^{s_1 q + d q - d q/p_1 - 1} (1 + \log t)^{b q} F^q_0(t) dt .
	\end{equation*}
	Then, the equivalence (\ref{DiffDim1}) follows from (\ref{Jf}) and Theorem \ref{Theorem 3.6}.
	
	On the other hand, by contradiction, one can show that $\widehat{GM}^d \cap B^{s_0,b}_{p_0,q}(\mathbb{R}^d) = \widehat{GM}^d \cap B^{s_1,b}_{p_1,q}(\mathbb{R}^d)$ implies (\ref{DiffDim}). For example, assume that $s_0-d/p_0 < s_1 -d/p_1$, and let
        		\begin{equation*}
    F_0(t) = \left\{\begin{array}{lcl}
                            t^{-\varepsilon} & ,  & 0 < t  < 1, \\
                            & & \\
                            t^{-\beta} & , & t \geq 1,
            \end{array}
            \right.
	\end{equation*}
        with $\varepsilon < d - d/p_0$ and $s_0 -d/p_0 + d < \beta < s_1 -d/p_1 +d$. According to Theorem \ref{Theorem 3.6}, we obtain  that the function $f(x) = f_0(|x|)$, where $f_0$ is given by (\ref{3.4new+}), satisfies that $f \in B^{s_0,b}_{p_0,q}(\mathbb{R}^d)$ and $f \not \in B^{s_1,b}_{p_1,q}(\mathbb{R}^d)$, which is a contradiction. Analogously, one can prove that $s_0-d/p_0 > s_1 -d/p_1$ yields a contradiction.
\end{proof}

\begin{rem}\label{Remark 4.15}
	The condition $J_f(p_0) < \infty$ used in Theorem \ref{SobolevEmbeddingImprovement}(i) cannot be dropped. For instance, let $f$ be the radial function whose Fourier-Hankel transform is
	\begin{equation*}
	F_0(t)=
    \left\{\begin{array}{lcl}
    				 t^{-\varepsilon} & \text{ if } & 0 < t \leq 1, \\

                            & & \\
                              t^{-\delta} & \text{ if }  & t >1,
            \end{array}
            \right.
	\end{equation*}
	where $d-d/p_0 \leq \varepsilon < d - d/p_1$ and $\delta > s_1 + d - d/p_1$. Note that $J_f(p_0) = \infty$. Furthermore, applying Theorem \ref{Theorem 3.6} we obtain that $f$ belongs to $B^{s_1,b}_{p_1,q}(\mathbb{R}^d)$ but not to $B^{s_0,b}_{p_0,q}(\mathbb{R}^d)$.
\end{rem}

\subsection{Franke-Jawerth embeddings}\label{subsection4.10}

Let us first recall the Franke-Jawerth embedding theorem for Besov and Sobolev spaces given in Proposition \ref{RecallEmb2}(ii): Let $1 < p_0 < p < p_1 < \infty$ and $-\infty < s_1 < s < s_0 < \infty$ satisfying
\begin{equation*}
	s_0 - \frac{d}{p_0} = s -\frac{d}{p} = s_1 -\frac{d}{p_1}.
\end{equation*}
Let $-\infty < b< \infty$. Then
\begin{equation*}
	B^{s_0,b}_{p_0,p} (\mathbb{R}^d) \hookrightarrow H^{s,b}_p(\mathbb{R}^d) \hookrightarrow B^{s_1,b}_{p_1,p}(\mathbb{R}^d).
\end{equation*}

Quite surprisingly, our next result claims that the previous embeddings become identities between Besov and Sobolev spaces when we deal with $GM$-functions.
\begin{thm}\label{-}
	Let $\frac{2d}{d+1} < p_0 < p < p_1 <\infty, -\infty < s_1 < s < s_0 < \infty$, and $-\infty < b < \infty$.
	\begin{enumerate}[\upshape(a)]
	 \item Let
	\begin{equation*}
	s_0 - \frac{d}{p_0} = s -\frac{d}{p} = s_1 -\frac{d}{p_1}.
\end{equation*}
Suppose that $f \in \widehat{GM}^d$.
\begin{enumerate}[\upshape(i)]
\item If $J_f(p_0)<\infty$, then
	\begin{equation}\label{-1-}
		f \in B^{s_0,b}_{p_0,p}(\mathbb{R}^d) \iff f \in H^{s,b}_{p}(\mathbb{R}^d).
	\end{equation}
\item If $J_f(p) < \infty$, then
	\begin{equation}\label{-2-}
		f \in H^{s,b}_p(\mathbb{R}^d) \iff f \in B^{s_1,b}_{p_1,p}(\mathbb{R}^d).
	\end{equation}
\end{enumerate}
\item Conversely, we have
	\begin{enumerate}[\upshape(i)]
		\item If $\widehat{GM}^d \cap B^{s_0,b}_{p_0,p}(\mathbb{R}^d) = \widehat{GM}^d \cap  H^{s,b}_{p}(\mathbb{R}^d)$ then $s_0-d/p_0 = s-d/p$.
		\item If $\widehat{GM}^d \cap  H^{s,b}_{p}(\mathbb{R}^d)=\widehat{GM}^d \cap B^{s_1,b}_{p_1,p}(\mathbb{R}^d)$ then $s-d/p = s_1-d/p_1$.
	\end{enumerate}
	\end{enumerate}
\end{thm}
\begin{proof}
	We start by showing (a). Since
	\begin{align*}
		\int_1^\infty t^{s_0 p + d p - d p/p_0 - 1} (1 + \log t)^{b p} F^p_0(t) dt & = \int_1^\infty t^{s p + d p - d  - 1} (1 + \log t)^{b p} F^p_0(t) dt \\
		& \hspace{-3.5cm}= \int_1^\infty t^{s_1 p + d p - d p/p_1 - 1} (1 + \log t)^{b p} F^p_0(t) dt
	\end{align*}
	and (\ref{Jf}) holds, the proofs of (i) and (ii) are immediate consequences of Theorems \ref{Theorem 3.9} and \ref{Theorem 3.6}.
	
	Next we prove (b). Let us assume that $\widehat{GM}^d \cap B^{s_0,b}_{p_0,p}(\mathbb{R}^d) = \widehat{GM}^d \cap  H^{s,b}_{p}(\mathbb{R}^d)$. If $s_0 -d/p_0 < s -d/p$, we set
        		\begin{equation*}
    F_0(t) = \left\{\begin{array}{lcl}
                            t^{-\varepsilon} & ,  & 0 < t  < 1, \\
                            & & \\
                            t^{-\beta} & , & t \geq 1,
            \end{array}
            \right.
	\end{equation*}
        where $\varepsilon < d - d/p_0$ and $s_0 -d/p_0 + d < \beta < s -d/p +d$, and $f(x) = f_0(|x|)$ where $f_0$ is given by (\ref{3.4new+}). According to Theorems \ref{Theorem 3.6} and \ref{Theorem 3.9}, we derive that $f \in B^{s_0,b}_{p_0,p}(\mathbb{R}^d)$ and $f \not \in H^{s,b}_{p}(\mathbb{R}^d)$. This contradicts our assumption. On the other hand, a similar argument shows that $s_0 -d/p_0 > s -d/p$ is not possible. Then we arrive at $s_0 -d/p_0 = s-d/p$, which proves (i).

        The proof of (ii) can be carried out similarly.
\end{proof}

\begin{rem}
	Similarly to Remark \ref{Remark 4.15} one can construct counterexamples showing that the conditions $J_f(p_0)< \infty$ and $J_f(p) < \infty$ given in Theorem \ref{-}(a) are indeed necessary.
\end{rem}

\subsection{Characterization of spaces $B^{s,b}_{p,q}(\mathbb{T}), H^{s,b}_p(\mathbb{T}),$ and $ \mathbf{B}^{s,b}_{p,q}(\mathbb{T})$}\label{Section 3.8}

The goal of this subsection is to obtain the periodic counterparts of the results given in Sections
\ref{SubsectionContinuous}--\ref{subsection4.10}.
 This will be done with the help of general monotone sequences (see \cite{Tikhonov, LiflyandTikhonov} and the
references given there). First, we recall the definition of general monotone sequences and some of their properties.

 A sequence $a=\{a_n\}_{n \in \mathbb{N}}$
is called \emph{general monotone}, written $a \in GM$, if there is a
constant $C > 0$ such that
\begin{equation*}
    \sum_{k=n}^{2n-1} |\Delta a_k| \leq C |a_n| \text{ for all } n
    \in \mathbb{N}.
\end{equation*}
Here $\Delta a_k = a_k - a_{k+1}$ and the constant $C$ is
independent of $n$. It is proved in \cite[p. 725]{Tikhonov} that $a \in GM$
if and only if
\begin{equation}\label{3.16}
    |a_\nu| \lesssim |a_n| \text{ for } n \leq \nu \leq 2n
\end{equation}
and
\begin{equation}\label{3.17}
    \sum_{k=n}^N |\Delta a_k| \lesssim |a_n| + \sum_{k=n+1}^N
    \frac{|a_k|}{k} \text{ for any } n < N.
\end{equation}

Note that conditions (\ref{3.16}) and (\ref{3.17})
    imply that there exists $c > 1$ such that
    \begin{equation}\label{3.17new}
        \sum_{k=n}^\infty |\Delta a_k| \lesssim
        \sum_{k=[n/c]}^\infty \frac{|a_k|}{k}
    \end{equation}
    for any $n \in \mathbb{N}$ (see (\ref{3.12})).

The multiplier property for $GM$ sequences reads as follows.
\begin{lem}[\cite{LiflyandTikhonov}]
\label{Section 6.11: Lemma}
	Let $a= \{a_n\}_{n \in \mathbb{N}}, b= \{b_n\}_{n \in \mathbb{N}} \in GM$, then $a b = \{a_n b_n\}_{n \in \mathbb{N}} \in GM$.
\end{lem}

Note that Lemma \ref{Section 6.11: Lemma} can be considered as the discrete counterpart of Lemma \ref{Lemma 3.1}.

We will need the following auxiliary result (see \cite{AskeyWainger}).

\begin{lem}\label{Lemma 3.5}
    Let $1 < p < \infty$, and let $\sum_{n=1}^\infty a_n \cos nx$ be
    the Fourier series of $f \in L_1(\mathbb{T})$.
    \begin{enumerate}[\upshape(i)]
        \item If the sequences $\{a_n\}_{n \in \mathbb{N}}$ and $\{\beta_n\}_{n \in
        \mathbb{N}}$ are such that
      $\displaystyle
            \sum_{k=n}^\infty |\Delta a_k| \lesssim \beta_n, n \in
            \mathbb{N},
      $
        then
        \begin{equation*}
            \|f\|^p_{L_p(\mathbb{T})} \lesssim \sum_{n=1}^\infty
            n^{p-2} \beta_n^p.
        \end{equation*}
        \item If $\{a_n\}_{n \in \mathbb{N}}$ is a nonnegative
        sequence, then
        \begin{equation*}
            \sum_{n=1}^\infty \Big(\sum_{k=[n/2]}^n a_k\Big)^p
            n^{-2} \lesssim \|f\|^p_{L_p(\mathbb{T})}.
        \end{equation*}
    \end{enumerate}
\end{lem}

\begin{thm}\label{TheoremBesovGMPer}
	Let $1 < p < \infty, 0 < q \leq \infty$, and $s,b \in \mathbb{R}$.
Let the Fourier series of
    $f \in L_1(\mathbb{T})$ be given by
	 \begin{equation*}
	 f(x) \sim \sum_{n=1}^\infty (a_n \cos n x + b_n \sin nx)
	 \end{equation*}
    where  $\{a_n\}_{n \in \mathbb{N}}, \{b_n\}_{n \in \mathbb{N}}$ are  nonnegative general monotone sequences.
 Then
	\begin{equation}\label{BesovGMPer}
		\|f\|_{B^{s,b}_{p,q}(\mathbb{T})} \asymp \left(\sum_{n=1}^\infty n^{sq + q -q/p -1} (1 + \log n)^{b q} (a_{n}^q + b_n^q)\right)^{1/q}.
	\end{equation}
\end{thm}

The proof of Theorem \ref{TheoremBesovGMPer} relies on the following (Lizorkin-type) representation of $B^{s,b}_{p,q}(\mathbb{T})$. Further details and references can be found in \cite[Section 3.5.3]{SchmeisserTriebel}.

\begin{lem}\label{LemmaBesovGMPer}
	Let $1 < p < \infty, 0 < q \leq \infty$, and $s,b \in \mathbb{R}$. We have that
	\begin{equation*}
		\|f\|_{B^{s,b}_{p,q}(\mathbb{T})} \asymp \left(\sum_{j=0}^\infty 2^{j s q} (1+j)^{b q} \Big\|\sum_{k \in K_j} \widehat{f}(k) e^{i k x}\Big\|^q_{L_p(\mathbb{T})}\right)^{1/q}
	\end{equation*}
	where $K_0 =\{0\}$ and $K_j = \{k : 2^{j-1} \leq |k| < 2^j\}$ for $j \in \mathbb{N}$.
\end{lem}

\begin{rem}
	Lemma \ref{LemmaBesovGMPer} claims that one can replace the classical smooth resolution of unity $\{\varphi_j\}_{j \in \mathbb{N}_0}$ given by (\ref{SmoothFunction}), (\ref{resolution}) in the definition of the space $B^{s,b}_{p,q}(\mathbb{T})$ (see (\ref{Section2:new})) by characteristic functions of the dyadic blocks $K_j$.
\end{rem}

\begin{proof}[Proof of Theorem \ref{TheoremBesovGMPer}]

	We first remark that it suffices to show (\ref{BesovGMPer}) for $f(x) \sim \sum_{n=1}^\infty a_n \cos n x$ with nonnegative $\{a_n\}_{n \in \mathbb{N}} \in GM$.
	
	We will prove that
	\begin{equation}\label{1new-}
		a_{2^{j+1}} 2^{j(1-1/p)} \lesssim \Big\|\sum_{n=2^j}^{2^{j+1}-1} a_n \cos n x\Big\|_{L_p(\mathbb{T})} \lesssim a_{2^j} 2^{j(1-1/p)}
	\end{equation}
	for $j \in \mathbb{N}_0$, see also \cite{Tikhonov}.
	
	By Lemma \ref{Lemma 3.5}(ii) and (\ref{3.16}) we have
	\begin{align*}
		\Big\|\sum_{n=2^j}^{2^{j+1}-1} a_n \cos n x\Big\|_{L_p(\mathbb{T})}^p & \gtrsim
\sum_{n=[\frac43 2^j]}^{2^{j+1}} \left(\sum_{k=
[\frac34 n]}^n a_k\right)^p n^{-2}  \gtrsim \sum_{n=[\frac43 2^j]}^{2^{j+1}} a_n^p n^{p-2}
 \gtrsim a_{2^{j+1}}^p 2^{j(p-1)}.
	\end{align*}
	This gives the lower estimate in (\ref{1new-}).
To obtain the upper bound, we apply Lemma \ref{Lemma 3.5}(i), together with the definition of $GM$ sequence:
	\begin{equation*}
		\Big\|\sum_{n=2^j}^{2^{j+1}-1} a_n \cos n x\Big\|_{L_p(\mathbb{T})}^p  \lesssim
\left(a_{2^j} + \sum_{n=2^j}^{2^{j+1}} \frac{a_n}{n}\right)^p 2^{j (p-1)}+
\sum_{n=2^j}^{2^{j+1}-2} n^{p-2} a_n^p+
2^{j(p-2)} a_{2^{j+1}-2}^p.
	\end{equation*}
By the monotonicity property (\ref{3.16}) the latter can be bounded by $\displaystyle a_{2^j}^p 2^{j(p-1)}.$

	With (\ref{1new-}) in hand, we complete the proof by using Lemma \ref{LemmaBesovGMPer}  and (\ref{3.16}).
\end{proof}

Using (\ref{1new-}), Lemma \ref{Section 6.11: Lemma}, and the Littlewood-Paley decomposition
together with property (\ref{3.16}),
  we can obtain the Hardy-Littlewood theorem for Sobolev spaces.

\begin{thm}\label{Section 6.11: Theorem}
	Let $1 < p < \infty$ and $-\infty < s, b < \infty$. Let the Fourier series of $f \in L_1(\mathbb{T})$ be given by
	\begin{equation*}
		f(x) \sim \sum_{n=1}^\infty (a_n \cos
    nx + b_n \sin nx), 
	\end{equation*}
where  $\{a_n\}_{n \in \mathbb{N}}, \{b_n\}_{n \in \mathbb{N}}$ are nonnegative general monotone sequences.
	Then,
	\begin{equation*}
		\|f\|_{H^{s,b}_p(\mathbb{T})}^p \asymp \sum_{n=1}^\infty n^{s p +p-2} (1 + \log n)^{b p} (a_n^p + b_n^p).
	\end{equation*}
	In particular,
	\begin{equation}\label{HLPer}
		\|f\|_{L_p(\mathbb{T})}^p \asymp \sum_{n=1}^\infty n^{p-2} (a_n^p + b_n^p).
	\end{equation}
\end{thm}
\begin{rem}
	Equivalence  (\ref{HLPer}) was already proved in \cite[Theorem 4.2]{Tikhonov}.
\end{rem}

Since periodic spaces $\mathbf{B}^{s,b}_{p,q}(\mathbb{T})$ and $B^{s,b}_{p,q}(\mathbb{T})$ coincide if $s > 0$,  Theorem \ref{TheoremBesovGMPer} gives a description of functions from $\mathbf{B}^{s,b}_{p,q}(\mathbb{T}), s > 0,$ in terms of their Fourier coefficients. Alternatively, we give a new approach to characterize $\mathbf{B}^{s,b}_{p,q}(\mathbb{T})$ which also covers the limiting case $s=0$. The following result can be considered as the periodic counterpart of Theorem \ref{Theorem 3.2}.

\begin{thm}\label{Theorem 3.4}
    Let $1 < p < \infty, 0 < q \leq \infty$, and $-\infty < b < \infty$. Let the Fourier series of
    $f \in L_1(\mathbb{T})$ be given by
    \begin{equation*}
     	f(x) \sim \sum_{n=1}^\infty (a_n \cos
    nx + b_n \sin nx),
    \end{equation*}
    where  $\{a_n\}_{n \in \mathbb{N}}, \{b_n\}_{n \in \mathbb{N}}$ are nonnegative general monotone sequences. If $s > 0$, then
    \begin{equation}\label{char}
        \|f\|_{\emph{\textbf{B}}^{s,b}_{p,q}(\mathbb{T})} \asymp
        \left(\sum_{n=1}^\infty (n^{s+1-1/p} (1 + \log n)^b (a_n + b_n))^q
        \frac{1}{n}\right)^{1/q}.
    \end{equation}
    In the case  $s=0$ we have
    \begin{equation}\label{char2}
        \|f\|_ {\emph{\textbf{B}}^{0,b}_{p,q}(\mathbb{T})} \asymp
        \left(\sum_{n=1}^\infty \left[(1 + \log n)^b \left(\sum_{k=n}^\infty k^{p-2} (a_k^p + b_k^p)\right)^{1/p}\right]^q
        \frac{1}{n}\right)^{1/q}.
    \end{equation}
\end{thm}

\begin{rem}
(i) Convergence of sums in (\ref{char}) and (\ref{char2}) implies that
$\displaystyle
\sum_{n=1}^\infty n^{p-2} (a_n^p + b_n^p)<\infty
 $ 
which is equivalent to $f\in L_p(\T)$ (see (\ref{HLPer})).
\\
(ii) Of course, (\ref{char2}) is only interesting for $b \geq -1/q$ if $q < \infty$ ($b > 0$ if $q=\infty$). See Remark \ref{RemarkContinuous}(ii).
\\
(iii) Similar to Remark \ref{RemarkContinuous}(iii), the example  $f(x) \sim \sum_{n=1}^\infty a_n \cos
    nx $ with $a_n=n^{1/p-1} (1 + \log n)^{-\beta}$ for certain $\beta$ shows that, in general, the sum in (\ref{char2}) is not equivalent to the sum
$$\left(\sum_{n=1}^\infty (n^{1-1/p} (1 + \log n)^{b+1/q} a_n)^q
        \frac{1}{n}\right)^{1/q}.$$
 The latter characterizes, by (\ref{BesovGMPer}), the Besov norm $\|f\|_{B^{0,b+1/q}_{p,q}(\mathbb{T})}$.
\\
(iv) Some partial results of Theorem \ref{Theorem 3.4} can be also found in \cite{etna}. For the case $p=\infty$ see the papers \cite{Tikhonov-hung, Tikhonov-jat}. Note also that the statements of Theorem \ref{Theorem 3.4} hold in fact for all general monotone sequences
$\{a_n\}_{n \in \mathbb{N}}, \{b_n\}_{n \in \mathbb{N}}$, not necessarily  nonnegative. This can be observed using the results from the  recent paper \cite{tikhonov-stu}.
\end{rem}

    For $n \in \mathbb{N}$, let $E_n^\ast(f)_{L_p(\mathbb{T})}$\index{\bigskip\textbf{Numbers, relations}!$E_n^\ast(f)_{L_p(\mathbb{T})}$}\label{ERRORPER} be the $L_p$-best approximation of $f$ by trigonometric polynomials $t_{n-1}$ of degree at most $n-1$, i.e.,
    \begin{equation}\label{errors}
    E^\ast_n(f)_{L_p(\mathbb{T})} = \inf_{t_{n-1}} \|f - t_{n-1}\|_{L_p(\mathbb{T})}.
    \end{equation}
    To show Theorem \ref{Theorem 3.4} we shall use the approximation description of the spaces
    $\textbf{B}^{s,b}_{p,q}(\mathbb{T})$ (see, e.g., \cite[Chapter 7, Theorem 9.2, p. 235]{DeVoreLorentz} and \cite[Corollary 7.1]{DeVoreRiemenschneiderSharpley}) given in the following

    \begin{lem}\label{LemmaBesovApprox}
    Let $s \geq 0, 1 \leq p \leq \infty, 0 < q \leq \infty$, and $-\infty < b < \infty$. Then,
    \begin{equation}\label{BesovApprox}
        \|f\|_{\mathbf{B}^{s,b}_{p,q}(\mathbb{T})} \asymp
        \left(\sum_{n=1}^\infty (n^s (1 + \log n)^b E^\ast_n(f)_{L_p(\mathbb{T})})^q
        \frac{1}{n}\right)^{1/q}+\|f\|_{L_p(\mathbb{T})}.
    \end{equation}
\end{lem}

Note that Lemma \ref{LemmaBesovApprox} is the periodic counterpart of Corollary \ref{cor4.8*}.

    Let us denote by $S_n (f)$\index{\bigskip\textbf{Operators}!$S_n(f)$}\label{PARTIALSUM} the $n$-th
    partial sum of Fourier series of $f$ and $S_0(f) = 0$. It is well known that $E^\ast_n(f)_{L_p(\mathbb{T})}  \asymp \|f -
    S_{n-1}(f)\|_{L_p(\mathbb{T})}$ and then one can rewrite (\ref{BesovApprox}) as
    \begin{equation}\label{3.18}
        \|f\|_{\textbf{B}^{s,b}_{p,q}(\mathbb{T})} \asymp
        \left(\sum_{n=1}^\infty (n^s (1 + \log n)^b \|f - S_{n-1}(f)\|_{L_p(\mathbb{T})})^q
        \frac{1}{n}\right)^{1/q}+\|f\|_{L_p(\mathbb{T})}.
    \end{equation}

\begin{proof}[Proof of Theorem \ref{Theorem 3.4}]
    It is easy to see that, without loss of generality, we may assume that $f (x) \sim \sum_{n=1}^\infty a_n \cos nx$ with $\{a_n\}_{n \in \mathbb{N}}$ nonnegative general monotone sequence. Next we prove
    that there exists a constant $c > 1$ such that
    \begin{equation}\label{3.19}
        \sum_{k=2n}^\infty a_k^p k^{p-2} \lesssim \|f -
        S_{n-1}(f)\|_{L_p(\mathbb{T})}^p \lesssim
        \sum_{k=[n/c]}^\infty a_k^p k^{p-2}.
    \end{equation}
    Indeed, it follows from Lemma \ref{Lemma 3.5}(ii) and
    (\ref{3.16}) that
    \begin{equation*}
        \|f -
        S_{n-1}(f)\|_{L_p(\mathbb{T})}^p  = \Big\|\sum_{k=n}^\infty a_k \cos kx
        \Big\|_{L_p(\mathbb{T})}^p
         \gtrsim \sum_{k=2n}^\infty a_k^p k^{p-2}.
    \end{equation*}
Moreover,
applying
    (\ref{3.17new}),  Lemma \ref{Lemma 3.5}(i), and Hardy's inequality (see (\ref{HardyInequal2*})), we obtain
    \begin{align*}
        \|f - S_{n-1}(f)\|^p_{L_p(\mathbb{T})} & \lesssim
        \sum_{k=1}^{n-1} k^{p-2} \left(\sum_{k=[n/c]}^\infty
        \frac{a_k}{k}\right)^p + \sum_{k=n}^\infty k^{p-2}
        \left(\sum_{\nu=[k/c]}^\infty \frac{a_\nu}{\nu}\right)^p \\
        & \lesssim  n^{p-1} \left(\sum_{k=[n/c]}^\infty
        \frac{a_k}{k}\right)^p + \sum_{k=n}^\infty k^{p-2} a_k^p.
    \end{align*}
Then, H\"older's inequality implies  the right-hand side estimate in (\ref{3.19}).

Combining (\ref{3.18}) and (\ref{3.19}), we derive that
    \begin{equation}\label{3.20}
         \|f\|_{\textbf{B}^{s,b}_{p,q}(\mathbb{T})} \asymp
        \Bigg(\sum_{n=1}^\infty \Bigg[n^s (1 + \log n)^b \Bigg(\sum_{k=n}^\infty a_k^p k^{p-2}\Bigg)^{1/p}\Bigg]^q
        \frac{1}{n}\Bigg)^{1/q},
    \end{equation}
which gives (\ref{char2}) in the case
    $s=0$.

    Furthermore, using  equivalence (\ref{3.20}) and property
    (\ref{3.16}), one obtains
    \begin{align*}
         \|f\|_{\textbf{B}^{s,b}_{p,q}(\mathbb{T})} 
        & \gtrsim \left(\sum_{n=1}^\infty [n^{s+1-1/p} (1 + \log n)^b a_n]^q
        \frac{1}{n}\right)^{1/q}.
    \end{align*}
    On the other hand, we shall show that the converse estimate also holds under the assumption that $s > 0$. Indeed, if $p \leq q$ we can apply Hardy's inequality (\ref{HardyInequal2*}) to the right-hand side of (\ref{3.20}) to derive
    \begin{align*}
        \|f\|_{\textbf{B}^{s,b}_{p,q}(\mathbb{T})} \lesssim \left(\sum_{n=1}^\infty [n^{s+1-1/p} (1 + \log n)^b a_n]^q
        \frac{1}{n}\right)^{1/q}.
    \end{align*}
    In the case  $p > q$ we use $\ell_q \hookrightarrow \ell_p$ which implies, together with (\ref{3.20}) and the monotonicity property (\ref{3.16}), that
    \begin{align*}
         \|f\|^q_{\textbf{B}^{s,b}_{p,q}(\mathbb{T})} &
         \asymp \sum_{n=0}^\infty \Bigg[2^{n s} (1 + n)^b \Bigg(\sum_{k=n}^\infty a_{2^k}^p 2^{k(1-1/p)p}\Bigg)^{1/p}\Bigg]^q+
         \Bigg(\sum_{k=1}^\infty a_{k}^p k^{p-2}\Bigg)^{q/p} \\
        & \lesssim
         \sum_{n=0}^\infty 2^{n s q} (1 + n)^{b q} \sum_{k=n}^\infty a_{2^k}^q 2^{k(1-1/p)q} +\sum_{k=1}^\infty a_{2^k}^q 2^{k(1-1/p)q}
          \\
        & \asymp \sum_{k=0}^\infty 2^{k(s+1-1/p)q} (1+k)^{b q}
        a_{2^k}^q \\       &
         \asymp \sum_{n=1}^\infty n^{(s+1-1/p) q-1} (1 + \log n)^{b q} a_n^q,
    \end{align*}
completing the proof.
\end{proof}

%

\begin{rem}
In the case 	$s = 0$, unlike the case $s>0$ (cf. Theorems \ref{TheoremBesovGMPer} and \ref{Theorem 3.4}), the
 characterizations of the spaces $B^{0,b}_{p,q}(\mathbb{T})$ and $\mathbf{B}^{0,b}_{p,q}(\mathbb{T})$ are different (cf. (\ref{BesovGMPer}) and (\ref{char2})).
\end{rem}

\begin{rem}
Applying now the characterizations given in Theorems \ref{TheoremBesovGMPer}, \ref{Section 6.11: Theorem} and \ref{Theorem 3.4} we can establish the corresponding results to Theorems \ref{TheoremEmbGM}, \ref{Theorem 3.10new}, \ref{Theorem 3.10}, \ref{Section 6.8: Theorem}, \ref{SobolevEmbeddingImprovement}, and \ref{-} for periodic functions. The proofs are obtained \emph{mutatis mutandis} and we leave the details to the reader. For instance, the periodic counterpart of Theorem \ref{TheoremEmbGM} reads as follows


\begin{thm}
 Let $1 < p < \infty$, and $0 < q \leq \infty$. Then, the inequality
		\begin{equation*}
	\|f\|_{L_p(\mathbb{T})} \lesssim \|f\|_{B^{0,b}_{p,q}(\mathbb{T})}
	\end{equation*}
	holds for all
	  \begin{equation*}
	 f(x) \sim \sum_{n=1}^\infty (a_n \cos n x + b_n \sin nx),
	 \end{equation*}
    where  $\{a_n\}_{n \in \mathbb{N}}, \{b_n\}_{n \in \mathbb{N}}$ are nonnegative general monotone sequences, if and only if one of the following conditions is satisfied:
    \begin{enumerate}[\upshape(i)]
    \item  $b \geq 0$ if  $p \geq q,$
    \item $b > 1/p - 1/q$ if $p < q.$
	\end{enumerate}

\end{thm}

Note that in the periodic setting the additional assumptions $J_f(p_0) < \infty$ (respectively, $J_f(p_0) < \infty$ and $J_f(p) < \infty$) given in Theorem \ref{SobolevEmbeddingImprovement} (respectively, Theorem \ref{-}) are superfluous because we have that $L_{p_1}(\mathbb{T}) \hookrightarrow L_{p_0}(\mathbb{T}), p_0 < p_1$.
\end{rem}

\newpage
\section{Characterizations and embedding theorems for lacunary Fourier series}\label{section5}

\subsection{Characterization of spaces $B^{s,b}_{p,q}(\mathbb{R}^d), F^{s,b}_{p.q}(\mathbb{R}^d)$, and $\mathbf{B}^{s,b}_{p,q}(\mathbb{R}^d)$}\label{Section 5.1}

In this section, we deal with lacunary Fourier series, that is, Fourier series of the form
\begin{equation}\label{4.1}
    W (x) \sim \sum_{j=3}^\infty b_j e^{i \lambda_j x_1} \text{ for } x=(x_1,\ldots,
    x_d) \in \mathbb{R}^d,
\end{equation}
where $\{b_j\}_{j \in \mathbb{N}} \subset \mathbb{C}$ and $\{\lambda_j\}_{j \in \mathbb{N}}$ is formed by positive integers for which there exists $\lambda$ such that $\frac{\lambda_{j+1}}{\lambda_j} > \lambda > 1$, $j \in \mathbb{N}$.

Let $\psi\in \mathcal{S}(\mathbb{R}^d) \backslash \{0\}$
 be a fixed function satisfying
\begin{equation}\label{4.2}
    \text{supp } \widehat{\psi} \subset \{\xi : |\xi| \leq 2\}
\end{equation}
and let $\mathfrak{L}$\index{\bigskip\textbf{Sets}!$\mathfrak{L}$}\label{LAC} be the set
\begin{equation*}
    \mathfrak{L} = \{ \psi W : W \text{ given by (\ref{4.1})}\}.
\end{equation*}
Throughout this section $\lambda_j = 2^j-2, j \geq 3$. It is not hard to check that
\begin{equation}\label{4.2new}
	(\varphi_j \widehat{f})^\vee(x) = b_j e^{i \lambda_j x_1} \psi (x) \text{ for } f \in \mathfrak{L},
\end{equation}
where $\varphi_j$ is defined by (\ref{resolution}).
Then, we are in a position to obtain the following characterizations of the Besov and Triebel-Lizorkin norms of functions from the class $\mathfrak{L}$ in terms of their Fourier coefficients.

\begin{prop}\label{Proposition 4.1}
	Let $1 \leq p \leq \infty, 0 < q \leq \infty$, and $s,b \in \mathbb{R}$. Then
	\begin{equation}\label{4.3}
		\|f\|_{B^{s,b}_{p,q}(\mathbb{R}^d)}  \asymp \|f\|_{F^{s,b}_{p,q}(\mathbb{R}^d)} \asymp \left(\sum_{j=3}^\infty (2^{j s} (1 + j)^{b} |b_j|)^q \right)^{1/q}
	\end{equation}
	for all $f \in \mathfrak{L}$.
	
	In particular, if $1 < p < \infty$ then
	\begin{equation}\label{4.4}
    		\|f\|_{L_p(\mathbb{R}^d)} \asymp \left(\sum_{j=3}^\infty
    		|b_j|^2\right)^{1/2} \text{ for } f \in \mathfrak{L}
	\end{equation}
	or, more generally,
	\begin{equation}\label{4.4new}
		\|f\|_{H^{s,b}_p(\mathbb{R}^d)} \asymp \left(\sum_{j=3}^\infty (2^{j s} (1 + j)^b |b_j|)^2 \right)^{1/2} \text{ for } f \in \mathfrak{L}.
	\end{equation}
\end{prop}

\begin{rem}
(i) It follows from (\ref{4.3})-(\ref{4.4new}) that the norms of $\mathfrak{L}$-functions
 in Besov, Triebel-Lizorkin, Lebesgue, and Sobolev spaces do not depend on $p$. In particular, (\ref{4.4}) implies the Zygmund--type property $\|f\|_{L_{p_1}(\mathbb{R}^d)}\asymp \|f\|_{L_{p_2}(\mathbb{R}^d)}, f \in \mathfrak{L},$
for $p_1\ne p_2$ (see, e.g., \cite[Theorem 3.7.4]{Grafakos}).

(ii) Note that the equivalence constants hidden in (\ref{4.3})-(\ref{4.4new}) depend on the $L_p(\mathbb{R}^d)$ norm of the chosen function $\psi$ in the definition of the class $\mathfrak{L}$.
\end{rem}

The corresponding characterization for spaces $\mathbf{B}^{s,b}_{p,q}(\mathbb{R}^d)$ is given as follows.

\begin{thm}\label{Theorem 4.2}
	Let $1 < p < \infty, 0 < q \leq \infty$, and $b \in \mathbb{R}$. Assume that $f \in \mathfrak{L}$. Then
	\begin{equation}\label{4.5}
		\|f\|_{\mathbf{B}^{s,b}_{p,q}(\mathbb{R}^d)} \asymp \left(\sum_{j=3}^\infty 2^{j s q} (1 + j)^{b q} |b_j|^q \right)^{1/q}  \text{ if } s > 0
	\end{equation}
	and
	\begin{equation}\label{4.6}
		\|f\|_{\mathbf{B}^{0,b}_{p,q}(\mathbb{R}^d)} \asymp \left(\sum_{j=3}^\infty (1 + j)^{b q} \left(\sum_{k=j}^\infty |b_k|^2 \right)^{q/2} \right)^{1/q}.
	\end{equation}
\end{thm}
\begin{rem}
	(i) If the right-hand sides in (\ref{4.5}) and (\ref{4.6}) are finite then
	$
		\sum_{j=3}^\infty |b_j|^2 < \infty
	$ 
	which is equivalent to $f \in L_p(\mathbb{R}^d)$ (see (\ref{4.4})).

(ii) The case of interest in (\ref{4.6}) is when $b \geq -1/q$ if $q < \infty$ ($b > 0$ if $q=\infty$). Otherwise, it is easy to check that both sides in (\ref{4.6}) are equivalent to $\|f\|_{L_p(\mathbb{R}^d)}$ (see (\ref{4.4})).

	(iii) In the case 	$s = 0$, unlike the case $s>0$ (see (\ref{BesovComparison})), 
 the characterizations of the spaces  $\mathbf{B}^{s,b}_{p,q}\mathbb{R}^d)$ and $B^{s,b}_{p,q}(\mathbb{R}^d)$ are different (cf. (\ref{4.3}) and (\ref{4.6})).
\end{rem}

We start with the definition of the  de la Vall\'ee-Poussin operator.
	    For $R > 0$, let $\eta _R$ be
the
    de la Vall\'ee-Poussin operator, that is, $\eta _R$\index{\bigskip\textbf{Operators}!$\eta_R$}\label{VALLEEPOUSSIN} is
    defined on $L_p(\mathbb{R}^d)$ by
    \begin{equation}\label{eta}
        (\eta_R f)^\wedge (\xi)= \eta \left(\frac{|\xi|}{R}\right)
        \widehat{f}(\xi)
    \end{equation}
    where $\eta \in C^\infty[0,\infty)$ with $\eta (u) = 1$ if $u \leq
    1$ and $\eta (u) = 0$ if $u \geq 3/2$. We have  (see, for example, \cite[Section 8.6]{Nikolskii}) that
    \begin{equation}\label{4.7}
        \|f - \eta_R f\|_{L_p(\mathbb{R}^d)} \lesssim E_R(f)_{L_p(\mathbb{R}^d)}, 
    \end{equation}
    where $E_R(f)_{L_p(\mathbb{R}^d)}$
   \index{\bigskip\textbf{Numbers, relations}!$E_k(f)_{L_p(\mathbb{R}^d)}$}\label{ERROR}
is  the best approximation of $f \in {L_p(\mathbb{R}^d)}$ by  entire functions of
spherical exponential type $R$.


    Note that Corollary \ref{cor4.8*} below together with (\ref{4.7}) allows us to derive
    \begin{equation}\label{4.8}
      \|f\|_{\mathbf{B}^{s,b}_{p,q}(\mathbb{R}^d)}  \asymp \|f\|_{L_p(\mathbb{R}^d)} + \left(\sum_{j=0}^\infty 2^{j s q} (1 + j)^{b q}
        \|f - \eta_{2^j}f\|_{L_p(\mathbb{R}^d)}^q\right)^{1/q} .
    \end{equation}

\begin{proof}[Proof of Theorem  \ref{Theorem 4.2}]
    Let $W$ be a function whose Fourier series expansion satisfies (\ref{4.1}) and let $\psi \in \mathcal{S}(\mathbb{R}^d) \backslash \{0\}$ be such that (\ref{4.2}) holds. By using assumptions on $\psi$ and $\eta$, it is plain to check that
    \begin{equation}\label{4.9}
    	\eta_{2^j} f (x) = \sum_{k=3}^j b_k e^{i \lambda_k x_1} \psi(x) \text{ for } j \geq 3.
    \end{equation}

    Inserting (\ref{4.9}) into (\ref{4.8}) and applying (\ref{4.4}) we get
    \begin{align}
    	\|f\|_{\mathbf{B}^{s,b}_{p,q}(\mathbb{R}^d)} & \asymp \|f\|_{L_p(\mathbb{R}^d)} + \left(\sum_{j=3}^\infty 2^{j s q} (1 + j)^{b q}
        \Big\|\sum_{k=j+1}^\infty b_k e^{i \lambda_k x_1} \psi\Big\|_{L_p(\mathbb{R}^d)}^q\right)^{1/q} \nonumber\\
        & \asymp  \left(\sum_{j=3}^\infty 2^{j s q} (1 + j)^{b q}
        \left(\sum_{k=j}^\infty |b_k|^2\right)^{q/2}\right)^{1/q} \label{4.10}.
    \end{align}
    This proves (\ref{4.6}) when $s=0$.

    Suppose now that $s > 0$. If $q \geq 2$ we can apply Hardy's inequality (\ref{HardyInequal2**}) in (\ref{4.10}) to derive
    \begin{equation*}
    	\|f\|_{\mathbf{B}^{s,b}_{p,q}(\mathbb{R}^d)}  \lesssim \left(\sum_{j=3}^\infty 2^{j s q} (1 + j)^{b q} |b_j|^q \right)^{1/q}.
    \end{equation*}
    If $q < 2$,  we have
    \begin{align*}
    	\|f\|^q_{\mathbf{B}^{s,b}_{p,q}(\mathbb{R}^d)}  & \lesssim \sum_{j=3}^\infty 2^{j s q} (1 + j)^{b q}
        \sum_{k=j}^\infty |b_k|^q  \lesssim \sum_{k=3}^\infty 2^{k s q} (1 + k)^{b q} |b_k|^q.
    \end{align*}

\end{proof}

\begin{rem}
	If $b > -1/q$ the characterization (\ref{4.6}) for the norms of lacunary Fourier series in the spaces $\mathbf{B}^{0,b}_{p,q}(\mathbb{R}^d)$ can also be shown by using their Fourier-analytical description given in Lemma \ref{lem: CDT}. Indeed, by (\ref{4.2new}), we have
	     \begin{align*}
        \|f\|_{\mathbf{B}^{0,b}_{p,q}(\mathbb{R}^d)} & \asymp \Bigg(\sum_{j=3}^\infty \Big[(1 + j)^b \Big\|\Big(\sum_{k=j}^\infty |(\varphi_k \widehat{f})^\vee
        (\cdot)|^2\Big)^{1/2}\Big\|_{L_p(\mathbb{R}^d)}\Big]^q\Bigg)^{1/q} \\
        &  = \|\psi\|_{L_p(\mathbb{R}^d)} \left(\sum_{j=3}^\infty (1 + j)^{b q} \left(\sum_{k=j}^\infty |b_k|^2 \right)^{q/2} \right)^{1/q}.
    \end{align*}
\end{rem}

\subsection{Embeddings of $B^{0,b}_{p,q}(\mathbb{R}^d)$ in $L_p(\mathbb{R}^d)$}
First, we recall that 
\begin{equation}\label{Loc2}
   B^{0,b}_{p,q}(\mathbb{R}^d) \hookrightarrow L_p(\mathbb{R}^d) \iff \left\{\begin{array}{lcl}
                            b \geq 0 & \text{ if }  & 0 < q \leq
                            \min\{2,p\}, \\
                            & & \\
                            b > 1/p - 1/q & \text{ if } & 1 < p \leq
                            2 \text{ and } p < q \leq \infty, \\
                            & & \\
                            b > 1/2 -1/q & \text{ if } & 2 < p <
                            \infty \text{ and } 2 < q \leq \infty,
            \end{array}
            \right.
\end{equation}
cf. (\ref{Loc}). We will see that this result can be  extended to a bigger range of the logarithmic smoothness $b$ when we deal with functions from the $\mathfrak{L}$ class. Moreover, we observe that assuming additional conditions on the growth of Fourier coefficients of functions $W$ given by (\ref{4.1}) implies even bigger range of the parameter $b$.

Let us introduce the subclass $\mathfrak{L}^{\text{wm}}$\index{\bigskip\textbf{Sets}!$\mathfrak{L}^{\text{wm}}$}\label{LACMON} of $\mathfrak{L}$, which consists of those functions from $\mathfrak{L}$ such that the  sequences $\{b_j\}_{j \in \mathbb{N}}$ in (\ref{4.1}) satisfy the weak monotone  condition (\ref{3.16}), that is,
\begin{equation*}
	\mathfrak{L}^{\text{wm}} = \left\{\psi W : W(x_1, \ldots, x_d) \sim \sum_{j=3}^\infty b_j e^{i \lambda_j x_1}, \quad |b_j| \lesssim |b_n| \quad \text{for} \quad n \leq j \leq 2n\right\}.
\end{equation*}
Here, $\psi \in \mathcal{S}(\mathbb{R}^d) \backslash \{0\}$ is a fixed function with (\ref{4.2}).
For more detail on weak monotone sequence see \cite{liflyand}.

\begin{thm}\label{TheoremEmbL}
	Let $1 < p < \infty$, and $0 < q \leq \infty$. Then
	\begin{equation}\label{LpLacunary}
	\mathfrak{L} \cap B^{0,b}_{p,q}(\mathbb{R}^d) \hookrightarrow L_p(\mathbb{R}^d) \iff
    \left\{\begin{array}{lcl}
    				 b \geq 0 & \text{ if } &q \leq 2 , \\

                            & & \\
                              b > 1/2 - 1/q & \text{ if }  & 2 < q
            \end{array}
            \right.
	\end{equation}
	and
	\begin{equation}\label{LpLacunaryDoubling}
		\mathfrak{L}^{\emph{wm}}
\cap B^{0,b}_{p,q}(\mathbb{R}^d) \hookrightarrow L_p(\mathbb{R}^d) \iff     \left\{\begin{array}{lcl}
    				 b \geq 1/2 - 1/q & \text{ if } & q \leq 2 , \\

                            & & \\
                              b > 1/2 - 1/q & \text{ if }  & 2 < q.
            \end{array}
            \right.
	\end{equation}
\end{thm}
\begin{rem}
Note that the integrability parameter $p$ is not important when we investigate the embedding $B^{0,b}_{p,q}(\mathbb{R}^d) \hookrightarrow L_p(\mathbb{R}^d)$ for $\mathfrak{L}$ and $\mathfrak{L}^{\text{wm}}$ functions (see (\ref{LpLacunary}) and (\ref{LpLacunaryDoubling})), but it only involves the relationships between $q, 2$, and $b$. This is in sharp contrast to the general case given in (\ref{Loc2}). In particular, working with functions from the $\mathfrak{L}$ class (respectively, $\mathfrak{L}^{\text{wm}}$ class), one can extend the range of $b$ given in (\ref{Loc2}) for which this embedding holds when $p < \min\{2,q\}$ (respectively, $\min\{p,q\} < 2$).
\end{rem}
\begin{proof}[Proof of Theorem \ref{TheoremEmbL}]
We start by showing the if-part in (\ref{LpLacunary}). Assume first that $q \leq 2$ and $b \geq 0$. Then, using (\ref{4.3}) and (\ref{4.4}), we obtain
\begin{equation*}
	\|f\|_{L_p(\mathbb{R}^d)} \asymp \left(\sum_{j=3}^\infty |b_j|^2\right)^{1/2} \leq \left(\sum_{j=3}^\infty (1 + j)^{b q} |b_j|^q\right)^{1/q} \asymp \|f\|_{B^{0,b}_{p,q}(\mathbb{R}^d)},
\end{equation*}
which yields the desired result. On the other hand, in the case $2 < q$ and $b > 1/2- 1/q$ the result follows from H\"older's inequality.

Next we show the only-if-part in (\ref{LpLacunary}). If $q > 2$ we proceed by contradiction. We distinguish two possible cases. Firstly, we assume that $b = 1/2- 1/q$. Then, putting
\begin{equation}\label{5.16new}
 \{b_j\}_{j \in \mathbb{N}} = \{(1 + j)^{-b - 1/q} (1 + \log j)^{-\beta}\}_{j \in \mathbb{N}}, \quad 1/q < \beta \leq 1/2,
 \end{equation}
 and
 \begin{equation}\label{5.17new}
 	f(x) \sim \psi (x) \sum_{j=3}^\infty b_j e^{i (2^j - 2) x_1},  \quad x \in \mathbb{R}^d,
\end{equation}
(with $\psi \in \mathcal{S}(\mathbb{R}^d) \backslash \{0\}$ satisfying (\ref{4.2})), we have by (\ref{4.3}) and (\ref{4.4}) that
\begin{equation*}
	\|f\|_{B^{0,b}_{p,q}(\mathbb{R}^d)}^q \asymp \sum_{j=3}^\infty (1 + \log j)^{-\beta q} \frac{1}{1+j} < \infty
\end{equation*}
but
\begin{equation*}
	\|f\|_{L_p(\mathbb{R}^d)}^2 \asymp \sum_{j=3}^\infty (1 + \log j)^{-2\beta } \frac{1}{1+j} = \infty.
\end{equation*}
Whence, this example shows that $\mathfrak{L} \cap B^{0,b}_{p,q}(\mathbb{R}^d) \not \hookrightarrow L_p(\mathbb{R}^d)$. Further, if we assume that $b < 1/2 - 1/q$ we arrive at a contradiction applying the trivial embedding $B^{0, 1/2-1/q}_{p,q}(\mathbb{R}^d) \hookrightarrow B^{0,b}_{p,q}(\mathbb{R}^d)$ together with the previous case.

Assume now that $\mathfrak{L} \cap B^{0,b}_{p,q}(\mathbb{R}^d) \hookrightarrow L_p(\mathbb{R}^d)$ holds for $q \leq 2$. This implies that $b \geq 0$. Indeed, for every $n \geq 3$, we introduce
 \begin{equation*}
 	f_n(x) \sim \psi (x) \sum_{j=3}^n 2^j e^{i (2^j - 2) x_1},  \quad x \in \mathbb{R}^d.
\end{equation*}
By our assumption,
\begin{align*}
	2^n &\asymp \Big(\sum_{j=3}^n 2^{2 j} \Big)^{1/2} \asymp \|f_n\|_{L_p(\mathbb{R}^d)} \lesssim \|f_n\|_{B^{0,b}_{p,q}(\mathbb{R}^d)}  \asymp \Big(\sum_{j=3}^n 2^{j q} (1 + j)^{b q} \Big)^{1/q}
	 \asymp 2^n (1 + n)^b
\end{align*}
for $n \geq 3$, which yields that $b \geq 0$. This completes the proof of (\ref{LpLacunary}).

Next we prove (\ref{LpLacunaryDoubling}). Assume that $b \geq 1/2 - 1/q$ and $2 \geq q$. Let $f \in \mathfrak{L}^\text{wm}$. Applying monotonicity properties, we get
\begin{align*}
 \|f\|_{L_p(\mathbb{R}^d)} & \asymp \Big(\sum_{j=3}^\infty |b_j|^2 \Big)^{1/2} \lesssim \Big(|b_{3}|^2 +\sum_{j=2}^\infty 2^j  |b_{2^j}|^2 \Big)^{1/2} \leq \Big(|b_{3}|^q +\sum_{j=2}^\infty 2^{j(b + 1/q) q} |b_{2^j}|^q  \Big)^{1/q} \\
 & \lesssim \Big(\sum_{j=3}^\infty (1 + j)^{bq} |b_j|^q\Big)^{1/q} \asymp \|f\|_{B^{0,b}_{p,q}(\mathbb{R}^d)},
\end{align*}
which implies the embedding of the left-hand side in (\ref{LpLacunaryDoubling}). On the other hand, if $2 < q$ and $b > 1/2 - 1/q$ the embedding is a direct consequence of the H\"older's inequality. Hence, the if-part in (\ref{LpLacunaryDoubling}) is shown.

Finally, we obtain the only-if-part in (\ref{LpLacunaryDoubling}). Suppose that $\mathfrak{L}^{\text{wm}} \cap B^{0,b}_{p,q}(\mathbb{R}^d) \hookrightarrow L_p(\mathbb{R}^d)$. We shall distinguish two cases. Firstly, if $q \leq 2$ we consider the functions
 \begin{equation*}
 	g_n(x) \sim \psi (x) \sum_{j=3}^n  e^{i (2^j - 2) x_1},  \quad x \in \mathbb{R}^d, \quad n \geq 3.
\end{equation*}
Note that $g_n \in \mathfrak{L}^{\text{wm}}$.
Then
\begin{equation*}
	n^{1/2} \asymp \|g_n\|_{L_p(\mathbb{R}^d)} \lesssim \|g_n\|_{B^{0,b}_{p,q}(\mathbb{R}^d)}
	 \asymp n^{b+1/q}
\end{equation*}
for $n \geq 3$, which yields that $b + 1/q - 1/2 \geq 0$.

Secondly, assume that $q > 2$. We observe that the Fourier series (\ref{5.17new}) with $\{b_j\}$ given by  (\ref{5.16new}) belongs to $\mathfrak{L}^{\text{wm}}$. Then, the proof by contradiction given above to show that $\mathfrak{L} \cap B^{0,b}_{p,q}(\mathbb{R}^d) \not \hookrightarrow L_p(\mathbb{R}^d)$ for $b \leq 1/2 - 1/q$ and $q > 2$ also allows us to derive the counterpart for the $\mathfrak{L}^{\text{wm}}$ class.

\end{proof}

\subsection{Embeddings between $F^{s,b}_{p,q}(\mathbb{R}^d)$ and $B^{s,b}_{p,q}(\mathbb{R}^d)$}

Let $1 < p < \infty, 0 < q \leq \infty,$ and $-\infty < b < \infty$. It is an immediate consequence of the characterizations (\ref{4.3}) that
\begin{equation}\label{LacunaryFB}
	\mathfrak{L} \cap F^{s,b}_{p,q}(\mathbb{R}^d) = \mathfrak{L} \cap B^{s,b}_{p,q}(\mathbb{R}^d) \quad \text{for} \quad s \in \mathbb{R}
\end{equation}
and, in particular,
\begin{equation}\label{LacunaryFB1}
	\mathfrak{L} \cap H^{s,b}_{p}(\mathbb{R}^d) = \mathfrak{L} \cap B^{s,b}_{p,2}(\mathbb{R}^d) \quad \text{for} \quad s \in \mathbb{R}.
\end{equation}
If $s > 0$ then (see (\ref{BesovComparison}))
  \begin{equation*}
    		\mathfrak{L} \cap F^{s,b}_{p,q}(\mathbb{R}^d) = \mathfrak{L} \cap \mathbf{B}^{s,b}_{p,q}(\mathbb{R}^d) \quad \text{and} \quad \mathfrak{L} \cap H^{s,b}_{p}(\mathbb{R}^d) = \mathfrak{L} \cap \mathbf{B}^{s,b}_{p,2}(\mathbb{R}^d)
\end{equation*}
(compare also (\ref{4.3}) and (\ref{4.4new}) with (\ref{4.5})).
Moreover, the characterizations of smoothness spaces obtained in Subsection \ref{Section 5.1} will allow us to refine the classical embeddings given in Proposition \ref{RecallEmb} for the $\mathfrak{L}$ class. To be more precise, the range of the parameter $q$ for which these embeddings hold is bigger if we restrict ourselves to functions with lacunary Fourier series.

\begin{prop}\label{Proposition 4.5}
Let $1 < p < \infty, 0 < q, r \leq \infty$, and $-\infty < s, b < \infty$. Then we have
\begin{enumerate}[\upshape(i)]
\item
    $\mathfrak{L} \cap F^{s,b}_{p,r}(\mathbb{R}^d) \hookrightarrow B^{s,b}_{p,q}(\mathbb{R}^d) \quad \text{ if and only if } \quad q \geq r$,

\item $\mathfrak{L} \cap  B^{s,b}_{p,q}(\mathbb{R}^d) \hookrightarrow F^{s,b}_{p,r}(\mathbb{R}^d) \quad \text{ if and only if } \quad q \leq r$,
\item $\mathfrak{L} \cap F^{s,b}_{p,r}(\mathbb{R}^d)  = \mathfrak{L} \cap  B^{s,b}_{p,q}(\mathbb{R}^d) \quad \text{ if and only if } \quad r=q.$
    \end{enumerate}

If, in addition, $s > 0$ then
    \begin{enumerate}[\upshape(i)]
    \setcounter{enumi}{3}
    \item      $\mathfrak{L} \cap F^{s,b}_{p,r}(\mathbb{R}^d) \hookrightarrow \mathbf{B}^{s,b}_{p,q}(\mathbb{R}^d) \quad \text{ if and only if } \quad q \geq r$,
        \item  $\mathfrak{L} \cap H^{s,b}_p(\mathbb{R}^d) \hookrightarrow \mathbf{B}^{s,b}_{p,q}(\mathbb{R}^d) \quad \text{ if and only if } \quad q \geq 2$,
        \item $\mathfrak{L} \cap \mathbf{B}^{s,b}_{p,q}(\mathbb{R}^d) \hookrightarrow F^{s,b}_{p,r}(\mathbb{R}^d) \quad \text{ if and only if } \quad q \leq
    r$,
    \item
    $\mathfrak{L} \cap \mathbf{B}^{s,b}_{p,q}(\mathbb{R}^d) \hookrightarrow H^{s,b}_p(\mathbb{R}^d) \quad \text{ if and only if } \quad q \leq
    2$,
    \item $\mathfrak{L} \cap F^{s,b}_{p,r}(\mathbb{R}^d)  = \mathfrak{L} \cap  \mathbf{B}^{s,b}_{p,q}(\mathbb{R}^d) \quad \text{ if and only if } \quad r=q$,
    \item $\mathfrak{L} \cap H^{s,b}_{p}(\mathbb{R}^d)  = \mathfrak{L} \cap  \mathbf{B}^{s,b}_{p,q}(\mathbb{R}^d) \quad \text{ if and only if } \quad 2=q$.
\end{enumerate}
\end{prop}
\begin{proof}
	Let us show (i). Assume first that $q \geq r$. Then, the inequality
	\begin{equation*}
	\left(\sum_{j=3}^\infty (2^{j s} (1 + j)^{b} |b_j|)^q \right)^{1/q} \leq \left(\sum_{j=3}^\infty (2^{j s} (1 + j)^{b} |b_j|)^r \right)^{1/r}
	\end{equation*}
	together with the characterizations  (\ref{4.3}) yield that $\mathfrak{L} \cap F^{s,b}_{p,r}(\mathbb{R}^d) \hookrightarrow B^{s,b}_{p,q}(\mathbb{R}^d)$.
	
	Conversely, suppose that
	\begin{equation*}
		\|f\|_{B^{s,b}_{p,q}(\mathbb{R}^d)} \lesssim \|f\|_{F^{s,b}_{p,r}(\mathbb{R}^d) } \quad \text{for all} \quad f \in \mathfrak{L}.
	\end{equation*}
	In particular, taking
	\begin{equation*}
	 f_n(x) \sim \psi(x) \sum_{j=3}^n 2^{-j s} (1+ j)^{-b} e^{i \lambda_j x_1}, \quad x \in \mathbb{R}^d, \quad n \geq 3,
	 \end{equation*}
	 where $\psi \in \mathcal{S}(\mathbb{R}^d) \backslash \{0\}$ is a fixed function such that (\ref{4.2}) holds, we derive
	 \begin{equation*}
	 	\|f_n\|_{B^{s,b}_{p,q}(\mathbb{R}^d)} \lesssim \|f_n\|_{F^{s,b}_{p,r}(\mathbb{R}^d) },  \quad n \geq 3.
	 \end{equation*}
	 These estimates together with (\ref{4.3}) yield that $n^{1/q} \lesssim n^{1/r}$ uniformly in $n \geq 3$, or equivalently, $r \leq q$. This finishes the proof of (i).
	
	The proof of (ii) follows the same lines as above and is left to the reader. Putting together (i) and (ii), we obtain (iii). Finally, we note that the statements (iv)--(ix) are special cases of (i)--(iii) for $s > 0$ (see (\ref{BesovComparison}) and (\ref{LPgeneral})).
	
\end{proof}


\subsection{Embeddings between $F^{0,b}_{p,q}(\mathbb{R}^d)$ and $\mathbf{B}^{0,b}_{p,q}(\mathbb{R}^d)$}
In this subsection we deal with embeddings between $F^{0,b}_{p,q}(\mathbb{R}^d)$ and $\mathbf{B}^{0,b}_{p,q}(\mathbb{R}^d)$.
More precisely, for $\mathfrak{L}$-functions we improve the range of the parameter $q$ for which the embeddings obtained in Theorem \ref{Theorem 2.1} hold. Moreover, we prove that this range is sharp with the help of Hardy's and Copson's inequality with weights. Let $\{a_j\}_{j \in \mathbb{N}}$ and $\{b_j\}_{j \in \mathbb{N}}$ be sequences of non-negative numbers and let $0 < p, q \leq \infty$. The Hardy's inequality claims that there exists $C> 0$ such that
\begin{equation}\label{hardy}
	\left(\sum_{j=1}^\infty \Big(a_j \sum_{k=j}^\infty b_k x_k\Big)^{q}\right)^{1/q} \leq C \left(\sum_{j=1}^\infty x_j^p\right)^{1/p}
\end{equation}
for all non-negative sequences $\{x_j\}_{j \in \mathbb{N}}$. In particular, the inequalities (\ref{HardyInequal2*}) and (\ref{HardyInequal2**}) are special cases of (\ref{hardy}).

Conversely, the Copson's inequality asserts that
\begin{equation}\label{copson}
	\left(\sum_{j=1}^\infty \Big(a_j \sum_{k=j}^\infty b_k x_k\Big)^{q}\right)^{1/q} \geq C \left(\sum_{j=1}^\infty x_j^p\right)^{1/p}
\end{equation}
for all non-negative sequences $\{x_j\}_{j \in \mathbb{N}}$.

A complete characterization of those $\{a_j\}_{j \in \mathbb{N}}$ and $\{b_j\}_{j \in \mathbb{N}}$ for which (\ref{hardy}) and (\ref{copson}) hold, was obtained by Bennett \cite{Bennett2, Bennett} and Grosse-Erdmann \cite[Chapter 3]{Grosse-Erdmann}.
For later use, we write down a special case of (\ref{hardy}) and (\ref{copson}) in Lemmas \ref{LemmaHardy} and \ref{LemmaCopson} below, respectively. For the proofs, we refer to \cite[(9.2ix'), p. 55 and Theorem 10.4, p. 60]{Grosse-Erdmann}.

\begin{lem}[\bf{Hardy's inequality; $0 < q < p \leq 1$}]\label{LemmaHardy}
Let $0 < q < p \leq 1$. The inequality (\ref{hardy}) holds if and only if
\begin{equation*}
	\sum_{j=1}^\infty a_j^q \Big(\sum_{k=1}^j a_k^q\Big)^{r/p}  \sup_{k \geq j} b_k^r < \infty,
\end{equation*}
where $1/r = 1/q-1/p$.
\end{lem}

\begin{lem}[\bf{Copson's inequality; $1 \leq p < q < \infty$}]\label{LemmaCopson}
Let $1 \leq p < q < \infty$ and $\sum_{j=1}^\infty a_j^q = \infty$. The inequality (\ref{copson}) holds if and only if
\begin{equation*}
	\sum_{j=1}^\infty a_{j+1}^q \Big(\sum_{k=1}^{j+1} a_k^q\Big)^{-1} \Big(\sum_{k=1}^{j} a_k^q\Big)^{-r/q}  \sup_{k \leq j} b_k^{-r} < \infty,
\end{equation*}
where $1/r = 1/p-1/q$.
\end{lem}

\begin{thm}\label{Theorem 4.5new}
    Let $1 < p < \infty, 0 < q,r \leq \infty$, and $b > -1/q$. Then
    \begin{equation}\label{5.13new}
    	\mathfrak{L} \cap F^{0,b+1/q+1/2-1/\max\{r,2\}}_{p,r} (\mathbb{R}^d) \hookrightarrow \mathbf{B}^{0,b}_{p,q}(\mathbb{R}^d) \quad \text{ if and only if }\quad q \geq r,
    \end{equation}
    and
    \begin{equation}\label{5.14new}
    	\mathfrak{L} \cap \mathbf{B}^{0,b}_{p,q}(\mathbb{R}^d) \hookrightarrow F^{0,b+1/q + 1/2 - 1/\min\{r,2\}}_{p,r}(\mathbb{R}^d)\quad \text{ if and only if } \quad q \leq  r.
    \end{equation}
    In particular,
    \begin{equation}\label{Section 7.4: new}
    	\mathfrak{L} \cap F^{0,b+1/\min\{2,q\}}_{p,q} (\mathbb{R}^d) \hookrightarrow \mathfrak{L} \cap \mathbf{B}^{0,b}_{p,q}(\mathbb{R}^d) \hookrightarrow \mathfrak{L} \cap F^{0,b+1/\max\{q,2\}}_{p,q} (\mathbb{R}^d).
    \end{equation}
\end{thm}
\begin{proof}
Let $q \geq r$. We consider two cases. Assume first that $r \leq 2$. By using (\ref{4.6}), (\ref{4.3}) and Minkowski's inequality with $q/r \geq 1$, we obtain
	\begin{align*}
		\|f\|_{\mathbf{B}^{0,b}_{p,q}(\mathbb{R}^d)} & \asymp \left(\sum_{j=3}^\infty (1 + j)^{b q} \Big(\sum_{k=j}^\infty |b_k|^2\Big)^{q/2}\right)^{1/q} \leq \left(\sum_{j=3}^\infty (1 + j)^{b q} \Big(\sum_{k=j}^\infty |b_k|^r\Big)^{q/r}\right)^{1/q} \\
		& \leq \left(\sum_{k=3}^\infty \Big(\sum_{j=3}^k (1 + j)^{b q}\Big)^{r/q} |b_k|^r \right)^{1/r}
\lesssim\|f\|_{F^{0,b+1/q}_{p,r}(\mathbb{R}^d)},
	\end{align*}
	i.e., the embedding given in (\ref{5.13new}) follows.
	
	Suppose now that $2 < r$. We take $\beta$ such that $1/2 - 1/r < \beta < 1/2 - 1/r + b + 1/q$. For $j \geq 3$, we have
	\begin{equation*}
		\Big(\sum_{k=j}^\infty |b_k|^2\Big)^{1/2} \lesssim (1 + j)^{-\beta + 1/2 - 1/r} \Big(\sum_{k=j}^\infty (1 + k)^{\beta r} |b_k|^r \Big)^{1/r}.
	\end{equation*}
	Combining  this estimate and (\ref{4.6}) and applying Minkowski's inequality, we derive
	\begin{align*}
		\|f\|_{\mathbf{B}^{0,b}_{p,q}(\mathbb{R}^d)} &  \lesssim \left(\sum_{j=3}^\infty (1 + j)^{(b -\beta + 1/2 -1/r)q} \Big(\sum_{k=j}^\infty (1 + k)^{\beta r} |b_k|^r\Big)^{q/r}\right)^{1/q} \\
		& \lesssim \left(\sum_{k=3}^\infty (1 + k)^{(b +1/q + 1/2 -1/r)r} |b_k|^r \right)^{1/r}.
	\end{align*}
	Therefore, the embedding in (\ref{5.13new}) follows from (\ref{4.3}).
	
	Conversely, let us show that if
	\begin{equation}\label{4.17new}
	\mathfrak{L} \cap F^{0,b+1/q+1/2-1/\max\{r,2\}}_{p,r} (\mathbb{R}^d) \hookrightarrow \mathbf{B}^{0,b}_{p,q}(\mathbb{R}^d)
	\end{equation}
	 then $q \geq r$. We will proceed by contradiction, that is, assume $q < r$. We distinguish two possible cases. Firstly, if $2 \leq r$, we let
	\begin{equation*}
		b_j = (1 + j)^{-(b+1/q+1/2)} (1 + \log j)^{-\delta}, \quad j \in \mathbb{N},
	\end{equation*}
	where $1/r < \delta < 1/q$. Thus, it follows from (\ref{4.3}) and (\ref{4.6}) that the Fourier series $f(x)\sim \psi(x) \sum_{j=3}^\infty b_j e^{i(2^j-2) x_1}, x \in \mathbb{R}^d$, where $\psi \in \mathcal{S}(\mathbb{R}^d) \backslash \{0\}$ with (\ref{4.2}), belongs to $F^{0,b+1/q+1/2-1/r}_{p,r} (\mathbb{R}^d)$ but not to $\mathbf{B}^{0,b}_{p,q}(\mathbb{R}^d)$. This leads to a contradiction.
	
	Secondly, suppose that $2 > r$. We notice that, by (\ref{4.3}) and (\ref{4.6}), the embedding (\ref{4.17new}) can be rewritten as
	\begin{equation*}
		\left(\sum_{j=3}^\infty \Big((1  + j)^{2 b} \sum_{k=j}^\infty |b_k|^2 \Big)^{q/2} \right)^{2/q} \leq C \left(\sum_{j=3}^\infty ((1 + j)^{2(b+1/q)} |b_j|^2)^{r/2}\right)^{2/r}
	\end{equation*}
	for all $\{b_j\}_{j \in \mathbb{N}}$. Since $0 < \frac{q}{2} < \frac{r}{2} < 1$, Lemma \ref{LemmaHardy} yields that
	\begin{equation*}
	\sum_{j=3}^\infty (1 + j)^{b q} \Big(\sum_{k=1}^j (1 + k)^{b q}  \Big)^{\frac{q}{r-q}}  \sup_{k \geq j} k^{-\frac{(b q +1)r}{r-q}} < \infty,
\end{equation*}
or equivalently, $\sum_{j=3}^\infty \frac{1}{j} < \infty$, which is not true. This completes the proof of (\ref{5.13new}).
	
	The proof of (\ref{5.14new}) can be done in a similar way but now using Lemma \ref{LemmaCopson}. The rigorous details can be safely left to the reader.
\end{proof}

Letting $r=2$ in the previous theorem we obtain the following improvement of Corollary \ref{Corollary 2.3}.

\begin{cor}\label{Theorem 4.5}
	Let $1 < p < \infty, 0 < q \leq \infty$, and $b > -1/q$. Then
	    \begin{equation}\label{4.17}
        \mathfrak{L} \cap H^{0,b+1/q}_p(\mathbb{R}^d) \hookrightarrow \mathbf{B}^{0,b}_{p,q}(\mathbb{R}^d)\quad \text{if and only if}\quad q \geq 2
    \end{equation}
    and
    \begin{equation}\label{4.18}
        \mathfrak{L} \cap \mathbf{B}^{0,b}_{p,q}(\mathbb{R}^d) \hookrightarrow H^{0,b+1/q}_p(\mathbb{R}^d) \quad\text{if and only if}\quad q \leq 2.
    \end{equation}
    In particular,
    \begin{equation}\label{LacunaryWB}
        \mathfrak{L} \cap \emph{\textbf{B}}^{0,b}_{p,q}(\mathbb{R}^d) =  \mathfrak{L} \cap H^{0,b+1/q}_p(\mathbb{R}^d) \quad \text{if and only if} \quad q=2.
    \end{equation}
\end{cor}
\begin{rem}
	In Propositions \ref{Proposition 5.2} and \ref{Proposition 5.5} below we will show that $\xi=b+1/q$ in $H^{0,\xi}_p(\mathbb{R}^d)$ is the best possible exponent of the logarithmic smoothness for which the embeddings (\ref{4.17}) and (\ref{4.18}) hold, respectively.
\end{rem}

\subsection{Embeddings between $B^{0,b}_{p,q}(\mathbb{R}^d)$ and $\mathbf{B}^{0,b}_{p,q}(\mathbb{R}^d)$}
Let $1 < p < \infty, 0 < q \leq \infty$, and $b > -1/q$. Recall that (cf. (\ref{1}))
\begin{equation}\label{1new}
    B^{0,b + 1/\min\{2,p,q\}}_{p,q}(\mathbb{R}^d) \hookrightarrow
    \textbf{B}^{0,b}_{p,q}(\mathbb{R}^d) \hookrightarrow
    B^{0,b+1/\max\{2,p,q\}}_{p,q}(\mathbb{R}^d).
\end{equation}
The goal of this subsection is to show that these embeddings can be sharpened when we work with the $\mathfrak{L}$ class, that is, the loss of logarithmic smoothness is less than $1/\min\{2,p,q\}$ (respectively, $1/\max\{2,p,q\}$) in the left-hand side embedding (respectively, right-hand side embedding) in (\ref{1new}).

\begin{thm}\label{Theorem 5.10new}
		Let $1 < p < \infty, 0 < q \leq \infty$, and $b > -1/q$. Then
	\begin{equation}\label{EmbLacunaryBB}
 	 \mathfrak{L} \cap  B^{0,b + 1/\min\{2,q\}}_{p,q}(\mathbb{R}^d) \hookrightarrow
   \mathfrak{L} \cap  \mathbf{B}^{0,b}_{p,q}(\mathbb{R}^d) \hookrightarrow
    \mathfrak{L} \cap  B^{0,b+1/\max\{2,q\}}_{p,q}(\mathbb{R}^d).
\end{equation}
In particular, we have
\begin{equation}\label{LacunaryBB}
	 \mathfrak{L} \cap  B^{0,b + 1/q}_{p,q}(\mathbb{R}^d) =
     \mathfrak{L} \cap  \mathbf{B}^{0,b}_{p,q}(\mathbb{R}^d) \quad \text{ if and only if } \quad q=2.
\end{equation}
\end{thm}
\begin{rem} The fact that $\mathfrak{L} \cap  B^{0,b + 1/2}_{p,2}(\mathbb{R}^d) =
     \mathfrak{L} \cap  \mathbf{B}^{0,b}_{p,2}(\mathbb{R}^d)$ also follows from (\ref{LacunaryFB1}) and (\ref{LacunaryWB}).

\end{rem}
\begin{proof}[Proof of Theorem \ref{Theorem 5.10new}]
The embeddings (\ref{EmbLacunaryBB}) are direct consequences of (\ref{Section 7.4: new}) and (\ref{LacunaryFB}).
    Hence it only remains to show the only-if-part of (\ref{LacunaryBB}). Suppose that $2 < q$. Let
	\begin{equation*}
		b_j = (1 + j)^{-\beta}, \quad j \in \mathbb{N},
	\end{equation*}
	where $\max\{b + 2/q,1/2\} < \beta < b + 1/q + 1/2$. Thus, applying the characterizations (\ref{4.3}) and (\ref{4.6}) to the Fourier series $f(x)\sim \psi(x) \sum_{j=3}^\infty b_j e^{i(2^j-2) x_1}, x \in \mathbb{R}^d$, where $\psi \in \mathcal{S}(\mathbb{R}^d) \backslash \{0\}$ with (\ref{4.2}), we obtain that $f \in B^{0,b+1/q}_{p,q}(\mathbb{R}^d)$, but $f \not \in \mathbf{B}^{0,b}_{p,q}(\mathbb{R}^d)$. We arrive at the contradiction. The case $2 > q$ can be treated similarly.
\end{proof}

\subsection{Characterization of spaces $B^{s,b}_{p,q}(\mathbb{T}), F^{s,b}_{p,q}(\mathbb{T}),$ and $\mathbf{B}^{s,b}_{p,q}(\mathbb{T})$}

Our goal in this subsection is to establish the periodic counterparts of results stated in Sections 7.1-7.5. We start by recalling the definition of lacunary sequences.


Fixing $\lambda>1$, we say that
$a=\{a_n\}_{n\in\N}$ is \emph{lacunary}, $a\in
\mathfrak{l}$\index{\bigskip\textbf{Sets}!$\mathfrak{l}$}\label{LACSEQ}, if there is $\lambda > 1$ such that
\begin{equation}\label{lacunarysequence}
	a_n=0 \quad\mbox{for all} \quad n\neq \lambda^j, j = 0, 1, \ldots .
\end{equation}
The following result extends the famous Zygmund property on lacunary series
for various smoothness spaces.
\begin{prop}\label{Proposition Lac}
	Let $1< p < \infty$, $0 < q \leq \infty$, and $s,b \in \mathbb{R}$. Let $f \in L_1(\mathbb{T})$ be such that
	\begin{equation*}
		f(x) \sim \sum_{n=1}^\infty (a_n \cos nx + b_n \sin nx) \text{ with } a, b \in \mathfrak{l}.
	\end{equation*}
	Then
	\begin{equation}\label{4.3per}
		\|f\|_{B^{s,b}_{p,q}(\mathbb{T})}  \asymp \|f\|_{F^{s,b}_{p,q}(\mathbb{T})} \asymp \left(\sum_{j=0}^\infty \Big(\lambda^{j s} (1 + j)^{b} (|a_{\lambda^j}| + |b_{\lambda^j}|)\Big)^q \right)^{1/q}.
	\end{equation}
	In particular, we have
	\begin{equation}\label{4.4per}
    		\|f\|_{L_p(\mathbb{T})} \asymp \left(\sum_{j=0}^\infty
    		(|a_{\lambda^j}| + |b_{\lambda^j}|)^2\right)^{1/2}
	\end{equation}
	or, more generally,
	\begin{equation}\label{4.4newper}
		\|f\|_{H^{s,b}_p(\mathbb{T})} \asymp \left(\sum_{j=0}^\infty \Big(\lambda^{j s} (1 + j)^b (|a_{\lambda^j}|+|b_j|)\Big)^2 \right)^{1/2}.
	\end{equation}
\end{prop}

\begin{rem}
Relation (\ref{4.4per}) was proved by Zygmund \cite[Chapter V]{Zygmund}. Note that it holds for any $0<p<\infty$.
As in the case of Lebesgue spaces, we see that   Besov, Triebel-Lizorkin,  and Sobolev norms of
  lacunary Fourier series
   are independent of the integrability parameter $p$.
\end{rem}

The proof of Proposition \ref{Proposition Lac} is an immediate consequence of the following equivalent decompositions of Besov and Triebel-Lizorkin spaces which extend (\ref{Section2:new}) and (\ref{Section 2.1: New2}) from the usual resolution of unity $\{\varphi_j\}_{j \in \mathbb{N}_0}$ with (\ref{SmoothFunction}) and (\ref{resolution}) to more general resolutions of unity in $\mathbb{R}^d$. We refer to \cite[Remark 3, p. 46]{Triebel1} for further details.

\begin{lem}
	Let $1 < p < \infty, 0 < q \leq \infty$, and $s,b \in \mathbb{R}$. Assume that $\lambda > 1$. Then there exists a resolution of unity $\{\varphi_j^\lambda\}_{j \in \mathbb{N}_0} \subset \mathcal{S}(\mathbb{R}^d)$  satisfying
	\begin{equation*}
		\emph{supp } \varphi_0^\lambda  \subset \{x \in \mathbb{R}^d : |x| \leq \lambda\},
	\end{equation*}
	\begin{equation*}
		\emph{supp } \varphi_j^\lambda \subset \{x \in \mathbb{R}^d : \lambda^{j-1} \leq |x| \leq \lambda^{j+1}\} \text{ if } j = 1,2,\ldots
	\end{equation*}
	for which
	\begin{equation*}
	\|f\|_{B^{s,b}_{p,q}(\mathbb{T})} \asymp \left(\sum_{j=0}^\infty \lambda^{j s q} (1 + j)^{b q} \Big\|\sum_{k=-\infty}^\infty \varphi^\lambda_j(k) \widehat{f}(k) e^{i k x}\Big\|_{L_p(\mathbb{T})}^q\right)^{1/q}
	\end{equation*}
	and
	\begin{equation*}
	\|f\|_{F^{s,b}_{p,q}(\mathbb{T})} \asymp \left\|\left(\sum_{j=0}^\infty \left(\lambda^{j s} (1 + j)^{b} \Big|\sum_{k=-\infty}^\infty \varphi^\lambda_j(k) \widehat{f}(k) e^{i k x} \Big|\right)^q \right)^{1/q}\right\|_{L_p(\mathbb{T})}.
\end{equation*}
\end{lem}



Let $1 < p < \infty$. It follows from (\ref{4.4per}) that
\begin{equation}\label{4.13}
E^\ast_n(f)_{L_p(\mathbb{T})}\,\,\asymp\,\, \|f - S_{n-1}(f)\|_{L_p(\mathbb{T})} \asymp \left(\sum\limits_{k=n}^\infty
(|a_k| + |b_k|)^2\right)^{1/2}
\end{equation}
where $E^\ast_n(f)_{L_p(\mathbb{T})}$ is given by (\ref{errors}).

\begin{thm}\label{Theorem 4.4}
	Let $1 < p < \infty, 0 < q \leq \infty$, and $-\infty < b < \infty$. Let $f \in L_1(\mathbb{T})$ be such that
	\begin{equation*}
		f(x) \sim \sum_{n=1}^\infty (a_n \cos nx + b_n \sin nx) \text{ with } a, b \in \mathfrak{l}.
	\end{equation*}
	If $s > 0$ then
	\begin{equation}\label{4.14}
		\|f\|_{\mathbf{B}^{s,b}_{p,q}(\mathbb{T})} \asymp \left(\sum_{j=0}^\infty \lambda^{j s q} (1 + j)^{b q} (|a_{\lambda^j}| + |b_{\lambda^j}|)^q \right)^{1/q}.
	\end{equation}
	In the case  $s =0$ we have
	\begin{equation}\label{4.15}
		\|f\|_{\mathbf{B}^{0,b}_{p,q}(\mathbb{T})} \asymp \left(\sum_{j=0}^\infty \left( (1 + j)^b \Big(\sum\limits_{k=j}^\infty
(|a_{\lambda^k}| + |b_{\lambda^k}|)^2\Big)^{1/2}\right)^q \right)^{1/q}.
	\end{equation}
\end{thm}

\begin{rem}
(i) The convergence of the sums in (\ref{4.14}) and (\ref{4.15}) yield that
$
	\sum_{j=0}^\infty (|a_{\lambda^j}| + |b_{\lambda^j}|)^2 < \infty,
$
and then, by (\ref{4.4per}), $f \in L_p(\mathbb{T})$.

(ii) The interesting case in (\ref{4.15}) is when $b \geq -1/q$ if $q < \infty$ ($b > 0$ if $q=\infty$). Otherwise, both terms in (\ref{4.15}) are equivalent to $\|f\|_{L_p(\mathbb{T})}$ (see (\ref{4.4per})).

\end{rem}
\begin{proof}[Proof of Theorem \ref{Theorem 4.4}]
	Using Lemma \ref{LemmaBesovApprox} and relation (\ref{4.13}), we derive that
	\begin{align}
		\|f\|_{\mathbf{B}^{s,b}_{p,q}(\mathbb{T})} & \asymp \left(\sum_{j=0}^\infty (\lambda^{j s} (1 + j)^b E^\ast_{\lambda^j}(f)_{L_p(\mathbb{T})})^q \right)^{1/q} + \|f\|_{L_p(\mathbb{T})} \nonumber\\
		& \asymp \left(\sum_{j=0}^\infty \left(\lambda^{j s} (1 + j)^b \Big(\sum\limits_{k=j}^\infty
(|a_{\lambda^k}| + |b_{\lambda^k}|)^2\Big)^{1/2}\right)^q \right)^{1/q}, \label{4.16}
	\end{align}
where we used (\ref{4.4per}) in the last estimate. Hence, (\ref{4.15}) holds.
Relation (\ref{4.14}) easily follows  from (\ref{4.16}) by Hardy's inequality ($q \geq 2$) (see (\ref{HardyInequal2**})) and by $\ell_q \hookrightarrow \ell_2$ ($q <2$).

\end{proof}

We are now in a position to establish the periodic counterparts of Proposition \ref{Proposition 4.5} and Theorem \ref{Theorem 4.5new}. The proofs follow along the same lines as the proofs of Proposition \ref{Proposition 4.5} and Theorem \ref{Theorem 4.5new} but now using Proposition \ref{Proposition Lac} and Theorem \ref{Theorem 4.4}.

\begin{prop}
Let $1 < p < \infty, 0 < q, r \leq \infty$, and $-\infty < s, b < \infty$. Then we have
\begin{enumerate}[\upshape(i)]
\item
    $\mathfrak{l} \cap F^{s,b}_{p,r}(\mathbb{T}) \hookrightarrow B^{s,b}_{p,q}(\mathbb{T}) \quad \text{ if and only if } \quad q \geq r$,

\item $\mathfrak{l} \cap  B^{s,b}_{p,q}(\mathbb{T}) \hookrightarrow F^{s,b}_{p,r}(\mathbb{T}) \quad \text{ if and only if } \quad  q \leq r$,
\item $\mathfrak{l} \cap F^{s,b}_{p,r}(\mathbb{T})  = \mathfrak{l} \cap  B^{s,b}_{p,q}(\mathbb{T}) \quad \text{ if and only if } \quad r=q.$
    \end{enumerate}

If, in addition, $s > 0$, then
    \begin{enumerate}[\upshape(i)]
    \setcounter{enumi}{3}
    \item      $\mathfrak{l} \cap F^{s,b}_{p,r}(\mathbb{T}) \hookrightarrow \mathbf{B}^{s,b}_{p,q}(\mathbb{T}) \quad \text{ if and only if } \quad  q \geq r$,
          \item  $\mathfrak{l} \cap H^{s,b}_p(\mathbb{T}) \hookrightarrow \mathbf{B}^{s,b}_{p,q}(\mathbb{T}) \quad \text{ if and only if } \quad  q \geq 2$,
        \item $\mathfrak{l} \cap \mathbf{B}^{s,b}_{p,q}(\mathbb{T}) \hookrightarrow F^{s,b}_{p,r}(\mathbb{T}) \quad \text{ if and only if } \quad  q \leq
    r$,
        \item $\mathfrak{l} \cap \mathbf{B}^{s,b}_{p,q}(\mathbb{T}) \hookrightarrow H^{s,b}_p(\mathbb{T}) \quad \text{ if and only if } \quad  q \leq
    2$,
        \item $\mathfrak{l} \cap F^{s,b}_{p,r}(\mathbb{T})  = \mathfrak{l} \cap  \mathbf{B}^{s,b}_{p,q}(\mathbb{T}) \quad \text{ if and only if } \quad r=q$,

     \item $\mathfrak{l} \cap H^{s,b}_{p}(\mathbb{T})  = \mathfrak{l} \cap \mathbf{B}^{s,b}_{p,q}(\mathbb{T}) \quad \text{ if and only if } \quad 2=q$.
\end{enumerate}
\end{prop}

\begin{thm}\label{Theorem 5.18new}
    Let $1 < p < \infty, 0 < q,r \leq \infty$, and $b > -1/q$. Then
    \begin{equation*}
    	\mathfrak{l} \cap F^{0,b+1/q+1/2-1/\max\{r,2\}}_{p,r} (\mathbb{T}) \hookrightarrow \mathbf{B}^{0,b}_{p,q}(\mathbb{T}) \quad \text{if and only if}\quad q \geq r
    \end{equation*}
    and
    \begin{equation*}
    	\mathfrak{l} \cap \mathbf{B}^{0,b}_{p,q}(\mathbb{T}) \hookrightarrow F^{0,b+1/q + 1/2 - 1/\min\{r,2\}}_{p,r}(\mathbb{T})\quad \text{if and only if} \quad q \leq  r.
    \end{equation*}
       In particular,
    \begin{equation*}
    	\mathfrak{l} \cap F^{0,b+1/\min\{q,2\}}_{p,q} (\mathbb{T}) \hookrightarrow \mathfrak{l} \cap \mathbf{B}^{0,b}_{p,q}(\mathbb{T}) \hookrightarrow \mathfrak{l} \cap F^{0,b+1/\max\{q,2\}}_{p,q} (\mathbb{T}).
    \end{equation*}
\end{thm}


\begin{cor}
	Let $1 < p < \infty, 0 < q \leq \infty$, and $b > -1/q$. Then
	    \begin{equation*}
        \mathfrak{l} \cap H^{0,b+1/q}_p(\mathbb{T}) \hookrightarrow \mathbf{B}^{0,b}_{p,q}(\mathbb{T}) \quad \text{if and only if} \quad q \geq 2
    \end{equation*}
    and
    \begin{equation*}
        \mathfrak{l} \cap \mathbf{B}^{0,b}_{p,q}(\mathbb{T}) \hookrightarrow H^{0,b+1/q}_p(\mathbb{T}) \quad \text{if and only if} \quad q \leq 2.
    \end{equation*}
    In particular,
    \begin{equation*}
        \mathfrak{l} \cap \emph{\textbf{B}}^{0,b}_{p,q}(\mathbb{T}) = \mathfrak{l} \cap H^{0,b+1/q}_p(\mathbb{T}) \quad \text{if and only if} \quad q=2.
    \end{equation*}
\end{cor}

\begin{rem}
Using Proposition \ref{Proposition Lac} and Theorem \ref{Theorem 4.4} one can also obtain the periodic counterparts of Theorems \ref{TheoremEmbL} and \ref{Theorem 5.10new}. We leave details to the reader.
\end{rem}

\newpage
\section{Optimality of Propositions \ref{RecallEmb*} and \ref{RecallEmb**}}\label{section-propositions}

In this section we will prove that the conditions given in Propositions \ref{RecallEmb*} and \ref{RecallEmb**} are indeed necessary. This will be done with the help of the characterizations of the norms of smoothness spaces for lacunary Fourier series obtained in Section \ref{section5}.

\begin{prop}\label{RecallEmb*optim}
	Let $1 < p < \infty, 0 < q_0, q_1 \leq \infty$, and $-\infty < s_0, s_1, b_0, b_1 < \infty$. Then, we have
	\begin{equation}\label{8.1new}
		B^{s_0, b_0}_{p,q_0}(\mathbb{R}^d) \hookrightarrow B^{s_1, b_1}_{p,q_1}(\mathbb{R}^d)
	\end{equation}
	if and only if one of the following conditions is valid
	\begin{enumerate}[\upshape(i)]
		\item $s_0 > s_1$,
		\item $s_0 = s_1, q_0 \leq q_1$, and $b_0 \geq b_1$,
		\item $s_0 = s_1, q_0 > q_1$, and $b_0 + \frac{1}{q_0} > b_1 + \frac{1}{q_1}$.
	\end{enumerate}
The similar assertions hold for $F$-spaces as well as for periodic $B$- and $F$-spaces. 
\end{prop}
\begin{proof}
	The if-part is covered by Proposition \ref{RecallEmb*}. Next we rely on lacunary Fourier series
	\begin{equation}\label{8.2new}
		f (x) \sim \sum_{j=3}^\infty b_j e^{i (2^j - 2) x_1} \psi (x), \quad x = (x_1, \ldots, x_d) \in \mathbb{R}^d,
	\end{equation}
	where $\psi \in \mathcal{S}(\mathbb{R}^d) \backslash \{0\}$ with (\ref{4.2}), and we will choose specific examples of sequences of complex numbers $\{b_j\}_{j \in \mathbb{N}}$ in (\ref{8.2new}) showing that the embedding (\ref{8.1new}) does not hold if the conditions (i), (ii), and (iii) are not satisfied.
	
	Assume first that $s_0 < s_1$. Let $0 < \varepsilon < s_1 - s_0$, and $b_j = 2^{-j (s_0 + \varepsilon)}, j \in \mathbb{N}$. According to Proposition \ref{Proposition 4.1}, we have
	\begin{equation*}
		\|f\|_{B^{s_0,b_0}_{p,q_0}(\mathbb{R}^d)}^{q_0} \asymp \sum_{j=3}^\infty 2^{-j \varepsilon q_0} (1 + j)^{b_0 q_0} < \infty
		\end{equation*}
		 and
		 \begin{equation*}
		 \|f\|_{B^{s_1,b_1}_{p,q_1}(\mathbb{R}^d)}^{q_1} \asymp \sum_{j=3}^\infty 2^{j(s_1 - s_0 - \varepsilon) q_1} (1 + j)^{b_1 q_1} = \infty.
	\end{equation*}
	Hence, (\ref{8.1new}) does not hold.
	
	Assume that $s = s_0 = s_1$ and $q_0 \leq q_1$. Let $\varepsilon > 0$. For $N \geq 3$, we put
	\begin{equation*}
			 b_{j,N} = \left\{\begin{array}{lcl}
                            2^{-j (s - \varepsilon)} & ,  & j \leq N, \\
                            & & \\
                            0 & , & j > N.
            \end{array}
            \right.
	\end{equation*}
	Applying Proposition \ref{Proposition 4.1} to $f_N$, where $f_N$ denotes the Fourier series (\ref{8.2new}) with coefficients $\{b_{j, N}\}_{j \in \mathbb{N}}$, we arrive at
	\begin{equation*}
		\|f_N\|_{B^{s,b_0}_{p,q_0}(\mathbb{R}^d)} \asymp \left(\sum_{j=3}^N 2^{j \varepsilon q_0} (1 + j)^{b_0 q_0} \right)^{1/q_0} \asymp 2^{N \varepsilon} (1 + N)^{b_0}
		\end{equation*}
		 and
		 \begin{equation*}
		 \|f_N\|_{B^{s,b_1}_{p,q_1}(\mathbb{R}^d)} \asymp \left(\sum_{j=3}^N 2^{j \varepsilon q_1} (1 + j)^{b_1 q_1} \right)^{1/q_1} \asymp 2^{N \varepsilon} (1 + N)^{b_1}.
	\end{equation*}
	If (\ref{8.1new}) holds then we have necessarily $b _0 \geq b_1$.
	
	Finally we treat the case $s=s_0=s_1, q_0 > q_1$, and $b_0 + 1/q_0 \leq b_1 + 1/q_1$. Using trivial embeddings between Besov spaces, it suffices to show that (\ref{8.1new}) is not true in the limiting case $b_0 + 1/q_0 = b_1 + 1/q_1$. Let $b_j = 2^{-j s} (1 + j)^{-(b_0 + 1/q_0)} (1 + \log j)^{-\beta}$ where $1/q_0 < \beta < 1/q_1$. Applying again Proposition \ref{Proposition 4.1}, it can easily be checked that (\ref{8.2new}) belongs to $B^{s,b_0}_{p,q_0}(\mathbb{R}^d)$ but not to $B^{s,b_1}_{p,q_1}(\mathbb{R}^d)$.
	
	Note that the above arguments also work with $F$-spaces, periodic $B$-spaces and periodic $F$-spaces.

\end{proof}

\begin{rem}
	Working with function spaces on $\mathbb{R}^d$, Proposition \ref{RecallEmb*optim} can also be proved by wavelet arguments which, allow us to obtain embeddings for function spaces via those for  sequence spaces. See \cite{Almeida} and \cite[Section 3.4.2]{Triebel10}.
\end{rem}

\begin{prop}\label{RecallEmb**optim}
	Let $1 < p < \infty, 0 < q_0, q_1 \leq \infty, s_0, s_1 \geq 0,$ and $-\infty< b_0, b_1 < \infty \, (b_i \geq -1/q_i \quad \text{if} \quad s_i = 0, i = 0, 1)$. Then, we have
	\begin{equation}\label{8.3new2}
		\mathbf{B}^{s_0, b_0}_{p,q_0}(\mathbb{R}^d) \hookrightarrow \mathbf{B}^{s_1, b_1}_{p,q_1}(\mathbb{R}^d)
	\end{equation}
	if and only if one of the following conditions is valid
	\begin{enumerate}[\upshape(i)]
		\item $s_0 > s_1$,
		\item $s_0 = s_1 > 0, q_0 \leq q_1$, and $b_0 \geq b_1$,
		\item $s_0 = s_1 > 0, q_0 > q_1$, and $b_0 + \frac{1}{q_0} > b_1 + \frac{1}{q_1}$,
		\item $s_0 = s_1 = 0, b_0 + \frac{1}{q_0} > b_1 + \frac{1}{q_1}$,
		\item $s_0 = s_1 = 0, b_0 + \frac{1}{q_0} = b_1 + \frac{1}{q_1}$, and $q_0 \leq q_1$.
	\end{enumerate}
	The corresponding assertion for periodic $\mathbf{B}$-spaces holds.
\end{prop}
\begin{proof}
	The if-part is covered by Proposition \ref{RecallEmb**}. So in what follows we restrict ourselves to the only-if-parts. For this purpose we use the lacunary Fourier series $f$ given in (\ref{8.2new}). Firstly, assume that $0 < s _ 0 < s_ 1$. We argue as in the proof of Proposition \ref{RecallEmb*optim} to show that $f$ with $b_j = 2^{-j (s_0 + \varepsilon)}, j \in \mathbb{N}$, where $0 < \varepsilon < s_1 - s_0$, belongs to $\mathbf{B}^{s_0, b_0}_{p,q_0}(\mathbb{R}^d)$ but not to $\mathbf{B}^{s_1, b_1}_{p,q_1}(\mathbb{R}^d)$. Next, we deal with the case $s_0 = 0 < s_1$. We define $b_j = (1+ j)^{-\beta}$ with $\beta > 1/2 + b_0 + 1/q_0$. Then, by Theorem \ref{Theorem 4.2}, we have
	\begin{equation*}
		\|f\|_{\mathbf{B}^{0,b_0}_{p,q_0}(\mathbb{R}^d)}^{q_0} \asymp \sum_{j=3}^\infty (1 + j)^{b_0 q_0} \left(\sum_{k=j}^\infty (1 + k)^{-2 \beta} \right)^{q_0/2} \asymp \sum_{j=3}^\infty (1 + j)^{(b_0 - \beta + 1/2) q_0} < \infty
	\end{equation*}
	and
	\begin{equation*}
		\|f\|_{\mathbf{B}^{s_1,b_1}_{p,q_1}(\mathbb{R}^d)}^{q_1} \asymp \sum_{j=3}^\infty 2^{j s_1 q_1} (1 + j)^{(b_1 - \beta) q_1} = \infty.
	\end{equation*}
	This shows that $f \in \mathbf{B}^{0,b_0}_{p,q_0}(\mathbb{R}^d)$ but $f \not \in \mathbf{B}^{s_1,b_1}_{p,q_1}(\mathbb{R}^d)$. Consequently, if (\ref{8.3new2}) holds then $s_0 \geq s_1 \geq 0$.
	
	Suppose now that $s_0 = s_1 = s \geq 0$. If $s=0$, we distinguish two possible cases. Assume first that $s=0$ and $b_0 + 1/q_0 < b_1 + 1/q_1$. We put $b_j = (1+ j)^{-\beta}$ where $1/2 + b_0 + 1/q_0 < \beta < 1/2 + b_1 + 1/q_1$. Using again Theorem \ref{Theorem 4.2}, we obtain
	\begin{equation*}
		\|f\|_{\mathbf{B}^{0,b_0}_{p,q_0}(\mathbb{R}^d)}^{q_0} \asymp \sum_{j=3}^\infty (1 + j)^{b_0 q_0} \left(\sum_{k=j}^\infty (1 + k)^{-2 \beta} \right)^{q_0/2} \asymp \sum_{j=3}^\infty (1 + j)^{(b_0 - \beta + 1/2) q_0} < \infty
	\end{equation*}
	and
		\begin{equation*}
		\|f\|_{\mathbf{B}^{0,b_1}_{p,q_1}(\mathbb{R}^d)}^{q_1} \asymp \sum_{j=3}^\infty (1 + j)^{(b_1 - \beta + 1/2) q_1} = \infty.
	\end{equation*}
	These estimates give us the desired counterexample.
	
	Secondly, let $s_0 = s_1= s=0$ and $b_0 + 1/q_0 = b_1 + 1/q_1 > 0$. Then the boundedness of (\ref{8.3new2}) implies $q_0 \leq q_1$. Indeed, if $q_0 > q_1$, we can define the lacunary Fourier series $f$ (see (\ref{8.2new})) with Fourier coefficients
	\begin{equation*}
	b_j = (1 + j)^{-(1/2 + b_0 + 1/q_0)} (1 + \log j)^{-\beta}, \quad \text{where} \quad 1/q_0 < \beta < 1/q_1.
	\end{equation*}
	Then, Theorem \ref{Theorem 4.2} yields that
	\begin{align*}
		\|f\|_{\mathbf{B}^{0,b_0}_{p,q_0}(\mathbb{R}^d)}^{q_0} & \asymp \sum_{j=3}^\infty (1 + j)^{b_0 q_0} \left(\sum_{k=j}^\infty (1 + k)^{-2 (1/2 + b_0 + 1/q_0)} (1 + \log k)^{-\beta 2} \right)^{q_0/2} \\
		& \asymp \sum_{j=3}^\infty (1 + \log j)^{-\beta q_0} \frac{1}{1+j} < \infty
	\end{align*}
	where we have used that $b_0 + 1/q_0 > 0$ in the second estimate. Analogously, we have
	\begin{equation*}
		\|f\|_{\mathbf{B}^{0,b_1}_{p,q_1}(\mathbb{R}^d)}^{q_1}  \asymp \sum_{j=3}^\infty (1 + \log j)^{-\beta q_1} \frac{1}{1+j} = \infty.
	\end{equation*}
	This contradicts our assumption. Hence, we have shown that $q_0 \leq q_1$. Furthermore, the case $s_0 = s_1 = 0$ and $b_0 + 1/q_0 = b_1 + 1/q_1 = 0$ can be treated analogously.

	
	Finally, if $s_0 = s_1 =s > 0$ then we may proceed similarly to the proof of Proposition \ref{RecallEmb*optim} to show that (ii) and (iii) are necessary conditions.
\end{proof}

\begin{rem}
	Our approach can also be applied to derive the optimality of further embedding results. For instance, invoking Theorem \ref{Theorem 3.6} we can easily obtain the characterization of continuous embeddings between Besov spaces with different metrics and constant differential dimension (i.e., $s_0 - \frac{d}{p_0} = s_1 -\frac{d}{p_1}$). More precisely, let $\frac{2d}{d+1} < p_0 \leq p_1 < \infty, 0 < q_0, q_1 \leq \infty, -\infty < s_0, s_1, b_0, b_1 < \infty$. Suppose that $s_0 - \frac{d}{p_0} = s_1 -\frac{d}{p_1}$. Then, we have
	\begin{equation}\label{8.4}
		B^{s_0,b_0}_{p_0,q_0}(\mathbb{R}^d) \hookrightarrow B^{s_1,b_1}_{p_1,q_1}(\mathbb{R}^d)
	\end{equation}
	if and only if one of the following conditions is valid
	\begin{enumerate}[\upshape(i)]
		\item $q_0 \leq q_1$ and $b_0 \geq b_1$,
		\item $q_0 > q_1$ and $b_0 + \frac{1}{q_0} > b_1 + \frac{1}{q_1}$.
	\end{enumerate}
	The same result also holds for periodic Besov spaces. Note that (\ref{8.4}) consists of a generalization of the Sobolev embedding given in Proposition \ref{RecallEmb2}(i). Further, in the special case where $p_0 = p_1$ and so, $s_0=s_1$, the previous characterization covers Proposition \ref{RecallEmb*optim}. A completely different method characterizing (\ref{8.4}) was previously proposed by Leopold \cite[Theorem 1]{Leopold98}. In this paper, the assumption $\frac{2d}{d+1} < p_0$ can be relaxed to cover the full range of parameters.
	
	Further, we can apply Theorem \ref{Theorem 3.9} to obtain the sharpness of the Sobolev embedding theorem for Sobolev spaces. More specifically, let $\frac{2d}{d+1} < p_0 < p_1 < \infty, -\infty < s_0, s_1, b_0, b_1 < \infty$. Suppose that $s_0 - \frac{d}{p_0} = s_1 -\frac{d}{p_1}$. Then, we have
	\begin{equation*}
		H^{s_0,b_0}_{p_0}(\mathbb{R}^d) \hookrightarrow H^{s_1,b_1}_{p_1}(\mathbb{R}^d) \quad \iff \quad  b_0 \geq b_1.
	\end{equation*}
	See also \cite[Proposition 5.3]{CaetanoHaroske}.
	
\end{rem}

\newpage
\section{Optimality of embeddings between Sobolev and Besov spaces with smoothness close to zero}\label{SectionOptimalityWB}

In Theorem \ref{Theorem 3.10} and Corollary \ref{Theorem 4.5}, we have shown  that ($b > -1/q$)
\begin{equation}\label{5.1new}
	 \widehat{GM}^d \cap H^{0,b+1/q}_p(\mathbb{R}^d) \hookrightarrow
        \mathbf{B}^{0,b}_{p,q}(\mathbb{R}^d) \quad \text{if} \quad q \geq p
\end{equation}
and
\begin{equation}\label{5.2new}
	  \mathfrak{L} \cap H^{0,b+1/q}_p(\mathbb{R}^d) \hookrightarrow \mathbf{B}^{0,b}_{p,q}(\mathbb{R}^d) \quad \text{if} \quad q \geq 2.
\end{equation}
 Let us prove that  $\xi = b + 1/q$ is the best possible logarithmic smoothness parameter in Sobolev spaces $H^{0,\xi}_p(\mathbb{R}^d)$
for embeddings  (\ref{5.1new}) and (\ref{5.2new}) to hold.

\begin{prop}\label{Proposition 5.1}
	Let $\frac{2d}{d+1} < p < \infty, p \leq q \leq \infty$, and $b > -1/q$. For any $\varepsilon > 0$, there is a function $f \in \widehat{GM}^d \cap H^{0,b + 1/q - \varepsilon}_p(\mathbb{R}^d)$ such that $f \not \in \mathbf{B}^{0,b}_{p,q}(\mathbb{R}^d)$.
\end{prop}

\begin{proof}
	 Set
\begin{equation*}
	F_0 (t) = t^{-d + d/p} (1 + |\log t|)^{-\beta}, \quad t > 0,
\end{equation*}
where $\beta$ is chosen such that $\max\{1/p, b + 1/q + 1/p - \varepsilon\} < \beta < b + 1/q + 1/p$.
Define
the function
$f(x)=f_{0}(|x|)$, where $f_{0}$ is given by (\ref{3.4new+}).
On the one hand,
noting that $F_0 \in GM$ and
$$
	\left(\int_0^1 t^{d p - d - 1} F_0^p(t) dt \right)^{1/p} + \left(\int_1^\infty t^{d p -d-1} (1 + \log t)^{(b+1/q-\varepsilon) p} F_0^p(t) dt \right)^{1/p} 
 < \infty,
$$
Theorem \ref{Theorem 3.9} implies that $f \in H^{0,b+1/q-\varepsilon}_p(\mathbb{R}^d)$.
On the other hand, since
$$
	\int_1^\infty (1 + \log t)^{b q} \left(\int_t^\infty u^{d p -d -1} F_0^p(u) du \right)^{q/p} \frac{dt}{t} \asymp \int_1^\infty (1 + \log t)^{(b - \beta +1/p) q} \frac{dt}{t} = \infty,
$$
  (\ref{3.3.3}) yields that $f \not \in \mathbf{B}^{0,b}_{p,q}(\mathbb{R}^d)$.

\end{proof}

\begin{prop}\label{Proposition 5.2}
	Let $1 < p < \infty, 2 \leq q \leq \infty$, and $b > -1/q$. For any $\varepsilon > 0$, there is a function $f \in \mathfrak{L} \cap H^{0,b + 1/q - \varepsilon}_p(\mathbb{R}^d)$ such that $f \not \in \mathbf{B}^{0,b}_{p,q}(\mathbb{R}^d)$.
\end{prop}

\begin{proof}
	We choose $\delta$ satisfying that $\max\{1/2, b + 1/q + 1/2 - \varepsilon\} < \delta < b +1/q + 1/2$ and put $\{b_j\}_{j \in \mathbb{N}} = \{(1 + j)^{-\delta}\}_{j \in \mathbb{N}}$. We define
	\begin{equation*}
		W(x) \sim \sum_{j=3}^\infty b_j e^{i (2^j -2) x_1},\quad  x \in \mathbb{R}^d,
	\end{equation*}
	and $f = \psi W$, where the function $\psi \in \mathcal{S}(\mathbb{R}^d) \backslash \{0\}$ fulfills (\ref{4.2}). It follows from (\ref{4.4new}) that $f \in H^{0, b+1/q -\varepsilon}_p(\mathbb{R}^d)$ because
\begin{equation*}
	\sum_{j=3}^\infty ((1 + j)^{b + 1/q - \varepsilon} |b_j|)^2 = \sum_{j=3}^\infty (1 + j)^{(b + 1/q - \delta - \varepsilon) 2} < \infty.
\end{equation*}
However, $f \not \in \mathbf{B}^{0,b}_{p,q}(\mathbb{R}^d)$ since
\begin{align*}
	\|f\|^q_{\mathbf{B}^{0,b}_{p,q}(\mathbb{R}^d)} &\asymp \sum_{j=3}^\infty (1 + j)^{b q} \left(\sum_{k=j}^\infty |b_k|^2\right)^{q/2}
	 \asymp \sum_{j=3}^\infty (1 + j)^{(b - \delta + 1/2)q} = \infty,
\end{align*}
where we have used (\ref{4.6}).
\end{proof}

It is worth mentioning that Propositions \ref{Proposition 5.1} and \ref{Proposition 5.2}  also yield the optimality of the embeddings (\ref{5.1new}) and (\ref{5.2new}) in the setting of generalized Sobolev spaces $H^{s,b,\Psi}_p(\mathbb{R}^d)$. Let us recall the definition of the spaces $H^{s,b,\Psi}_p(\mathbb{R}^d)$.
	
	 A function $\Psi$ is said to be \emph{slowly varying} on $(0,\infty)$ if for each $\varepsilon > 0$, the function $t^\varepsilon \Psi(t)$ is equivalent to an increasing function and the function $t^{-\varepsilon} \Psi(t)$ is equivalent to a decreasing function. We denote by $SV$\index{\bigskip\textbf{Sets}!$SV$}\label{SV} the class formed by all slowly varying functions on $(0,\infty)$.
	
	Examples of $SV$ functions are powers of logarithms $\ell^\alpha(t), \alpha \in \mathbb{R}$, compositions of appropriate $\log$ functions, the $\ell^{\mathbb{A}}(t)$ functions introduced in (\ref{brokenlog}), $e^{|\log t|^\alpha}$ with $\alpha \in (0,1)$, etc. For further examples of $SV$ functions, as well as some of their properties, we refer the interested reader to \cite{BinghamGoldieTeugels} and \cite{EdmundsEvans}.
	
	Let $1 < p < \infty, -\infty < s , b < \infty$, and $\Psi \in SV$. The Sobolev space
$H^{s,b,\Psi}_p(\mathbb{R}^d)$\index{\bigskip\textbf{Spaces}!$H^{s,b,\Psi}_p(\mathbb{R}^d)$}\label{SOBGENERAL} consists of all $f \in \mathcal{S}'(\mathbb{R}^d)$ such that
\begin{equation*}
    \|f\|_{H^{s,b,\Psi}_p(\mathbb{R}^d)} = \|((1 + |x|^2)^{s/2} (1 + \log (1 + |x|^2))^b \Psi(1 + \log (1 + |x|^2)) \widehat{f})^\vee\|_{L_p(\mathbb{R}^d)} < \infty.
\end{equation*}
In particular, if $\Psi \equiv 1$ we recover the spaces $H^{s,b}_p(\mathbb{R}^d)$ (cf. (\ref{def:SobolevSpace})).

The following sharpness assertions hold.

\begin{prop}\label{Proposition 5.1 SV}
	Let $\frac{2d}{d+1} < p < \infty, p \leq q \leq \infty$, and $b > -1/q$. Let $\Psi \in SV$. For any $\varepsilon > 0$, there is a function $f \in \widehat{GM}^d \cap H^{0,b + 1/q - \varepsilon, \Psi}_p(\mathbb{R}^d)$ such that $f \not \in \mathbf{B}^{0,b}_{p,q}(\mathbb{R}^d)$.
\end{prop}

\begin{prop}\label{Proposition 5.2 SV}
	Let $1 < p < \infty, 2 \leq q \leq \infty$, and $b > -1/q$. Let $\Psi \in SV$. For any $\varepsilon > 0$, there is a function $f \in \mathfrak{L} \cap H^{0,b + 1/q - \varepsilon, \Psi}_p(\mathbb{R}^d)$ such that $f \not \in \mathbf{B}^{0,b}_{p,q}(\mathbb{R}^d)$.
\end{prop}

Propositions \ref{Proposition 5.1 SV} and \ref{Proposition 5.2 SV} are simple consequences of Propositions \ref{Proposition 5.1} and \ref{Proposition 5.2}, respectively, due to the trivial embedding
\begin{equation*}
	H^{s,b_0}_p(\mathbb{R}^d) \hookrightarrow H^{s,b_1,\Psi}_p(\mathbb{R}^d) \quad \text{whenever} \quad b_0 > b_1.
\end{equation*}

Similarly one can deal with the sharpness of the logarithmic parameter of Sobolev spaces in the converse  embeddings
\begin{equation*}
	 \widehat{GM}^d  \cap  \mathbf{B}^{0,b}_{p,q}(\mathbb{R}^d) \hookrightarrow
        H^{0,b+1/q}_p(\mathbb{R}^d) \quad \text{if} \quad q \leq p
\end{equation*}
and
\begin{equation*}
	  \mathfrak{L} \cap \mathbf{B}^{0,b}_{p,q}(\mathbb{R}^d) \hookrightarrow H^{0,b+1/q}_p(\mathbb{R}^d)  \quad \text{if} \quad q \leq 2,
\end{equation*}
obtained in Theorem \ref{Theorem 3.10} and Corollary \ref{Theorem 4.5}, respectively. 

\begin{prop}\label{Proposition 5.4}
	Let $\frac{2d}{d+1} < p < \infty, q \leq p$, and $b > -1/q$. For any $\varepsilon > 0$, there is a function $f \in \widehat{GM}^d \cap \mathbf{B}^{0,b}_{p,q}(\mathbb{R}^d)$ such that $f \not \in H^{0,b + 1/q + \varepsilon}_p(\mathbb{R}^d)$.
\end{prop}
\begin{proof}
	Let
	\begin{equation*}
		F_0(t) = t^{-d + d/p} (1 + |\log t|)^{-\beta},\quad t > 0,
	\end{equation*}
	where $b + 1/q + 1/p < \beta \leq b + 1/q +1/p + \varepsilon$.
 Define
the function
$f(x)=f_{0}(|x|)$, where $f_{0}$ is given by (\ref{3.4new+}).
 Applying Theorems \ref{Theorem 3.2} and \ref{Theorem 3.9},
 we obtain that $f \in\mathbf{B}^{0,b}_{p,q}(\mathbb{R}^d)$ but  $f \notin H^{0,b + 1/q + \varepsilon}_p(\mathbb{R}^d)$.
\end{proof}

\begin{prop}\label{Proposition 5.5}
	Let $1 < p < \infty, q \leq 2$, and $b > -1/q$. For any $\varepsilon > 0$, there is a function $f \in \mathfrak{L} \cap \mathbf{B}^{0,b}_{p,q}(\mathbb{R}^d)$ such that $f \not \in H^{0,b + 1/q + \varepsilon}_p(\mathbb{R}^d)$.
\end{prop}

\begin{proof}
	Put
	\begin{equation*}
		\{b_j\}_{j \in \mathbb{N}} = \{(1 + j)^{-\delta}\}_{j \in \mathbb{N}}
	\end{equation*}
	with $b + 1/q + 1/2 < \delta \leq b + 1/q + 1/2 + \varepsilon$. Let $f = \psi W$, where $\psi \in \mathcal{S}(\mathbb{R}^d)\backslash \{0\}$ satisfies (\ref{4.2})  and the Fourier series of $W$ is given by
	\begin{equation*}
		 \sum_{j=3}^\infty b_j e^{i(2^j - 2) x_1},\quad x \in \mathbb{R}^d.
	\end{equation*}
	Therefore, Proposition \ref{Proposition 4.1} and Theorem \ref{Theorem 4.2} imply that $f \in \mathbf{B}^{0,b}_{p,q}(\mathbb{R}^d) \backslash H^{0,b+1/q+\varepsilon}_p(\mathbb{R}^d)$.
\end{proof}

Analogously to Propositions \ref{Proposition 5.1 SV} and \ref{Proposition 5.2 SV}, we write down the corresponding results to Propositions \ref{Proposition 5.4} and \ref{Proposition 5.5}.

\begin{prop}\label{Proposition 5.4 SV}
	Let $\frac{2d}{d+1} < p < \infty, q \leq p$, and $b > -1/q$. Let $\Psi \in SV$. For any $\varepsilon > 0$, there is a function $f \in \widehat{GM}^d \cap \mathbf{B}^{0,b}_{p,q}(\mathbb{R}^d)$ such that $f \not \in H^{0,b + 1/q + \varepsilon,\Psi}_p(\mathbb{R}^d)$.
\end{prop}

\begin{prop}\label{Proposition 5.5 SV}
	Let $1 < p < \infty, q \leq 2$, and $b > -1/q$. Let $\Psi \in SV$. For any $\varepsilon > 0$, there is a function $f \in \mathfrak{L} \cap \mathbf{B}^{0,b}_{p,q}(\mathbb{R}^d)$ such that $f \not \in H^{0,b + 1/q + \varepsilon,\Psi}_p(\mathbb{R}^d)$.
\end{prop}

Propositions \ref{Proposition 5.1}, \ref{Proposition 5.2}, \ref{Proposition 5.4}, and \ref{Proposition 5.5}
 imply the following optimality result, which complements Corollary \ref{Corollary 2.3}.

\begin{thm}\label{Proposition 5.3}
	Let $\frac{2d}{d+1} < p < \infty, 0 < q \leq \infty$, and $b > -1/q$.
\begin{enumerate}[\upshape(i)]
\item If $\max\{p,2\} \leq q$, then
	\begin{equation*}
	H^{0,\xi}_p(\mathbb{R}^d) \hookrightarrow \mathbf{B}^{0,b}_{p,q}(\mathbb{R}^d) \text{ \qquad if and only if  \qquad } \xi \geq b + \frac{1}{q}.
\end{equation*}

\item If $q \leq \min\{p,2\}$, then
	\begin{equation*}
	\mathbf{B}^{0,b}_{p,q}(\mathbb{R}^d) \hookrightarrow H^{0,\xi}_p(\mathbb{R}^d) \text{ \qquad  if and only if  \qquad } \xi \leq b + \frac{1}{q}.
\end{equation*}
\end{enumerate}

\end{thm}
Part (i) follows from Propositions \ref{Proposition 5.1} and \ref{Proposition 5.2}
and part (ii) from
 Propositions \ref{Proposition 5.4} and \ref{Proposition 5.5}.
In particular, we see that  $\xi = b+1/q$ is an optimal parameter in embeddings (\ref{2.2}) and
 (\ref{2.3}).

We finish this section by showing another sharpness assertions (see (\ref{Sharp2.2})--(\ref{SharpBesovSobolev})).

\begin{thm}\label{Theorem 6.6}
	Let $\frac{2d}{d+1} < p < \infty, 0 < q \leq \infty, b > -1/q$, and $\xi \in \mathbb{R}$. Then
		     \begin{equation}\label{Sharp2.2New}
        H^{0,b+1/q}_p(\mathbb{R}^d) \hookrightarrow \mathbf{B}^{0,b}_{p,q}(\mathbb{R}^d) \qquad \text{if and only if} \qquad q \geq \max\{p,2\},
    \end{equation}
        \begin{equation}\label{Sharp2.3New}
        \mathbf{B}^{0,b}_{p,q}(\mathbb{R}^d) \hookrightarrow H^{0,b+1/q}_p(\mathbb{R}^d) \qquad \text{if and only if} \qquad q \leq \min\{p,2\},
    \end{equation}
    and
	\begin{equation}\label{SharpBesovSobolevNew}
		\mathbf{B}^{0,b}_{p,q}(\mathbb{R}^d) = H^{0,\xi}_p(\mathbb{R}^d) \qquad \text{if and only if} \qquad p=q=2 \quad \text{and} \quad \xi = b+1/2.
	\end{equation}
\end{thm}
\begin{proof}
	The if-parts of (\ref{Sharp2.2New}), (\ref{Sharp2.3New}) and (\ref{SharpBesovSobolevNew}) were already proved in Corollary \ref{Corollary 2.3}. Then it remains to show their only-if-parts. Concerning (\ref{Sharp2.2New}) and (\ref{Sharp2.3New}), they follow immediately from Theorem \ref{Theorem 3.10} and Corollary \ref{Theorem 4.5}.

 We proceed to the only-if-part in (\ref{SharpBesovSobolevNew}), that is, let us show that if $\mathbf{B}^{0,b}_{p,q}(\mathbb{R}^d) = H^{0,\xi}_p(\mathbb{R}^d)$ then $p=q=2$ and $\xi = b+1/2$. For $0 < s < 1$, applying (\ref{BesovComparison}), (\ref{PrelimInterpolationW}), (\ref{PrelimInterpolationnew}), Lemma \ref{PrelimLemma7.2}(i) (with $\alpha=0$) and (\ref{PrelimInterpolation}), we get
	\begin{align*}
		\mathbf{B}^{s, (1-s) \xi}_{p,q}(\mathbb{R}^d) & = (H^{0,\xi}_p(\mathbb{R}^d), W^1_p(\mathbb{R}^d))_{s,q} =  (\mathbf{B}^{0,b}_{p,q}(\mathbb{R}^d), W^1_p(\mathbb{R}^d))_{s,q} \\
		& \hspace{-2cm}= ((L_p(\mathbb{R}^d), W^1_p(\mathbb{R}^d))_{(0,b),q}, W^1_p(\mathbb{R}^d))_{s,q} = (L_p(\mathbb{R}^d), W^1_p(\mathbb{R}^d))_{s,q;(1-s)(b + 1/q)} \\
		& \hspace{-2cm}= \mathbf{B}^{s, (1-s)(b+1/q)}_{p,q}(\mathbb{R}^d).
	\end{align*}
	This implies that $\xi = b+1/q$ (see Proposition \ref{RecallEmb**optim}). Hence, we have $\mathbf{B}^{0,b}_{p,q}(\mathbb{R}^d) = H^{0,b+1/q}_p(\mathbb{R}^d)$. Applying now (\ref{sharpemb4}) and (\ref{LacunaryWB}) we get $q=p$ and $q=2$, respectively. This finishes the proof.

\end{proof}

\newpage
\section{Comparison between different kinds of smoothness spaces involving only logarithmic smoothness}\label{section7}

Throughout this section we shall assume that $1 < p < \infty$ and $0 < q \leq \infty$, unless otherwise stated. Suppose further that $q \geq \max\{p,2\}$.
It follows from (\ref{1}) and Corollary \ref{Corollary 2.3} that the following embeddings hold
\begin{equation}\label{6.1new}
	H^{0,b+1/q}_p(\mathbb{R}^d)\cup   B^{0,b +
    1/\min\{2,p\}}_{p,q}(\mathbb{R}^d) \hookrightarrow \mathbf{B}^{0,b}_{p,q}(\mathbb{R}^d)  \hookrightarrow B^{0,b+1/q}_{p,q}(\mathbb{R}^d).
\end{equation}
Here, $b > -1/q$.

 Our goal in this section is to prove the sharpness of (\ref{6.1new}). In more detail, we will see that the spaces $H^{0,b+1/q}_p(\mathbb{R}^d) $ and $B^{0,b+1/\min\{2,p\}}_{p,q}(\mathbb{R}^d) $ are not comparable and the embedding $ \mathbf{B}^{0,b}_{p,q}(\mathbb{R}^d)
\hookrightarrow B^{0,b+1/q}_{p,q}(\mathbb{R}^d) $ is strict.
Fig. 3 illustrates the embeddings (\ref{6.1new}) 
and results of Propositions \ref{Proposition 6.1} and \ref{Proposition 6.2} below.

\bigskip

\begin{center}
\begin{tikzpicture}[fill opacity=0.05, ,xscale=0.8,yscale=0.8]

\fill (0,0) ellipse (5.0 and 3.0);
\draw[fill opacity=0.2] (0,0) ellipse (5.0 and 3.0);

\fill (0,0) ellipse (4.0 and 2.2);
\draw[fill opacity=0.2] (0,0) ellipse (4.0 and 2.2);

\fill[white] (-1.8,0) ellipse (2.5 and 1.1);
\draw[fill opacity=0.2](-1.4,0) ellipse (2.5 and 1.1);

\fill[white] (1.8,0) circle (2.55 and 1.1);
\draw[fill opacity=0.2] (1.4,0) circle (2.55 and 1.1);

\node[fill opacity=2,xscale=0.8,yscale=0.8] at (0,1.6) {$\mathbf{B}^{0,b}_{p,q}(\mathbb{R}^d)$};
\node[fill opacity=2,xscale=0.8,yscale=0.8] at (-2.5,0) {$H^{0,b+\frac{1}{q}}_p(\mathbb{R}^d)$};
\node[fill opacity=2,xscale=0.8,yscale=0.8] at (2.5,0) {$B^{0,b + \frac{1}{\min\{2,p\}}}_{p,q}(\mathbb{R}^d)$};
\node[fill opacity=2,xscale=0.8,yscale=0.8] at (0,-2.6) {$B^{0,b + \frac{1}{q}}_{p,q}(\mathbb{R}^d)$};

\end{tikzpicture}
\end{center}

\phantom{qqq}

\begin{center}
{\small \textbf{Fig. 3:} Relationships between the Besov and Sobolev  spaces involving only logarithmic smoothness in the case $q \geq \max\{p,2\}$.}
\end{center}

\bigskip

We start by showing that the spaces $H^{0,b+1/q}_p(\mathbb{R}^d)$ and $B^{0,b+1/\min\{2,p\}}_{p,q}(\mathbb{R}^d)$
are not comparable except the case $p=q=2$, cf. 
(\ref{LPgeneral}) and Proposition \ref{RecallEmb}.

\begin{prop}\label{Proposition 6.1}
	Let $q \geq \max\{p,2\} (q > 2 \text{ if } p=2)$ and $-\infty < b < \infty$. Then
\begin{enumerate}[\upshape (i)]
  \item	$H^{0,b+1/q}_p(\mathbb{R}^d)$ \text{ is not continuously embedded into } $B^{0,b+1/\min\{2,p\}}_{p,q}(\mathbb{R}^d)$,

  \item if $\frac{2 d}{d+1} < p < \infty$ then	 $B^{0,b+1/\min\{2,p\}}_{p,q}(\mathbb{R}^d)$ \text{ is not continuously embedded into } $H^{0,b+1/q}_p(\mathbb{R}^d).$

\end{enumerate}
\end{prop}
\begin{proof}

Assume first that $p < 2$ and let us prove (i). Let $W$ be the Fourier series given by (\ref{4.1}) with $\{b_j\}_{j \in \mathbb{N}} = \{(1 +
j)^{-b-1/q-\beta}\}_{j \in \mathbb{N}}, 1 / 2 < \beta < 1/p,$ and put $f=\psi W$, where $\psi \in \mathcal{S}(\mathbb{R}^d) \backslash \{0\}$ satisfies (\ref{4.2}). On the one hand, using (\ref{4.4new}), we derive that $f \in H^{0,b+1/q}_p(\mathbb{R}^d)$ because
\begin{equation*}
    \|f\|_{H^{0,b+1/q}_p(\mathbb{R}^d)}^2 \asymp \sum_{j=3}^\infty ((1 + j)^{b+1/q} |b_j|)^2 = \sum_{j=3}^\infty
    (1 + j)^{-2 \beta} < \infty.
\end{equation*}
On the other hand, $f \not \in B^{0,b+1/p}_{p,q}(\mathbb{R}^d)$ since
\begin{equation*}
    \|f\|_{B^{0,b+1/p}_{p,q}(\mathbb{R}^d)}^q \asymp \sum_{j=3}^\infty ((1 + j)^{b+1/p} |b_j|)^q = \sum_{j=3}^\infty
    (1 + j)^{(1/p -1/q - \beta)q} = \infty,
\end{equation*}
where we have used (\ref{4.3}). This shows (i).

Conversely, we construct a function $f$ which belongs to $B^{0,b+1/p}_{p,q}(\mathbb{R}^d)$ but not to $H^{0,b+1/q}_p(\mathbb{R}^d)$. Let $\varepsilon > 0$ and $1/q < \eta \leq 1/p$. For $t >0$, set
	\begin{equation*}
		 F_0(t) = \left\{\begin{array}{lcl}
                            t^{-d + d/p + \varepsilon} & ,  & 0 < t  < 1, \\
                            & & \\
                            t^{-d + d/p} (1 + \log t)^{-(b + 1/p + 1/q)} (1 + \log (1 + \log t))^{-\eta} & , & t \geq 1,
            \end{array}
            \right.
	\end{equation*}
 and $\widehat{f}(x) = F_0(|x|), x \in \mathbb{R}^d$. It is clear that $F_0 \in GM$. Furthermore,
 \begin{align*}
 	\left(\int_0^1 t^{d p-d - 1} F_0^p(t) dt\right)^{1/p} + \left(\int_1^\infty t^{dq -d q/p - 1} (1 + \log t)^{(b + 1/p) q} F_0^q(t) dt\right)^{1/q} 
< \infty
 \end{align*}
and by Theorem \ref{Theorem 3.6} we get that $f \in B^{0,b+1/p}_{p,q}(\mathbb{R}^d)$.
However, in light of
 Theorem \ref{Theorem 3.9} and
\begin{equation*}
	\int_1^\infty t^{d p - d -1} (1 + \log t)^{(b+1/q) p} F_0^p(t) dt= \int_1^\infty (1 + \log (1 + \log t))^{-\eta p} \frac{dt}{t (1 + \log t)} = \infty,
\end{equation*}
 we arrive at $f \not \in H^{0,b+1/q}_p(\mathbb{R}^d)$. Thus, (ii) holds.

Now assume that $p \geq 2$. First, we will show (ii), that is, there exist  functions from the space
$B^{0,b+1/2}_{p,q}(\mathbb{R}^d)$ but not from $H^{0,b+1/q}_p(\mathbb{R}^d)$. As an example take  
 the lacunary Fourier series associated to sequences
\begin{equation*}
\{b_j\}_{j \in \mathbb{N}} = \{(1 + j)^{-b - 1/2 -1/q} (1+ \log j)^{-\delta}\}_{j \in \mathbb{N}}, \quad 1/q < \delta
< 1/2.
\end{equation*}
Indeed, we have
\begin{equation*}
    \sum_{j=3}^\infty ((1+j)^{b+1/2} |b_j|)^q = \sum_{j=3}^\infty (1+ \log j)^{-\delta
    q} \frac{1}{1+j} < \infty
\end{equation*}
and
\begin{equation*}
    \sum_{j=3}^\infty ((1 + j)^{b+1/q} |b_j|)^2 = \sum_{j=3}^\infty (1 + \log
    j)^{-\delta 2} \frac{1}{1+j} = \infty
\end{equation*}
and the result follows from (\ref{4.3}) and (\ref{4.4new}).

Finally, we are going to prove (i), that is,
\begin{equation}\label{6.1}
	H^{0,b+1/q}_p(\mathbb{R}^d) \not \hookrightarrow B^{0,b+1/2}_{p,q}(\mathbb{R}^d).
\end{equation}
This will be done by contradiction. Let us assume that $H^{0,b+1/q}_p(\mathbb{R}^d)$ is continuously embedded into $B^{0,b+1/2}_{p,q}(\mathbb{R}^d)$. For $\nu \in \mathbb{N}_0$, we define $(F_0)_\nu(t) = \chi_{(0,2^\nu)}(t), t > 0$. Then
\begin{align}
	 \left(\int_0^1 t^{d p-d - 1} (F_0)_\nu^p(t) dt\right)^{1/p}  + \left(\int_1^\infty t^{dp-d-1} (1 + \log t)^{(b + 1/q) p} (F_0)_\nu^p(t) dt \right)^{1/p} \nonumber \\
	& \hspace{-10cm}\asymp 2^{\nu d (1-1/p)} (1 + \nu)^{b+1/q} \label{6.2}
\end{align}
and
\begin{align}
\left(\int_0^1 t^{d p-d - 1} (F_0)_\nu^p(t) dt\right)^{1/p}  + \left(\int_1^\infty t^{d q-d q/p-1} (1 + \log t)^{(b + 1/2) q} (F_0)_\nu^q(t) dt \right)^{1/q} \nonumber \\
	& \hspace{-10cm} \asymp  2^{\nu d(1 - 1/p)} (1 + \nu)^{b+1/2}. \label{6.3}
\end{align}
Therefore, applying Theorems \ref{Theorem 3.6} and \ref{Theorem 3.9} and estimates (\ref{6.2}) and (\ref{6.3}) we derive that
\begin{equation*}
	2^{\nu d(1 - 1/p)} (1 + \nu)^{b+1/2} \lesssim 2^{\nu d (1-1/p)} (1 + \nu)^{b+1/q} \quad \text{for all} \quad  \nu \in \mathbb{N}_0,
\end{equation*}
which contradicts the condition $q > 2$. 
\end{proof}

Next we show that the right-hand side embedding in (\ref{6.1new}), that is,
$
	\mathbf{B}^{0,b}_{p,q}(\mathbb{R}^d) \hookrightarrow B^{0,b+1/q}_{p,q}(\mathbb{R}^d),
 $ 
is strict, except the case $q=p=2$ where these  spaces coincide (cf. (\ref{BesovZero})).

\begin{prop}\label{Proposition 6.2}
	Let $\frac{2 d}{d+1} < p < \infty, q \geq \max\{p,2\} (q > 2 \text{ if } p=2)$ and $b > -1/q$. Then  $B^{0,b+1/q}_{p,q}(\mathbb{R}^d)$ is not continuously embedded into $\mathbf{B}^{0,b}_{p,q}(\mathbb{R}^d)$.
\end{prop}
\begin{proof}
	Let us assume that $q > p$. Let $f$ be the radial function such that $\widehat{f}(x) = F_0(|x|), x \in \mathbb{R}^d$, where
	\begin{equation*}
			 F_0(t) = \left\{\begin{array}{lcl}
                            t^{-d + d/p + \varepsilon} & ,  & 0 < t  < 1, \\
                            & & \\
                            t^{-d + d/p} (1 + \log t)^{-(b + 1/q + \beta)} & , & t \geq 1,
            \end{array}
            \right.
	\end{equation*}
	with $\varepsilon > 0$ and $\max\{1/q, 1/p -b - 1/q\} < \beta < 1/p$. We have
	\begin{equation*}
		\left(\int_0^1 t^{d p-d - 1} F_0^p(t) dt\right)^{1/p} + \left(\int_1^\infty t^{dq -d q/p - 1} (1 + \log t)^{(b + 1/q) q} F_0^q(t) dt\right)^{1/q} < \infty
	\end{equation*}
	and
	\begin{equation*}
		\int_1^\infty (1 + \log t)^{b q} \left(\int_t^\infty u^{d p -d-1} F_0^p(u) du\right)^{q/p}
        \frac{dt}{t} \asymp \int_1^\infty (1 + \log t)^{(1/p - 1/q - \beta) q} \frac{dt}{t} = \infty.
	\end{equation*}
	Then, applying Theorems \ref{Theorem 3.2} and \ref{Theorem 3.6}, we derive that $f \in B^{0,b+1/q}_{p,q}(\mathbb{R}^d)$ and $f \not \in \mathbf{B}^{0,b}_{p,q}(\mathbb{R}^d)$.
	
	In the remaining case where $q=p > 2$ one can proceed as follows. We consider the function
$f=\psi W$, where $\psi \in \mathcal{S}(\mathbb{R}^d) \backslash \{0\}$ with $(\ref{4.2})$ and
	\begin{equation*}
		W(x) \sim \sum_{j=3}^\infty (1 + j)^{-(b + 1/p + \varepsilon)} e^{i (2^j -2) x_1}, \quad x \in \mathbb{R}^d,
	\end{equation*}
	 where $\max\{1/p, 1/2 - 1/p -b \} < \varepsilon \leq 1/2$. Using Proposition \ref{Proposition 4.1} and Theorem \ref{Theorem 4.2} we obtain
	  $\displaystyle 
	 	\|f\|^p_{B^{0,b + 1/p}_{p,p}(\mathbb{R}^d)}   \asymp \sum_{j=3}^\infty (1 + j)^{- \varepsilon p} < \infty
	 $ 
	 and
	 \begin{align*}
	 	\|f\|_{\mathbf{B}^{0,b}_{p,p}(\mathbb{R}^d)}^p &\asymp \sum_{j=3}^\infty (1 + j)^{b p} \left(\sum_{k=j}^\infty (1 + k)^{-(b + 1/p + \varepsilon)2} \right)^{p/2}  
 = \infty.
	 \end{align*}
\end{proof}

The rest of this section refers to relationships between smoothness spaces with classical smoothness zero when $q \leq \min \{p,2\}$. Recall that, under this assumption, taking into account (\ref{1}) and Corollary \ref{Corollary 2.3}, we have that
\begin{equation*}
	B^{0,b + 1/q}_{p,q}(\mathbb{R}^d) \hookrightarrow \mathbf{B}^{0,b}_{p,q}(\mathbb{R}^d) \hookrightarrow B^{0, b+ 1/\max\{2,p\}}_{p,q}(\mathbb{R}^d) \cap H^{0,b+1/q}_p(\mathbb{R}^d), \quad b > -1/q.
\end{equation*}
Applying similar arguments to those used in the proofs of Propositions \ref{Proposition 6.1} and \ref{Proposition 6.2}, one can show that the spaces $H^{0,b+1/q}_p(\mathbb{R}^d)$ and $B^{0,b+1/\max\{2,p\}}_{p,q}(\mathbb{R}^d)$ are not comparable and, moreover, the embedding
 $ B^{0,b + 1/q}_{p,q}(\mathbb{R}^d)\hookrightarrow \mathbf{B}^{0,b}_{p,q}(\mathbb{R}^d) $ is strict.


\bigskip

\begin{center}
\begin{tikzpicture}[fill opacity=0.05, ,xscale=0.8,yscale=0.8]

\fill (-1,0) ellipse (3.9 and 2.1);
\draw[fill opacity=0.2] (-1,0) ellipse (3.9 and 2.1);

\fill (1.55,0) ellipse (4.2 and 2.1);
\draw[fill opacity=0.2] (1.55,0) ellipse (4.2 and 2.1);

\fill[white] (-1.4,0) ellipse (2.6 and 1.1);
\draw[fill opacity=0.2]
(0,0) ellipse (2.6 and 1.1);

\fill[white] (1.4,0) circle (1.0 and 0.8);
\draw[fill opacity=0.2] (0,0) circle (1.0 and 0.8);

\node[fill opacity=2,xscale=0.8,yscale=0.8] at (0,0) {$B^{0,b + \frac{1}{q}}_{p,q}(\mathbb{R}^d)$};
\node[fill opacity=2,xscale=0.8,yscale=0.8] at (-1.8,0) {$\mathbf{B}^{0,b}_{p,q}(\mathbb{R}^d)$};
\node[fill opacity=2,xscale=0.8,yscale=0.8] at (4.3,0) {$B^{0,b + \frac{1}{\max\{2,p\}}}_{p,q}(\mathbb{R}^d)$};
\node[fill opacity=2,xscale=0.8,yscale=0.8] at (-3.7,0) {$H^{0,b + \frac{1}{q}}_p(\mathbb{R}^d)$};

\end{tikzpicture}
\end{center}

\phantom{qqq}

\begin{center}
{\small \textbf{Fig. 4:} Relationships between  the Besov and Sobolev spaces involving only logarithmic smoothness 
  in the case  $q \leq \min \{p,2\}$.}
\end{center}


Analogues of Propositions \ref{Proposition 6.1} and \ref{Proposition 6.2} read as follows.
\begin{prop}\label{Proposition 6.3}
	Let $q \leq \min\{p,2\} (q < 2 \text{ if } p=2)$ and $-\infty < b < \infty$. Then

\begin{enumerate}[\upshape (i)]
  \item 	$	H^{0,b+1/q}_p(\mathbb{R}^d)$ \text{ is not continuously embedded into } $B^{0,b+1/\max\{2,p\}}_{p,q}(\mathbb{R}^d)$,

  \item if $\frac{2 d}{d+1} < p < \infty$ then 	
  	$B^{0,b+1/\max\{2,p\}}_{p,q}(\mathbb{R}^d)$ \text{ is not continuously embedded into } $H^{0,b+1/q}_p(\mathbb{R}^d)$,

  \item 	
if $b > -1/q$,  
 then $\mathbf{B}^{0,b}_{p,q}(\mathbb{R}^d)$ is not continuously embedded into $B^{0,b+1/q}_{p,q}(\mathbb{R}^d)$.

\end{enumerate}


\end{prop}

\newpage
\section{Optimality of embeddings between Besov spaces}\label{section8}

The goal of the subsequent Sections \ref{Optimality of embeddings with smoothness near zero} and \ref{Optimality of Sobolev embeddings} is to show some sharpness assertions  with respect to smoothness parameters involved in the corresponding embeddings. Section
\ref{Comparison between embeddings with smoothness near zero and Sobolev embeddings} deals  with optimality of embedding theorems for various
Besov spaces, namely $\mathbf{B}^{0,b}_{p,q}(\mathbb{R}^d)$, ${B}^{0,b+\frac12}_{p,q}(\mathbb{R}^d)$, and ${B}^{s,b+\xi(p,q)}_{p,q}(\mathbb{R}^d)$, for  various $p, q$ and $\xi(p,q)$ being either $1/\max(p,q)$ or $1/\min(p,q)$.
The proofs of these results will be given in Sections \ref{section8.5}-\ref{section8.10} and based on
 applications of results obtained in Sections \ref{section4} and \ref{section5} together with new interpolation formulas and duality properties, which will be given in Section \ref{section8.4}.

\subsection{Optimality of embeddings with smoothness near zero}\label{Optimality of embeddings with smoothness near zero}
Let us recall that
\begin{equation}\label{7.1}
    B^{0,b + 1/\min\{2,p,q\}}_{p,q}(\mathbb{R}^d) \hookrightarrow
    \mathbf{B}^{0,b}_{p,q}(\mathbb{R}^d)
\end{equation}
and
\begin{equation}\label{7.2}
	 \mathbf{B}^{0,b}_{p,q}(\mathbb{R}^d) \hookrightarrow
    B^{0,b+1/\max\{2,p,q\}}_{p,q}(\mathbb{R}^d),
\end{equation}
see (\ref{1}). Here, $1 < p < \infty, 0 < q \leq \infty$, and $b > -1/q$.


The next result provides the sharpness of embeddings (\ref{7.1}).

\begin{prop}\label{Proposition 7.1}
	For any $\varepsilon > 0$, there is a function $f \in B^{0,b+1/\min\{2,p,q\} - \varepsilon}_{p,q}(\mathbb{R}^d)$ such that $f \not \in \mathbf{B}^{0,b}_{p,q}(\mathbb{R}^d)$ in any of the following cases:
	\begin{enumerate}[\upshape (a)]
		\item $\frac{2d}{d + 1} < p < \infty, 0 < q \leq \infty, b > -1/q$, and $p= \min \{2,p,q\}$.
		\item $1 < p < \infty, 0 < q \leq \infty, b > -1/q$, and $2 = \min \{2,p,q\}$.
		\item $1 < p < \infty, 0 <  q \leq \infty, b > -1/q$, and $q = \min \{2,p,q\}$.
	\end{enumerate}
\end{prop}


The corresponding result for (\ref{7.2}) reads as follows.

\begin{prop}\label{Proposition 7.2}
	Let $1 < p < \infty, 0 < q \leq \infty$, and $b >- 1/q$. For any $\varepsilon > 0$, there is a function $f \in \mathbf{B}^{0,b}_{p,q}(\mathbb{R}^d)$ such that $f \not \in B^{0,b+1/\max\{2,p,q\} + \varepsilon}_{p,q}(\mathbb{R}^d)$.
\end{prop}

\begin{rem}\label{Section 11.1: Remark}
(i) Note that Propositions \ref{Proposition 7.1} and \ref{Proposition 7.2} were already obtained in \cite[Remark 5.7]{CobosDominguezTriebel} but only for the case when $p = \min \{2,p,q\}$  and $p = \max \{2,p,q\}$, respectively. The method used there is based on wavelet decompositions for Besov spaces.

(ii)  It follows from Propositions \ref{Proposition 7.1} and \ref{Proposition 7.2} that the embeddings (\ref{7.1}) and (\ref{7.2}) are also optimal in the setting of Besov spaces of generalized smoothness. Before we state the precise sharpness assertions, let us recall the definition of Besov spaces of generalized smoothness (see, e.g., \cite{Moura}, \cite{HaroskeMoura}, \cite{FarkasLeopold}). Let $1 < p < \infty, 0 < q \leq \infty, -\infty< s, b < \infty,$ and $\Psi \in SV$. Then, the space $B^{s,b,\Psi}_{p,q}(\mathbb{R}^d)$\index{\bigskip\textbf{Spaces}!$B^{s,b,\Psi}_{p,q}(\mathbb{R}^d)$}\label{BESOVGENERAL} consists of all $f \in \mathcal{S}'(\mathbb{R}^d)$ for which
	\begin{equation}\label{BesovGeneralSmoothness}
		 \|f\|_{B^{s,b,\Psi}_{p,q}(\mathbb{R}^d)} = \left(\sum_{j=0}^\infty \Big(2^{j s} (1 + j)^b \Psi (1 + j) \|(\varphi_j
    \widehat{f})^\vee\|_{L_p(\mathbb{R}^d)}\Big)^q\right)^{1/q} < \infty.
	\end{equation}
Note that if $\Psi(t) = \ell^\xi(t), \xi \in \mathbb{R},$ then we get the spaces $B^{s,b,\xi}_{p,q}(\mathbb{R}^d)$ (see (\ref{BesovIteratedSmoothness})). In particular, if $\xi=0$ then we recover the spaces $B^{s,b}_{p,q}(\mathbb{R}^d)$ .

Now we are ready to state

\begin{prop}\label{Proposition 7.1 SV}
	Let $\Psi \in SV$. For any $\varepsilon > 0$, there is a function $f \in B^{0,b+1/\min\{2,p,q\} - \varepsilon, \Psi}_{p,q}(\mathbb{R}^d)$ such that $f \not \in \mathbf{B}^{0,b}_{p,q}(\mathbb{R}^d)$ in any of the following cases:
	\begin{enumerate}[\upshape (a)]
		\item $\frac{2d}{d + 1} < p < \infty, 0 < q \leq \infty, b > -1/q$, and $p= \min \{2,p,q\}$.
		\item $1 < p < \infty, 0 < q \leq \infty, b > -1/q$, and $2 = \min \{2,p,q\}$.
		\item $1 < p < \infty, 0 < q \leq \infty, b > -1/q$, and $q = \min \{2,p,q\}$.
	\end{enumerate}
\end{prop}

\begin{prop}\label{Proposition 7.2 SV}
	Let $1 < p < \infty, 0 < q \leq \infty$, and $b >- 1/q$. Let $\Psi \in SV$. For any $\varepsilon > 0$, there is a function $f \in \mathbf{B}^{0,b}_{p,q}(\mathbb{R}^d)$ such that $f \not \in B^{0,b+1/\max\{2,p,q\} + \varepsilon, \Psi}_{p,q}(\mathbb{R}^d)$.
\end{prop}

Both Propositions \ref{Proposition 7.1 SV} and \ref{Proposition 7.2 SV} are simple corollaries from Propositions \ref{Proposition 7.1} and \ref{Proposition 7.2}, respectively, taking into account that
\begin{equation*}
	B^{s,b_0}_{p,q}(\mathbb{R}^d) \hookrightarrow B^{s,b_1, \Psi}_{p,q}(\mathbb{R}^d) \hookrightarrow B^{s,b_2}_{p,q}(\mathbb{R}^d)  \quad \text{for} \quad b_0 > b_1 > b_2.
\end{equation*}
\end{rem}

Embeddings (\ref{7.1}) and (\ref{7.2}) yield that $\mathbf{B}^{0,b}_{2,2}(\mathbb{R}^d) = B^{0,b+1/2}_{2,2}(\mathbb{R}^d)$; see (\ref{BesovZero}). Our next result establishes that this is indeed the only possible case where both approaches to introduce Besov spaces with smoothness zero coincide. This answers a conjecture of W. Sickel \cite{Sickel}.

\begin{thm}\label{Proposition 7.3}
		Let $\frac{2d}{d+1} < p < \infty, 0 < q \leq \infty, b > -1/q$, and $\xi \in \mathbb{R}$. Then
	\begin{equation*}
		\mathbf{B}^{0,b}_{p,q}(\mathbb{R}^d) = B^{0,\xi}_{p,q}(\mathbb{R}^d) \quad \text{if and only if} \quad p=q=2 \quad \text{and} \quad \xi= b+1/2.
	\end{equation*}
\end{thm}

\subsection{Optimality of Sobolev embeddings for $\mathbf{B}^{0,b}_{p,q}(\mathbb{R}^d)$}\label{Optimality of Sobolev embeddings}

	This subsection deals with the Sobolev embeddings for $\mathbf{B}^{0,b}_{p,q}(\mathbb{R}^d)$ given in Theorem \ref{TheoremSobolevEmb}. That is, let $1 < p_0 < p < p_1 < \infty, -\infty < s_1 < 0 < s_0 < \infty, 0 < q \leq \infty$, and $b > -1/q$. Assume that
	\begin{equation}\label{differentialdimension2}
		s_0 - \frac{d}{p_0} = - \frac{d}{p} = s_1 - \frac{d}{p_1}.
	\end{equation}
	 Then
	\begin{equation}\label{8.3new}
		B^{s_0, b + 1/\min\{p,q\}}_{p_0,q}(\mathbb{R}^d) \hookrightarrow \mathbf{B}^{0,b}_{p,q}(\mathbb{R}^d)
	\end{equation}
	and
	\begin{equation}\label{8.4new}
	\mathbf{B}^{0,b}_{p,q}(\mathbb{R}^d) \hookrightarrow B^{s_1, b + 1/\max\{p,q\}}_{p_1,q}(\mathbb{R}^d).
	\end{equation}
	
	Now we present the following assertions showing that  both (\ref{8.3new}) and (\ref{8.4new}) are sharp.
	
	\begin{prop}\label{Proposition 8.4}
	For any $\varepsilon > 0$, there is a function $f \in B^{s_0,b+1/\min\{p,q\} - \varepsilon}_{p_0,q}(\mathbb{R}^d)$ such that $f \not \in \mathbf{B}^{0,b}_{p,q}(\mathbb{R}^d)$ in any of the following cases:
	\begin{enumerate}[\upshape (a)]
		\item $\frac{2d}{d + 1} < p < \infty, 0 < q \leq \infty, b > -1/q$, and $p= \min \{p,q\}$.
		\item $1 < p < \infty, 0 <  q \leq \infty, b > -1/q$, and $q = \min \{p,q\}$.
	\end{enumerate}
\end{prop}


\begin{prop}\label{Proposition 8.5}
	For any $\varepsilon > 0$, there is a function $f \in \mathbf{B}^{0,b}_{p,q}(\mathbb{R}^d)$ such that $f \not \in B^{s_1,b+1/\max\{p,q\} + \varepsilon}_{p_1,q}(\mathbb{R}^d)$ in any of the following cases:
	\begin{enumerate}[\upshape (a)]
		\item $\frac{2d}{d + 1} < p < \infty, 0 < q \leq \infty, b > -1/q$, and $p= \max \{p,q\}$.
		\item $1  < p < \infty, 0 < q < \infty, b > -1/q$, and $q = \max \{p,q\}$.
	\end{enumerate}
\end{prop}


\begin{rem}
	Similarly to Remark \ref{Section 11.1: Remark}(ii), one can establish the corresponding sharpness assertions to Propositions \ref{Proposition 8.4} and \ref{Proposition 8.5} for Besov spaces of generalized smoothness (cf. (\ref{BesovGeneralSmoothness})).
\end{rem}

Another important question is to study the limiting case of (\ref{8.3new}) when $p=\infty$, which corresponds to the critical smoothness $s_0=d/{p_0}$ (cf. (\ref{differentialdimension2})).
In particular, in \cite[Remark 9.11]{HaroskeBook}, Haroske posed the following question: Let $0 < {p_0} \leq \infty$ and $0 < q \leq 1$. Is it true that there exists $0 < r < \infty$ such that
 \begin{equation}\label{11}
 	B^{d/{p_0}}_{{p_0},q}(\mathbb{R}^d) \hookrightarrow \mathbf{B}^{0,0}_{\infty,r}(\mathbb{R}^d) ?
 \end{equation}
 Note that if $r=\infty$ the answer is affirmative by using
  $\mathbf{B}^{0,0}_{\infty,\infty}(\mathbb{R}^d) = L_\infty(\mathbb{R}^d)$
   and the
well-known fact that
\begin{equation}\label{10}
 	B^{d/{p_0}}_{{p_0},q}(\mathbb{R}^d) \hookrightarrow L_\infty(\mathbb{R}^d) \quad \text{if and only if} \quad 0 < {p_0} \leq \infty, 0 < q \leq 1;
 \end{equation}
see Lemma \ref{LemBianchiniContinuous} below.

  However, if $r < \infty$ then (\ref{11}) fails to hold as it was proved in \cite[Theorem 3.5]{MouraNevesPiotrowski} with the help of atomic decompositions.

 As can be seen in the book by Haroske \cite[Section 9.2]{HaroskeBook}, a lift argument is useful to achieve sharp embeddings for Besov spaces in the borderline case $s_0=d/p_0 + 1$. However, this approach does not work for the case $s_0=d/{p_0}$.
 Hence, it would be desirable to give a new lifting argument that can be applied to investigate embedding (\ref{11}) and complements the decomposition method given in \cite{MouraNevesPiotrowski}. Accordingly, we propose below such a lifting approach based on limiting interpolation which allows us to switch from $s_0=d/{p_0}$ to the well-known case $d/{p_0} < s_0 < d/{p_0} +1$.

 Let us assume that there is $0 < r < \infty$ such that
 \begin{equation}\label{12}
 	B^{d/{p_0}}_{{p_0},q}(\mathbb{R}^d) \hookrightarrow \mathbf{B}^{0,0}_{\infty,r}(\mathbb{R}^d) .
 \end{equation}
 Let $d/{p_0} < s_0 < d/{p_0} +1$. Then, there exists $\theta \in (0,1)$ such that $s_0 = (1-\theta) \frac{d}{{p_0}} + \theta \Big(\frac{d}{{p_0}} + 1\Big)$. Since the following embedding holds (see \cite[Proposition 7.15(i)]{HaroskeBook})
 \begin{equation*}
 	B^{d/{p_0}+1}_{{p_0},q}(\mathbb{R}^d) \hookrightarrow \text{Lip}(\mathbb{R}^d)\equiv \text{Lip}^{(1,0)}_{\infty,\infty}(\mathbb{R}^d),
 \end{equation*}
 one can apply the interpolation property to derive
 \begin{equation}\label{13}
 	(B^{d/{p_0}}_{{p_0},q}(\mathbb{R}^d), B^{d/{p_0}+1}_{{p_0},q}(\mathbb{R}^d))_{\theta,q} \hookrightarrow (\mathbf{B}^{0,0}_{\infty,r}(\mathbb{R}^d), \text{Lip}(\mathbb{R}^d))_{\theta,q}.
 \end{equation}
 By Lemma \ref{PrelimLemmaCF} we have $(B^{d/{p_0}}_{{p_0},q}(\mathbb{R}^d), B^{d/{p_0}+1}_{{p_0},q}(\mathbb{R}^d))_{\theta,q} = B^s_{{p_0},q}(\mathbb{R}^d)$. Concerning the target space in (\ref{13}), we use (\ref{PrelimInterpolationnew}), (\ref{interpolationLipschitz1}), (\ref{PrelimInterpolation}) and (\ref{BesovComparison}) to get
 \begin{align*}
 	(\mathbf{B}^{0,0}_{\infty,r}(\mathbb{R}^d), \text{Lip}(\mathbb{R}^d))_{\theta,q} & = ((L_\infty(\mathbb{R}^d), W^1_\infty(\mathbb{R}^d))_{(0,0),r}, (L_\infty(\mathbb{R}^d), W^1_\infty(\mathbb{R}^d))_{(1,0),\infty})_{\theta,q} \\
	& \hspace{-2.5cm}= (L_\infty(\mathbb{R}^d), W^1_\infty(\mathbb{R}^d))_{\theta,q;\frac{1-\theta}{r}} = B^{s_0-d/{p_0}, (1-\theta)/r}_{\infty,q}(\mathbb{R}^d),
 \end{align*}
 where we have applied \cite[Theorem 5.8]{Fernandez-MartinezSignes} in the second equivalence. Therefore, we obtain that
 \begin{equation}
 	B^{s_0}_{{p_0},q}(\mathbb{R}^d) \hookrightarrow B^{{s_0}-d/{p_0}, (1-\theta)/r}_{\infty,q}(\mathbb{R}^d),
 \end{equation}
 which contradicts \cite[Theorem 9.2]{HaroskeBook} because $r \neq \infty$ (see also (\ref{8.4})). Consequently, embedding (\ref{12}) does not hold true.

\subsection{Comparison between embeddings with smoothness near zero and Sobolev embeddings}\label{Comparison between embeddings with smoothness near zero and Sobolev embeddings}

Combining (\ref{7.1}) and (\ref{8.3new}) we arrive at
\begin{equation}\label{ComparisonNew}
	 B^{0,b + 1/\min\{2,p,q\}}_{p,q}(\mathbb{R}^d) \cup B^{s_0, b + 1/\min\{p,q\}}_{p_0,q}(\mathbb{R}^d) \hookrightarrow \mathbf{B}^{0,b}_{p,q}(\mathbb{R}^d).
\end{equation}
Here, $1 < p_0 < p < \infty, 0 < q \leq \infty, b > -1/q$, and $0 < s_0 < \infty$ satisfying $s_0 - d/p_0 = -d/p$.
It is worth mentioning that if $2 \ge \min\{p,q\}$, then (\ref{ComparisonNew}) is equivalent to
$$	 B^{0,b + 1/\min\{2,p,q\}}_{p,q}(\mathbb{R}^d) \hookrightarrow \mathbf{B}^{0,b}_{p,q}(\mathbb{R}^d).
$$ This follows from Remark \ref{Remark 3.5}(iv).
 Therefore, here we consider the nontrivial case when $2 < \min\{p,q\}$.
  Under this assumption, we will show that the spaces $ B^{0,b + 1/2}_{p,q}(\mathbb{R}^d) $ and $B^{s_0, b + 1/\min\{p,q\}}_{p_0,q}(\mathbb{R}^d)$ are not comparable and their union 
   does not coincide with $\mathbf{B}^{0,b}_{p,q}(\mathbb{R}^d)$. Fig. 5 shows the sharpness assertions given in Proposition \ref{Proposition 8.6new}.

\bigskip

\begin{center}
\begin{tikzpicture}[fill opacity=0.05, ,xscale=0.8,yscale=0.8]


\fill (0,0) ellipse (3.6 and 2.0);
\draw[fill opacity=0.2] (0,0) ellipse (3.6 and 2.0);

\fill[white] (-1.8,0) ellipse (2.0 and 1.0);
\draw[fill opacity=0.2](-1.6,0) ellipse (2.0 and 1.0);

\fill[white] (1.9,0) circle (2.1 and 1.0);
\draw[fill opacity=0.2] (1.4,0) circle (2.1 and 1.0);

\node[fill opacity=2,xscale=0.8,yscale=0.8] at (0,1.3) {$\mathbf{B}^{0,b}_{p,q}(\mathbb{R}^d)$};
\node[fill opacity=2,xscale=0.8,yscale=0.8] at (2.0,0) {$B^{s_0,b+\frac{1}{\min\{p,q\}}}_{p_0,q}(\mathbb{R}^d)$};
\node[fill opacity=2,xscale=0.8,yscale=0.8] at (-2.2,0) {$B^{0,b + \frac{1}{2}}_{p,q}(\mathbb{R}^d)$};

\end{tikzpicture}
\end{center}

\phantom{qqq}

\begin{center}
{\small \textbf{Fig. 5:} Relationships between the Besov spaces $\mathbf{B}^{0,b}_{p,q}(\mathbb{R}^d), B^{0,b+1/2}_{p,q}(\mathbb{R}^d)$, and $B^{s_0,b+1/\min\{p,q\}}_{p_0,q}(\mathbb{R}^d)$ in the case $2 < \min\{p,q\}$.}
\end{center}

\bigskip


\begin{prop}\label{Proposition 8.6new}
	Let $1 < p_0 < p < \infty, 0 < q \leq \infty, b > -1/q$, and $0 < s_0 < \infty$ with $s_0 - d/p_0 = -d/p$. Assume $2 < \min\{p,q\}$. The following assertions hold
\begin{enumerate}[\upshape (i)]
  \item $B^{s_0,b+1/\min\{p,q\}}_{p_0,q}(\mathbb{R}^d)$ \quad \text{is not continuously embedded into}
  \quad
       $B^{0,b+1/2}_{p,q}(\mathbb{R}^d)$,

  \item	$B^{0,b+1/2}_{p,q}(\mathbb{R}^d)$ \quad \text{is not continuously embedded into} \quad
   $B^{s_0,b+1/\min\{p,q\}}_{p_0,q}(\mathbb{R}^d)$,

  \item  $B^{0,b + 1/2}_{p,q}(\mathbb{R}^d) \cup B^{s_0, b + 1/\min\{p,q\}}_{p_0,q}(\mathbb{R}^d) \neq \mathbf{B}^{0,b}_{p,q}(\mathbb{R}^d).$

\end{enumerate}
\end{prop}
\begin{rem}
	Proposition \ref{Proposition 8.6new}(iii) is not necessarily true if we restrict to the $GM$ class. For example, if $2 < q=p$, then
	\begin{equation*}
		\widehat{GM}^d \cap (B^{0,b + 1/2}_{p,p}(\mathbb{R}^d) \cup B^{s_0, b + 1/p}_{p_0,p}(\mathbb{R}^d)) = \widehat{GM}^d \cap  B^{0, b + 1/p}_{p,p}(\mathbb{R}^d) = \widehat{GM}^d \cap  \mathbf{B}^{0,b}_{p,p}(\mathbb{R}^d),
	\end{equation*}
	where we have applied Theorem \ref{SobolevEmbeddingImprovement} and (\ref{EmbBesovGMNew}).
\end{rem}

On the other hand, according to (\ref{7.2}) and (\ref{8.4new}) we have
\begin{equation}\label{ComparisonNew2}
	\mathbf{B}^{0,b}_{p,q}(\mathbb{R}^d) \hookrightarrow  B^{0,b + 1/\max\{2,p,q\}}_{p,q}(\mathbb{R}^d) \cap B^{s_1, b + 1/\max\{p,q\}}_{p_1,q}(\mathbb{R}^d) .
\end{equation}
Here, $1 < p< p_1 < \infty, 0 < q \leq \infty, b > -1/q$, and $-\infty < s_1 < 0$ satisfying $s_1 - d/p_1 = -d/p$.
If $2 \le  \max\{p,q\}$, then, by Remark \ref{Remark 3.5}(iv), (\ref{ComparisonNew2}) is equivalent to
$$
	\mathbf{B}^{0,b}_{p,q}(\mathbb{R}^d) \hookrightarrow  B^{0,b + 1/\max\{2,p,q\}}_{p,q}(\mathbb{R}^d).
$$

Below we consider the most interesting case $2 > \max\{p,q\}$.
 We will show that 
  the spaces $ B^{0,b + 1/2}_{p,q}(\mathbb{R}^d) $ and $B^{s_1, b + 1/\max\{p,q\}}_{p_1,q}(\mathbb{R}^d)$ are not comparable and their intersection 
   differs from $\mathbf{B}^{0,b}_{p,q}(\mathbb{R}^d)$. Fig. 6 illustrates the sharpness assertions given in Proposition \ref{Proposition 8.7new}.

\bigskip

\begin{center}
\begin{tikzpicture}[fill opacity=0.05, ,xscale=0.8,yscale=0.8]

\fill (-1.4,0) ellipse (2.6 and 1.5);
\draw[fill opacity=0.2] (-1.4,0) ellipse (2.6 and 1.5);

\fill (1.5,0) ellipse (2.8 and 1.5);
\draw[fill opacity=0.2] (1.5,0) ellipse (2.8 and 1.5);

\fill[white] (0.0,0) circle (1.0 and 0.8);
\draw[fill opacity=0.2]
 (0,0) circle (1.0 and 0.8);

\node[fill opacity=2,xscale=0.8,yscale=0.8] at (0,0) {$\mathbf{B}^{0,b}_{p,q}(\mathbb{R}^d)$};
\node[fill opacity=2,xscale=0.8,yscale=0.8] at (2.7,0) {$B^{s_1,b + \frac{1}{\max\{p,q\}}}_{p_1,q}(\mathbb{R}^d)$};
\node[fill opacity=2,xscale=0.8,yscale=0.8] at (-2.6,0) {$B^{0,b + \frac{1}{2}}_{p,q}(\mathbb{R}^d)$};

\end{tikzpicture}
\end{center}

\phantom{qqq}

\begin{center}
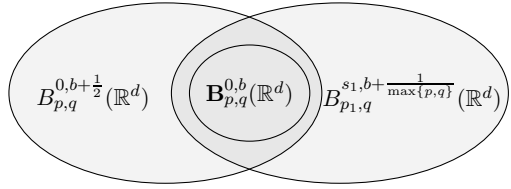
{\small \textbf{Fig. 6:} Relationships between the Besov spaces $\mathbf{B}^{0,b}_{p,q}(\mathbb{R}^d), B^{0,b+1/2}_{p,q}(\mathbb{R}^d)$, and $B^{s_1,b+1/\max\{p,q\}}_{p_1,q}(\mathbb{R}^d)$ in the case $2 > \max\{p,q\}$.}
\end{center}

\bigskip

\begin{prop}\label{Proposition 8.7new}
	Let $1 < p< p_1 < \infty, 0 < q \leq \infty, b > -1/q$, and $-\infty < s_1 < 0$ with $s_1 - d/p_1 = -d/p$. Assume $2 > \max\{p,q\}$. The following assertions hold
\begin{enumerate}[\upshape (i)]
  \item	$B^{s_1,b+1/\max\{p,q\}}_{p_1,q}(\mathbb{R}^d)$ \quad \text{is not continuously embedded into} \quad $B^{0,b+1/2}_{p,q}(\mathbb{R}^d)$,

  \item if $\frac{2 d}{d+1} < p < \infty$ then $B^{0,b+1/2}_{p,q}(\mathbb{R}^d)$ \quad \text{is not continuously embedded}
   \\
  \text{ into} \quad $B^{s_1,b+1/\max\{p,q\}}_{p_1,q}(\mathbb{R}^d)$,

  \item  $B^{0,b + 1/2}_{p,q}(\mathbb{R}^d) \cap B^{s_1, b + 1/\max\{p,q\}}_{p_1,q}(\mathbb{R}^d) \neq \mathbf{B}^{0,b}_{p,q}(\mathbb{R}^d).$

\end{enumerate}
\end{prop}

\begin{rem}
	Proposition \ref{Proposition 8.7new}(iii) is not necessarily true if we restrict to the $GM$ class. For example, if $2 > q=p$, then
	\begin{equation*}
		\widehat{GM}^d \cap (B^{0,b + 1/2}_{p,p}(\mathbb{R}^d) \cap B^{s_1, b + 1/p}_{p_1,p}(\mathbb{R}^d)) = \widehat{GM}^d \cap  B^{0, b + 1/p}_{p,p}(\mathbb{R}^d) = \widehat{GM}^d \cap  \mathbf{B}^{0,b}_{p,p}(\mathbb{R}^d),
	\end{equation*}
	where we have used Theorem \ref{SobolevEmbeddingImprovement} and (\ref{EmbBesovGMNew}).
\end{rem}


\subsection{New interpolation formulas and duality theorems}\label{section8.4}

To prove Propositions \ref{Proposition 7.1}, \ref{Proposition 7.2}, \ref{Proposition 8.4} and \ref{Proposition 8.5} and Theorem \ref{Proposition 7.3} we need to establish new interpolation formulas (see Lemma \ref{PrelimLemmaT2} below). These formulas  are interesting in themselves. For the sake of completeness, we also recall some auxiliary results, 
such as the characterizations of the dual spaces of $B^{s,b}_{p,q}(\mathbb{R}^d)$ and $\mathbf{B}^{0,b}_{p,q}(\mathbb{R}^d)$.

Some interpolation formulae for Besov spaces were already collected in Section \ref{Section:Interpolation methods}. In particular, Lemma \ref{PrelimLemmaT} concerns with interpolation of classical Besov spaces with different integrability. Namely, let $1 < p_0, p_1 < \infty, p_0 \neq p_1, 1 \leq q_0, q_1 < \infty, -\infty < s_0, s_1 < \infty,$ and $0 < \theta < 1$. Assume that
\begin{equation}\label{zv}
			s = (1-\theta) s_0 + \theta s_1 \quad \text{and} \quad \frac{1-\theta}{p_0} + \frac{\theta}{p_1} = \frac{1}{p} = \frac{1-\theta}{q_0} + \frac{\theta}{q_1}.
		\end{equation}
		Then we have
\begin{equation}\label{8.6new}
(B^{s_0}_{p_0,q_0}(\mathbb{R}^d), B^{s_1}_{p_1,q_1}(\mathbb{R}^d))_{\theta,p} = B^{s}_{p,p}(\mathbb{R}^d).
\end{equation}

We are going to extend the interpolation formula (\ref{8.6new}) to Besov spaces of generalized smoothness with the help of the lifting operator $\mathcal{J}_b$\index{\bigskip\textbf{Operators}!$\mathcal{J}_b$}\label{LIFTLOG} given by
\begin{equation}\label{LiftLog}
	\mathcal{J}_b  f = \Big(\sum_{j=0}^\infty (1 + j)^b \varphi_j (\xi)\widehat{f}\Big)^\vee, \quad f \in \mathcal{S}'(\mathbb{R}^d),
\end{equation}
and its inverse $(\mathcal{J}_b)^{-1}$.  Here $-\infty < b < \infty$ and $\{\varphi_j\}_{j \in \mathbb{N}_0}$ is given by (\ref{SmoothFunction}), (\ref{resolution}). According to \cite[Proposition 3.2]{CaetanoMoura},
\begin{equation}\label{8.7new}
	\mathcal{J}_b \quad \text{is an isomorphism from} \quad B^{s,b}_{p,q}(\mathbb{R}^d) \quad \text{onto} \quad B^{s}_{p,q}(\mathbb{R}^d)
\end{equation}
for any $1 \leq p \leq \infty, 0 < q \leq \infty$ and $-\infty < s < \infty$.

\begin{lem}\label{PrelimLemmaT2}
		Let $1 < p_0, p_1 < \infty, p_0 \neq p_1, 1 \leq q_0, q_1 < \infty, -\infty < s_0, s_1, b < \infty,$ and $0 < \theta < 1$.
Assume that (\ref{zv}) holds.
		Then we have
\begin{equation*}
(B^{s_0,b}_{p_0,q_0}(\mathbb{R}^d), B^{s_1,b}_{p_1,q_1}(\mathbb{R}^d))_{\theta,p} = B^{s,b}_{p,p}(\mathbb{R}^d).
\end{equation*}
\end{lem}
\begin{proof}
By (\ref{8.7new}), we have
\begin{equation*}
	\mathcal{J}_b : B^{s_i,b}_{p_i,q_i}(\mathbb{R}^d) \longrightarrow B^{s_i}_{p_i,q_i}(\mathbb{R}^d), \quad i=0,1.
\end{equation*}
Hence, using the interpolation property and (\ref{8.6new}), we get
\begin{equation}\label{8.8new}
	\mathcal{J}_b : (B^{s_0,b}_{p_0,q_0}(\mathbb{R}^d), B^{s_1,b}_{p_1,q_1}(\mathbb{R}^d))_{\theta,p} \longrightarrow (B^{s_0}_{p_0,q_0}(\mathbb{R}^d), B^{s_1}_{p_1,q_1}(\mathbb{R}^d))_{\theta,p}  = B^{s}_{p,p}(\mathbb{R}^d).
\end{equation}
Since $(\mathcal{J}_b)^{-1}$ maps boundedly $B^{s}_{p,p}(\mathbb{R}^d)$ into $B^{s,b}_{p,p}(\mathbb{R}^d)$, it follows from (\ref{8.8new}) that
\begin{equation*}
	(B^{s_0,b}_{p_0,q_0}(\mathbb{R}^d), B^{s_1,b}_{p_1,q_1}(\mathbb{R}^d))_{\theta,p} \hookrightarrow B^{s,b}_{p,p}(\mathbb{R}^d).
\end{equation*}
Analogously, one can prove the reverse embedding, which completes the proof.
\end{proof}

As usual, we denote by $A'$ the dual space of a Banach space $A$\index{\bigskip\textbf{Spaces}!$A'$}\label{DUAL}.

The characterization of the dual spaces of $B^{s,b}_{p,q}(\mathbb{R}^d)$ was studied in \cite[Theorem 2.11.2]{Triebel1} and \cite[Theorem 3.1.10]{FarkasLeopold} (see also \cite[p. 78]{CobosDominguez3}). It reads as follows.

\begin{lem}\label{LemmaDual2}
	Let $1 < p < \infty, 1 \leq q < \infty$, and $-\infty < s, b< \infty$. Then
	\begin{equation*}
		(B^{s,b}_{p,q}(\mathbb{R}^d))' = B^{-s,-b}_{p',q'}(\mathbb{R}^d).
	\end{equation*}
\end{lem}

On the other hand, the dual space of $\mathbf{B}^{0,b}_{p,q}(\mathbb{R}^d)$ was characterized in \cite[Theorem 4.3]{CobosDominguez3} with the help of the logarithmic Lipschitz spaces $\text{Lip}^{(1,-b)}_{p,q}(\mathbb{R}^d)$ (see (\ref{Lipschitz})) and the lifting operator $I_\sigma$ given in (\ref{liftdef}).

\begin{lem}\label{LemmaDual}
	Let $1 < p < \infty, 1 \leq q < \infty$, and $b > -1/q$. Then $f$ belongs to $(\mathbf{B}^{0,b}_{p,q}(\mathbb{R}^d))'$ if and only if $I_{-1}f$ belongs to $\emph{Lip}^{(1,-b-1)}_{p',q'}(\mathbb{R}^d)$. Furthermore, we have
	\begin{equation*}
		\|f\|_{(\mathbf{B}^{0,b}_{p,q}(\mathbb{R}^d))'} \asymp \|I_{-1}f\|_{\emph{Lip}^{(1,-b-1)}_{p',q'}(\mathbb{R}^d)}.
	\end{equation*}
\end{lem}

\subsection{Proof of Proposition \ref{Proposition 7.1}}
\label{section8.5}
Assume first that $p = \min\{2,p,q\}$ and $p > \frac{2 d}{d+1}$. Let $f$ be such that $\widehat{f}(\xi)=F_0(|\xi|)$ with
\begin{equation*}
	F_0(t) = t^{-d + d/p} (1 + |\log t|)^{-\beta}, \quad t > 0,
\end{equation*}
where $\max \{1/p, b + 1/q + 1/p - \varepsilon\} < \beta < b + 1/q + 1/p$. We have
\begin{align*}
	\left(\int_0^1 t^{d p - d - 1} F_0^p(t) dt\right)^{1/p} + \left(\int_1^\infty t^{d q - d q/p - 1} (1 + \log t)^{(b + 1/p - \varepsilon) q} F_0^q(t) d t\right)^{1/q} \\
	& \hspace{-11cm}\asymp \left(\int_0^1 (1 - \log t)^{- \beta p} \frac{dt}{t}\right)^{1/p} + \left(\int_1^\infty (1 + \log t)^{(b + 1/p - \varepsilon - \beta) q} \frac{d t}{t}\right)^{1/q} < \infty
\end{align*}
which means, according to Theorem \ref{Theorem 3.6}, that $f \in B^{0,b+1/p-\varepsilon}_{p,q}(\mathbb{R}^d)$. However, using (\ref{3.3.3}) and
\begin{align*}
	\int_1^\infty (1 + \log t)^{b q} \left(\int_t^\infty u^{d p - d-1} F_0^p(u) d u\right)^{q/p} \frac{dt}{t} \\
	& \hspace{-5cm}= \int_1^\infty (1 + \log t)^{b q} \left(\int_t^\infty (1 + \log u)^{- \beta p} \frac{du}{u}\right)^{q/p} \frac{dt}{t} \\
	& \hspace{-5cm} \asymp \int_1^\infty (1 + \log t)^{(b - \beta + 1/p) q} \frac{dt}{t} = \infty
\end{align*}
we derive that $f \not \in \mathbf{B}^{0,b}_{p,q}(\mathbb{R}^d)$. This proves the case (a).

Suppose now that $2 = \min \{2,p,q\}$. Set
\begin{equation*}
	\{b_j\}_{j \in \mathbb{N}_0} = \{(1 + j)^{- \delta}\}_{j \in \mathbb{N}_0}, \quad \max\{1/2, b+1/q+1/2-\varepsilon\} < \delta < b + 1/q +1/2,
\end{equation*}
and
\begin{equation*}
	W(x) \sim \sum_{j=3}^\infty b_j e^{i (2^j - 2) x_1}, \quad x \in \mathbb{R}^d.
\end{equation*}
Let $f = \psi W$ with $\psi \in \mathcal{S}(\mathbb{R}^d) \backslash \{0\}$ satisfying (\ref{4.2}). Applying (\ref{4.3}), we obtain
\begin{align*}
	\|f\|^q_{B^{0,b+1/2-\varepsilon}_{p,q}(\mathbb{R}^d)} & \asymp \sum_{j=3}^\infty (1 + j)^{(b+1/2 - \varepsilon) q} |b_j|^q \\
	& = \sum_{j=3}^\infty (1 + j)^{(b + 1/2 - \varepsilon - \delta) q} < \infty
\end{align*}
and by (\ref{4.6}) we get
\begin{align*}
	\|f\|^q_{\mathbf{B}^{0,b}_{p,q}(\mathbb{R}^d)} & \asymp \sum_{j=3}^\infty (1 + j)^{b q} \left(\sum_{k=j}^\infty |b_k|^2\right)^{q/2} \\
	& = \sum_{j=3}^\infty (1 + j)^{b q} \left(\sum_{k=j}^\infty (1 + k)^{- 2 \delta} \right)^{q/2} \\
	& \asymp \sum_{j=3}^\infty (1 + j)^{(b - \delta +1/2) q} = \infty.
\end{align*}
So, $f \in B^{0,b+1/2-\varepsilon}_{p,q}(\mathbb{R}^d) \backslash \mathbf{B}^{0,b}_{p,q}(\mathbb{R}^d)$. The case (b) is shown.


It remains to show the case (c). This will be done by contradiction. Let us assume that there is $\varepsilon > 0$ such that the embedding
\begin{equation}\label{7.5}
	B^{0,b+1/q-\varepsilon}_{p,q}(\mathbb{R}^d) \hookrightarrow \mathbf{B}^{0,b}_{p,q}(\mathbb{R}^d)
\end{equation}
holds. Let $\theta \in (0,1)$. Then, by (\ref{PrelimInterpolationBWnew}), (\ref{7.5}), (\ref{PrelimInterpolationnew2.3}) and (\ref{BesovComparison}), we have
\begin{align*}
	B^{\theta, (1-\theta)(b+1/q-\varepsilon)}_{p,q}(\mathbb{R}^d) & = (B^{0,b+1/q-\varepsilon}_{p,q}(\mathbb{R}^d), W^1_p)_{\theta,q}  \\
	&\hookrightarrow (\mathbf{B}^{0,b}_{p,q}(\mathbb{R}^d), W^1_p)_{\theta,q}
	= B^{\theta, (1-\theta)(b+1/q)}_{p,q}(\mathbb{R}^d),
\end{align*}
which is not true (see Proposition \ref{RecallEmb*optim}).

\qed

\subsection{Proof of Proposition \ref{Proposition 7.2}}

Similar ideas to those given in the proof of Proposition \ref{Proposition 7.1} yield the desired results in the cases $p=\max\{2,p,q\}$ and $2 = \max\{2,p,q\}$. For instance, if $p=\max\{2,p,q\}$ and the function $f(x)$ is such that
$\widehat{f}(\xi)=F_{0}(|\xi|)$ with
\begin{equation*}
	F_0(t) = t^{-d+d/p} (1 + |\log t|)^{-\beta},\quad  t > 0,
\end{equation*}
where $b + 1/q + 1/p < \beta < b + 1/q + 1/p + \varepsilon$, then
 $f \in \mathbf{B}^{0,b}_{p,q}(\mathbb{R}^d) \backslash B^{0,b+1/p+\varepsilon}_{p,q}(\mathbb{R}^d)$.


If $q = \max\{2,p,q\}$ (and so, $q > 1$), let us assume that there exists $\varepsilon > 0$ such that
\begin{equation}\label{7.13}
	\mathbf{B}^{0,b}_{p,q}(\mathbb{R}^d) \hookrightarrow B^{0,b+1/q+\varepsilon}_{p,q}(\mathbb{R}^d).
\end{equation}
Then, we note that $\mathbf{B}^{0,b}_{p,q}(\mathbb{R}^d)$ is densely embedded into $B^{0,b+1/q+\varepsilon}_{p,q}(\mathbb{R}^d)$. Indeed, it is clear that
\begin{equation*}
	\mathbf{B}^s_{p,q}(\mathbb{R}^d) \hookrightarrow \mathbf{B}^{0,b}_{p,q}(\mathbb{R}^d) \quad \text{for any} \quad  s > 0
\end{equation*}
and, by (\ref{PrelimInterpolationEmbDense}), $\mathbf{B}^s_{p,q}(\mathbb{R}^d)=B^s_{p,q}(\mathbb{R}^d)$ (see (\ref{BesovComparison})) is dense in $B^{0,b+1/q+\varepsilon}_{p,q}(\mathbb{R}^d)$. Remark also that  $B^{0,b+1/q+\varepsilon}_{p,q}(\mathbb{R}^d)=(B^{-s}_{p,q}(\mathbb{R}^d), B^s_{p,q}(\mathbb{R}^d))_{1/2,q;b+1/q+\varepsilon}$ (cf. Lemma \ref{PrelimLemmaCF}).

Taking duals in (\ref{7.13}) and applying Lemmas \ref{LemmaDual} and \ref{LemmaDual2}, we derive that
\begin{equation}\label{7.14}
	\|I_{-1} f\|_{\text{Lip}^{(1,-b-1)}_{p',q'}(\mathbb{R}^d)} \asymp \|f\|_{(\mathbf{B}^{0,b}_{p,q}(\mathbb{R}^d))'} \lesssim \|f\|_{(B^{0,b+1/q+\varepsilon}_{p,q}(\mathbb{R}^d))'} \asymp \|f\|_{B^{0,-b-1/q-\varepsilon}_{p',q'}(\mathbb{R}^d)},
\end{equation}
for all $f \in B^{0,-b-1/q-\varepsilon}_{p',q'}(\mathbb{R}^d)$, where $I_{-1} f$ is the lifting operator given by
(\ref{liftdef}).
Hence, one can rewrite (\ref{7.14}) as
\begin{equation}\label{7.15}
	I_{-1} : B^{0,-b-1/q-\varepsilon}_{p',q'}(\mathbb{R}^d) \longrightarrow \text{Lip}^{(1,-b-1)}_{p',q'}(\mathbb{R}^d).
\end{equation}
Since $\text{id}_{\mathbb{R}^d} = I_{-1} \circ I_1$, using Lemma \ref{LemmaLift}, we obtain the embedding
\begin{equation}\label{7.16}
	 B^{1,-b-1/q-\varepsilon}_{p',q'}(\mathbb{R}^d) \hookrightarrow \text{Lip}^{(1,-b-1)}_{p',q'}(\mathbb{R}^d).
\end{equation}
Let us show that  (\ref{7.16}) does not hold. Indeed, for $0 < s < 1$, according to (\ref{PrelimInterpolationnew4}) we have
\begin{align*}
	\text{Lip}^{(1,-b-1)}_{p',q'}(\mathbb{R}^d) = (\mathbf{B}^s_{p',q'}(\mathbb{R}^d), W^1_{p'} (\mathbb{R}^d))_{(1,-b-1),q'}.
\end{align*}
Then, by the interpolation property, (\ref{7.16}) and (\ref{BesovComparison}), we get
\begin{equation}\label{7.17}
	(B^s_{p',q'}(\mathbb{R}^d), B^{1,-b-1/q-\varepsilon}_{p',q'}(\mathbb{R}^d))_{\theta,q'} \hookrightarrow (B^s_{p',q'}(\mathbb{R}^d), (B^s_{p',q'}(\mathbb{R}^d), W^1_{p'} (\mathbb{R}^d))_{(1,-b-1),q'})_{\theta,q'}
\end{equation}
for any $0 < \theta < 1$.

The left-hand side space in (\ref{7.17}) can be computed by using Lemma \ref{PrelimLemmaCF}. We have
\begin{equation*}
	(B^s_{p',q'}(\mathbb{R}^d), B^{1,-b-1/q-\varepsilon}_{p',q'}(\mathbb{R}^d))_{\theta,q'}  = B^{\xi, \theta (-b-1/q-\varepsilon)}_{p',q'}(\mathbb{R}^d)
\end{equation*}
where $\xi = (1-\theta) s + \theta$.

Next we deal with the right-hand side space in (\ref{7.17}).
Applying Lemma \ref{PrelimLemma7.2}(ii) with $\alpha = 0$ and (\ref{PrelimInterpolationBWnew}), we obtain
\begin{align*}
	(B^s_{p',q'}(\mathbb{R}^d), (B^s_{p',q'}(\mathbb{R}^d), W^1_{p'} (\mathbb{R}^d))_{(1,-b-1),q'})_{\theta,q'}  \\
	& \hspace{-5cm}= (B^s_{p',q'}(\mathbb{R}^d), W^1_{p'} (\mathbb{R}^d))_{\theta, q'; \theta (-b-1+1/q')} \\
	& \hspace{-5cm} = B^{\xi, \theta (-b -1/q)}_{p',q'}(\mathbb{R}^d)
\end{align*}
where we have also used the fact that $W^1_{p'}(\mathbb{R}^d) = H^1_{p'}(\mathbb{R}^d)$.

Hence, we conclude that
\begin{equation*}
	B^{\xi, \theta (-b-1/q-\varepsilon)}_{p',q'}(\mathbb{R}^d) \hookrightarrow B^{\xi, \theta (-b -1/q)}_{p',q'}(\mathbb{R}^d)
\end{equation*}
which is not true (see Proposition \ref{RecallEmb*optim}). Consequently, (\ref{7.16}) and, therefore, (\ref{7.13}) do not hold.
 \qed

\subsection{Proof of Theorem \ref{Proposition 7.3}}
We concentrate on the only-if-part. This essentially follows the proof of Theorem \ref{Theorem 6.6}. Let us assume that $\mathbf{B}^{0,b}_{p,q}(\mathbb{R}^d) = B^{0,\xi}_{p,q}(\mathbb{R}^d)$. For $0 < s < 1$, applying (\ref{PrelimInterpolationBWnew}) and (\ref{PrelimInterpolationnew2.3}), we have
	\begin{equation*}
		B^{s, (1-s) \xi}_{p,q}(\mathbb{R}^d)  = (B^{0,\xi}_{p,q}(\mathbb{R}^d), W^1_p(\mathbb{R}^d))_{s,q}  =  (\mathbf{B}^{0,b}_{p,q}(\mathbb{R}^d), W^1_p(\mathbb{R}^d))_{s,q}  = B^{s, (1-s)(b+1/q)}_{p,q}(\mathbb{R}^d),
	\end{equation*}
	which yields, by Proposition \ref{RecallEmb*optim}, that $\xi = b+1/q$.
	
	In order to show that $q=2$ we proceed by contradiction. Suppose that, e.g., $2 < q$. Let $\max\{b+2/q, 1/2\} < \beta < b + 1/q+ 1/2$,
	\begin{equation*}
		W (x) \sim \sum_{j=3}^\infty (1 + j)^{-\beta} e^{i (2^j - 2) x_1}, \quad x \in \mathbb{R}^d,
	\end{equation*}
	and $f = \psi W$ with $\psi \in \mathcal{S}(\mathbb{R}^d) \backslash \{0\}$ such that (\ref{4.2}) holds. According to (\ref{4.3}) and (\ref{4.6}), one can easily check that $f \in B^{0,b+1/q}_{p,q}(\mathbb{R}^d) \backslash \mathbf{B}^{0,b}_{p,q}(\mathbb{R}^d)$. Using similar ideas one can show that $q < 2$ is not possible. Hence, $q=2$ and so, $\xi = b+1/2$.
	
	Arguing again by contradiction we show that $p = 2$. Suppose first that $p > 2$. Take any $\delta$ such that $b + 1/p + 1/2 < \delta < b + 1$. Set
	\begin{equation*}
		F_0(t) = t^{-(d - d/p)} (1 + |\log t|)^{- \delta}, \quad t > 0,
	\end{equation*}
	and $\widehat{f}(\xi) = F_0(|\xi|), \xi \in \mathbb{R}^d$. Applying Theorems \ref{Theorem 3.6} and \ref{Theorem 3.2}, we derive that $f \in \mathbf{B}^{0,b}_{p,2}(\mathbb{R}^d) \backslash B^{0,b+1/2}_{p,2}(\mathbb{R}^d)$, which is not possible. Analogously, if $p < 2$ we arrive at a contradiction. Thus $p=2$, which  completes the proof. 	\qed
	
\subsection{Proof of Proposition \ref{Proposition 8.4}}

Assume that the conditions given in (a) hold. Let $\beta$ be such that $\max\{b + 1/p +1/q - \varepsilon, 1/p\} < \beta < b +1/p + 1/q$. We define
\begin{equation*}
		 F_0(t) = \left\{\begin{array}{lcl}
                            t^{-d + d/p_0 + \delta} & ,  & 0 < t  < 1, \\
                            & & \\
                            t^{-s_0 + d/p_0 -d} (1 + \log t)^{-\beta} & , & t \geq 1,
            \end{array}
            \right.
	\end{equation*}
where $\delta > 0$, and $\widehat{f}(\xi) = F_0(|\xi|), \xi \in \mathbb{R}^d$. Then, taking into account (\ref{differentialdimension2}), Theorems \ref{Theorem 3.2} and \ref{Theorem 3.6} imply that $f \in B^{s_0, b+1/p - \varepsilon}_{p_0,q}(\mathbb{R}^d) \backslash \mathbf{B}^{0,b}_{p,q}(\mathbb{R}^d)$. This proves (a).


To prove (b), suppose that there is $\varepsilon > 0$ such that the following embedding holds
\begin{equation}\label{8.31}
	B^{s_0, b+1/q - \varepsilon}_{p_0,q} (\mathbb{R}^d) \hookrightarrow \mathbf{B}^{0,b}_{p,q}(\mathbb{R}^d).
\end{equation}
Let $k \in \mathbb{N}$ be such that $k > s_0 = d/p_0-d/p$. Let $s' = k -s_0 =k -d/p_0+d/p > 0$. First, we note that $k > s'$ because $s_0 > 0$. Therefore, by (\ref{PrelimInterpolationBWnew}), (\ref{8.31}), (\ref{e0*}), (\ref{BesovComparison}), (\ref{PrelimInterpolationnew3}), we get
\begin{align*}
	B^{(1-\theta) s_0 + \theta k, (1-\theta)(b+1/q-\varepsilon)}_{p_0,q}(\mathbb{R}^d) &= (B^{s_0, b+1/q - \varepsilon}_{p_0,q} (\mathbb{R}^d) , W^k_{p_0}(\mathbb{R}^d))_{\theta, q} \\
	& \hspace{-3cm}\hookrightarrow ( \mathbf{B}^{0,b}_{p,q}(\mathbb{R}^d), \mathbf{B}^{s'}_{p,p_0}(\mathbb{R}^d))_{\theta,q} = B^{\theta s', (1-\theta)(b+1/q)}_{p,q}(\mathbb{R}^d)
\end{align*}
with $\theta s' - d/p = (1-\theta) s_0 + \theta k -d/p_0$. This contradicts (\ref{8.4}) because $\varepsilon > 0$. 	\qed

\subsection{Proof of Proposition \ref{Proposition 8.5}}

Using similar ideas to those given in the proof of Proposition \ref{Proposition 8.4}, we can show that the function
\begin{equation*}
		 F_0(t) = \left\{\begin{array}{lcl}
                            t^{-d + d/p_1 + \delta} & ,  & 0 < t  < 1, \\
                            & & \\
                            t^{-s_1 + d/p_1 -d} (1 + \log t)^{-\beta} & , & t \geq 1,
            \end{array}
            \right.
	\end{equation*}
	where $\delta > 0$ and $b + 1/q + 1/p < \beta < b + 1/q + 1/p + \varepsilon$, allows us to construct a desired counterexample, which implies (a).
	
We now proceed to prove (b). This will be done by contradiction. Let us assume that there is $\varepsilon > 0$ for which the embedding
\begin{equation}\label{8.34}
	\mathbf{B}^{0,b}_{p,q}(\mathbb{R}^d) \hookrightarrow B^{s_1, b+1/q+\varepsilon}_{p_1,q}(\mathbb{R}^d)
\end{equation}
holds. In fact, the previous embedding is dense because $\mathcal{S}(\mathbb{R}^d)$ is dense in $B^{s_1, b+1/q+\varepsilon}_{p_1,q}(\mathbb{R}^d)$ (see \cite[Theorem 2.3.3]{Triebel1} and \cite[p. 13]{FarkasLeopold}) and, obviously, $\mathcal{S}(\mathbb{R}^d) \subset \mathbf{B}^{0,b}_{p,q}(\mathbb{R}^d)$. Noting that, by assumption, $q > 1$ and then, applying a duality argument similar to that given in the proof of Proposition \ref{Proposition 7.2}, which relies on Lemmas \ref{LemmaDual2} and \ref{LemmaDual}, we derive that
\begin{equation*}
	I_{-1} : B^{-s_1, -b-1/q-\varepsilon}_{p_1',q'}(\mathbb{R}^d) \longrightarrow \text{Lip}^{(1,-b-1)}_{p',q'}(\mathbb{R}^d).
\end{equation*}
On the other hand, by Lemma \ref{LemmaLift}, we know that $I_1 : B^{-s_1 + 1, -b-1/q-\varepsilon}_{p_1',q'}(\mathbb{R}^d) \longrightarrow B^{-s_1, -b-1/q-\varepsilon}_{p_1',q'}(\mathbb{R}^d)$, which results in
\begin{equation}\label{8.35}
	B^{-s_1 + 1, -b-1/q-\varepsilon}_{p_1',q'}(\mathbb{R}^d) \hookrightarrow \text{Lip}^{(1,-b-1)}_{p',q'}(\mathbb{R}^d).
\end{equation}

Take any $p_1' < u < p'$ with $u \neq q'$. In particular, there is $\theta \in (0,1)$ such that
\begin{equation}\label{8.36}
 \frac{1-\theta}{p'} + \frac{\theta}{p_1'} = \frac{1}{u}.
 \end{equation}
 We may assume, without loss of generality, that $u < q'$ and so, we define $v$ given by the equation
 \begin{equation}\label{8.37}
 	\frac{1-\theta}{v} + \frac{\theta}{q'} = \frac{1}{u}.
 \end{equation}
Let $\delta = -b-1/q-\varepsilon$ and $0 < s < 1$. According to (\ref{PrelimInterpolationnew4}), we have
\begin{equation*}
	\text{Lip}^{(1,-b-1)}_{p',q'}(\mathbb{R}^d) = (B^{s,\delta}_{p',v}(\mathbb{R}^d), W^1_{p'}(\mathbb{R}^d))_{(1,-b-1),q'}
\end{equation*}
where we have also applied (\ref{BesovComparison}).
Hence, one may rewrite (\ref{8.35}) as
\begin{equation}\label{8.38}
	B^{-s_1 + 1, \delta}_{p_1',q'}(\mathbb{R}^d) \hookrightarrow (B^{s,\delta}_{p',v}(\mathbb{R}^d), W^1_{p'}(\mathbb{R}^d))_{(1,-b-1),q'}.
\end{equation}
Interpolating with parameters $(\theta,u)$ the embedding (\ref{8.38}) and the trivial embedding $B^{s,\delta}_{p',v}(\mathbb{R}^d) \hookrightarrow B^{s,\delta}_{p',v}(\mathbb{R}^d)$ implies
\begin{equation}\label{8.39}
(B^{s,\delta}_{p',v}(\mathbb{R}^d), B^{-s_1 + 1, \delta}_{p_1',q'}(\mathbb{R}^d))_{\theta,u} \hookrightarrow (B^{s,\delta}_{p',v}(\mathbb{R}^d), (B^{s,\delta}_{p',v}(\mathbb{R}^d), W^1_{p'}(\mathbb{R}^d))_{(1,-b-1),q'})_{\theta,u}.
\end{equation}

Next we identify the interpolation spaces in (\ref{8.39}) with Besov spaces. Concerning the target space, by Lemma \ref{PrelimLemma7.2}(ii) with $\alpha = 0$ and (\ref{PrelimInterpolationBWnew}), we have
\begin{align}
	(B^{s,\delta}_{p',v}(\mathbb{R}^d), (B^{s,\delta}_{p',v}(\mathbb{R}^d), W^1_{p'}(\mathbb{R}^d))_{(1,-b-1),q'})_{\theta,u} & = (B^{s,\delta}_{p',v}(\mathbb{R}^d), W^1_{p'}(\mathbb{R}^d))_{\theta, u; \theta (-b-1/q)} \nonumber \\
	&\hspace{-5cm} = B^{(1-\theta) s + \theta, \delta + \varepsilon \theta}_{p',u}(\mathbb{R}^d). \label{8.40}
\end{align}
Regarding the source space in (\ref{8.39}), using Lemma \ref{PrelimLemmaT2} with (\ref{8.36}) and (\ref{8.37}), we obtain
\begin{equation}\label{8.42}
	(B^{s,\delta}_{p',v}(\mathbb{R}^d), B^{-s_1 + 1, \delta}_{p_1',q'}(\mathbb{R}^d))_{\theta,u} = B^{(1-\theta) s + \theta(-s_1+1), \delta}_{u,u}(\mathbb{R}^d).
\end{equation}

Finally, by (\ref{8.39}), (\ref{8.40}) and (\ref{8.42}), we have
\begin{equation}\label{8.43}
	B^{(1-\theta) s + \theta(-s_1+1), \delta}_{u,u}(\mathbb{R}^d) \hookrightarrow B^{(1-\theta) s + \theta, \delta + \varepsilon \theta}_{p',u}(\mathbb{R}^d)
\end{equation}
satisfying that
\begin{equation*}
	(1-\theta) s + \theta (-s_1 + 1) -\frac{d}{u} = (1-\theta)s + \theta - \frac{d}{p'}.
\end{equation*}
However, since $\varepsilon > 0$, the embedding (\ref{8.43}) contradicts (\ref{8.4}).
Therefore, (\ref{8.34}) does not hold.		\qed

\subsection{Proofs of Proposition \ref{Proposition 8.6new} and \ref{Proposition 8.7new}}\label{section8.10} Firstly we prove Proposition \ref{Proposition 8.6new}. The statement (i) follows from (\ref{8.4}).


In order to prove (ii), we consider the function
$f=\psi W$, where $\psi \in \mathcal{S}(\mathbb{R}^d) \backslash \{0\}$ with (\ref{4.2}) and
	\begin{equation*}
		W(x) \sim \sum_{j=3}^\infty (1 + j)^{-\beta_1} e^{i (2^j -2) x_1}, \quad x \in \mathbb{R}^d,
	\end{equation*}
	 where $\beta_1 > b + 1/2 + 1/q$. Clearly, Proposition \ref{Proposition 4.1} yields that $f \not \in B^{s_0,b+1/\min\{p,q\}}_{p_0,q}(\mathbb{R}^d)$ because $s_0 > 0$, but $f \in B^{0,b+1/2}_{p,q}(\mathbb{R}^d)$.

Finally we show (iii). We proceed by contradiction and assume that
\begin{equation}\label{8.45new}
B^{0,b + 1/2}_{p,q}(\mathbb{R}^d) \cup B^{s_0, b + 1/\min\{p,q\}}_{p_0,q}(\mathbb{R}^d) = \mathbf{B}^{0,b}_{p,q}(\mathbb{R}^d).
\end{equation}
In particular, we have that
\begin{equation}\label{8.46new}
(\mathfrak{L} \cap B^{0,b + 1/2}_{p,q}(\mathbb{R}^d)) \cup (\mathfrak{L} \cap B^{s_0, b + 1/\min\{p,q\}}_{p_0,q}(\mathbb{R}^d)) = \mathfrak{L}  \cap \mathbf{B}^{0,b}_{p,q}(\mathbb{R}^d).
\end{equation}
Since $s_0 > 0$, it is clear from Proposition \ref{Proposition 4.1} that the left-hand side in (\ref{8.46new}) coincides with $\mathfrak{L} \cap B^{0,b + 1/2}_{p,q}(\mathbb{R}^d)$ and then,
\begin{equation}\label{8.47new}
\mathfrak{L} \cap B^{0,b + 1/2}_{p,q}(\mathbb{R}^d) = \mathfrak{L}  \cap \mathbf{B}^{0,b}_{p,q}(\mathbb{R}^d).
\end{equation}
However, (\ref{8.47new}) is no longer true. Indeed, let us consider the sequence of Fourier series
 \begin{equation*}
 	f_n(x) \sim \psi (x) e^{i (2^n - 2) x_1},  \quad x \in \mathbb{R}^d, \quad  n \geq 3.
\end{equation*}
Here, $\psi \in \mathcal{S}(\mathbb{R}^d) \backslash \{0\}$ such that (\ref{4.2}) holds. Then, applying Proposition \ref{Proposition 4.1} and Theorem \ref{Theorem 4.2}, we obtain
\begin{equation*}
	\|f_n\|_{B^{0,b+1/2}_{p,q}(\mathbb{R}^d)} \asymp (1 + n)^{b+1/2} \quad \text{and} \quad \|f_n\|_{\mathbf{B}^{0,b}_{p,q}(\mathbb{R}^d)} \asymp (1 + n)^{b+1/q}.
\end{equation*}
In light of the condition $2 < q$, this implies that (\ref{8.47new}) is not valid. Hence, (\ref{8.45new}) is not true and the proof of Proposition \ref{Proposition 8.6new} is completed.

The proof of Proposition \ref{Proposition 8.7new} follows the same lines as the proof for Proposition  \ref{Proposition 8.6new}. The details are left to the reader.

\qed

\newpage
\section{Various characterizations of Besov  spaces}\label{kfunctional}

Let $1 \leq p \leq \infty$ and $s > 0$. The Riesz potential space $\dot{\mathscr{L}}^{s}_p(\mathbb{R}^d)$\index{\bigskip\textbf{Spaces}!$\dot{\mathscr{L}}^{s}_p(\mathbb{R}^d)$}\label{RIESZ} is defined by
\begin{equation}\label{RieszSpace}
\dot{\mathscr{L}}^{s}_p(\mathbb{R}^d)=\{f\in L_p(\mathbb{R}^d): \|f\|_{\dot{\mathscr{L}}^{s}_p(\mathbb{R}^d)}=\|J_s f\|_{L_p(\mathbb{R}^d)}<\infty \},
\end{equation}
where $J_s f$\index{\bigskip\textbf{Operators}!$J_s$}\label{RIESZOPERATOR} denotes the Riesz potential of a function $f$, i.e.,
\begin{equation}\label{RieszPotential}
	J_s f= f*g_s, \quad \widehat{g_s}(\xi)=|\xi|^s.
\end{equation}
For any $s\in\mathbb{R}$,
 one can introduce the spaces $\dot{\mathscr{L}}^{s}_p(\mathbb{R}^d)$ as subspaces of $\mathcal{S}'(\mathbb{R}^d)$ with $\|J_s f\|_{L_p(\mathbb{R}^d)}<\infty$  modulo polynomials. We refer, e.g., to \cite[Chapter 3]{Peetre} and \cite[Chapter 5]{Triebel1} for further details.

Let $1\leq p \leq \infty$ and $k \in \mathbb{N}$. The (homogeneous) Sobolev space $\dot{W}^k_p(\mathbb{R}^d)$\index{\bigskip\textbf{Spaces}!$\dot{W}^k_p(\mathbb{R}^d)$}\label{SOBHOM} is formed by all $f \in L_p(\mathbb{R}^d)$ for which
\begin{equation*}
	  \|f\|_{\dot{W}^{k}_p(\mathbb{R}^d)} = \sum_{|\alpha| = k} \|D^\alpha f\|_{L_p(\mathbb{R}^d)} < \infty.
\end{equation*}
It is well known that
\begin{equation}\label{LPHom}
 \dot{\mathscr{L}}^{k}_p(\mathbb{R}^d) = \dot{W}^k_p(\mathbb{R}^d), \quad 1 < p < \infty,
 \end{equation}
with equivalence of (semi-) norms.

For future use, we recall two basic properties of homogeneous Besov spaces. The lifting property of the homogeneous Besov spaces $\dot{B}^s_{p,q}(\mathbb{R}^d)$\index{\bigskip\textbf{Spaces}!$\dot{B}^s_{p,q}(\mathbb{R}^d)$}\label{BESOVHOM} (see \cite[Theorem 1, Section 5.2.3, p. 242]{Triebel1}) claims that
	\begin{equation}\label{5.5new3}
		\|J_{s_0} f\|_{\dot{B}^{s_1}_{p, q} (\mathbb{R}^d)} \asymp \|f\|_{\dot{B}^{s_0 + s_1}_{p, q} (\mathbb{R}^d)}, \quad \text{for} \quad -\infty < s_0, s_1 < \infty \quad \text{and}  \quad 0 < p, q \leq \infty.
	\end{equation}
	 One can compare homogeneous Besov spaces with their nonhomogeneous counterparts (cf. \cite[p. 148]{BerghLofstrom} and \cite[Remark 3, Section 5.2.3, p. 242]{Triebel1}). More specifically, we have
	\begin{equation}\label{5.5new2}
	B^s_{p,q}(\mathbb{R}^d) = L_p(\mathbb{R}^d) \cap \dot{B}^s_{p,q}(\mathbb{R}^d) \quad \text{and} \quad \|f\|_{B^s_{p,q}(\mathbb{R}^d)} \asymp \|f\|_{L_p(\mathbb{R}^d) } + \|f\|_{\dot{B}^s_{p,q}(\mathbb{R}^d) }
	\end{equation}
	if $s > 0, 1 \leq p \leq \infty$, and $1 \leq q \leq \infty$.

\subsection{Characterizations involving Riesz potential space}
The goal of this section is to provide an abstract characterization of the spaces $\mathbf{B}^{s,b}_{p,q}(\mathbb{R}^d)$ in terms of the $K$-functional (\ref{Peetre}) 
 for the couple
 $(L_p(\mathbb{R}^d),\dot{\mathscr{L}}^{s}_p(\mathbb{R}^d))$, paying special attention to the delicate cases $p=1, \infty$ and $s=0$. As applications, we will obtain descriptions of the spaces $\mathbf{B}^{s,b}_{p,q}(\mathbb{R}^d)$ defined by the moduli of smoothness in terms of other means such as ball and sphere averages (Subsection \ref{balls1}), Bochner-Riesz means (Subsection \ref{balls2}), heat kernels or thermic extensions (Subsection \ref{balls3}).

\begin{thm}\label{Theorem $K$-functional}
Assume that $s \geq 0, 1 \le p \le \infty,$ $0 < q \leq \infty,$ and $-\infty < b < \infty$. Let $\alpha> s$. Then
\begin{equation}\label{$K$ functional characterization 1}
\|f\|_{\mathbf{B}^{s,b}_{p,q}(\mathbb{R}^d)}
\asymp \|f\|_{L_p(\mathbb{R}^d)} + \left(\int_0^1 t^{-s q} (1-\log t)^{b q}
K(t^{\alpha},f;L_p(\mathbb{R}^d), \dot{\mathscr{L}}^{\alpha}_p(\mathbb{R}^d))^q
 \frac{dt}{t}\right)^{1/q}.
 \end{equation}
The corresponding result for periodic spaces also holds true.
\end{thm}
\begin{rem}
 Note that if $s =0$ and $b < -1/q$ then the characterization (\ref{$K$ functional characterization 1}) holds trivially because both left- and right-hand sides are equivalent to $ \|f\|_{L_p(\mathbb{R}^d)}$ (see Section \ref{Preliminaries}). Hence, (\ref{$K$ functional characterization 1}) with $s=0$ is interesting if $b \geq -1/q$.
\end{rem}

Before proving  Theorem \ref{Theorem $K$-functional}, let us recall the characterization of the $K$-functional for the couple $(L_p(\mathbb{R}^d), \dot{\mathscr{L}}^{\alpha}_p(\mathbb{R}^d))$ in terms of the modulus of smoothness $\omega_\alpha (f,t)_p$\index{\bigskip\textbf{Functionals and functions}!$\omega_\alpha (f,t)_p$}\label{MODA} of fractional order $\alpha > 0$ of a function $f \in L_p(\mathbb{R}^d)$, given by (cf. \cite[p. 788]{ButzerDyckhoffGorlichStens})
\begin{equation*}
	\omega_\alpha(f,t)_p = \sup_{|h| \leq t} \|\Delta^\kappa_h f\|_{L_p(\mathbb{R}^d)}, \quad \Delta^\kappa_h f (x) = \sum_{\nu=0}^\infty (-1)^\nu {{\kappa} \choose {\nu}} f(x + \nu h).
\end{equation*}
Note that $\Delta^\kappa_h f$\index{\bigskip\textbf{Functionals and functions}!$\Delta^\kappa_h$}\label{DELTAFRAC} is well defined for $f\in L_p$ since $\sum_{\nu=0}^\infty | {{\kappa} \choose {\nu}}|\le C(\kappa)$, $\kappa>0$.
If $\alpha = k \in \mathbb{N}$, then we obtain the classical modulus of smoothness $\omega_k(f,t)_p$ given in (\ref{modulus}). Analogously, one can introduce the periodic counterpart for functions $f \in L_p(\mathbb{T})$.

\begin{lem}[\cite{ButzerDyckhoffGorlichStens, Wilmes, Wilmes2}]\label{LemmaModuli}
Let $1 < p < \infty,$ and $\alpha >0$. Then,
\begin{equation}\label{8+}
	K(t^\alpha,f;L_p(\mathbb{R}^d), \dot{\mathscr{L}}^{\alpha}_p(\mathbb{R}^d))
\asymp \omega_{\alpha}(f,t)_p.
\end{equation}
The periodic counterpart also holds true.
\end{lem}

Applying the Marchaud inequality for the modulus of smoothness $\omega_\alpha(f,t)_p$ (see, e.g.,
 \cite[Corollary 2]{ButzerDyckhoffGorlichStens} and \cite[Theorem 4.4]{Kol})
\begin{equation}\label{Kol}
	\omega_\alpha(f,t)_p \lesssim t^\alpha \left(\|f\|_{L_p(\mathbb{R}^d)} + \int_t^1 u^{-\alpha} \omega_\beta(f,u)_p \frac{du}{u}\right), \quad \beta > \alpha, \quad 0 < t <1,
\end{equation}
it is readily seen that Besov spaces can also be introduced through  $\omega_\alpha(f,t)_p$. For the sake of completeness, we give below the proof of this fact.

\begin{lem}\label{LemmaFractionalBesov}
	Let $1 \leq p \leq \infty, 0 < q \leq \infty, -\infty < b < \infty,$ and $s \geq 0$. Assume that $\alpha\in \mathbb{R}$ and  $\alpha > s$. Then
	\begin{equation*}
		\|f\|_{\mathbf{B}^{s,b}_{p,q}(\mathbb{R}^d)} \asymp \|f\|_{L_p(\mathbb{R}^d)} +  \left(\int_0^1 (t^{-s} (1 - \log t)^b \omega_\alpha(f,t)_p)^q
    \frac{dt}{t}\right)^{1/q} .
	\end{equation*}
	The periodic counterpart also holds true.
\end{lem}
\begin{proof}
	Let $k \in \mathbb{N}$ be such that $k > \alpha$. Then,
		\begin{equation}\label{aux4.9}
		\|f\|_{\mathbf{B}^{s,b}_{p,q}(\mathbb{R}^d)} \asymp \|f\|_{L_p(\mathbb{R}^d)} +  \left(\int_0^1 (t^{-s} (1 - \log t)^b \omega_k(f,t)_p)^q
    \frac{dt}{t}\right)^{1/q} .
	\end{equation}
	Since $\omega_k(f,t)_p \lesssim \omega_\alpha(f,t)_p$ (see \cite[Lemma 6]{ButzerDyckhoffGorlichStens}), it is clear that
	\begin{equation*}
		\|f\|_{\mathbf{B}^{s,b}_{p,q}(\mathbb{R}^d)} \lesssim \|f\|_{L_p(\mathbb{R}^d)} +  \left(\int_0^1 (t^{-s} (1 - \log t)^b \omega_\alpha(f,t)_p)^q
    \frac{dt}{t}\right)^{1/q} .
	\end{equation*}
	
	To prove the converse estimate, we will make use of (\ref{Kol}) together with the Hardy's inequality (\ref{HardyInequal2}) (noting that $\omega_k(f,u)_p/u^k$ is equivalent to a decreasing function) to get
	\begin{align*}
		 \left(\int_0^1 (t^{-s} (1 - \log t)^b \omega_\alpha(f,t)_p)^q \frac{dt}{t} \right)^{1/q} & \\
		 &\hspace{-4cm}\lesssim \|f\|_{L_p(\mathbb{R}^d)} +  \left(\int_0^1 \left(t^{\alpha-s} (1 - \log t)^b \int_t^1 \frac{\omega_k(f,u)_p}{u^\alpha} \frac{du}{u}\right)^q \frac{dt}{t}\right)^{1/q} \\
		 &\hspace{-4cm} \lesssim \|f\|_{L_p(\mathbb{R}^d)} +  \left(\int_0^1 (t^{-s} (1 - \log t)^b \omega_k(f,t)_p)^q  \frac{dt}{t}\right)^{1/q} .
	\end{align*}
	The result now follows from (\ref{aux4.9}).
\end{proof}

We are now ready to give the

\begin{proof}[Proof of Theorem \ref{Theorem $K$-functional}]

If $1 < p < \infty$, the result follows immediately from Lemmas \ref{LemmaModuli} and \ref{LemmaFractionalBesov}.

	  Let now $p=1, \infty$. It follows from the trivial embeddings
	\begin{equation*}
		\dot{B}^0_{p,1}(\mathbb{R}^d) \hookrightarrow L_p(\mathbb{R}^d) \hookrightarrow \dot{B}^0_{p, \infty} (\mathbb{R}^d)
	\end{equation*}
	that
	\begin{equation*}
		\|J_\alpha f\|_{\dot{B}^0_{p, \infty} (\mathbb{R}^d)} \lesssim \|f\|_{\dot{\mathscr{L}}^{\alpha}_p(\mathbb{R}^d)} \lesssim \|J_\alpha f\|_{\dot{B}^0_{p, 1} (\mathbb{R}^d)}.
	\end{equation*}
	Applying the lifting property (\ref{5.5new3}) in Besov spaces, the above estimates are equivalent to the chain of embeddings
	\begin{equation*}
		\dot{B}^\alpha_{p, 1} (\mathbb{R}^d) \hookrightarrow \dot{\mathscr{L}}^{\alpha}_p(\mathbb{R}^d) \hookrightarrow \dot{B}^\alpha_{p, \infty} (\mathbb{R}^d),
	\end{equation*}
	which yields to its non-homogeneous counterpart
		\begin{equation}\label{5.6new}
		\mathbf{B}^\alpha_{p, 1} (\mathbb{R}^d) \hookrightarrow L_p(\mathbb{R}^d) \cap \dot{\mathscr{L}}^{\alpha}_p(\mathbb{R}^d) \hookrightarrow \mathbf{B}^\alpha_{p, \infty} (\mathbb{R}^d),
	\end{equation}
	where we have used (\ref{5.5new2}) and (\ref{BesovComparison}).
	
	We now distinguish two possible cases. Assume first that $s=0$. Then, by (\ref{PrelimInterpolationnew2}) (with $\xi=0$), we get
	\begin{align*}
		\mathbf{B}^{0,b}_{p,q}(\mathbb{R}^d) & =  (L_p(\mathbb{R}^d), \mathbf{B}^{\alpha}_{p,1}(\mathbb{R}^d))_{(0,b),q} \hookrightarrow (L_p(\mathbb{R}^d), L_p(\mathbb{R}^d) \cap \dot{\mathscr{L}}^{\alpha}_p(\mathbb{R}^d))_{(0,b),q}  \\
		& \hookrightarrow  (L_p(\mathbb{R}^d), \mathbf{B}^{\alpha}_{p,\infty}(\mathbb{R}^d))_{(0,b),q} =  \mathbf{B}^{0,b}_{p,q}(\mathbb{R}^d),
	\end{align*}
	that is, $\mathbf{B}^{0,b}_{p,q}(\mathbb{R}^d) =  (L_p(\mathbb{R}^d), L_p(\mathbb{R}^d) \cap \dot{\mathscr{L}}^{\alpha}_p(\mathbb{R}^d))_{(0,b),q}$. Since
	\begin{equation}\label{5.8new}
		K(t,f; L_p(\mathbb{R}^d), L_p(\mathbb{R}^d) \cap \dot{\mathscr{L}}^{\alpha}_p(\mathbb{R}^d)) \asymp t \|f\|_{ L_p(\mathbb{R}^d)} + K(t,f; L_p(\mathbb{R}^d), \dot{\mathscr{L}}^{\alpha}_p(\mathbb{R}^d)), \quad 0 < t \leq 1,
	\end{equation}
	 (see \cite[(6.4), p. 4415]{CobosDominguezTriebel}), we obtain that
	 \begin{align*}
	 	\|f\|_{\mathbf{B}^{0,b}_{p,q}(\mathbb{R}^d)} & \asymp \|f\|_{(L_p(\mathbb{R}^d), L_p(\mathbb{R}^d) \cap \dot{\mathscr{L}}^{\alpha}_p(\mathbb{R}^d))_{(0,b),q}} \\
		& \asymp  \|f\|_{L_p(\mathbb{R}^d)} + \left(\int_0^1 (1-\log t)^{b q}
K(t^{\alpha},f;L_p(\mathbb{R}^d), \dot{\mathscr{L}}^{\alpha}_p(\mathbb{R}^d))^q
 \frac{dt}{t}\right)^{1/q}.
	 \end{align*}
	 This gives the desired characterization for $s=0$.
	
	 If $s > 0$ and $\alpha > s$, we can apply (\ref{PrelimInterpolationnew2.2}) (with $\xi=0$) and (\ref{5.6new})  to derive
		\begin{align*}
		\mathbf{B}^{s,b}_{p,q}(\mathbb{R}^d) & =  (L_p(\mathbb{R}^d), \mathbf{B}^{\alpha}_{p,1}(\mathbb{R}^d))_{s/\alpha,q;b} \hookrightarrow (L_p(\mathbb{R}^d), L_p(\mathbb{R}^d) \cap \dot{\mathscr{L}}^{\alpha}_p(\mathbb{R}^d))_{s/\alpha,q;b}   \\
		& \hookrightarrow  (L_p(\mathbb{R}^d), \mathbf{B}^{\alpha}_{p,\infty}(\mathbb{R}^d))_{s/\alpha,q;b}  =  \mathbf{B}^{s,b}_{p,q}(\mathbb{R}^d),
	\end{align*}
	that is, $\mathbf{B}^{s,b}_{p,q}(\mathbb{R}^d) =  (L_p(\mathbb{R}^d), L_p(\mathbb{R}^d) \cap \dot{\mathscr{L}}^{\alpha}_p(\mathbb{R}^d))_{s/\alpha,q;b}$. Then, by (\ref{5.8new}),
	 \begin{equation*}
	 	\|f\|_{\mathbf{B}^{s,b}_{p,q}(\mathbb{R}^d)}  \asymp  \|f\|_{L_p(\mathbb{R}^d)} + \left(\int_0^1 t^{-\frac{s q}{\alpha}}(1-\log t)^{b q}
K(t,f;L_p(\mathbb{R}^d), \dot{\mathscr{L}}^{\alpha}_p(\mathbb{R}^d))^q
 \frac{dt}{t}\right)^{1/q}.
	 \end{equation*}
	 A simple change of variables completes the proof.
\end{proof}

\begin{rem} Here
	we give an alternative proof of Theorem \ref{Theorem $K$-functional}. Since
	\begin{equation}\label{aux4.12}
		\omega_\beta(f,t)_p \lesssim K(t^\alpha, f: L_p(\mathbb{R}^d),  \dot{\mathscr{L}}^{\alpha}_p(\mathbb{R}^d)) \lesssim \omega_\gamma(f,t)_p, \quad 0 < \gamma < \alpha < \beta,
	\end{equation}
	(see \cite[Theorem 7, (4.3)]{Wilmes}), the equivalence (\ref{$K$ functional characterization 1}) follows now from (\ref{aux4.12}) and Lemma \ref{LemmaFractionalBesov}.

In the periodic setting, the proof follows the same lines because the periodic counterparts of the estimates given in (\ref{aux4.12}) hold; see \cite[Theorem 3, (22)]{Wilmes2}.
	
	\end{rem}

\subsection{Characterizations via averages on balls}\label{balls1}

For $t >0$ and $x \in \mathbb{R}^d$ (or $\mathbb{T}$), let $B_t(x)$\index{\bigskip\textbf{Sets}!$B_t(x)$}\label{BALL} denote the ball of radius $t$ centered at $x$. For a locally integrable function $f$ on $\mathbb{R}^d$ (or $\mathbb{T}$), we define the averaging operator $B_t f$\index{\bigskip\textbf{Operators}!$B_t$}\label{AVERAGE} by
\begin{equation*}
	B_t f (x) = \frac{1}{|B_t(x)|_d} \int_{B_t(x)} f(y) dy
\end{equation*}
and the $l$-th order average, $l \geq 1$, by
\begin{equation*}
	B_{l,t} f(x) = \frac{-2}{{{2l} \choose{l}}} \sum_{j=1}^l (-1)^j {{2l} \choose {l-j}} B_{j t} f(x).
\end{equation*}
Obviously, $B_{1,t} f = B_t f$.

Recently, there has been a growing interest in obtaining characterizations of smoothness spaces in terms of the operators $f-B_{l,t}f$; see, e.g., \cite{AlabernMateuVerdera, DaiGogatishviliYangYuan2, DaiGogatishviliYangYuan, ZhuoSickelYangYuan}. Besides their intrinsic interest, such characterizations allow us to introduce smoothness spaces on metric measure spaces because the means $f-B_{l,t}f$ only depend on the metric of $\mathbb{R}^d$ and the  measure. In particular, the following result was shown in \cite[Theorem 3.1] {DaiGogatishviliYangYuan}: Let $s > 0, 1 < p \leq \infty$, and $0 < q \leq \infty$. Then
\begin{equation}\label{DGYY}
	\|f\|_{\mathbf{B}^s_{p,q}(\mathbb{R}^d)} \asymp  \|f\|_{L_p(\mathbb{R}^d)} + \left(\int_0^1 t^{-s q}  \|f - B_{l,t} f\|_{L_p(\mathbb{R}^d)}^q \frac{dt}{t}\right)^{1/q}.
\end{equation}
Here, $l \in \mathbb{N}$ such that $l > s/2$. The proof of the above characterization given in \cite{DaiGogatishviliYangYuan} relies on maximal functions and vector-valued maximal inequalities. However, this method does not cover the interesting case when $p=1$ and/or $s=0$ in (\ref{DGYY}). Further, it is worth mentioning that the characterization of the space $\mathbf{B}^0_{1,1}(\mathbb{T})$ in terms of ball averages plays a crucial role in order to formulate the Bressan's mixing problem, as can be seen in \cite{Bianchini, HadzicSeegerSmartStreet}. In particular, Bianchini \cite{Bianchini} showed that
 \begin{equation}\label{BianchiniMeasurable}
		\|\chi_A\|_{\mathbf{B}^{0,0}_{1,1}(\mathbb{R}^d)} \asymp \|\chi_A\|_{L_1(\mathbb{R}^d)} + \int_0^1 \left\|\chi_A - B_t \chi_A \right\|_{L_1(\mathbb{R}^d)} \frac{dt}{t}
	\end{equation}
	for any measurable subset $A \subset \mathbb{R}^d$.
	
In this section, we propose a completely different approach to the one given in \cite{DaiGogatishviliYangYuan} to obtain the characterization of the spaces $\mathbf{B}^{s,b}_{p,q}(\mathbb{R}^d)$ via ball averages.
We extend  (\ref{BianchiniMeasurable}) to any locally integrable function
for all ranges of parameters $s, b, p, q$, including the important cases $p=1$ and $s=0$, which were left open in  \cite{DaiGogatishviliYangYuan}.

\begin{thm}
Let $s \geq 0, 1 \le p \le \infty,$ $0 < q \leq \infty$, and $-\infty < b < \infty$. Assume that $l \in \mathbb{N}$ such that $l > s/2$. Then
\begin{equation*}
	\|f\|_{\mathbf{B}^{s,b}_{p,q}(\mathbb{R}^d)} \asymp  \|f\|_{L_p(\mathbb{R}^d)} + \left(\int_0^1 t^{-s q} (1 - \log t)^{b q}  \|f - B_{l,t} f\|_{L_p(\mathbb{R}^d)}^q \frac{dt}{t}\right)^{1/q}.
\end{equation*}
 The corresponding result for periodic spaces also holds true.
\end{thm}

\begin{proof}
The proof is based on the following result proved in \cite[Theorem 1]{DaiWang}. For $f$ sufficiently regular, we define\index{\bigskip\textbf{Operators}!$\Delta^l$}\label{LAPLACEP}
\begin{equation}\label{Laplacian}
\Delta f = \frac{\partial^2 f}{\partial x_1^2} + \cdots +
\frac{\partial^2 f}{\partial x_d^2},\qquad \Delta^l f = \Delta (\Delta^{l-1} f), \quad  l > 1.
\end{equation}
Note that (see (\ref{RieszPotential}))
\begin{equation}\label{LaplacianLift}
 \widehat{\Delta^l f}(\xi) = (-1)^l |\xi|^{2 l} \widehat{f}(\xi) = (-1)^l  \widehat{J_{2l} f}(\xi).
 \end{equation}
  This latter formula allows us to introduce the $l$th order Laplacian in the distributional sense.

\begin{lem}
	Let $l \in \mathbb{N}$ and $1 \leq p \leq \infty$. Then
	\begin{equation}\label{KfunctionalBallAverage}
		\|f - B_{l,t} f\|_{L_p(\mathbb{R}^d)} \asymp \inf_{\Delta^l g \in L_p(\mathbb{R}^d)} \Big(\|f-g\|_{L_p(\mathbb{R}^d)} + t^{2 l} \|\Delta^l g\|_{L_p(\mathbb{R}^d)}\Big).
	\end{equation}
	The corresponding result for periodic functions also holds true.
\end{lem}

We point out that (\ref{KfunctionalBallAverage}) can be written in terms of the spaces $\dot{\mathscr{L}}^{2 l}_p(\mathbb{R}^d)$ (cf. (\ref{RieszSpace})) and the $K$-functional (cf. (\ref{Peetre})), that is,
	\begin{equation*}
		\|f - B_{l,t} f\|_{L_p(\mathbb{R}^d)} \asymp K(t^{2 l}, f ; L_p(\mathbb{R}^d), \dot{\mathscr{L}}^{2 l}_p(\mathbb{R}^d)).
	\end{equation*}
 Now the proof is a simple application of Theorem \ref{Theorem $K$-functional}.

\end{proof}

\begin{rem}
This method can also be  applied to characterize the spaces $\mathbf{B}^{s,b}_{p,q}(\mathbb{R}^d)$ in terms of different kinds of averages that satisfy the estimates corresponding to (\ref{KfunctionalBallAverage}). For instance, one may consider the averages on a sphere given by\index{\bigskip\textbf{Operators}!$V_t$}\label{AVERAGESPH}
\begin{equation*}
	V_t f (x) = \frac{1}{m_t} \int_{|y-x| = t} f(y) dy, \quad t > 0,
\end{equation*}
and, for $l \in \mathbb{N}$,
\begin{equation*}
	V_{l,t} f(x) = \frac{-2}{{{2l} \choose{l}}} \sum_{j=1}^l (-1)^j {{2l} \choose {l-j}} V_{j t} f(x)
\end{equation*}
(see \cite{DaiDitzian}). Here $m_t$\index{\bigskip\textbf{Numbers, relations}!$m_t$}\label{MEAS} is the (surface) measure of the sphere of radius $t$ in $\mathbb{R}^d, d >2$. Further examples are averages on the box with center at $x$ studied in \cite{DitzianIvanov}.
\end{rem}

	\subsection{Characterizations in terms of differences} In this section we establish several descriptions of the space $\mathbf{B}^{s, b}_{p,q}(\mathbb{R}^d)$ in terms of $k$-th differences.
	
	\begin{thm}\label{ThmRonStin}
		 Let $0 \leq s < k \in \mathbb{N}, -\infty < b < \infty,1 \leq p \leq \infty$, and $0 < q \leq \infty$. Then,
		 \begin{equation}\label{BesovDiff*General}
		 	\|f\|_{\mathbf{B}^{s, b}_{p,q}(\mathbb{R}^d)} \asymp \|f\|_{L_p(\mathbb{R}^d)} + \left( \int_{B_1(0)} \frac{(1 - \log |h|)^{b q} \|\Delta^k_h f\|_{L_p(\mathbb{R}^d)}^q}{|h|^{d + s q}} d h  \right)^{1/q}.
		 \end{equation}
		 The corresponding result for periodic spaces also holds true.
	\end{thm}
	
	\begin{rem}
	It is easy to see that both sides in (\ref{BesovDiff*General}) are equivalent to $\|f\|_{L_p(\mathbb{R}^d)}$ if $s=0$ and $b < -1/q \, (b \leq 0 \text{ if } q = \infty)$.
	\end{rem}
	
	\begin{proof}[Proof of Theorem \ref{ThmRonStin}]
	Let $q < \infty$. It is not hard to see that
\begin{equation*}
	\omega_k(f,t)_p \asymp \left(t^{-d} \int_{|h| \leq t} \|\Delta^k_h f\|^q_{L_p(\mathbb{R}^d)} d h \right)^{1/q}, \quad k \in \mathbb{N},
\end{equation*}
(see \cite[(10), p. 4]{KolyadaLerner} and \cite[Appendix A]{KaradzhovMilmanXiao}), and hence,
\begin{align*}
	|f|_{\mathbf{B}^{s, b}_{p,q}(\mathbb{R}^d)}^q &\asymp \int_0^1 t^{-d -s q} (1 - \log t)^{b q}  \int_{|h| \leq t} \|\Delta^k_h f\|^q_{L_p(\mathbb{R}^d)} d h \frac{dt}{t} \\
	& \asymp \int_{B_1(0)}  \|\Delta^k_h f\|^q_{L_p(\mathbb{R}^d)} \frac{(1 - \log |h|)^{b q}}{|h|^{d + s q}} dh
\end{align*}
which yields (\ref{BesovDiff*General}). If $q=\infty$ the proof of (\ref{BesovDiff*General}) is easier and we omit the details.
	\end{proof}
	
	Putting $p=q$ in Theorem \ref{ThmRonStin} we obtain the following

\begin{cor}\label{CorRonStin}
		 Let $0 \leq s < k \in \mathbb{N}, -\infty < b < \infty,$ and $1 \leq p \leq \infty$. Then,
		 \begin{equation}\label{BesovDiff*}
		 	\|f\|_{\mathbf{B}^{s, b}_{p,p}(\mathbb{R}^d)} \asymp \|f\|_{L_p(\mathbb{R}^d)} + \left(\int_{\mathbb{R}^d} \int_{B_1(0)} \frac{(1 - \log |h|)^{b p}|\Delta^k_h f (x)|^p}{|h|^{d + s p}} d h d x \right)^{1/p}, \quad p < \infty,
		 \end{equation}
		 and
		 	 \begin{equation}\label{BesovDiff**}
		 	\|f\|_{\mathbf{B}^{s, b}_{\infty,\infty}(\mathbb{R}^d)} \asymp \|f\|_{L_\infty(\mathbb{R}^d)} + \sup_{\substack{x, h \in \mathbb{R}^d \\ 0 < |h| < 1}} \frac{(1 - \log |h|)^{b}|\Delta^k_h f (x)|}{|h|^{s}}.
		 \end{equation}
		 In particular, if $0 \leq s < 1$, then
		 \begin{equation}\label{BesovDiff}
	\|f\|_{\mathbf{B}^{s, b}_{p,p}(\mathbb{R}^d)} \asymp  \|f\|_{L_p(\mathbb{R}^d)} +  \left(\int_{\mathbb{R}^d} \int_{B_1(0)} \frac{(1 - \log |x-y|)^{b p}|f(x) - f(y)|^p}{|x-y|^{d + s p}} d y d x \right)^{1/p}, \quad p < \infty,
\end{equation}
and
		 \begin{equation*}
	\|f\|_{\mathbf{B}^{s, b}_{\infty,\infty}(\mathbb{R}^d)} \asymp  \|f\|_{L_\infty(\mathbb{R}^d)} +  \sup_{\substack{x, y \in \mathbb{R}^d \\ 0 < |x-y| < 1}} \frac{(1 - \log |x-y|)^b | f(x)- f(y)|}{|x-y|^s}.
\end{equation*}
\end{cor}

\begin{rem}
Using the fact that $\omega_k(f,t)_p\lesssim\|f\|_p$, it is easy to see that for $s > 0$
 \begin{equation}\label{14.23**}
 	\|f\|_{\mathbf{B}^{s,b}_{p,q}(\mathbb{R}^d)} \asymp \|f\|_{L_p(\mathbb{R}^d)} + \left(\int_0^\infty (t^{-s} (1 + |\log t|)^b \omega_k(f,t)_p)^q
    \frac{dt}{t}\right)^{1/q}, \quad -\infty < b<  \infty.
 \end{equation}
 However, this is not true if $s=0$ because the right-hand side of (\ref{14.23**}) becomes trivial (recall that $b \geq -1/q$ and $\omega_k(f,t)_p$ is an increasing function). Taking into account (\ref{14.23**}) and following the proof of Theorem \ref{ThmRonStin}, one can show that if $s > 0$ then the formula (\ref{BesovDiff*General}) can be replaced by
 	 \begin{equation*}
		 	\|f\|_{\mathbf{B}^{s, b}_{p,q}(\mathbb{R}^d)} \asymp \|f\|_{L_p(\mathbb{R}^d)} + \left( \int_{\mathbb{R}^d} \frac{(1 + |\log |h||)^{b q} \|\Delta^k_h f\|_{L_p(\mathbb{R}^d)}^q}{|h|^{d + s q}} d h  \right)^{1/q}.
		 \end{equation*}
 In particular, if $q = p$ then we have
 		 \begin{equation}\label{BesovDiff*2}
		 	\|f\|_{\mathbf{B}^{s, b}_{p,p}(\mathbb{R}^d)} \asymp \|f\|_{L_p(\mathbb{R}^d)} + \left(\int_{\mathbb{R}^d} \int_{\mathbb{R}^d} \frac{(1 + |\log |h||)^{b p}|\Delta^k_h f (x)|^p}{|h|^{d + s p}} d h d x \right)^{1/p}, \quad p < \infty,
		 \end{equation}
		 and
		  \begin{equation}\label{BesovDiff**2}
		 	\|f\|_{\mathbf{B}^{s, b}_{\infty,\infty}(\mathbb{R}^d)} \asymp \|f\|_{L_\infty(\mathbb{R}^d)} + \sup_{x, h \in \mathbb{R}^d} \frac{(1 +| \log |h||)^{b}|\Delta^k_h f (x)|}{|h|^{s}},
		 \end{equation}
		 which complement (\ref{BesovDiff*}) and (\ref{BesovDiff**}),
		 respectively. Note that if $0 < s < k=1$ then the previous characterizations can be written as
		 		 \begin{equation}\label{BesovDiff*3}
	\|f\|_{\mathbf{B}^{s, b}_{p,p}(\mathbb{R}^d)} \asymp  \|f\|_{L_p(\mathbb{R}^d)} +  \left(\int_{\mathbb{R}^d} \int_{\mathbb{R}^d} \frac{(1 + |\log |x-y||)^{b p}|f(x) - f(y)|^p}{|x-y|^{d + s p}} d x d y \right)^{1/p}, \quad p < \infty,
\end{equation}
and
		 \begin{equation}\label{BesovDiff**3}
	\|f\|_{\mathbf{B}^{s, b}_{\infty,\infty}(\mathbb{R}^d)} \asymp  \|f\|_{L_\infty(\mathbb{R}^d)} +  \sup_{x, y \in \mathbb{R}^d} \frac{(1 +| \log |x-y||)^b | f(x)- f(y)|}{|x-y|^s}.
\end{equation}
The family of norms given in the right-hand side of (\ref{BesovDiff*2})--(\ref{BesovDiff**3}) is widely used in the study of regularity problems of PDE's. See the survey of Di Nezza, Palatucci and Valdinoci \cite{DiNPalVald} and the references therein. Sometimes, these norms are called Aronszajn-Slobodeckij-Sobolev norms, but the space $\dot{\mathbf{B}}^{s}_{p,p}(\mathbb{R}^d)$ should be carefully distinguished from the Sobolev space $\dot{\mathscr{L}}^{s}_p(\mathbb{R}^d)$ introduced in (\ref{RieszSpace}). In fact, if $s > 0$ then
\begin{equation}\label{BesovDiff2}
	\dot{\mathbf{B}}^{s}_{p,q}(\mathbb{R}^d) = \dot{\mathscr{L}}^{s}_p(\mathbb{R}^d) \iff p=q = 2.
\end{equation}
See \cite[Chapter V]{Stein}.
\end{rem}

\begin{rem}
	Further characterizations of Besov-type spaces in terms of differences can be found in \cite{BesovIlinNikolskii}.
\end{rem}

\subsection{Characterizations via approximation processes}\label{balls2}

Besov norms can also be characterized in terms of various approximation methods.
We start with the classical characterization 
 in terms of best approximations.

Let $E_n(f)_{L_p(\mathbb{R}^d)}$ 
be the best approximation of $f \in {L_p(\mathbb{R}^d)}$ by  entire functions of
spherical exponential type $n$. Following the characterization of the classical Besov space (see \cite[Section 5.6]{Nikolskii}) one can prove the following result.
 \begin{cor}[\bf{Approximation description of $\B^{s,b}_{p,q}(\mathbb{R}^d)$}]\label{cor4.8*}
Assume that $s \geq 0, 1 \le p \le \infty,$ $0 < q \leq \infty$, and $-\infty < b < \infty$. Then
$$
\|f\|_{\mathbf{B}^{s,b}_{p,q}(\mathbb{R}^d)}
\asymp \|f\|_{L_p(\mathbb{R}^d)} + \left(
\sum\limits_{\nu=0}^\infty 2^{\nu s q} (1+\nu)^{b q} E_{2^\nu}^q(f)_{L_p(\mathbb{R}^d)} \right)^{1/q}.
$$
\end{cor}

We also consider the approximation behavior of the generalized Bochner--Riesz means defined by \index{\bigskip\textbf{Operators}!$S^{\lambda,\alpha}_t$}\label{BOCHNERRIESZ}
$$S^{\lambda,\alpha}_t f(x)= \int_{|\xi|\le t} \widehat{f}(\xi)\Big(1-\Big(\frac{|\xi|}{t}\Big)^\alpha\Big)^{\lambda}e^{ix \cdot \xi}d\xi,\quad\alpha>0.
$$
If $\alpha =2$, then we get the classical Bochner-Riesz operator that has been widely used in harmonic analysis. See \cite{Bourgain, Christ, LeeSeeger, LuYan, Tao}, to mention just a few references. It is known that if $\lambda>\frac{d-1}2$, then the Bochner--Riesz means are bounded in $L_p(\mathbb{R}^d)$ and
\begin{equation*}
	\|S^{\lambda,\alpha}_t f\|_{L_p(\mathbb{R}^d)} \leq C \|f\|_{L_p(\mathbb{R}^d)}, \quad f \in L_p(\mathbb{R}^d),
\end{equation*}
 where $C$ is a positive constant which is independent of $f$ and $t$.

 It is known that the approximation method given by $S^{\lambda, \alpha}_t$ in $L_p(\mathbb{R}^d)$ can be compared with the one given by Weierstrasss means $W^\alpha_t$. Let\index{\bigskip\textbf{Operators}!$W^\alpha_t$}\label{WEIERSTRASS}
 \begin{equation}\label{WeierstrassMeans}
 	W^\alpha_t f (x) = \int_{\mathbb{R}^d} \widehat{f}(\xi) e^{-(t |\xi|)^\alpha} e^{ix \cdot \xi}d\xi, \quad \alpha > 0.
 \end{equation}
 If $\lambda > \frac{d-1}{2}$, then (cf. \cite[Theorem 4.6]{Trebels})
 \begin{equation}\label{5.12new}
 	\|f - S^{\lambda, \alpha}_{1/t} f\|_{L_p(\mathbb{R}^d)} \asymp \|f - W^{\alpha}_t f\|_{L_p(\mathbb{R}^d)},
 \end{equation}
where  the  constants in this equivalence are independent of $f$ and $t$. As a consequence, the approximation behavior of the Bochner-Riesz means only depends on the index $\alpha$, that is,
$$\|f-S^{\lambda_1,\alpha}_t f\|_{L_p(\mathbb{R}^d)}\asymp \|f-S^{\lambda_2,\alpha}_t f\|_{L_p(\mathbb{R}^d)},\quad \lambda_1,\lambda_2> \frac{d-1}2.$$

It is folklore that the degree of approximation by Weierstrass means can be expressed in terms of the $K$-functional. Namely,
\begin{equation}\label{5.13new2}
\|f-W^{\alpha}_{t} f\|_{L_p(\mathbb{R}^d)} \asymp
K(t^\alpha,f;L_p(\mathbb{R}^d), \dot{\mathscr{L}}^{\alpha}_p(\mathbb{R}^d)),
\quad 1\le p  \le\infty,
\end{equation}
and so, by (\ref{5.12new}),
\begin{equation}\label{5.14new2}
\|f-S^{\lambda,\alpha}_{1/t} f\|_{L_p(\mathbb{R}^d)}\asymp
K(t^\alpha,f;L_p(\mathbb{R}^d), \dot{\mathscr{L}}^{\alpha}_p(\mathbb{R}^d)),
\qquad 1\le p  \le\infty, \quad \lambda > \frac{d-1}{2}.
\end{equation}
See also \cite{Trebels-lecture} and \cite[Theorem 2]{Ditzian}.

As a immediate consequence of Theorem \ref{Theorem $K$-functional}, (\ref{5.13new2}), and (\ref{5.14new2}), we obtain the following.

 \begin{cor}\label{Corollary 5.8new}
Assume that $s \geq 0, 1 \le p \le \infty,$ $0 < q \leq \infty$, and $-\infty < b < \infty$. Let $\alpha > s$. Then
$$
\|f\|_{\mathbf{B}^{s,b}_{p,q}(\mathbb{R}^d)}
\asymp \|f\|_{L_p(\mathbb{R}^d)} + \left(\int_0^1 t^{-s q} (1-\log t)^{b q}
\|f-S^{\lambda,\alpha}_{1/t} f\|_{L_p(\mathbb{R}^d)}^q \frac{dt}{t}\right)^{1/q}, \quad \lambda > \frac{d-1}{2},
$$
and
\begin{equation}\label{BesovWeierstrass}
\|f\|_{\mathbf{B}^{s,b}_{p,q}(\mathbb{R}^d)}
\asymp \|f\|_{L_p(\mathbb{R}^d)} + \left(\int_0^1 t^{-s q} (1-\log t)^{b q}
\|f-W^{\alpha}_{t} f\|_{L_p(\mathbb{R}^d)}^q \frac{dt}{t}\right)^{1/q}.
\end{equation}
\end{cor}

\begin{rem}
(i) Similar results can also be written for different moduli of smoothness given in \cite{Kol, run}.

(ii) The periodic counterparts of the characterizations given in Corollary \ref{Corollary 5.8new} can also be established. We leave the details to the reader. Further characterizations of periodic Besov spaces with positive smoothness in terms of approximation processes can be found in \cite{SchmeisserSickel}.
\end{rem}

\subsection{Characterizations by means of heat and Poisson kernels}\label{balls3}

Let\index{\bigskip\textbf{Operators}!$W_t$}\label{GAUSSWEIERSTRASS}
\begin{equation}\label{WeierstrassSemiGroup}
	W_t f (x) = \frac{1}{(4 \pi t)^{d/2}} \int_{\mathbb{R}^d} e^{- \frac{|x-y|^2}{4 t}} f(y) dy = \Big(e^{- t |\xi|^2} \widehat{f} (\xi)\Big)^\vee (x), \qquad t > 0,  \quad x \in \mathbb{R}^d,
\end{equation}
be the Gauss-Weierstrass semi-group and $W_0 = \text{id}_{\mathbb{R}^d}$ (identity). We mention that $u(x,t) = W_t f (x)$ is a solution of the heat equation in the upper half-space $\mathbb{R}^{d+1}_+ = \{(x,t): x \in \mathbb{R}^d, t > 0\}$\index{\bigskip\textbf{Sets}!$\mathbb{R}^{d+1}_+$}\label{HALFSPACE}, with $u(x,0) = f(x)$. In other words, $W_t f$ is a thermic (caloric) extension of $f$ from $\mathbb{R}^d$ to $\mathbb{R}^{d+1}_+$.

The Cauchy-Poisson semi-group is given by\index{\bigskip\textbf{Operators}!$P_t$}\label{POISSON}
\begin{equation}\label{PoissonSemiGroup}
	P_t f (x) = c_d \int_{\mathbb{R}^d} \frac{t}{(|x-y|^2 + t^2)^{\frac{d+1}{2}}} f(y) dy = \Big(e^{- t |\xi|} \widehat{f} (\xi)\Big)^\vee (x), \qquad t > 0,  \quad x \in \mathbb{R}^d,
\end{equation}
with $c_d \Big\|(1 + |x|^2)^{-\frac{d+1}{2}}\Big\|_{L_1(\mathbb{R}^d)} = 1$, and complemented by $P_0 = \text{id}_{\mathbb{R}^d}$. We have that $u(x,t) = P_t f (x)$ is a harmonic function in $\mathbb{R}^{d+1}_+$, that is,
\begin{equation*}
\frac{\partial^2 u(x,t)}{\partial t^2} + \Delta u(x,t) = 0 \quad \text{in} \quad \mathbb{R}^{d+1}_+,
\end{equation*}
with $u(x,0) = f(x)$. In other words, $P_t f$ is a harmonic extension of $f$ from $\mathbb{R}^d$ to $\mathbb{R}^{d+1}_+$.

Characterizations of smoothness spaces in terms of heat and Poisson kernels have a long history, which traced back to the works by Taibleson \cite{Taibleson1, Taibleson2}. They are intimately connected to the theory of semi-groups and PDEs (see, e.g., \cite{ButzerBerens, Triebel}). The following descriptions of Besov spaces with positive smoothness are well known: Let $1 \leq p \leq \infty, 0 < q \leq \infty$, and $s > 0$. Let $m \in \mathbb{N}$ with $m > s/2$. Then,
\begin{equation}\label{5.15new}
	\|f\|_{\mathbf{B}^s_{p,q}(\mathbb{R}^d)} \asymp \|f\|_{L_p(\mathbb{R}^d)} + \left(\int_0^1 t^{-\frac{s q}{2}} \| [W_t - \text{id}_{\mathbb{R}^d}]^m f\|_{L_p(\mathbb{R}^d)}^q \frac{dt}{t}\right)^{1/q}
\end{equation}
(see \cite[Theorem 3.4.6, p. 198]{ButzerBerens} and \cite[Section 1.13.2, pp. 76--81]{Triebel}).

In this section, we are mainly interested in the extreme case $s=0$ in (\ref{5.15new}). In particular, it was recently proved in \cite[Theorem 6.2]{CobosDominguezTriebel} that if $1 < p < \infty, 0 < q \leq \infty, b \geq -1/q$, and $m \in \mathbb{N}$, then,
\begin{equation}\label{5.16new2'}
	\|f\|_{\mathbf{B}^{0,b}_{p,q}(\mathbb{R}^d)} \asymp \|f\|_{L_p(\mathbb{R}^d)} + \left(\int_0^1 (1-\log t)^{b q} \| [W_t - \text{id}_{\mathbb{R}^d}]^m f\|_{L_p(\mathbb{R}^d)}^q \frac{dt}{t}\right)^{1/q}.
\end{equation}
However, the method given in \cite{CobosDominguezTriebel} does not work with the extreme cases $p=1$ or $p=\infty$ in (\ref{5.16new2'}). Next, we are able to overcome this obstruction and extend (\ref{5.16new2'}) to the full range of parameters.

\begin{thm}\label{Theorem Heat Kernels}
Let $s \geq 0, 1 \leq p \leq \infty, 0 < q \leq \infty,$ and $-\infty < b < \infty$.
\begin{enumerate}[\upshape(i)]
\item Let $m \in \mathbb{N}$ with $m > s/2$. Then,
\begin{equation}\label{BesovHeat}
	\|f\|_{\mathbf{B}^{s,b}_{p,q}(\mathbb{R}^d)} \asymp \|f\|_{L_p(\mathbb{R}^d)} + \left(\int_0^1 t^{-\frac{s q}{2}}(1-\log t)^{b q} \| [W_t - \emph{id}_{\mathbb{R}^d}]^m f\|_{L_p(\mathbb{R}^d)}^q \frac{dt}{t}\right)^{1/q}.
\end{equation}
\item Let $m \in \mathbb{N}$ with $m > s$. Then,
\begin{equation}\label{BesovPoisson}
	\|f\|_{\mathbf{B}^{s,b}_{p,q}(\mathbb{R}^d)} \asymp \|f\|_{L_p(\mathbb{R}^d)} + \left(\int_0^1 t^{- s q}(1-\log t)^{b q} \| [P_t - \emph{id}_{\mathbb{R}^d}]^m f\|_{L_p(\mathbb{R}^d)}^q \frac{dt}{t}\right)^{1/q}.
\end{equation}
\end{enumerate}
\end{thm}

	Before we show Theorem \ref{Theorem Heat Kernels} we collect some basic definitions on semi-groups of operators. We essentially follow \cite{ButzerBerens, Triebel}. Let $A$ be a Banach space and let $\{T_t\}_{t \geq 0}$ be an \emph{equibounded semi-group of class $C_0$} of linear operators from $A$ into itself, that is,
	\begin{eqnarray*}
		&T_{t_1} T_{t_2} = T_{t_1 + t_2} \quad &\text{for} \quad t_1, t_2 \geq 0, \quad T_0 = \text{id}_A,
\\	
		&\|T_t a\|_A \leq C \|a\| \quad &\text{for all} \quad t \geq 0, \quad C \,\,\, \text{being a constant},
\\	
		&\lim_{t \downarrow 0} \|T_t a - a\|_A = 0 \quad &\text{for all} \quad a \in A.
	\end{eqnarray*}
	The infinitesimal generator $\Lambda$ of the semi-group is defined by\index{\bigskip\textbf{Operators}!$\Lambda$}\label{GEN}
	\begin{equation}\label{5.17new2}
		\Lambda a = \lim_{t \downarrow 0} t^{-1} (T_t - \text{id}_A) a
	\end{equation}
	whenever the limit exists. The domain $D(\Lambda)$\index{\bigskip\textbf{Spaces}!$D(\Lambda)$}\label{DOMGEN} of $\Lambda$ is formed by all those $a \in A$ for which (\ref{5.17new2}) exists. For every $m \in \mathbb{N}$, the domain of $\Lambda^m$ is endowed with the (semi-)norm
	\begin{equation*}
		\|a\|_{D(\Lambda^m)} = \|\Lambda^m a\|_A, \quad a \in D(\Lambda^m).
	\end{equation*}
	Such a semi-group is called \emph{analytic} if $T_t$ maps $A$ into $D(\Lambda)$ and if $\sup_{t > 0} t \|\Lambda T_t a\|_A < \infty$.
	
	We write down two auxiliary results that will be useful in later considerations.
	
	\begin{lem}[\cite{DitzianIvanov, trebels}]\label{LemmaK-FunctionalDI}
	Let $A$ be a Banach space and $\{T_t\}_{t \geq 0}$ be an analytic semi-group of operators in $A$. For $m \in \mathbb{N}$, we have
	\begin{equation*}
		K(t^m,a; A, D(\Lambda^m)) \asymp \|[T_t - \emph{id}_{A}]^m a\|_A.
	\end{equation*}
	\end{lem}
	
	\begin{lem}[\cite{DeVoreRiemenschneiderSharpley}]\label{LemmaMarchaud}
		Let $A$ be a Banach space and $\{T_t\}_{t \geq 0}$ be an equibounded $C_0$ semi-group of operators in $A$. The following Marchaud-type inequality holds
		\begin{equation*}
			K(t^m, a; A, D(\Lambda^m)) \lesssim t^m \int_t^\infty \frac{K(u^{m+r}, a; A, D(\Lambda^{m+r}))}{u^m} \frac{du}{u}
		\end{equation*}
		for $m, r \in \mathbb{N}$.
	\end{lem}
	
	\begin{proof}[Proof of Theorem \ref{Theorem Heat Kernels}]
		We know that $\Lambda = \Delta$ is the infinitesimal generator relative to the analytic semi-group $\{W_t\}_{t \geq 0}$ in $L_p(\mathbb{R}^d)$ (see \cite[Theorem 4.3.11, p. 261]{ButzerBerens} and \cite[Section 2.5.2, pp. 190--192]{Triebel}), and so, $\Lambda^m = \Delta^m$ (cf. (\ref{Laplacian})). Thus, $D(\Lambda^m) = D(\Delta^m) = \dot{\mathscr{L}}^{2 m}_p(\mathbb{R}^d)$ (cf. (\ref{LaplacianLift}) and (\ref{RieszSpace})).  Hence, in virtue of Lemma \ref{LemmaK-FunctionalDI}, we have
		\begin{equation*}
			K(t^m, f; L_p(\mathbb{R}^d), \dot{\mathscr{L}}^{2 m}_p(\mathbb{R}^d)) \asymp  \|[W_t - \text{id}_{\mathbb{R}^d}]^m f\|_{L_p(\mathbb{R}^d)}.
		\end{equation*}
	We conclude the proof of (i) by invoking Theorem \ref{Theorem $K$-functional}.
	
	Next we proceed with the proof of (ii). We have that $\{P_t\}_{t \geq 0}$ is an analytic semi-group of operators in $L_p(\mathbb{R}^d)$ with
\begin{equation*}
	\Lambda^{2 m} = (- \Delta)^m, \quad D(\Lambda^{2 m}) =  \dot{\mathscr{L}}^{2 m}_p(\mathbb{R}^d).
\end{equation*}
See \cite[Section 2.5.3, pp. 192--196]{Triebel}. According to Lemma \ref{LemmaK-FunctionalDI},
		\begin{equation*}
			K(t^{2m}, f; L_p(\mathbb{R}^d), \dot{\mathscr{L}}^{2 m}_p(\mathbb{R}^d)) \asymp \|[P_t - \text{id}_{\mathbb{R}^d}]^{2 m} f\|_{L_p(\mathbb{R}^d)}
		\end{equation*}
		and then, by Theorem \ref{Theorem $K$-functional}, we get
		\begin{equation}\label{5.16new2}
\|f\|_{\mathbf{B}^{s,b}_{p,q}(\mathbb{R}^d)}
\asymp \|f\|_{L_p(\mathbb{R}^d)} + \left(\int_0^1 t^{-s q} (1-\log t)^{b q}
\|[P_t - \text{id}_{\mathbb{R}^d}]^{2 m} f\|_{L_p(\mathbb{R}^d)}^q
 \frac{dt}{t}\right)^{1/q}.
 \end{equation}
 This shows (ii) for even powers of $P_t - \text{id}_{\mathbb{R}^d}$. It remains to prove that (\ref{5.16new2}) holds for every natural number. This is a simple consequence of the elementary inequality
 \begin{equation*}
 	\|[P_t - \text{id}_{\mathbb{R}^d}]^{m + 1} f\|_{L_p(\mathbb{R}^d)} \lesssim \|[P_t - \text{id}_{\mathbb{R}^d}]^{m} f\|_{L_p(\mathbb{R}^d)},
 \end{equation*}
 and the converse estimate
 \begin{equation*}
 	\|[P_t - \text{id}_{\mathbb{R}^d}]^{m} f\|_{L_p(\mathbb{R}^d)} \lesssim t^m \int_t^\infty \frac{\|[P_u - \text{id}_{\mathbb{R}^d}]^{m + 1} f\|_{L_p(\mathbb{R}^d)}}{u^m} \frac{du}{u},
 \end{equation*}
 which follows from Lemmas \ref{LemmaK-FunctionalDI} and \ref{LemmaMarchaud} and the Hardy's inequality (\ref{HardyInequal2}) (noting that, by Lemma \ref{LemmaK-FunctionalDI}, $\frac{\|[P_t - \text{id}_{\mathbb{R}^d}]^{m} f\|_{L_p(\mathbb{R}^d)}}{t^m} \asymp \frac{K(t^m, f;  L_p(\mathbb{R}^d), D(\Lambda^{m}))}{t^m}$ and the latter is a decreasing function).
	\end{proof}
	
	\begin{rem}
		Note that the heat kernel $W_t$ (see (\ref{WeierstrassSemiGroup})) and the Poisson kernel $P_t$ (see (\ref{PoissonSemiGroup})) are particular cases of Weierstrass means (see (\ref{WeierstrassMeans})). More precisely, we have $W_t  = W^2_{t^{1/2}}$ and $P_t = W^1_t$. Consequently, if $m=1$ and $\alpha =2$ (respectively, $\alpha = 1$) then the characterizations given in (\ref{BesovWeierstrass}) and (\ref{BesovHeat}) (respectively, (\ref{BesovPoisson})) coincide.
	\end{rem}

\newpage
\section{Besov and  Bianchini norms}\label{section3.5}
%

The following characterization of $\mathbf{B}^{0,0}_{1,1}(\mathbb{R}^d)$ in terms of minimization problems for bounded variation functions plays
 a key role  to establish new results related to Bressan's mixing problem in the paper by Bianchini \cite{Bianchini}. Namely,
\begin{equation}\label{1+}
	\|f\|_{\mathbf{B}^{0,0}_{1,1}(\mathbb{R}^d)} \asymp \int_0^1  \inf_{g \in \text{BV}(\mathbb{R}^d)} \Big\{\|f-g\|_{L_1(\mathbb{R}^d)}+t \|g\|_{\text{BV}(\mathbb{R}^d)} \Big\} \frac{dt}{t}.
\end{equation}
Note that the above expression can be understood in terms of the $K$-functional of $(L_1(\mathbb{R}^d), \text{BV}(\mathbb{R}^d))$ (see (\ref{Peetre})) and interpolation spaces with $\theta=0$ (cf. (\ref{limitinginterpolation})), that is,
\begin{equation*}
	\|f\|_{\mathbf{B}^{0,0}_{1,1}(\mathbb{R}^d)} \asymp \int_0^1  K(t,f; L_1(\mathbb{R}^d), \text{BV}(\mathbb{R}^d)) \frac{dt}{t} = \|f\|_{(L_1(\mathbb{R}^d), \text{BV}(\mathbb{R}^d))_{(0,0),1}}.
\end{equation*}
In particular, this formula shows that there is a relationship between the measures of smoothness $\omega_k(f,t)_1$ and $K(t,f;L_1(\mathbb{R}^d), \text{BV}(\mathbb{R}^d))$.

Our first goal in this section is to describe this relationship explicitly and
 obtain sharp estimates between $\omega_k(f,t)_p$ and $K(t,f;L_p(\mathbb{R}^d), \text{Lip}^{(k,-\alpha)}_{p,r}(\mathbb{R}^d))$ for $1 \leq p \leq \infty$.
 This can be viewed as
 a quantitative version of (\ref{1+});  recall that $\text{Lip}^{(1,0)}_{1,\infty}(\mathbb{R}^d) = \text{BV}(\mathbb{R}^d)$ (cf. (\ref{def:BV})).

 Moreover, our approach can also be applied to obtain characterizations of $K$-functionals
  for another important function couples such as, e.g., $(\mathbf{B}^{0,b}_{p,q}(\mathbb{R}^d), W^k_p(\mathbb{R}^d))$ and $(\mathbf{B}^{0,b}_{p,q}(\mathbb{R}^d), \text{Lip}^{(k,-\alpha)}_{p,r}(\mathbb{R}^d))$, in terms of $\omega_k(f,t)_p$. These characterizations will complement the well-known estimate (see (\ref{KFunctional}))
\begin{equation}\label{aux5.2}
		K(t^k,f;L_p(\mathbb{R}^d), W^k_p(\mathbb{R}^d)) \asymp  t^k  \|f\|_{L_p(\mathbb{R}^d)} + \omega_k(f,t)_p, \quad 0 < t < 1.
\end{equation}

\subsection{Characterizations of various $K$-functionals}
 \begin{thm}\label{thmKfunctionalBV}
 	Let $k \in \mathbb{N}, 1 \leq p \leq \infty,$ and $0 < q, r \leq \infty$.
	\begin{enumerate}[\upshape(i)]
	\item Assume $0 < s < k$ and $-\infty < b < \infty$. Then,
	\begin{align}
		K(t^{s/k}(1-\log t)^{-b}, f: L_p(\mathbb{R}^d), \mathbf{B}^{s,b}_{p,q}(\mathbb{R}^d)) & \asymp t^{s/k}(1-\log t)^b \|f\|_{L_p(\mathbb{R}^d)} \nonumber \\
		&  \hspace{-5cm}+ t^{s/k}(1-\log t)^b \left(\int_{t^{1/k}}^1 (u^{-s} (1 - \log u)^b \omega_k(f,u)_p)^q \frac{du}{u}\right)^{1/q} \label{KFunctLpBp}
	\end{align}
	for all $0 < t < 1$.
	\item Assume $0 < s < k$ and $-\infty < b < \infty$. Then,
	\begin{align}
		K(t^{1- s/k}(1-\log t)^{b}, f: \mathbf{B}^{s,b}_{p,q}(\mathbb{R}^d), W^k_p(\mathbb{R}^d)) & \asymp t^{1- s/k}(1-\log t)^b \|f\|_{L_p(\mathbb{R}^d)} \nonumber \\
		&  \hspace{-5cm}+ \left(\int^{t^{1/k}}_0 (u^{-s} (1 - \log u)^b \omega_k(f,u)_p)^q \frac{du}{u}\right)^{1/q} \label{KFunctBpWp}
	\end{align}
	for all $0 < t < 1$.
	\item Assume $\alpha > 1/r \, (\alpha \geq 0 \text{ if } r = \infty)$. Then,
	\begin{align}
		K(t^k (1-\log t)^{\alpha -\frac{1}{r}}, f ; L_p(\mathbb{R}^d), \emph{Lip}^{(k,-\alpha)}_{p,r}(\mathbb{R}^d)) & \asymp t^k (1-\log t)^{\alpha -\frac{1}{r}} \|f\|_{L_p(\mathbb{R}^d)} \nonumber \\
		 &  \hspace{-7cm}	+ \omega_k(f,t)_p
	 + t^k (1-\log t)^{\alpha -\frac{1}{r}} \left(\int_t^1 (u^{-k} (1-\log u)^{-\alpha} \omega_k(f,u)_p)^r \frac{du}{u}\right)^{\frac{1}{r}} \label{5.3new+}
	\end{align}
	for all $0 < t < 1$.
	\item Assume $b > -1/q$. Then,
	\begin{align}
	K(t^k(1-\log t)^{b + \frac{1}{q}}, f; \mathbf{B}^{0,b}_{p,q}(\mathbb{R}^d), W^k_p(\mathbb{R}^d)) &  \asymp t^k (1-\log t)^{b + \frac{1}{q}} \|f\|_{L_p(\mathbb{R}^d)} \nonumber \\
	& \hspace{-6.5cm} + (1 - \log t)^{b+\frac{1}{q}} \omega_k(f,t)_p + \left(\int_0^t ((1-\log u)^b \omega_k(f,u)_p)^q \frac{du}{u}\right)^{\frac{1}{q}} \label{5.4new+}
	\end{align}
	for all $0 < t < 1$.
	\item  Assume $b > -1/q$. Then,
	\begin{align}
	K((1-\log t)^{-b - \frac{1}{q}}, f; L_p(\mathbb{R}^d), \mathbf{B}^{0,b}_{p,q}(\mathbb{R}^d)) &  \asymp  (1-\log t)^{-b - \frac{1}{q}} \|f\|_{L_p(\mathbb{R}^d)} \nonumber \\
	& \hspace{-5.5cm} + (1 - \log t)^{-b-\frac{1}{q}}  \left(\int_t^1 ((1-\log u)^b \omega_k(f,u)_p)^q \frac{du}{u}\right)^{\frac{1}{q}} \label{5.4new++}
	\end{align}
	for all $0 < t < 1$.
	\item Assume $\alpha > 1/r \, (\alpha \geq 0 \text{ if } r = \infty)$. Then,
	\begin{align}
		K( (1-\log t)^{-\alpha +\frac{1}{r}}, f ; \emph{Lip}^{(k,-\alpha)}_{p,r}(\mathbb{R}^d), W^k_p(\mathbb{R}^d)) & \asymp (1-\log t)^{-\alpha +\frac{1}{r}} \|f\|_{L_p(\mathbb{R}^d)} \nonumber \\
		 &  \hspace{-7cm}	+ \left(\int_0^t (u^{-k} (1-\log u)^{-\alpha} \omega_k(f,u)_p)^r \frac{du}{u}\right)^{\frac{1}{r}} \label{5.3new++}
	\end{align}
	for all $0 < t < 1$.
	\item Assume $b > -1/q$ and $\alpha > 1/r$. Then,
	\begin{align}
		K(t^k (1-\log t)^{b+\frac{1}{q} +\alpha -\frac{1}{r}}, f ; \mathbf{B}^{0,b}_{p,q}(\mathbb{R}^d), \emph{Lip}^{(k,-\alpha)}_{p,r}(\mathbb{R}^d) ) \nonumber \\
		& \hspace{-6.5cm}\asymp t^k  (1-\log t)^{b+\frac{1}{q} +\alpha -\frac{1}{r}} \|f\|_{L_p(\mathbb{R}^d)} \nonumber \\
		&\hspace{-6.25cm}+ (1-\log t)^{b+\frac{1}{q}} \omega_k(f,t)_p + \left(\int_0^t ((1-\log u)^b \omega_k(f,u)_p)^q \frac{du}{u}\right)^{\frac{1}{q}} \label{5.5new+} \\
		& \hspace{-6.25cm}+ t^k (1-\log t)^{b + \frac{1}{q} + \alpha -\frac{1}{r}} \left(\int_t^1 (u^{-k} (1-\log u)^{-\alpha} \omega_k(f,u)_p)^r \frac{du}{u}\right)^{\frac{1}{r}} \nonumber
	\end{align}
	for all $0 < t < 1$.
	\end{enumerate}
	The same estimates also hold for periodic functions.
 \end{thm}
 \begin{rem}
 	(i) The estimate (\ref{KFunctLpBp}) with $b=0$ can also be found in \cite[(4.9)]{Nilsson} and \cite[Theorem 3.1]{DeVoreYu}.
	
	
	(ii) The homogeneous counterpart of (\ref{aux5.2}) reads as follows
	\begin{equation}\label{aux5.2hom}
		K(t^k,f;L_p(\mathbb{R}^d), \dot{W}^k_p(\mathbb{R}^d)) \asymp \omega_k(f,t)_p, \quad 0 < t < 1,
\end{equation}
see (\ref{8+}) and (\ref{LPHom}). Using (\ref{aux5.2hom}) one can also establish the homogeneous version of the formulas (\ref{KFunctLpBp})-(\ref{5.5new+}). In such cases the terms involving $\|f\|_{L_p(\mathbb{R}^d)}$ in (\ref{KFunctLpBp})-(\ref{5.5new+}) can be dropped. For example, we have
	\begin{align*}
		K(t^{s/k}(1-\log t)^{-b}, f: L_p(\mathbb{R}^d), \dot{\mathbf{B}}^{s,b}_{p,q}(\mathbb{R}^d)) & \\
		& \hspace{-3cm}\asymp
		 t^{s/k}(1-\log t)^b \left(\int_{t^{1/k}}^1 (u^{-s} (1 - \log u)^b \omega_k(f,u)_p)^q \frac{du}{u}\right)^{1/q} .
	\end{align*}
	Here, $1 \leq p \leq \infty, 0 < q \leq \infty, 0 < s < k \in \mathbb{N}$ and $-\infty < b < \infty$.
 \end{rem}

 Before we show Theorem \ref{thmKfunctionalBV}, let us write down some special cases.
 \begin{cor}\label{corKfunctionalBV}
 	The following estimates hold for any $0 < t < 1$,
	\begin{enumerate}[\upshape(i)]
		\item $K(t,f; L_1(\mathbb{R}^d), \emph{BV}(\mathbb{R}^d)) \asymp t \|f\|_{L_1(\mathbb{R}^d)} + \omega_1(f,t)_1 \asymp K(t,f;L_1(\mathbb{R}^d), W^1_1(\mathbb{R}^d))$,
		\item $K(t,f; L_\infty(\mathbb{R}^d), \emph{Lip}(\mathbb{R}^d)) \asymp t \|f\|_{L_\infty(\mathbb{R}^d)} + \omega_1(f,t)_\infty \asymp K(t,f;L_\infty(\mathbb{R}^d), W^1_\infty(\mathbb{R}^d))$,

%
%

		\item \begin{align*}
		K(t(1-\log t), f; \mathbf{B}^{0,0}_{1,1}, W^1_1(\mathbb{R}^d))& \asymp t (1-\log t) \|f\|_{L_1(\mathbb{R}^d)} + (1-\log t) \omega_1(f,t)_1 \\
		& \hspace{1cm}+ \int_0^t \omega_1(f,u)_1 \frac{du}{u},
		\end{align*}
		
		
		\item Let $b > 0$. Then,
		
		 \begin{align*}
			K(t(1-\log t)^b, f ; \mathbf{B}^{0,b}_{\infty,\infty}(\mathbb{R}^d), W^1_\infty(\mathbb{R}^d)) & \asymp  t (1-\log t)^b \|f\|_{L_\infty(\mathbb{R}^d)} + (1-\log t)^b \omega_1(f,t)_\infty \\
		& \hspace{1cm}+ \sup_{0 \leq u \leq t} (1-\log u)^b \omega_1(f,u)_\infty .
		\end{align*}
	\end{enumerate}
 \end{cor}

 \begin{rem}
 	The homogeneous counterpart of Corollary \ref{corKfunctionalBV} also holds. For instance, we have
	\begin{equation*}
		K(t,f; L_1(\mathbb{R}^d), \dot{\text{BV}}(\mathbb{R}^d)) \asymp \omega_1(f,t)_1 \asymp K(t,f;L_1(\mathbb{R}^d), \dot{W}^1_1(\mathbb{R}^d))
	\end{equation*}
	and
	\begin{equation*}
		K(t,f; L_\infty(\mathbb{R}^d), \dot{\text{Lip}}(\mathbb{R}^d)) \asymp  \omega_1(f,t)_\infty \asymp K(t,f;L_\infty(\mathbb{R}^d), \dot{W}^1_\infty(\mathbb{R}^d)).
	\end{equation*}
 \end{rem}

 \begin{proof}[Proof of Theorem \ref{thmKfunctionalBV}]
 	By (\ref{interpolationLipschitz}), we have
	\begin{equation*}
		K(t, f; L_p(\mathbb{R}^d), \text{Lip}^{(k,-\alpha)}_{p,r}(\mathbb{R}^d)) \asymp K(t, f;L_p(\mathbb{R}^d),  (L_p(\mathbb{R}^d), W^k_p(\mathbb{R}^d))_{(1,-\alpha),r}).
	\end{equation*}
	Thus, according to Lemma \ref{PrelimHolmstedt}(iii) and (\ref{aux5.2}),
	\begin{align*}
		K(t^k (1-\log t)^{\alpha -\frac{1}{r}}, f; L_p(\mathbb{R}^d), \text{Lip}^{(k,-\alpha)}_{p,r}(\mathbb{R}^d)) \asymp t^k \|f\|_{L_p(\mathbb{R}^d)} + \omega_k(f,t)_p \\
		& \hspace{-10cm} + t^k (1-\log t)^{\alpha -\frac{1}{r}} \left(\int_t^1 (1-\log u)^{-\alpha r} \frac{du}{u}\right)^{\frac{1}{r}} \|f\|_{L_p(\mathbb{R}^d)} \\
		& \hspace{-10cm}+ t^k (1-\log t)^{\alpha -\frac{1}{r}}  \left(\int_t^1 (u^{-k} (1-\log u)^{-\alpha} \omega_k(f,u)_p)^r \frac{du}{u}\right)^{\frac{1}{r}}.
	\end{align*}
	Since $\int_0^1 (1-\log u)^{-\alpha r} \frac{du}{u} < \infty \, (\alpha > 1/r)$, we obtain (iii).
	
	Next we show (vii). Applying (\ref{PrelimInterpolationnew2}) and (\ref{interpolationLipschitz}), one gets
	\begin{align*}
		K(t, f ; \mathbf{B}^{0,b}_{p,q}(\mathbb{R}^d), \text{Lip}^{(k,-\alpha)}_{p,r}(\mathbb{R}^d) )  \\
		& \hspace{-4cm} \asymp K(t, f;(L_p(\mathbb{R}^d), W^k_p(\mathbb{R}^d)_{(0,b),q},  (L_p(\mathbb{R}^d), W^k_p(\mathbb{R}^d))_{(1,-\alpha),r}).
	\end{align*}
	Then, using Lemma \ref{PrelimHolmstedt}(vii) and (\ref{aux5.2}), we have
	\begin{align*}
		K(t^k (1-\log t)^{b+1/q+\alpha-1/r}, f ; \mathbf{B}^{0,b}_{p,q}(\mathbb{R}^d), \text{Lip}^{(k,-\alpha)}_{p,r}(\mathbb{R}^d) ) & \asymp t^k (1-\log t)^{b+1/q} \|f\|_{L_p(\mathbb{R}^d)} \\
		& \hspace{-7cm}+ (1-\log t)^{b+1/q} \omega_k(f,t)_p + \left(\int_0^t (u^k (1-\log u)^b)^q \frac{du}{u}\right)^{1/q} \|f\|_{L_p(\mathbb{R}^d)} \\
		& \hspace{-7cm} + \left(\int_0^t ((1-\log u)^b \omega_k(f,u)_p)^q \frac{du}{u}\right)^{1/q} \\
		& \hspace{-7cm} + t^k (1-\log t)^{b+1/q+\alpha-1/r} \left(\int_t^1 (1-\log u)^{-\alpha r} \frac{du}{u}\right)^{1/r} \|f\|_{L_p(\mathbb{R}^d)} \\
		& \hspace{-7cm} + t^k (1-\log t)^{b+1/q+\alpha-1/r} \left(\int_t^1 (u^{-k} (1-\log u)^{-\alpha} \omega_k(f,u)_p)^r \frac{du}{u}\right)^{1/r} \\
		& \hspace{-7cm} \asymp t^k (1-\log t)^{b+1/q+\alpha-1/r} \|f\|_{L_p(\mathbb{R}^d)} + (1-\log t)^{b+1/q} \omega_k(f,t)_p  \\
		& \hspace{-6.5cm}  + \left(\int_0^t ((1-\log u)^b \omega_k(f,u)_p)^q \frac{du}{u}\right)^{1/q} \\
		& \hspace{-6.5cm} + t^k (1-\log t)^{b+1/q+\alpha-1/r} \left(\int_t^1 (u^{-k} (1-\log u)^{-\alpha} \omega_k(f,u)_p)^r \frac{du}{u}\right)^{1/r}.
	\end{align*}
	
	The proofs of (i), (ii), (iv), (v), and (vi) follow similar ideas as those of (iii) and (vii) but now applying Lemma \ref{PrelimHolmstedt}(i), (ii), (iv), (v) and (vi), respectively, and are left to the reader.
 \end{proof}

\begin{rem}\label{Remark 5.3}
 	The terms involving the moduli of smoothness on the right-hand side of (\ref{5.3new+}) are independent each other. Let
	\begin{equation*}
		I(f, t) = t^k (1-\log t)^{\alpha -\frac{1}{r}} \left(\int_t^1 (u^{-k} (1-\log u)^{-\alpha} \omega_k(f,u)_p)^r \frac{du}{u}\right)^{\frac{1}{r}}.
	\end{equation*}
	There exists a sufficiently regular $f$ such that $\omega_k(f,t)_p \asymp t^\beta$ with $0 < \beta < k$, see \cite{Tikhonov-real}.
Elementary computations yield that $I(f,t) \asymp t^\beta (1-\log t)^{-1/r}$. On the other hand, if $\omega_k(f,t)_p \asymp t^k (1-\log t)^{\alpha -1/r}$, then $I(f,t) \asymp t^k (1-\log t)^{\alpha-1/r} \log(1-\log t) \gtrsim \omega_k(f,t)_p$. Analogously, one can show the independence of the corresponding terms in (\ref{5.4new+}) and (\ref{5.5new+}). For instance, consider $\omega_k(f,t)_p \asymp t^k$ and $\omega_k(f,t)_p \asymp (1-\log t)^{-b-1/q} (1 + \log (1-\log t))^{-\delta}, \delta > 1/q,$ to deal with the terms
	\begin{equation*}
		\left(\int_0^t ((1-\log u)^b \omega_k(f,u)_p)^q \frac{du}{u}\right)^{\frac{1}{q}}
	\end{equation*}
	and
	\begin{equation*}
		t^k (1-\log t)^{b + \frac{1}{q} + \alpha -\frac{1}{r}} \left(\int_t^1 (u^{-k} (1-\log u)^{-\alpha} \omega_k(f,u)_p)^r \frac{du}{u}\right)^{\frac{1}{r}}
	\end{equation*}
	on the right-hand side of (\ref{5.5new+}).
\end{rem}

\subsection{Equivalences between  Besov and  Bianchini norms}
As application of Theorem \ref{thmKfunctionalBV} we can extend the Bianchini's characterization (\ref{1+}) for $\mathbf{B}^{0,0}_{1,1}(\mathbb{R}^d)$ to the full range of parameters of the scale $\mathbf{B}^{s,b}_{p,q}(\mathbb{R}^d)$. We will use Lipschitz-type spaces as well as Besov spaces with logarithmic smoothness.
We give a rather elementary proof based on estimates of the Peetre's $K$-functional.
It is different from that given by Bianchini for (\ref{1+}) (\cite[Proposition 4.2]{Bianchini}), which was relied on min max principle and the description of $\text{BV}(\mathbb{R}^d)$ as a dual space.

%


\begin{thm}\label{Thm 5.6new}
	Let $1 \leq p \leq \infty, 0 < q \leq \infty, s \geq 0$, and $-\infty < b < \infty$.
		\begin{enumerate}[\upshape(i)]
	\item Let $k \in \mathbb{N}, 0 < r \leq \infty$, and $\alpha - 1/r > 0 \, (\alpha \geq 0 \text{ if } r = \infty).$ The following characterization holds
	 \begin{multline}
	\|f\|_{\mathbf{B}^{s,b}_{p,q}(\mathbb{R}^d)}	
\\
\asymp \left(\int_0^1 t^{-\frac{s q}{k}} (1 - \log t)^{b q+ \frac{s q}{k} (\alpha -\frac{1}{r})} \inf_{g \in \emph{Lip}^{(k,-\alpha)}_{p,r}(\mathbb{R}^d)} \Big\{ \|f-g\|_{L_p(\mathbb{R}^d)} + t \|g\|_{\emph{Lip}^{(k,-\alpha)}_{p,r}(\mathbb{R}^d)} \Big\}^q \frac{dt}{t}\right)^{1/q} \label{Bianchini 4.2}
\end{multline}
if and only if $s < k$.

	\item Let $k \in \mathbb{N}, 0 < r \leq \infty$ and $\alpha > -1/r$. The following characterization holds
	 \begin{multline}
	\|f\|_{\mathbf{B}^{s,b}_{p,q}(\mathbb{R}^d)}
\asymp \left(\int_0^1 t^{-\frac{s q}{k}} (1 - \log t)^{b q- (1-\frac{s}{k}) (\alpha +\frac{1}{r})q} \right.
\\
\left.
\inf_{g \in W^{k}_{p}(\mathbb{R}^d)} \Big\{ \|f-g\|_{\mathbf{B}^{0,\alpha}_{p,r}(\mathbb{R}^d)} + t \|g\|_{W^k_p(\mathbb{R}^d)} \Big\}^q \frac{dt}{t}\right)^{1/q} \label{Bianchini 4.3new}
\end{multline}
if and only if one of the following conditions is satisfied
\begin{enumerate}[\upshape(a)]
\item
	$0 < s < k,$
\item
	$s=0$, and $\alpha + 1/r < b + 1/q$.
\end{enumerate}

\item Let $ k \in \mathbb{N}, 0 < r_1, r_2 \leq \infty, \alpha_1 > -1/r_1$, and $\alpha_2 > 1/r_2$. The following characterization holds

 \begin{multline*}
	\|f\|_{\mathbf{B}^{s,b}_{p,q}(\mathbb{R}^d)} 
\asymp \left(\int_0^1 t^{-\frac{s q}{k}} (1 - \log t)^{b q - (1-\frac{s}{k}) (\alpha_1 +\frac{1}{r_1})q+ \frac{s q}{k} (\alpha_2 -\frac{1}{r_2})} \right.
\\
\left.\inf_{g \in \emph{Lip}^{(k,-\alpha_2)}_{p,r_2}(\mathbb{R}^d)} \Big\{ \|f-g\|_{\mathbf{B}^{0,\alpha_1}_{p,r_1}(\mathbb{R}^d)} + t \|g\|_{\emph{Lip}^{(k,-\alpha_2)}_{p,r_2}(\mathbb{R}^d)} \Big\}^q \frac{dt}{t}\right)^{1/q}
 \end{multline*}
	if and only if one of the following conditions is satisfied
\begin{enumerate}[\upshape(a)]
\item
	$0 < s < k,$
\item
	$s=0$, and $\alpha_1 + 1/r_1 < b + 1/q$.
\end{enumerate}

\end{enumerate}

	The previous characterizations also hold true in the periodic setting.
\end{thm}
%
\begin{proof}
	(i)  By a change of variables and (\ref{5.3new+}), we have
	\begin{align}
		\int_0^1 t^{-\frac{s q}{k}} (1 - \log t)^{b q+ \frac{s q}{k} (\alpha -\frac{1}{r})} \inf_{g \in \text{Lip}^{(k,-\alpha)}_{p,r}(\mathbb{R}^d)} \Big\{ \|f-g\|_{L_p(\mathbb{R}^d)} + t \|g\|_{\text{Lip}^{(k,-\alpha)}_{p,r}(\mathbb{R}^d)} \Big\}^q \frac{dt}{t} & \nonumber \\
		&\hspace{-12 cm}=  \int_0^1 t^{-\frac{s q}{k}} (1 - \log t)^{b q+ \frac{s q}{k} (\alpha -\frac{1}{r})} K(t,f; L_p(\mathbb{R}^d), \text{Lip}^{(k,-\alpha)}_{p,r}(\mathbb{R}^d))^q \frac{dt}{t} \nonumber \\
		& \hspace{-12cm} \asymp \int_0^1 t^{-s q} (1 - \log t)^{b q} K(t^k (1-\log t)^{\alpha -1/r},f; L_p(\mathbb{R}^d), \text{Lip}^{(k,-\alpha)}_{p,r}(\mathbb{R}^d))^q \frac{dt}{t} \nonumber \\
		& \hspace{-12cm} \asymp I_1 + I_2 + I_3,\label{5.7new+}
	\end{align}
	where
	\begin{equation}\label{I}
	I_1 := \int_0^1 t^{(k-s) q} (1-\log t)^{(b+\alpha-1/r) q} \frac{dt}{t} \|f\|_{L_p(\mathbb{R}^d)}^q ,
	\end{equation}
	\begin{equation}\label{5.8new2}
		I_2 :=  \int_0^1 t^{-s q} (1-\log t)^{b q} \omega_k(f,t)_p^q \frac{dt}{t},
	\end{equation}
	and
	\begin{equation*}
		I_3:= \int_0^1 t^{-(s-k) q} (1-\log t)^{(b + \alpha -1/r) q} \left(\int_t^1 (u^{-k} (1-\log u)^{-\alpha} \omega_k(f,u)_p)^r \frac{du}{u}\right)^{q/r} \frac{dt}{t}.
	\end{equation*}
	
	Suppose that $s < k$. Then, $I_1 \asymp \|f\|_{L_p(\mathbb{R}^d)}^q$ and $I_2 \asymp  |f|_{\mathbf{B}^{s,b}_{p,q}(\mathbb{R}^d)}^q$. Further, we claim that
	\begin{equation}\label{5.9new2}
	I_3 \lesssim I_2.
	\end{equation}
	Once this has been proved, the proof of (\ref{Bianchini 4.2}) follows from (\ref{5.7new+}).
	
	It remains to show (\ref{5.9new2}). We distinguish two possible cases. First, assume $q \geq r$. Applying Hardy's inequality ($s < k$) given in (\ref{HardyInequal2}), we get
	\begin{equation*}
		I_3 \lesssim  \int_0^1 t^{-s q} (1-\log t)^{(b-1/r) q} \omega_k(f,t)_p^q \frac{dt}{t} \leq I_2.
	\end{equation*}
	Second, if $q < r$ then
	\begin{align*}
		I_3 & \asymp \sum_{\nu=0}^\infty 2^{\nu(s-k) q} (1 + \nu)^{(b+\alpha-1/r)q} \left(\sum_{\mu=0}^\nu 2^{\mu k r} (1 + \mu)^{-\alpha r} \omega_k(f,2^{-\mu})_p^r\right)^{q/r} \\
		& \leq  \sum_{\nu=0}^\infty 2^{\nu(s-k) q} (1 + \nu)^{(b+\alpha-1/r)q} \sum_{\mu=0}^\nu 2^{\mu k q} (1 + \mu)^{-\alpha q} \omega_k(f,2^{-\mu})_p^q \\
		& \asymp \sum_{\mu=0}^\infty 2^{\mu s q} (1 + \mu)^{(b-1/r) q} \omega_k(f,2^{-\mu})_p^q  \lesssim  \int_0^1 t^{-s q} (1-\log t)^{b q} \omega_k(f,t)_p^q \frac{dt}{t} \leq I_2.
	\end{align*}
	
	Conversely, let us prove that if (\ref{Bianchini 4.2}) holds then $s < k$. We will proceed by contradiction. Assume first that $s=k$. Note that if $b + \alpha -1/r + 1/q \geq 0$ then the right-hand side space of (\ref{Bianchini 4.2}) is trivial because
	\begin{equation*}
	  \inf_{g \in \text{Lip}^{(k,-\alpha)}_{p,r}(\mathbb{R}^d)} \left\{\frac{\|f-g\|_{L_p(\mathbb{R}^d)}}{t} + \|g\|_{\text{Lip}^{(k,-\alpha)}_{p,r}(\mathbb{R}^d)}\right\} \gtrsim \|f\|_{L_p(\mathbb{R}^d)}, \quad 0 < t < 1.
	 \end{equation*}
	 In case $b + \alpha -1/r + 1/q < 0$ we will invoke the sharp Jackson inequality (see \cite[Theorem 2.1]{Trebels} and \cite[(2.15)]{DaiDitzianTikhonov})
	 \begin{equation}\label{SharpJackson}
	 	\left(\int_t^1 (u^{-k} \omega_{k+m}(f,u)_p)^{\max\{p,2\}} \frac{du}{u}\right)^{1/\max\{p,2\}} \lesssim t^{-k} \omega_k(f,t)_p, \quad 0 < t < 1.
	 \end{equation}
	 Here, $k, m \in \mathbb{N}$ and $1 < p < \infty$. Choose $f \in L_p(\mathbb{R}^d)$ such that $\omega_{k+1}(f,t)_p \asymp t^k (1-\log t)^{-\varepsilon}$ where $b + 1/q < \varepsilon < b + 1/q + 1/\max\{p,2\}$, see \cite{Tikhonov-real}. Applying (\ref{SharpJackson}) (with $m=1$) we get $\omega_k(f,t)_p \gtrsim t^k (1-\log t)^{-\varepsilon + 1/\max\{p,2\}}$ because $\varepsilon < 1/\max\{p,2\}$ (noting that $b+1/q < 1/r -\alpha < 0$). Therefore, $f$ belongs to $\mathbf{B}^{k,b}_{p,q}(\mathbb{R}^d)$ but not to the right-hand side space of (\ref{Bianchini 4.2}) because the term $I_2$ defined in (\ref{5.8new2}) is not finite.
	
	 If $s > k$ then $I_1 = \infty$ (see (\ref{I})).
	 Hence, (\ref{5.7new+}) yields that the right-hand side space of (\ref{Bianchini 4.2}) is trivial.
	
	(ii) Using (\ref{5.4new+}) we have
	\begin{align}
		\int_0^1 t^{-\frac{s q}{k}} (1 - \log t)^{b q- (1-\frac{s}{k}) (\alpha +\frac{1}{r})q} \inf_{g \in W^{k}_{p}(\mathbb{R}^d)} \Big\{ \|f-g\|_{\mathbf{B}^{0,\alpha}_{p,r}(\mathbb{R}^d)} + t \|g\|_{W^k_p(\mathbb{R}^d)} \Big\}^q \frac{dt}{t} \nonumber \\
		&\hspace{-10cm} \asymp J_1 + J_2+ J_3, \label{J*}
	\end{align}
	where
	\begin{equation}\label{J}
		J_1 := \int_0^1 t^{(k-s) q} (1-\log t)^{b q} \frac{dt}{t} \|f\|_{L_p(\mathbb{R}^d)}^q,
	\end{equation}
	\begin{equation}\label{JJ}
		J_2:= \int_0^1 t^{-s q} (1-\log t)^{b q} \omega_k(f,t)_p^q \frac{dt}{t},
	\end{equation}
	and
	\begin{equation}\label{JJJ}
		J_3:= \int_0^1 t^{-s q} (1-\log t)^{(b - \alpha -1/r) q} \left(\int^t_0 ( (1-\log u)^{\alpha} \omega_k(f,u)_p)^r \frac{du}{u}\right)^{q/r} \frac{dt}{t}.
	\end{equation}
	
	Under the stated assumptions (a) the proof of (\ref{Bianchini 4.3new}) is similar to that given in (i) to show (\ref{Bianchini 4.2}). The rigorous details can be safely left to the reader. Assume now that (b) holds. Clearly, we have $J_1 \asymp \|f\|_{L_p(\mathbb{R}^d)}^q$ and $J_2 = |f|_{\mathbf{B}^{0,b}_{p,q}(\mathbb{R}^d)}$.
To complete the proof of (\ref{Bianchini 4.3new}),
it is enough to show that
		\begin{equation}\label{Bianchini 5.12+}
		J_3 \lesssim  \|f\|_{\mathbf{B}^{0,b}_{p,q}(\mathbb{R}^d)}^q.
	\end{equation}


	 If $r \leq q$,  (\ref{Bianchini 5.12+}) easily follows from the Hardy's inequality with logarithmic weights ($\alpha +1/r < b+1/q$) (see (\ref{HardyInequal3})). Suppose $r > q$. Then,
	\begin{align*}
		J_3 & \asymp \sum_{j=0}^\infty 2^{j (b-\alpha -1/r +1/q)q} \left(\sum_{\nu=j}^\infty 2^{\nu \alpha r} \omega_k(f,2^{-2^\nu})_p^r 2^{\nu}\right)^{q/r} \\
		& \leq  \sum_{j=0}^\infty 2^{j (b-\alpha -1/r +1/q)q} \sum_{\nu=j}^\infty 2^{\nu \alpha q} \omega_k(f,2^{-2^\nu})_p^q 2^{\nu q/r} \\
		& \asymp \sum_{\nu=0}^\infty  2^{\nu (b +1/q)q} \omega_k(f,2^{-2^\nu})_p^q  \asymp \|f\|_{\mathbf{B}^{0,b}_{p,q}(\mathbb{R}^d)}^q
	\end{align*}
	where we have used $\alpha +1/r < b+1/q$ in the penultimate estimate.
	
	Conversely, suppose that (\ref{Bianchini 4.3new}) holds. Then, $s < k$. Indeed, if $s > k$ then by (\ref{J}) and (\ref{J*}) the right-hand side space of (\ref{Bianchini 4.3new}) is trivial. This also happens if $s=k$ and $b + 1/q \geq 0$ because
	\begin{equation*}
		 \inf_{g \in W^{k}_{p}(\mathbb{R}^d)} \Big\{ \frac{\|f-g\|_{\mathbf{B}^{0,\alpha}_{p,r}(\mathbb{R}^d)}}{t} +  \|g\|_{W^k_p(\mathbb{R}^d)} \Big\} \gtrsim \|f\|_{\mathbf{B}^{0,\alpha}_{p,r}(\mathbb{R}^d)}, \quad 0 < t < 1.
	\end{equation*}
	On the other hand, if $s=k$ and $b+1/q < 0$ we may follow the same approach  as above, based on (\ref{SharpJackson}), to disprove (\ref{Bianchini 4.3new}). Finally, we show that (\ref{Bianchini 4.3new}) with $s=0$, that is,
	\begin{equation}\label{Bianchini 5.13}
	\|f\|_{\mathbf{B}^{0,b}_{p,q}(\mathbb{R}^d)} \asymp \left(\int_0^1 (1 - \log t)^{(b -\alpha -1/r)q} \inf_{g \in W^{k}_{p}(\mathbb{R}^d)} \Big\{ \|f-g\|_{\mathbf{B}^{0,\alpha}_{p,r}(\mathbb{R}^d)} + t \|g\|_{W^{k}_{p}(\mathbb{R}^d)} \Big\}^q \frac{dt}{t}\right)^{1/q}
	\end{equation}
	 is not true if $\alpha + 1/r \geq b + 1/q$. In light of (\ref{J*}), it suffices to construct a function $f \in \mathbf{B}^{0,b}_{p,q}(\mathbb{R}^d)$ such that $J_3 = \infty$, where $J_3$ is given by (\ref{JJJ}) with $s=0$. For instance, if $\alpha + 1/r > b + 1/q$, we put $f\in L_p(\mathbb{R}^d)$ with $\omega_k(f,t)_p \asymp (1-\log t)^{-\varepsilon}, \, 0 < b + 1/q < \varepsilon < \alpha + 1/r$. It is readily seen that $f \in  \mathbf{B}^{0,b}_{p,q}(\mathbb{R}^d)$, but
	 \begin{equation*}
	 	J_3 \gtrsim \left(\int_0^{1/2} ((1-\log t)^\alpha \omega_k(f,t)_p)^r \frac{dt}{t}\right)^{1/r} = \infty.
	 \end{equation*}
	 If $\alpha + 1/r = b + 1/q$ , consider $\omega_k(f,t)_p \asymp (1-\log t)^{-b-1/q} (1 + \log (1-\log t))^{-\delta}$ where $ \max\{1/q,1/r\} < \delta < 1/q + 1/r$.
	
	 The proof of (iii) is similar to that of (ii) (with the only difference that we use (\ref{5.5new+}) rather than (\ref{5.4new+})).
\end{proof}

\begin{cor}
	Let $0 \leq s < 1, 0 < q \leq \infty$, and $-\infty < b < \infty$. Then
	\begin{equation}\label{Bianchini 4.3}
	\|f\|_{\mathbf{B}^{s,b}_{1,q}(\mathbb{R}^d)} \asymp \left(\int_0^1  t^{-s q} (1 - \log t)^{b q} \inf_{g \in \emph{BV}(\mathbb{R}^d)} \Big\{\|f-g\|_{L_1(\mathbb{R}^d)} + t \|g\|_{\emph{BV}(\mathbb{R}^d)} \Big\}^q \frac{dt}{t}\right)^{1/q}
\end{equation}
and
	\begin{equation*}
	\|f\|_{\mathbf{B}^{s,b}_{\infty,q}(\mathbb{R}^d)} \asymp \left(\int_0^1  t^{-s q} (1 - \log t)^{b q} \inf_{g \in \emph{Lip}(\mathbb{R}^d)} \Big\{\|f-g\|_{L_\infty(\mathbb{R}^d)} + t \|g\|_{\emph{Lip}(\mathbb{R}^d)}  \Big\}^q \frac{dt}{t}\right)^{1/q}.
	\end{equation*}
\end{cor}
\subsection{Besov norm and functions of bounded p-variation}
In the rest of this section, we provide another  extension of (\ref{1+}) in the one dimensional setting given in terms of minimization problems for bounded $p$-variation functions. For $1 \leq p < \infty$, we denote by $V_p(\mathbb{R})$\index{\bigskip\textbf{Spaces}!$V_p(\mathbb{R})$}\label{pVAR} the set formed by all $p$-integrable functions $f : \mathbb{R} \longrightarrow \mathbb{R}$ such that there exists $c > 0$ such that
\begin{equation*}
	 \sum_{k=1}^N |f(x_k) - f(x_{k-1})|^p < c
\end{equation*}
for some $c>0$ and for all finite sequences $x_0 < x_1 < \cdots < x_N$. Set
\begin{equation*}
	\|f\|_{V_p(\mathbb{R})} = \|f\|_{L_p(\mathbb{R})} + \sup_{x_0 < x_1 < \cdots < x_N} \Big( \sum_{k=1}^N |f(x_k) - f(x_{k-1})|^p\Big)^{1/p}.
\end{equation*}
Analogously, $V_p(\mathbb{T})$\index{\bigskip\textbf{Spaces}!$V_p(\mathbb{T})$}\label{pVARPER} is formed by all periodic functions $f$ with the period 1 on the real line $\mathbb{R}$ for which
\begin{equation*}
	\|f\|_{V_p(\mathbb{T})} = \|f\|_{L_p(\mathbb{T})} + \sup_{\substack{x_0 < x_1 < \cdots < x_N \\ x_N = x_0 +1}} \Big( \sum_{k=1}^N |f(x_k) - f(x_{k-1})|^p\Big)^{1/p}
\end{equation*}
is finite.

As usual, we denote by $\dot{V}_p(\mathbb{R})$\index{\bigskip\textbf{Spaces}!$\dot{V}_p(\mathbb{R})$}\label{pVARHOM} (respectively, $\dot{V}_p(\mathbb{T})$\index{\bigskip\textbf{Spaces}!$\dot{V}_p(\mathbb{T})$}\label{pVARPERHOM}) the homogeneous counterpart of $V_p(\mathbb{R})$ (respectively, $V_p(\mathbb{T})$).

The study of the space $V_p(\mathbb{R})$, as well as its periodic counterpart $V_p(\mathbb{T})$,  has a long history. It
was initiated by Wiener \cite{Wiener} and nowadays it is of considerable importance in probability, harmonic analysis and PDEs. For further details and basic properties on spaces of $p$-bounded variation functions, as well as some applications, we refer to the monograph \cite{KochTataruVisan}.
If $p=1$, Hardy and Littlewood \cite{HardyLittlewood}  showed that an integrable function $f$ is equivalent to a function in $V_1(\mathbb{T})$ if and only if $f \in \text{BV}(\mathbb{T})$. For similar result regarding the spaces 
$ V_1(\mathbb{R})$ and $\text{BV}(\mathbb{R})$ see \cite[Theorem 3.27, p. 135]{AmbrosioFuscoPallara}.


Next we obtain two-sided estimates for the $K$-functional for the couple  $(L_p(\mathbb{R}), V_p(\mathbb{R}))$ in terms of classical  moduli of continuity $\omega_1(f,t)_p$ which complement those obtained in Corollary \ref{corKfunctionalBV} for $p=1$. Besides their intrinsic interest, these estimates will enable us to establish characterizations of Besov spaces through minimization problems for $V_p(\mathbb{R})$ spaces (cf. Theorem \ref{Theorem Bianchini 4.3} below).

\begin{thm}\label{BianchiniKfunctVp}
	
	\begin{enumerate}[\upshape(i)]
	\item Let $1< p < \infty$. Then,
	\begin{align*}
		\|f\|_{L_p(\mathbb{R})} + \sup_{t^p < u < 1} u^{-1/p} \omega_1(f,u)_p& \lesssim \frac{K(t, f; L_p(\mathbb{R}), V_p(\mathbb{R}))}{t} \\
		&\hspace{-4cm}\lesssim \|f\|_{L_p(\mathbb{R})} + \int_{t^p}^1 u^{-1/p} \omega_1(f,u)_p \frac{du}{u}
	\end{align*}
	for all $0 < t < 1$.
	\item Let $1 < p < \infty, 0 < q \leq \infty$, and $b > -1/q$. Then,
	\begin{align*}
	t (1-\log t)^{b+1/q} \|f\|_{L_p(\mathbb{R})} + t (1-\log t)^{b+1/q} \sup_{t^p < u < 1} u^{-1/p} \omega_1(f,u)_p & \\
		&\hspace{-7.5cm}+ \left(\int_0^{t^p} ((1-\log u)^b \omega_1(f,u)_p)^q \frac{du}{u}\right)^{1/q}\\
		& \hspace{-8.5cm} \lesssim K(t(1-\log t)^{b+1/q}, f; \mathbf{B}^{0,b}_{p,q}(\mathbb{R}), V_p(\mathbb{R})) \\
		& \hspace{-8.5cm}\lesssim t (1-\log t)^{b+1/q} \|f\|_{L_p(\mathbb{R})} + t (1-\log t)^{b+1/q} \int_{t^p}^1 u^{-1/p} \omega_1(f,u)_p \frac{du}{u} \\
		&\hspace{-7.5cm}+ \left(\int_0^{t^p} ((1-\log u)^b \omega_1(f,u)_p)^q \frac{du}{u}\right)^{1/q}
	\end{align*}
	for all $0 < t < 1$.
	
		\item Let $1 < p < \infty, 0 < q \leq \infty, 0 < s < 1/p$, and $b \in \mathbb{R}$. Then,
	\begin{align*}
	t^{1-s p} (1-\log t)^{b} \|f\|_{L_p(\mathbb{R})} + t^{1 - s p} (1-\log t)^{b} \sup_{t^p < u < 1} u^{-1/p} \omega_1(f,u)_p & \\
		&\hspace{-7.5cm}+ \left(\int_0^{t^p} (u^{-s} (1-\log u)^b \omega_1(f,u)_p)^q \frac{du}{u}\right)^{1/q}\\
		& \hspace{-8.5cm} \lesssim K(t^{1- s p}(1-\log t)^{b}, f; \mathbf{B}^{s,b}_{p,q}(\mathbb{R}), V_p(\mathbb{R})) \\
		& \hspace{-8.5cm}\lesssim t^{1-s p} (1-\log t)^{b} \|f\|_{L_p(\mathbb{R})} + t^{1 - s p} (1-\log t)^{b} \int_{t^p}^1 u^{-1/p} \omega_1(f,u)_p \frac{du}{u} \\
		&\hspace{-7.5cm}+ \left(\int_0^{t^p} (u^{-s} (1-\log u)^b \omega_1(f,u)_p)^q \frac{du}{u}\right)^{1/q}
	\end{align*}
	for all $0 < t < 1$.
	\end{enumerate}
	
		The corresponding estimates for periodic spaces also hold true.

\end{thm}

\begin{rem}
(i) Theorem \ref{BianchiniKfunctVp} can also be formulated (with obvious modifications) for the corresponding homogeneous spaces. In particular, we have
\begin{equation*}
		\sup_{t^p < u < 1} u^{-1/p} \omega_1(f,u)_p \lesssim \frac{K(t, f; L_p(\mathbb{R}), \dot{V}_p(\mathbb{R}))}{t} \lesssim \int_{t^p}^1 u^{-1/p} \omega_1(f,u)_p \frac{du}{u}, \quad 0 < t < 1.
	\end{equation*}
	Here, $1 < p < \infty$.
	
	(ii) Similar arguments to those given in Remark \ref{Remark 5.3} allow us to observe that both terms involving the moduli of continuity in part (ii) of Theorem \ref{BianchiniKfunctVp}, namely,
	\begin{equation*}
	 t (1-\log t)^{b+1/q} \sup_{t^p < u < 1} u^{-1/p} \omega_1(f,u)_p \quad \text{and} \quad t (1-\log t)^{b+1/q} \int_{t^p}^1 u^{-1/p} \omega_1(f,u)_p \frac{du}{u}
	\end{equation*}
  are independent of
	\begin{equation*}
		\left(\int_0^{t^p} ((1-\log u)^b \omega_1(f,u)_p)^q \frac{du}{u}\right)^{1/q}.
	\end{equation*}
A similar comment applies to Theorem \ref{BianchiniKfunctVp}(iii). More precisely, the terms
	\begin{equation*}
	 t^{1-sp} (1-\log t)^{b} \sup_{t^p < u < 1} u^{-1/p} \omega_1(f,u)_p \quad \text{and} \quad t^{1- s p} (1-\log t)^{b} \int_{t^p}^1 u^{-1/p} \omega_1(f,u)_p \frac{du}{u}
	\end{equation*}
	are independent of
	\begin{equation*}
		\left(\int_0^{t^p} (u^{-s} (1-\log u)^b \omega_1(f,u)_p)^q \frac{du}{u}\right)^{1/q}.
	\end{equation*}
\end{rem}

Before we prove Theorem \ref{BianchiniKfunctVp}, let us recall the relationships between $V_p(\mathbb{R})$ and Besov spaces due to Peetre \cite[Theorem 7, Chapter 5, p. 112]{Peetre} (cf. also \cite[Theorem 5]{BourdaudLanzaDeCristoforisSickel}, \cite{Terehin}, \cite{Terehin2}, and \cite{KolyadaLind}).

\begin{lem}\label{Peetre1}
	Let $1 \leq p < \infty$. Then, we have
	\begin{equation*}
		B^{1/p}_{p,1}(\mathbb{R}) \hookrightarrow V_p(\mathbb{R}) \hookrightarrow B^{1/p}_{p,\infty}(\mathbb{R}).
	\end{equation*}
	
	The corresponding embeddings for periodic spaces and homogeneous spaces also hold true.
\end{lem}

\begin{proof}[Proof of Theorem \ref{BianchiniKfunctVp}]
	(i) In view of (\ref{BesovComparison}) and (\ref{KFunctLpBp}), we have
	\begin{equation}
		K(t^{1/p}, f; L_p(\mathbb{R}), B^{1/p}_{p,1}(\mathbb{R})) \asymp t^{1/p} \left(\|f\|_{L_p(\mathbb{R})} + \int_t^1 u^{-1/p} \omega_1(f,u)_p \frac{du}{u}\right) \label{Bianchini 5.16}
	\end{equation}
	and
	\begin{equation}\label{Bianchini 5.17}
		K(t^{1/p}, f; L_p(\mathbb{R}), B^{1/p}_{p,\infty}(\mathbb{R})) \asymp  t^{1/p} \left(\|f\|_{L_p(\mathbb{R})} + \sup_{t < u < 1} u^{-1/p} \omega_1(f,u)_p  \right).
	\end{equation}
	On the other hand, by Lemma \ref{Peetre1} we derive
	\begin{equation*}
		K(t^{1/p}, f; L_p(\mathbb{R}), B^{1/p}_{p,\infty}(\mathbb{R})) \lesssim K(t^{1/p}, f; L_p(\mathbb{R}), V_p(\mathbb{R})) \lesssim K(t^{1/p}, f; L_p(\mathbb{R}), B^{1/p}_{p,1}(\mathbb{R})),
	\end{equation*}
	and thus, (\ref{Bianchini 5.16}) and (\ref{Bianchini 5.17}) yield the desired estimates.
	
	(ii) Firstly we claim that the following interpolation formula holds
\begin{equation}\label{Bianchini 4.5}
 \mathbf{B}^{0,b}_{p,q}(\mathbb{R}) = (L_p(\mathbb{R}), V_p(\mathbb{R}))_{(0,b),q}.
 \end{equation}
 Indeed, applying the embeddings given in Lemma \ref{Peetre1}, (\ref{BesovComparison}) and (\ref{PrelimInterpolationnew2}), we have
 \begin{align*}
 	 \mathbf{B}^{0,b}_{p,q}(\mathbb{R}) &= 	 (L_p(\mathbb{R}), B^{1/p}_{p,1}(\mathbb{R}))_{(0,b),q} \hookrightarrow (L_p(\mathbb{R}), V_p(\mathbb{R}))_{(0,b),q} \\
	 & \hookrightarrow (L_p(\mathbb{R}), B^{1/p}_{p,\infty}(\mathbb{R}))_{(0,b),q} = \mathbf{B}^{0,b}_{p,q}(\mathbb{R}).
 \end{align*}
This proves (\ref{Bianchini 4.5}).

 By (\ref{Bianchini 4.5}), Lemma \ref{PrelimHolmstedt}(iv), and the inequality obtained in part (i), we get
 \begin{align}
 	K(t(1-\log t)^{b+1/q}, f; \mathbf{B}^{0,b}_{p,q}(\mathbb{R}), V_p(\mathbb{R})) \nonumber \\
	& \hspace{-6cm}\asymp K(t(1-\log t)^{b+1/q}, f; (L_p(\mathbb{R}), V_p(\mathbb{R}))_{(0,b),q}, V_p(\mathbb{R}))  \nonumber  \\
	& \hspace{-6cm} \asymp (1-\log t)^{b+1/q} K(t,f; L_p(\mathbb{R}), V_p(\mathbb{R}))  \nonumber \\
	& \hspace{-5cm}+ \left(\int_0^t ((1-\log u)^b K(u,f; L_p(\mathbb{R}), V_p(\mathbb{R})))^q \frac{du}{u}\right)^{1/q}  \nonumber \\
	& \hspace{-6cm} \lesssim t (1 - \log t)^{b+1/q} \|f\|_{L_p(\mathbb{R})} + t (1 -\log t)^{b+1/q} \int_{t^p}^1 u^{-1/p} \omega_1(f,u)_p \frac{du}{u}  \nonumber \\
	& \hspace{-5cm} + \left(\int_0^t  \left(u (1-\log u)^b \int_u^t v^{-1} \omega_1(f,v^p)_p \frac{dv}{v}\right)^q \frac{du}{u}\right)^{1/q}. \label{5.19 Bianchini new}
 \end{align}
 Further, applying Hardy's inequality (\ref{HardyInequal2}) (note that it can also be applied if $q < 1$ taking into account the monotonicity properties of $\omega_1(f,v^p)_p$), we have
 \begin{align*}
 	 \left(\int_0^t  \left(u (1-\log u)^b \int_u^t v^{-1} \omega_1(f,v^p)_p \frac{dv}{v}\right)^q \frac{du}{u}\right)^{1/q} \\
	 & \hspace{-6cm}\lesssim \left(\int_0^{t^p} ((1-\log u)^b \omega_1(f,u)_p)^q \frac{du}{u}\right)^{1/q}.
 \end{align*}
 Plugging this estimate into (\ref{5.19 Bianchini new}), we obtain
 \begin{align*}
 	K(t(1-\log t)^{b+1/q}, f; \mathbf{B}^{0,b}_{p,q}(\mathbb{R}), V_p(\mathbb{R})) & \lesssim t (1 - \log t)^{b+1/q} \|f\|_{L_p(\mathbb{R})}  \\
	& \hspace{-6cm}+ t (1 -\log t)^{b+1/q} \int_{t^p}^1 u^{-1/p} \omega_1(f,u)_p \frac{du}{u} +  \left(\int_0^{t^p} ((1-\log u)^b \omega_1(f,u)_p)^q \frac{du}{u}\right)^{1/q}.
 \end{align*}
A similar argument gives the proof of the first estimate of (ii).

(iii) Let $\theta = s p$. Then, we have
\begin{equation}\label{Bianchini 4.5new}
 \mathbf{B}^{s,b}_{p,q}(\mathbb{R}) = (L_p(\mathbb{R}), V_p(\mathbb{R}))_{\theta,q;b}.
 \end{equation}
 The proof of (\ref{Bianchini 4.5new}) can be carried out in a similar way as the proof of (\ref{Bianchini 4.5}) but now applying (\ref{PrelimInterpolationnew2.2}) rather than (\ref{PrelimInterpolationnew2}).

The estimates given in (iii) can be obtained by using similar arguments as in the proof of (ii) with the help of (\ref{Bianchini 4.5new}) and Lemma \ref{PrelimHolmstedt}(ii).
\end{proof}

\begin{rem}
	The estimates obtained in Theorem \ref{BianchiniKfunctVp}(i) are optimal, in the sense that,
	\begin{equation}\label{5.18Bianchini2}
		\frac{K(t, f; L_p(\mathbb{R}), V_p(\mathbb{R}))}{t} \lesssim \|f\|_{L_p(\mathbb{R})} + \left(\int_{t^p}^1 (u^{-1/p} \omega_1(f,u)_p)^q \frac{du}{u}\right)^{1/q} \iff 0 < q \leq 1
	\end{equation}
	and
	\begin{equation}\label{5.18Bianchini2*}
		\|f\|_{L_p(\mathbb{R})} + \left(\int_{t^p}^1 (u^{-1/p} \omega_1(f,u)_p)^q \frac{du}{u} \right)^{1/q} \lesssim \frac{K(t, f; L_p(\mathbb{R}), V_p(\mathbb{R}))}{t} \iff q = \infty.
	\end{equation}
	
	To prove (\ref{5.18Bianchini2}), we will make use of the following well-known results.
	
	\begin{lem}[\bf{$K$-functional for $(L_p(\mathbb{R}^d), L_\infty(\mathbb{R}^d))$}; \cite{BerghLofstrom}]\label{5.19Bianchini}
		Let $0 < p < \infty$. Then, we have
		\begin{equation*}
		K(t, f; L_p(\mathbb{R}^d), L_\infty(\mathbb{R}^d)) \asymp \left(\int_0^{t^p} f^\ast(s)^p ds \right)^{1/p}
	\end{equation*}
	where the equivalence constants depend only on $p$ and $d$. As usual, $f^\ast$ is the non-increasing rearrangement of
$f$.
	\end{lem}
	
	\begin{lem}[\bf{Characterization of embeddings of Besov spaces into $L_\infty(\mathbb{R}^d)$}; \cite{SickelTriebel}]\label{LemBianchiniContinuous}
		Let $s \in \mathbb{R}, 0 < p \leq \infty$ and $0 < q \leq \infty$. Then, the following statements are equivalent:
		\begin{enumerate}[\upshape(i)]
		\item $B^{s}_{p,q}(\mathbb{R}^d) \hookrightarrow L_\infty(\mathbb{R}^d)$,
		\item either $s > d/p$ or $s=d/p$ and $q \leq 1$.
		\end{enumerate}
	\end{lem}
	
	Assume that
	\begin{equation*}
		\frac{K(t, f; L_p(\mathbb{R}), V_p(\mathbb{R}))}{t} \lesssim \|f\|_{L_p(\mathbb{R})} + \left(\int_{t^p}^1 (u^{-1/p} \omega_1(f,u)_p)^q \frac{du}{u}\right)^{1/q}.
	\end{equation*}
	In particular, since $V_p(\mathbb{R}) \hookrightarrow L_\infty(\mathbb{R})$, we get
		\begin{equation*}
		\frac{K(t, f; L_p(\mathbb{R}), L_\infty(\mathbb{R}))}{t} \lesssim \|f\|_{L_p(\mathbb{R})} + \left(\int_{t^p}^1 (u^{-1/p} \omega_1(f,u)_p)^q \frac{du}{u}\right)^{1/q}.
	\end{equation*}
	By Lemma \ref{5.19Bianchini}, we can rewrite the previous estimate as
	\begin{equation}\label{5.20Bianchini}
		\left(\frac{1}{t^p} \int_0^{t^p} f^\ast(s)^p ds \right)^{1/p} \lesssim  \|f\|_{L_p(\mathbb{R})} + \left(\int_{t^p}^1 (u^{-1/p} \omega_1(f,u)_p)^q \frac{du}{u}\right)^{1/q}.
	\end{equation}
	Passing to the limit $t \downarrow 0$, (\ref{5.20Bianchini}) shows that
	\begin{equation*}
		\|f\|_{L_\infty(\mathbb{R})} \asymp f^\ast(0) \lesssim \|f\|_{\mathbf{B}^{1/p}_{p,q}(\mathbb{R})} \asymp  \|f\|_{B^{1/p}_{p,q}(\mathbb{R})}
	\end{equation*}
	where we used (\ref{BesovComparison}). Therefore, Lemma \ref{LemBianchiniContinuous} implies that $q \leq 1$.
	
	The converse implication of (\ref{5.18Bianchini2}) follows easily from Theorem \ref{BianchiniKfunctVp}.
	
	Suppose now that the inequality stated in (\ref{5.18Bianchini2*}) holds for some $q < \infty$. In particular, we have
	\begin{equation*}
		\|f\|_{L_p(\mathbb{R})} + \left(\int_{t^p}^1 (u^{-1/p} \omega_1(f,u)_p)^q \frac{du}{u} \right)^{1/q} \lesssim \|f\|_{V_p(\mathbb{T})}, \quad 0 < t < 1,
	\end{equation*}
	which yields that $V_p(\mathbb{R}) \hookrightarrow \mathbf{B}^{1/p}_{p,q}(\mathbb{R})$. On the other hand, we have $L_p(\mathbb{R}) \cap \mathbf{B}^{1/p}_{\infty,\infty}(\mathbb{R}) \hookrightarrow V_p(\mathbb{R})$ because
	\begin{equation*}
		\|f\|_{\mathbf{B}^{1/p}_{\infty,\infty}(\mathbb{R}) } \asymp \|f\|_{L_\infty(\mathbb{R})} + \sup_{|x-y| \leq 1} \frac{|f(x)-f(y)|}{|x-y|^{1/p}},
	\end{equation*}
	see, e.g., \cite[Theorem 2.5.12, p. 110]{Triebel1}. Consequently, we derive
	\begin{equation*}
		L_p(\mathbb{R}) \cap \mathbf{B}^{1/p}_{\infty,\infty}(\mathbb{R}) \hookrightarrow   \mathbf{B}^{1/p}_{p,q}(\mathbb{R}).
	\end{equation*}
	However, this embedding is not true. Take the lacunary Fourier series
	\begin{equation*}
		f(x) \sim \sum_{j=3}^\infty 2^{-j/p} e^{i(2^j-2)x_1}\psi(x), \quad x = (x_1, \ldots, x_d) \in \mathbb{R}^d,
	\end{equation*}
	where $\psi \in \mathcal{S}(\mathbb{R}^d) \backslash \{0\}$ with $ \text{supp } \widehat{\psi} \subset \{\xi : |\xi| \leq 2\}$. Using the characterization given in Proposition \ref{Proposition 4.1} below and (\ref{BesovComparison}), one can see that $f \in L_p(\mathbb{R}) \cap \mathbf{B}^{1/p}_{\infty,\infty}(\mathbb{R})$ but $f \not \in \mathbf{B}^{1/p}_{p,q}(\mathbb{R})$ because $q \neq \infty$. This proves the sharpness assertion (\ref{5.18Bianchini2*}).
	
\end{rem}

Next we establish the counterpart of Theorem \ref{Thm 5.6new} for the spaces $V_p(\mathbb{R})$.

\begin{thm}\label{Theorem Bianchini 4.3}
	Let $1 < p < \infty, 0 < q < \infty, s \geq 0,$ and $-\infty < b < \infty$.
	\begin{enumerate}[\upshape(i)]
	 \item The following characterization holds
	\begin{equation}\label{Bianchini 4.4}
	\|f\|_{\mathbf{B}^{s,b}_{p,q}(\mathbb{R})} \asymp \left(\int_0^1 t^{-s p q}(1 - \log t)^{b q} \inf_{g \in V_p(\mathbb{R})} \Big\{\|f-g\|_{L_p(\mathbb{R})}  + t \|g\|_{V_{p}(\mathbb{R})} \Big\}^q \frac{dt}{t}\right)^{1/q}
\end{equation}
if and only if $s < 1/p$.
\item Let $0 < r \leq \infty$, and $\alpha > -1/r$. The following characterization holds
	\begin{equation}\label{5.23 Bianchini}
	\|f\|_{\mathbf{B}^{s,b}_{p,q}(\mathbb{R})} \asymp \left(\int_0^1 t^{-s p q}(1 - \log t)^{b q-(1-s p)(\alpha + \frac{1}{r})q} \inf_{g \in V_p(\mathbb{R})} \Big\{\|f-g\|_{\mathbf{B}^{0,\alpha}_{p,r}(\mathbb{R})}  + t \|g\|_{V_{p}(\mathbb{R})} \Big\}^q \frac{dt}{t}\right)^{1/q}
\end{equation}
if and only if
\begin{enumerate}[\upshape(a)]
\item
	$0 < s < 1/p$
\item
	$s=0$, and $ \alpha +1/r < b+1/q$.
\end{enumerate}
\item Let $0 < r \leq \infty, 0 < \lambda < 1/p$, and $-\infty < \alpha < \infty$. The following characterization holds
	\begin{equation}\label{5.23 Bianchini*}
	\|f\|_{\mathbf{B}^{s,b}_{p,q}(\mathbb{R})} \asymp \left(\int_0^1 t^{-\frac{(s - \lambda)p q}{1 - \lambda p}}(1 - \log t)^{b q- \frac{1 - sp}{1 - \lambda p} \alpha q} \inf_{g \in V_p(\mathbb{R})} \Big\{\|f-g\|_{\mathbf{B}^{\lambda,\alpha}_{p,r}(\mathbb{R})}  + t \|g\|_{V_{p}(\mathbb{R})} \Big\}^q \frac{dt}{t}\right)^{1/q}
\end{equation}
if and only if $ \lambda < s < 1/p$.

\end{enumerate}

The previous characterizations also hold true in the periodic setting.
\end{thm}


\begin{rem}
(i) The formula (\ref{Bianchini 4.4}) can be considered as an extension of (\ref{Bianchini 4.3}) for $p > 1$. Recall that $\text{BV}(\mathbb{R}) = V_1(\mathbb{R})$.

(ii) All if-parts of Theorem \ref{Theorem Bianchini 4.3} still hold true in the limiting case $q=\infty$.

(iii) The condition $\alpha > -1/r$ in Theorem \ref{Theorem Bianchini 4.3}(ii) is necessary in order to get (\ref{5.23 Bianchini}). To be more precise, one can show that if $\alpha = -1/r$, then the corresponding characterization to (\ref{5.23 Bianchini}) involves iterated logarithmic weights.
\end{rem}

	\begin{proof}[Proof of Theorem \ref{Theorem Bianchini 4.3}]
	(i) The proof of (\ref{Bianchini 4.4}) under the condition $s < 1/p$ is a simple consequence of Theorem \ref{BianchiniKfunctVp}(i).
	
	Next we show that the formula (\ref{Bianchini 4.4}) does not hold for $s \geq 1/p$. Assume first that either $s > 1/p$ or $s=1/p$ and $b \geq -1/q$. Then,
	\begin{align*}
		\int_0^1 t^{(1-s p) q}(1 - \log t)^{b q} \inf_{g \in V_p(\mathbb{R})} \Big\{\frac{\|f-g\|_{L_p(\mathbb{R})}}{t}  +  \|g\|_{V_{p}(\mathbb{R})} \Big\}^q \frac{dt}{t} & \\
		& \hspace{-8cm}\geq  \inf_{g \in V_p(\mathbb{R})} \Big\{ \|f-g\|_{L_p(\mathbb{R})}  +  \|g\|_{V_{p}(\mathbb{R})}  \Big\}^q \int_0^1 t^{(1-s p) q}(1 - \log t)^{b q} \frac{dt}{t} \\
		& \hspace{-8cm} \gtrsim \|f\|_{L_p(\mathbb{R})}^q \int_0^1 t^{(1-s p) q}(1 - \log t)^{b q} \frac{dt}{t} = \infty,
	\end{align*}
	which implies that the right-hand side space of  (\ref{Bianchini 4.4}) is trivial. Consequently, (\ref{Bianchini 4.4}) is not true.
	
	Suppose now that $s=1/p$ and $b < -1/q$. First we observe that the right-hand side space of (\ref{Bianchini 4.4}) is not trivial. Indeed, it contains the space $V_p(\mathbb{R})$ because
	\begin{align}
		\int_0^1 t^{-q}(1 - \log t)^{b q} \inf_{g \in V_p(\mathbb{R})} \Big\{\|f-g\|_{L_p(\mathbb{R})}  + t \|g\|_{V_{p}(\mathbb{R})} \Big\}^q \frac{dt}{t} & \label{5.24 Bianchini new}\\
		& \hspace{-6cm} \leq \|f\|_{V_p(\mathbb{R})}^q \int_0^1 (1 - \log t)^{b q} \frac{dt}{t} \lesssim  \|f\|_{V_p(\mathbb{R})}^q \nonumber.
	\end{align}
 To disprove (\ref{Bianchini 4.4}), we will proceed by contradiction. Assume that
		\begin{equation}\label{5.26 Bianchini}
	\|f\|_{\mathbf{B}^{1/p,b}_{p,q}(\mathbb{R})} \asymp \left(\int_0^1 t^{- q}(1 - \log t)^{b q} \inf_{g \in V_p(\mathbb{R})} \Big\{\|f-g\|_{L_p(\mathbb{R})}  + t \|g\|_{V_{p}(\mathbb{R})} \Big\}^q \frac{dt}{t}\right)^{1/q},
\end{equation}
that is, $\mathbf{B}^{1/p,b}_{p,q}(\mathbb{R}) = (L_p(\mathbb{R}), V_{p}(\mathbb{R}))_{(1,b),q}$. Then, for $0 < \theta < 1$, using (\ref{PrelimInterpolationnew2.2}), Lemma \ref{PrelimLemma7.2}(ii) ($b < -1/q$) and (\ref{Bianchini 4.5new}), we have
\begin{align*}
	\mathbf{B}^{\theta/p, \theta b}_{p,q}(\mathbb{R}) & = (L_p(\mathbb{R}), \mathbf{B}^{1/p,  b}_{p,q}(\mathbb{R}))_{\theta, q} = (L_p(\mathbb{R}), (L_p(\mathbb{R}), V_{p}(\mathbb{R}))_{(1,b),q})_{\theta, q} \\
	& = (L_p(\mathbb{R}), V_p(\mathbb{R}))_{\theta,q;\theta (b+1/q)} = \mathbf{B}^{\theta/p, \theta (b + 1/q)}_{p,q}(\mathbb{R})
\end{align*}
which implies that $q=\infty$, a contradiction. Hence, (\ref{Bianchini 4.4}) does not hold when $s=1/p$ and $b < -1/q$.

	(ii) Assume that either (a) or (b) holds. Then, (\ref{5.23 Bianchini}) is a simple application of Theorem \ref{BianchiniKfunctVp}(ii).
	
	Suppose that (\ref{5.23 Bianchini}) holds. Then, a similar argument to that given in the proof of (i) but now replacing $L_p(\mathbb{R})$ by $\mathbf{B}^{0,\alpha}_{p,r}(\mathbb{R})$ leads to a contradiction, either to $s > 1/p$ or to $s=1/p$ and $b \geq -1/q$. Suppose that $s=1/p$ and $b < -1/q$. Hence, (\ref{5.23 Bianchini}) is equivalent to the fact that $\mathbf{B}^{1/p,b}_{p,q}(\mathbb{R}) = (\mathbf{B}^{0,\alpha}_{p,r}(\mathbb{R}), V_p(\mathbb{R}))_{(1,b),q}$. Let $\theta \in (0,1)$. Then, by (\ref{PrelimInterpolationnew3}) and Lemma \ref{PrelimLemma7.2}(ii) ($b < -1/q$), we have
	\begin{align}
		\mathbf{B}^{\theta/p, (1-\theta)(\alpha + 1/r) + \theta b}_{p,q}(\mathbb{R}) & = (\mathbf{B}^{0,\alpha}_{p,r}(\mathbb{R}), \mathbf{B}^{1/p,b}_{p,q}(\mathbb{R}))_{\theta, q}\nonumber \\
		& \hspace{-3cm} = (\mathbf{B}^{0,\alpha}_{p,r}(\mathbb{R}), (\mathbf{B}^{0,\alpha}_{p,r}(\mathbb{R}), V_p(\mathbb{R}))_{(1,b),q})_{\theta, q} = (\mathbf{B}^{0,\alpha}_{p,r}(\mathbb{R}), V_p(\mathbb{R}))_{\theta, q; \theta (b+1/q)} . \label{aux5.40}
	\end{align}
	Further, it follows from Lemma \ref{Peetre1}, (\ref{BesovComparison}), and (\ref{PrelimInterpolationnew3}) that
	\begin{align}
		(\mathbf{B}^{0,\alpha}_{p,r}(\mathbb{R}), V_p(\mathbb{R}))_{\theta, q; \theta (b+1/q)} & \hookrightarrow (\mathbf{B}^{0,\alpha}_{p,r}(\mathbb{R}), \mathbf{B}^{1/p}_{p,\infty}(\mathbb{R}))_{\theta, q; \theta (b+1/q)} \nonumber\\
		&  = \mathbf{B}^{\theta/p, (1-\theta)(\alpha + 1/r) + \theta (b+1/q)}_{p,q}(\mathbb{R}). \label{aux5.42}
	\end{align}
	Therefore, combining (\ref{aux5.40}) and (\ref{aux5.42}) we arrive at
	\begin{equation*}
	\mathbf{B}^{\theta/p, (1-\theta)(\alpha + 1/r) + \theta b}_{p,q}(\mathbb{R}) \hookrightarrow \mathbf{B}^{\theta/p, (1-\theta)(\alpha + 1/r) + \theta (b+1/q)}_{p,q}(\mathbb{R}).
	\end{equation*}
	This is a contradiction because $q < \infty$ and $\theta > 0$.

	It remains to disprove that  (\ref{5.23 Bianchini}) does not hold for the case $s=0$ and $\alpha +1/r \geq b+1/q$. According to Theorem \ref{BianchiniKfunctVp}(ii), we have
	\begin{equation*}
		 \int_0^1 (1 - \log t)^{b q-(\alpha + \frac{1}{r})q} \inf_{g \in V_p(\mathbb{R})} \Big\{\|f-g\|_{\mathbf{B}^{0,\alpha}_{p,r}(\mathbb{R})}  + t \|g\|_{V_{p}(\mathbb{R})} \Big\}^q \frac{dt}{t} \gtrsim J_3,
	\end{equation*}
	where $J_3$ is defined by (\ref{JJJ}) with $s=0$ and $k=1$. Therefore, we can construct a function $f \in \mathbf{B}^{0,b}_{p,q}(\mathbb{R})$ satisfying $J_3 = \infty$ by repeating the arguments given in the proof of Theorem \ref{Thm 5.6new}(ii) \emph{mutatis mutandis}.
	
	(iii) Let  $\lambda < s < 1/p$. The proof of (\ref{5.23 Bianchini*}) can be easily derived from Theorem \ref{BianchiniKfunctVp}(iii). Further details are left to the reader.
	
	Conversely, suppose that (\ref{5.23 Bianchini*}) holds. Then, we have $s < 1/p$. Indeed, if either $s > 1/p$ or $s=1/p$ and $b \geq -1/q$ holds then we arrive at a contradiction applying the same argument as in the proof of (i), but now $L_p(\mathbb{R})$ should be replaced by $\mathbf{B}^{\lambda,\alpha}_{p,r}(\mathbb{R})$. Suppose that $s=1/p$ and $b < -1/q$. Hence, (\ref{5.23 Bianchini*}) can be rewritten as $\mathbf{B}^{1/p,b}_{p,q}(\mathbb{R}) = (\mathbf{B}^{\lambda,\alpha}_{p,r}(\mathbb{R}), V_p(\mathbb{R}))_{(1,b),q}$. Let $\theta \in (0,1)$. Using Lemma \ref{PrelimLemma7.2}(ii) ($b < -1/q$), we have
	\begin{align}
		(\mathbf{B}^{\lambda,\alpha}_{p,r}(\mathbb{R}), \mathbf{B}^{1/p,b}_{p,q}(\mathbb{R}))_{\theta, q} &= (\mathbf{B}^{\lambda,\alpha}_{p,r}(\mathbb{R}), (\mathbf{B}^{\lambda,\alpha}_{p,r}(\mathbb{R}), V_p(\mathbb{R}))_{(1,b),q})_{\theta, q} \nonumber \\
		& = (\mathbf{B}^{\lambda,\alpha}_{p,r}(\mathbb{R}), V_p(\mathbb{R}))_{\theta, q; \theta (b+1/q)} . \label{aux5.40*}
	\end{align}
	By Lemma \ref{PrelimLemmaCF} and (\ref{BesovComparison}), we get
	\begin{equation}\label{aux5.41*}
		(\mathbf{B}^{\lambda,\alpha}_{p,r}(\mathbb{R}), \mathbf{B}^{1/p,b}_{p,q}(\mathbb{R}))_{\theta, q}    = \mathbf{B}^{(1-\theta) \lambda + \theta/p, (1-\theta)\alpha + \theta  b}_{p,q}(\mathbb{R}).
	\end{equation}
	On the other hand, according to Lemma \ref{Peetre1} and (\ref{BesovComparison}) we have
	\begin{align}
		(\mathbf{B}^{\lambda,\alpha}_{p,r}(\mathbb{R}), V_p(\mathbb{R}))_{\theta, q; \theta (b+1/q)} & \hookrightarrow (\mathbf{B}^{\lambda,\alpha}_{p,r}(\mathbb{R}), \mathbf{B}^{1/p}_{p,\infty}(\mathbb{R}))_{\theta, q; \theta (b+1/q)}\nonumber \\
		& = \mathbf{B}^{(1-\theta) \lambda + \theta/p, (1-\theta)\alpha + \theta (b+1/q)}_{p,q}(\mathbb{R}) \label{aux5.42*}
	\end{align}
	where we have also used Lemma \ref{PrelimLemmaCF} in the last estimate.
	Therefore, combining (\ref{aux5.40*})-(\ref{aux5.42*}) we obtain
	\begin{equation*}
	\mathbf{B}^{(1-\theta) \lambda + \theta/p, (1-\theta)\alpha + \theta  b}_{p,q}(\mathbb{R}) \hookrightarrow  \mathbf{B}^{(1-\theta) \lambda + \theta/p, (1-\theta)\alpha + \theta (b+1/q)}_{p,q}(\mathbb{R})
	\end{equation*}
	which is not true because $q < \infty$ and $\theta > 0$.

 In order to show that $s > \lambda$ is necessary we will again argue by contradiction, that is, we assume that if (\ref{5.23 Bianchini*}) is true for some $s \leq \lambda$ then we arrive at a contradiction. If either $s < \lambda$ or $s=\lambda$ and $b-\alpha < -1/q$, then the right-hand side of (\ref{5.23 Bianchini*}) can be estimated from above by
	\begin{equation*}
		\left(\int_0^1 t^{-\frac{(s - \lambda)p q}{1 - \lambda p}}(1 - \log t)^{b q- \frac{1 - sp}{1 - \lambda p} \alpha q}  \frac{dt}{t}\right)^{1/q} \|f\|_{\mathbf{B}^{\lambda,\alpha}_{p,r}(\mathbb{R})} \asymp  \|f\|_{\mathbf{B}^{\lambda,\alpha}_{p,r}(\mathbb{R})}
	\end{equation*}
	and, in addition, from below by (note that $s < 1/p$)
		\begin{equation*}
		\left(\int_0^1 t^{q -\frac{(s - \lambda)p q}{1 - \lambda p}}(1 - \log t)^{b q- \frac{1 - sp}{1 - \lambda p} \alpha q}  \frac{dt}{t}\right)^{1/q} \|f\|_{\mathbf{B}^{\lambda,\alpha}_{p,r}(\mathbb{R})} \asymp  \|f\|_{\mathbf{B}^{\lambda,\alpha}_{p,r}(\mathbb{R})}.
	\end{equation*}
	 Hence, by (\ref{5.23 Bianchini*}), we derive that $\mathbf{B}^{s,b}_{p,q}(\mathbb{R}) = \mathbf{B}^{\lambda,\alpha}_{p,r}(\mathbb{R})$, which is not possible (see Proposition \ref{RecallEmb**optim} below). Next we suppose that $s=\lambda$ and $b - \alpha \geq -1/q$. Then, by Theorem \ref{BianchiniKfunctVp}(iii) (recall that $0 < \lambda < 1/p$), we obtain the following estimate of the right-hand side of (\ref{5.23 Bianchini*})
	\begin{align}
		 \left(\int_0^1 (1 - \log t)^{(b - \alpha) q} \inf_{g \in V_p(\mathbb{R})} \Big\{\|f-g\|_{\mathbf{B}^{\lambda,\alpha}_{p,r}(\mathbb{R})}  + t \|g\|_{V_{p}(\mathbb{R})} \Big\}^q \frac{dt}{t}\right)^{1/q} & \nonumber \\
		 & \hspace{-9cm}\gtrsim \left(\int_0^1 (1 - \log t)^{(b - \alpha) q} \left(\int_0^t (u^{-\lambda} (1-\log u)^\alpha \omega_1(f,u)_p)^r \frac{du}{u}\right)^{q/r}\right)^{1/q}. \label{5.40}
	\end{align}
	Now we can consider a function $f \in L_p(\mathbb{R})$ such that $f \in \mathbf{B}^{\lambda,b}_{p,q}(\mathbb{R})$ but (\ref{5.40}) is not finite and thus we arrive at a contradiction. In fact, if $b - \alpha > -1/q$, then we take $f \in L_p(\mathbb{R})$ satisfying $\omega_1(f,t)_p \asymp t^\lambda (1-\log t)^{-\varepsilon}$ where $\max\{b+1/q, \alpha + 1/r\} < \varepsilon < b + 1/q + 1/r$. If $b - \alpha = -1/q$, we consider $f \in L_p(\mathbb{R})$ such that $\omega_1(f,t)_p \asymp t^\lambda (1-\log t)^{-(b+1/q+1/r)} (1 + \log (1 - \log t))^{-\delta}$, where $1/r < \delta < 1/r + 1/q$.
	
	\end{proof}

\newpage
\section{Functions and their derivatives in Besov spaces}\label{section9}

 This section deals with the relationships between the smoothness properties of functions  and their derivatives.
  More specifically, in the first part of this section we will establish sharp inequalities between moduli of smoothness of the derivatives of a function and those of  function itself. As application, in Section \ref{Section: Sobolev norms} we will shed some lights on Sobolev-type characterizations of Besov spaces.

 \subsection{Inequalities for moduli of smoothness of functions and their derivatives}

The following estimate of the moduli of smoothness for derivatives of $f$ was obtained  in \cite[Theorem 2.3]{Trebels2}.

 \begin{prop}\label{Prop Trebels}
 	Let $1 < p < \infty$ and $k, m \in \mathbb{N}$. Let $f \in L_p(\mathbb{R}^d)$. Then,
	\begin{equation}\label{inequalTrebels}
		\omega_k(D^\beta f, t)_p \lesssim \left(\int_0^t (u^{-m} \omega_{k+m}(f,u)_p)^{\min\{p,2\}} \frac{du}{u}\right)^{1/\min\{p,2\}}
	\end{equation}
	for all $\beta \in \mathbb{N}_0^d$ with $|\beta| =m$.
 \end{prop}
More delicate estimates involving the mixed derivatives were derived in \cite{pst, Potapov}. For the case $0<p\le\infty$ see also \cite{DitzianTikhonov}.
 Note that (\ref{inequalTrebels}) consists of a refinement of the well-known Marchaud inequality (see \cite[Chapter 6, Theorem 3.1, p. 178]{DeVoreLorentz})
 \begin{equation}\label{MarchaudClassic}
		\omega_k(D^\beta f, t)_p \lesssim \int_0^t u^{-m} \omega_{k+m}(f,u)_p \frac{du}{u}, \quad 1 \leq p \leq \infty.
	\end{equation}
For reasons of completeness we will give the proof of (\ref{inequalTrebels}) in Theorem \ref{Thm Trebels2} below following \cite{Trebels2}.
 Furthermore, we will also prove that (\ref{inequalTrebels}) is optimal in the sense of integrability parameter $\min\{p,2\}$.

  Concerning converse estimates of (\ref{inequalTrebels}), we mention the 
  classical inequality
  \begin{equation}\label{convTrebels}
  	 t^{-m} \omega_{k+m}(f,t)_p \lesssim \sup_{|\beta| =m} \omega_k(D^\beta f, t)_p, \quad 1 \leq p \leq \infty.
  \end{equation}
  See, e.g., \cite[(7.13), p. 46]{DeVoreLorentz}.
  For sharper results see \cite[p. 1398]{Simonov} (for functions from $L_p(\T)$) and \cite{pst, Potapov} for mixed derivatives and moduli of smoothness in $L_p(\mathbb{T}^d)$.
  In Theorem \ref{Thm Trebels2}, we obtain the sharp version of (\ref{convTrebels}) which complements the inequality  (\ref{inequalTrebels}).

 \begin{thm}\label{Thm Trebels2}
 	Let $1 < p < \infty, 0 < q \leq \infty$ and $k, m \in \mathbb{N}$.
	\begin{enumerate}[\upshape(i)]
	\item Let $f \in L_p(\mathbb{R}^d)$. Then,
	\begin{equation*}
		\sup_{|\beta| =m} \omega_k(D^\beta f, t)_p \lesssim \left(\int_0^t (u^{-m} \omega_{k+m}(f,u)_p)^{q} \frac{du}{u}\right)^{1/q} \iff q \leq \min\{p,2\}.
	\end{equation*}
	\item Let $f$ be such that $\frac{\partial^m f}{\partial x_j^m} \in L_p(\mathbb{R}^d)$ for $j=1, \ldots, d$. Then,
	\begin{equation}\label{inequalTrebels2}
		  \left(\int_0^t (u^{-m} \omega_{k+m}(f,u)_p)^q \frac{du}{u}\right)^{1/q} \lesssim \sup_{j=1, \ldots, d} \omega_k\left(\frac{\partial^m f}{\partial x_j^m}, t\right)_p \iff q \geq \max\{p,2\}.
	\end{equation}
	\end{enumerate}
	
	In particular, we have
	\begin{equation*}
		\left(\int_0^t (u^{-m} \omega_{k+m}(f,u)_2)^2 \frac{du}{u}\right)^{1/2} \asymp \sup_{|\beta| =m} \omega_k(D^\beta f, t)_2 \asymp  \sup_{j=1, \ldots, d} \omega_k\left(\frac{\partial^m f}{\partial x_j^m}, t\right)_2.
	\end{equation*}
	
	The corresponding result for periodic functions also holds true.
 \end{thm}


 \begin{rem}
 	 (i) The inequality (\ref{convTrebels}) follows from (\ref{inequalTrebels2}) with $q=\infty$.
	
	 (ii) The inequality given in (\ref{inequalTrebels2}) improves (\ref{convTrebels}) for $1 < p < \infty$ in several ways: firstly, one can refine the behavior of the moduli of smoothness of $f$ in terms of its derivatives with the help of the fine index $q$; and secondly, we observe that it is possible to bound the behavior of moduli of smoothness of $f$ using only the directional derivatives $\frac{\partial^m f}{\partial x_j^m}, j = 1, \ldots, d$, rather than all partial derivatives $D^\beta f$ with $|\beta| = m$.
 \end{rem}

 \begin{proof}[Proof of Theorem \ref{Thm Trebels2}]
	(i) Assume that $q \leq \min\{p,2\}$. Take $\beta \in \mathbb{N}_0^d$ such that $|\beta| = m$. Then, it is clear that
	\begin{equation}\label{12.7new2}
	D^\beta : \dot{W}^{k+m}_p(\mathbb{R}^d) \longrightarrow \dot{W}^{k}_p(\mathbb{R}^d).
	\end{equation}
	 Furthermore, by Lemma \ref{CharSobolevNorm}(ii) below we have $D^\beta : \dot{B}^{m}_{p,q}(\mathbb{R}^d) \longrightarrow \dot{B}^{0}_{p,q}(\mathbb{R}^d)$, and thus
	 \begin{equation}\label{12.8new}
	  D^\beta : \dot{B}^{m}_{p,q}(\mathbb{R}^d) \longrightarrow L_{p}(\mathbb{R}^d)
	  \end{equation}
	   because $\dot{B}^{0}_{p,q}(\mathbb{R}^d) \hookrightarrow L_p(\mathbb{R}^d)$. According to Lemma \ref{LemmaModuli}, (\ref{LPHom}), (\ref{12.7new2}), and (\ref{12.8new}), we derive
	   \begin{equation}\label{12.9new2}
	   	\omega_k(D^\beta f, t)_p \asymp K(t^k, D^\beta f; L_p(\mathbb{R}^d), \dot{W}^{k}_p(\mathbb{R}^d)) \lesssim K(t^k, f; \dot{B}^{m}_{p,q}(\mathbb{R}^d), \dot{W}^{k+m}_p(\mathbb{R}^d)).
	   \end{equation}
	  On the other hand, the well-known interpolation formula (cf. \cite[Theorem 6.3.1, p. 147]{BerghLofstrom}; see also (\ref{PrelimInterpolation}))
	\begin{equation*}
		 \dot{B}^m_{p,q}(\mathbb{R}^d) = (L_p(\mathbb{R}^d), \dot{W}^{k + m}_p(\mathbb{R}^d))_{\frac{m}{k+m},q},
	\end{equation*}
 in combination with Lemmas \ref{PrelimHolmstedt}(ii) and \ref{LemmaModuli}, gives that
	\begin{align}
		K(t^k,f;  \dot{B}^m_{p,q}(\mathbb{R}^d), \dot{W}^{k + m}_p(\mathbb{R}^d)) &\asymp \left(\int_0^t (u^{-m} K(u^{k+m},f;L_p(\mathbb{R}^d), \dot{W}^{k + m}_p(\mathbb{R}^d)))^q \frac{du}{u}\right)^{1/q}\nonumber \\
		& \asymp \left(\int_0^t (u^{-m} \omega_{k+m}(f,u)_p)^q \frac{du}{u}\right)^{1/q} \label{KfunctBWHom}.
	\end{align}
	Inserting this estimate into (\ref{12.9new2}) we arrive at
	\begin{equation}\label{12.11new2}
		\omega_k(D^\beta f, t)_p  \lesssim  \left(\int_0^t (u^{-m} \omega_{k+m}(f,u)_p)^q \frac{du}{u}\right)^{1/q} .
	\end{equation}
	
	Conversely, assume that (\ref{12.11new2}) holds for all $|\beta| = m$ and let $f \in \dot{B}^m_{p,q}(\mathbb{R}^d)$. Then,
	\begin{equation}\label{12.12new2}
		\sup_{t > 0} \omega_k(D^\beta f, t)_p \lesssim \left(\int_0^\infty (u^{-m} \omega_{k+m}(f,u)_p)^q \frac{du}{u}\right)^{1/q} \asymp \|f\|_{\dot{B}^m_{p,q}(\mathbb{R}^d)}, \quad |\beta| = m,
	\end{equation}
	where the last equivalence follows from the homogeneous counterpart of (\ref{BesovComparison}); see \cite[Section 6.3]{BerghLofstrom} for further details. Further, we remark that
	\begin{equation}\label{AsympMod}
	 \sup_{t > 0} \omega_k(f, t)_p \asymp \|f\|_{L_p(\mathbb{R}^d)} \quad \text{for}  \quad 1 \leq p < \infty
	 \end{equation}
	  (see, e.g., \cite[(3.6)]{Milman2}). Hence, by (\ref{12.12new2}) we obtain $\dot{B}^m_{p,q}(\mathbb{R}^d) \hookrightarrow \dot{W}^m_p(\mathbb{R}^d)$ which implies $q \leq \min\{p,2\}$ (cf. \cite[Lemma 1 and Remark 1]{HansenSickel}; see also Proposition \ref{RecallEmb}(ii) for the inhomogeneous case).

 	(ii) Let $ q \geq \max\{p,2\}$. We first remark that $\eta_R \frac{\partial^m  f}{\partial x_j^m} = \frac{\partial^m (\eta_R f)}{\partial x_j^m}$ where $\eta_R$ is the de la Vall\'ee-Poussin operator introduced in (\ref{eta}). This is a simple consequence of the fact that
	\begin{equation}\label{FourierDerivative}
	 \widehat{\frac{\partial^m f}{\partial x_j^m}}(\xi) = \xi_j^m \widehat{f}(\xi), \quad  \xi = (\xi_1, \ldots, \xi_d) \in \mathbb{R}^d.
	 \end{equation}
	 Further, we will make use of the following realization result (cf. \cite[Lemma 1.4]{GogatishviliOpicTikhonovTrebels})
	\begin{equation}\label{KfunctEta}
		K(t^\alpha,f;L_p(\mathbb{R}^d), \dot{\mathscr{L}}^{\alpha}_p(\mathbb{R}^d)) \asymp \|f-\eta_{1/t}f\|_{L_p(\mathbb{R}^d)} + t^\alpha \|\eta_{1/t}f\|_{\dot{\mathscr{L}}^{\alpha}_p(\mathbb{R}^d)}.
	\end{equation}

	By Lemma \ref{LemmaModuli} (with $\alpha = k$) and (\ref{KfunctEta}), we have for each $j=1, \ldots, d$,
	\begin{align*}
		\omega_k\left(\frac{\partial^m f}{\partial x_j^m}, t\right)_p &\asymp \left\|\frac{ \partial^m (f- \eta_{1/t} f)}{\partial x_j^m}\right\|_{L_p(\mathbb{R}^d)} + t^k \left\|\frac{\partial^m \eta_{1/t}f}{\partial x_j^m}\right\|_{\dot{\mathscr{L}}^{k}_p(\mathbb{R}^d)} \\
		& \gtrsim  \left\|\frac{ \partial^m (f- \eta_{1/t} f)}{\partial x_j^m}\right\|_{\dot{B}_{p,q}^0(\mathbb{R}^d)} + t^k \left\|\frac{\partial^m \eta_{1/t}f}{\partial x_j^m}\right\|_{\dot{\mathscr{L}}^{k}_p(\mathbb{R}^d)}
	\end{align*}
	because $L_p(\mathbb{R}^d) \hookrightarrow \dot{B}_{p,q}^0(\mathbb{R}^d)$ (see \cite{Jawerth}; this is the homogeneous counterpart of the embedding given in Proposition \ref{RecallEmb}(i) with $s=b=0$ and $r=2$). Taking the supremum over all $j=1, \ldots, d$ and then applying Lemma \ref{CharSobolevNorm}(ii) below and (\ref{KfunctBWHom}) we get
	\begin{align*}
		 \sup_{j=1, \ldots, d} \omega_k\left(\frac{\partial^m f}{\partial x_j^m}, t\right)_p & \gtrsim  \sum_{j=1}^d  \left\|\frac{ \partial^m (f- \eta_{1/t} f)}{\partial x_j^m}\right\|_{\dot{B}_{p,q}^0(\mathbb{R}^d)} + t^k \sum_{j=1}^d  \left\|\frac{\partial^m \eta_{1/t}f}{\partial x_j^m}\right\|_{\dot{W}^{k}_p(\mathbb{R}^d)} \nonumber \\
		 &  \asymp \|f-\eta_{1/t} f\|_{\dot{B}^m_{p,q}(\mathbb{R}^d)} + t^k \|\eta_{1/t} f\|_{\dot{W}^{k + m}_p(\mathbb{R}^d)}\nonumber \\
		 & \geq K(t^k, f; \dot{B}^m_{p,q}(\mathbb{R}^d), \dot{W}^{k + m}_p(\mathbb{R}^d)) \\
		 & \asymp  \left(\int_0^t (u^{-m} \omega_{k+m}(f,u)_p)^q \frac{du}{u}\right)^{1/q}
	\end{align*}
	where we also used (\ref{LPHom}) in the first estimate.
	
	 Conversely, assume that the inequality given in (\ref{inequalTrebels2}) is fulfilled. In particular, one has
	 \begin{equation*}
	 	 \left(\int_0^t (u^{-m} \omega_{k+m}(f,u)_p)^q \frac{du}{u}\right)^{1/q} \lesssim \sup_{j=1, \ldots, d} \left\|\frac{\partial^m f}{\partial x_j^m}\right\|_{L_p(\mathbb{R}^d)} \leq \|f\|_{\dot{W}^m_p(\mathbb{R}^d)}, \quad t > 0,
	 \end{equation*}
	 and taking the supremum over all $t > 0$, we obtain
	 \begin{equation*}
	  \|f\|_{\dot{B}^m_{p,q}(\mathbb{R}^d)} \asymp  \left(\int_0^\infty (u^{-m} \omega_{k+m}(f,u)_p)^q \frac{du}{u}\right)^{1/q} \lesssim \|f\|_{\dot{W}^m_p(\mathbb{R}^d)}.
	  \end{equation*}
	  Note that the latter embedding implies $q \geq \max\{p,2\}$ (cf. \cite[Lemma 1 and Remark 1]{HansenSickel}; see also Proposition \ref{RecallEmb}(i) for the inhomogeneous case).
	
	  The proofs of (i) and (ii) given above carry over verbatim to the periodic case.
 \end{proof}

 \subsection{Sobolev-type characterizations of Besov spaces}\label{Section: Sobolev norms}
 We start by recalling the characterizations of classical (homogenous and inhomogenous) Besov spaces  in terms of Sobolev-type norms.

 \begin{lem}\label{CharSobolevNorm}
 \begin{enumerate}[\upshape(i)]
 \item Let $s \in \mathbb{R}, 1 \leq  p \leq \infty, 0 < q \leq \infty$, and $m \in \mathbb{N}$. Then,
  \begin{equation*}
 	\|f\|_{B^s_{p,q}(\mathbb{R}^d)} \asymp \sum_{|\beta| \leq m} \|D^\beta f\|_{B^{s-m}_{p,q}(\mathbb{R}^d)} \asymp \|f\|_{B^{s-m}_{p,q}(\mathbb{R}^d)} + \sum_{j=1}^d \Big\|\frac{\partial^m f}{\partial x_j^m}\Big\|_{B^{s-m}_{p,q}(\mathbb{R}^d)}.
 \end{equation*}
 \item Let $s \in \mathbb{R}, 1 <  p < \infty, 0 < q \leq \infty$, and $m \in \mathbb{N}$. Then,
   \begin{equation*}
 	\|f\|_{\dot{B}^s_{p,q}(\mathbb{R}^d)} \asymp \sum_{|\beta| = m} \|D^\beta f\|_{\dot{B}^{s-m}_{p,q}(\mathbb{R}^d)} \asymp \sum_{j=1}^d \Big\|\frac{\partial^m f}{\partial x_j^m}\Big\|_{\dot{B}^{s-m}_{p,q}(\mathbb{R}^d)}.
 \end{equation*}
 \end{enumerate}
 \end{lem}

 The proof of Lemma \ref{CharSobolevNorm} can be found in \cite[Proposition 9, p. 147]{Stein}, \cite[Theorem 2.3.8]{Triebel1} and \cite[Proposition 2.18]{Triebel15}. We mention that we have not found an explicit formulation of the second equivalence in Lemma \ref{CharSobolevNorm}(ii), but its proof carries over from the inhomogeneous case (see, e.g., \cite[Theorem 2.3.8]{Triebel1}) with appropriate modifications based on the Fourier multiplier theorem stated in \cite[Theorem 5.2.2]{Triebel1}.

 With the help of the lift operator $\mathcal{J}_b$ given in (\ref{LiftLog}) we can extend Lemma \ref{CharSobolevNorm}(i) to generalized smoothness. This result seems to be new.

  \begin{lem}\label{CharSobolevNorm2}
 Let $s, b \in \mathbb{R}, 1 \leq  p \leq \infty, 0 < q \leq \infty$, and $m \in \mathbb{N}$. Then,
  \begin{equation}\label{11.1i}
 	\|f\|_{B^{s,b}_{p,q}(\mathbb{R}^d)} \asymp \sum_{|\beta| \leq m} \|D^\beta f\|_{B^{s-m,b}_{p,q}(\mathbb{R}^d)} \asymp \|f\|_{B^{s-m,b}_{p,q}(\mathbb{R}^d)} + \sum_{j=1}^d \Big\|\frac{\partial^m f}{\partial x_j^m}\Big\|_{B^{s-m,b}_{p,q}(\mathbb{R}^d)}.
 \end{equation}
 \end{lem}
 \begin{proof}
 	Let $\beta = (\beta_1, \ldots, \beta_d) \in \mathbb{N}_0^d$. Using (\ref{FourierDerivative}),
	  it is clear that $D^\beta \mathcal{J}_b f = \mathcal{J}_b D^\beta f.$ Thus, applying (\ref{8.7new}) twice and Lemma \ref{CharSobolevNorm}, we get
	\begin{equation*}
		\|f\|_{B^{s,b}_{p,q}(\mathbb{R}^d)}   \asymp \|\mathcal{J}_b f\|_{B^{s}_{p,q}(\mathbb{R}^d)} \asymp  \sum_{|\beta| \leq m} \|\mathcal{J}_b D^\beta f\|_{B^{s-m}_{p,q}(\mathbb{R}^d)} \asymp  \sum_{|\beta| \leq m} \|D^\beta f\|_{B^{s-m,b}_{p,q}(\mathbb{R}^d)}.
	\end{equation*}
	This shows the left-hand side equivalence in (\ref{11.1i}). Analogously, one can show the right-hand side equivalence in (\ref{11.1i}).
 \end{proof}

 Let $m$ be a positive integer and let $s > m$. According to Lemma \ref{CharSobolevNorm2} and (\ref{BesovComparison}), we have
 \begin{equation}\label{11.1new}
\text{if} \qquad f
 \in\mathbf{B}^{s,b}_{p,q}(\mathbb{R}^d),\qquad\mbox{ then }
  \qquad D^\beta f\in \mathbf{B}^{s-m,b}_{p,q}(\mathbb{R}^d),
\end{equation}
 where $\beta \in \mathbb{N}_0^d$ with $|\beta| = m$. At the endpoint case $s=m$, there is a loss of logarithmic smoothness of $f$ in order to get that $D^\beta f \in \mathbf{B}^{0,b}_{p,q}(\mathbb{R}^d)$. More precisely, DeVore, Riemenschneider and Sharpley \cite[p. 70]{DeVoreRiemenschneiderSharpley} showed that
 \begin{equation}\label{11.2new}
 	 \text{if} \qquad f
 \in\mathbf{B}^{m,b+1}_{p,q}(\mathbb{R}^d)\qquad\mbox{ then}
  \qquad D^\beta f \in \mathbf{B}^{0,b}_{p,q}(\mathbb{R}^d).
 \end{equation}
 Here, $1 \leq p \leq \infty, 1 \leq q \leq \infty, b > -1/q,$ and $|\beta| = m$.

 Our next result improves (\ref{11.2new}) with respect to the logarithmic factor.

  \begin{thm}\label{Theorem 11.1}
 		Let $m \in \mathbb{N}, 0 < q \leq \infty$, and $b > -1/q$.
		
		\begin{enumerate}[\upshape(i)]
		\item Assume that $1 < p < \infty$. We have
		\begin{equation*}
			\|f\|_{L_p(\mathbb{R}^d)} + \sum_{|\beta| = m} \|D^\beta f\|_{\mathbf{B}^{0,b}_{p,q}(\mathbb{R}^d)} \lesssim \|f\|_{ \mathbf{B}^{m,b+1/\min\{2,p,q\}}_{p,q}(\mathbb{R}^d)}.
		\end{equation*}
		In particular, if $f \in \mathbf{B}^{m,b+1/\min\{2,p,q\}}_{p,q}(\mathbb{R}^d)$, then $D^\beta f \in \mathbf{B}^{0,b}_{p,q}(\mathbb{R}^d)$ for any multi-index $\beta$ with $|\beta| = m$.
		
		\item Let $p=1, \infty$. We have
		\begin{equation*}
			\|f\|_{L_p(\mathbb{R}^d)} + \sum_{|\beta| = m} \|D^\beta f\|_{\mathbf{B}^{0,b}_{p,q}(\mathbb{R}^d)} \lesssim \|f\|_{ \mathbf{B}^{m,b+1/\min\{1,q\}}_{p,q}(\mathbb{R}^d)}.
		\end{equation*}
		 In particular, if $f \in \mathbf{B}^{m,b+1/\min\{1,q\}}_{p,q}(\mathbb{R}^d)$, then $D^\beta f \in \mathbf{B}^{0,b}_{p,q}(\mathbb{R}^d)$ for any multi-index $\beta$ with $|\beta| = m$.
		\end{enumerate}
		
		The same results also hold in the periodic setting.
 \end{thm}

  \begin{rem}
Let us compare the results given in Theorem \ref{Theorem 11.1} for Besov spaces defined by differences with those given in Lemma \ref{CharSobolevNorm2} for their Fourier-analytically defined counterparts. We observe that the additional shifts $1/\min\{2,p,q\}$ and $1/\min\{1,q\}$ in the logarithmic smoothness given in Theorem \ref{Theorem 11.1} are not needed  when dealing with the Besov spaces $B^{s,b}_{p,q}(\mathbb{R}^d)$ and, in particular, the space $B^{0,b}_{p,q}(\mathbb{R}^d)$. More precisely,
 \begin{equation*}
 	 \text{if} \qquad f
 \in B^{m,b}_{p,q}(\mathbb{R}^d),\qquad\mbox{ then}
  \qquad D^\beta f \in B^{0,b}_{p,q}(\mathbb{R}^d), \quad |\beta| = m.
 \end{equation*}
 This shows another important distinction between the spaces $B^{0,b}_{p,q}(\mathbb{R}^d)$ and $\mathbf{B}^{0,b}_{p,q}(\mathbb{R}^d)$.
 \end{rem}

 \begin{proof}[Proof of Theorem \ref{Theorem 11.1}]
 	Let $s > m$. We know that (cf. (\ref{11.1new}))
	\begin{equation*}
		D^\beta : W^m_p(\mathbb{R}^d) \longrightarrow L_p(\mathbb{R}^d) \quad \text{and} \quad D^\beta : \mathbf{B}^s_{p,q}(\mathbb{R}^d) \longrightarrow \mathbf{B}^{s-m}_{p,q}(\mathbb{R}^d).
	\end{equation*}
	Applying the interpolation property and (\ref{PrelimInterpolationnew2}) (with $\xi=0$), we obtain the boundedness of the operator
	\begin{equation}\label{11.1}
		D^\beta : (W^m_p(\mathbb{R}^d), \mathbf{B}^s_{p,q}(\mathbb{R}^d))_{(0,b),q} \longrightarrow (L_p(\mathbb{R}^d), \mathbf{B}^{s-m}_{p,q}(\mathbb{R}^d))_{(0,b),q}= \mathbf{B}^{0,b}_{p,q}(\mathbb{R}^d).
	\end{equation}
	 Let $s_0 > 0$ be such that $m > s_0$. Then, there exists $\theta \in (0,1)$ with $m = (1-\theta) s_0 + \theta s$. Assume first that $1 < p < \infty$. Since $W^m_p(\mathbb{R}^d) = H^m_p(\mathbb{R}^d)  \hookleftarrow \mathbf{B}^m_{p, \min\{p,2\}}(\mathbb{R}^d)$ (cf. Proposition \ref{RecallEmb}(vi)), we have
	\begin{align}
		(W^m_p(\mathbb{R}^d), \mathbf{B}^s_{p,q}(\mathbb{R}^d))_{(0,b),q}  & \hookleftarrow  (\mathbf{B}^m_{p, \min\{p,2\}}(\mathbb{R}^d), \mathbf{B}^s_{p,q}(\mathbb{R}^d))_{(0,b),q} \nonumber \\
		& \hspace{-3cm}= ((\mathbf{B}^{s_0}_{p,q}(\mathbb{R}^d), \mathbf{B}^s_{p,q}(\mathbb{R}^d))_{\theta,\min\{p,2\}} , \mathbf{B}^s_{p,q}(\mathbb{R}^d))_{(0,b),q} \nonumber\\
		& \hspace{-3cm} \hookleftarrow (\mathbf{B}^{s_0}_{p,q}(\mathbb{R}^d), \mathbf{B}^s_{p,q}(\mathbb{R}^d))_{\theta, q ; b + 1/\min\{2,p,q\}} = \mathbf{B}^{m, b + 1/\min\{2,p,q\}}_{p,q}(\mathbb{R}^d). \label{11.4}
	\end{align}
	where we have applied (\ref{BesovComparison}), Lemma \ref{PrelimLemmaCF} twice and Lemma \ref{PrelimLemma7.2}(vi).
	
	Let $p=1, \infty$. Then, $W^m_p(\mathbb{R}^d) \hookleftarrow \mathbf{B}^m_{p,1}(\mathbb{R}^d) $. This embedding follows easily from the known embedding $L_p(\mathbb{R}^d) \hookleftarrow B^0_{p,1}(\mathbb{R}^d)$ (see \cite[(3.1.3), (3.1.6), p. 112]{SickelTriebel}), (\ref{11.1i}), and (\ref{BesovComparison}). Applying now the same argument as above gives
	\begin{equation}\label{11.5}
	(W^m_p(\mathbb{R}^d), \mathbf{B}^s_{p,q}(\mathbb{R}^d))_{(0,b),q} \hookleftarrow (\mathbf{B}^m_{p,1}(\mathbb{R}^d), \mathbf{B}^s_{p,q}(\mathbb{R}^d))_{(0,b),q} \hookleftarrow \mathbf{B}^{m,b+1/\min\{1,q\}}_{p,q}(\mathbb{R}^d).
	\end{equation}
	
	Finally, the proof follows from (\ref{11.1}), (\ref{11.4}) and (\ref{11.5}).
	
	The same method of proof also works when $\mathbb{R}^d$ is replaced by $\mathbb{T}^d$.
 \end{proof}

\begin{rem}\label{Remark 12.8}
(i) Working with periodic functions, Theorem \ref{Theorem 11.1}(i) was already shown in \cite[Theorem 4.7] {CobosDominguez2} by using the approximation structure of $\mathbf{B}^{0,b}_{p,q}(\mathbb{T})$ given in Lemma \ref{LemmaBesovApprox}.

(ii) If $b=-1/q$ the corresponding result to Theorem \ref{Theorem 11.1} reads as follows: Let $m \in \mathbb{N}, 0 < q \leq \infty$ and $\beta \in \mathbb{N}_0^d$ with $|\beta| = m$. If $1 < p < \infty$ then
\begin{equation*}
	D^\beta f \in \mathbf{B}^{0,-1/q}_{p,q}(\mathbb{R}^d) \quad \text{whenever} \quad f \in \mathbf{B}^{m,-1/q+1/\min\{2,p,q\}, 1/\min\{2,p,q\}}_{p,q}(\mathbb{R}^d),
\end{equation*}
and, if $p=1, \infty$ then
\begin{equation*}
	D^\beta f \in \mathbf{B}^{0,-1/q}_{p,q}(\mathbb{R}^d) \quad \text{whenever} \quad f \in \mathbf{B}^{m,-1/q+1/\min\{1,q\}, 1/\min\{1,q\}}_{p,q}(\mathbb{R}^d).
\end{equation*}
	Here, $\mathbf{B}^{s,b,\xi}_{p,q}(\mathbb{R}^d)$\index{\bigskip\textbf{Spaces}!$\mathbf{B}^{s,b,\xi}_{p,q}(\mathbb{R}^d)$}\label{BESOVDLOG} is equipped with the quasi-norm
	\begin{equation*}
		 \|f\|_{\mathbf{B}^{s,b,\xi}_{p,q}(\mathbb{R}^d)} =  \|f\|_{L_p(\mathbb{R}^d)} + \left(\int_0^1 (t^{-s} (1 - \log t)^b (1 + \log (1 - \log t))^\xi \omega_l(f,t)_p)^q
    \frac{dt}{t}\right)^{1/q}
	\end{equation*}
	where $l \in \mathbb{N}$ with $l > s$.
	The proof is similar to that of Theorem \ref{Theorem 11.1} but now using the reiteration formula (\ref{limiting interpolation reiteration formula}).
	
	(iii) Alternatively we may use Theorem \ref{Thm Trebels2} to show Theorem \ref{Theorem 11.1}. For example, the proof of Theorem \ref{Theorem 11.1}(i) can be addressed as follows. First, we observe that $\|D^\beta f\|_{L_p(\mathbb{R}^d)} \lesssim \|f\|_{\mathbf{B}^{m,b+1/\min\{2,p,q\}}_{p,q}(\mathbb{R}^d)}$ for any $|\beta|=m$ because the following chain of embeddings hold
	\begin{equation*}
		\mathbf{B}^{m,b+1/\min\{2,p,q\}}_{p,q}(\mathbb{R}^d) \hookrightarrow \mathbf{B}^m_{p,\min\{p,2\}}(\mathbb{R}^d) \hookrightarrow W^m_p(\mathbb{R}^d).
	\end{equation*}
	See Propositions \ref{RecallEmb**} and \ref{RecallEmb}(vi) (with $s=m$ and $b=0$). On the other hand, Theorem \ref{Thm Trebels2}(i) yields that
	\begin{equation}\label{12.18new+}
		|D^\beta f|^q_{\mathbf{B}^{0,b}_{p,q}(\mathbb{R}^d)}  \lesssim \int_0^1 (1-\log t)^{b q}  \left(\int_0^t (u^{-m} \omega_{k+m}(f,u)_p)^{\min\{p,2\}} \frac{du}{u}\right)^{q/\min\{p,2\}} \frac{dt}{t}.
	\end{equation}
	Let $q \geq \min\{p,2\}$. Thus applying the Hardy's inequality (\ref{HardyInequal3}) in the previous estimate we arrive at
	\begin{equation*}
		 |D^\beta f|^q_{\mathbf{B}^{0,b}_{p,q}(\mathbb{R}^d)} \lesssim \int_0^1 (t^{-m} (1-\log t)^{b+ 1/\min\{p,2\}}
		  \omega_{k+m}(f,t)_p)^q \frac{dt}{t} \leq \|f\|_{\mathbf{B}^{m,b+1/\min\{p,2\}}_{p,q}(\mathbb{R}^d)}^q.
	\end{equation*}
	If $q \leq \min\{p,2\}$ then the right-hand side of (\ref{12.18new+}) can be estimated by
	\begin{align*}
		\sum_{j=0}^\infty (1 + j)^{b q} \left(\sum_{\nu=j}^\infty (2^{\nu m} \omega_{k+m}(f,2^{-\nu})_p)^{\min\{p,2\}}\right)^{q/\min\{p,2\}} \\
		& \hspace{-7cm}\leq \sum_{\nu =0}^\infty (2^{\nu m} \omega_{k+m}(f,2^{-\nu})_p)^q \sum_{j=0}^\nu (1 + j)^{b q} \\
		& \hspace{-7cm}\lesssim \sum_{\nu =0}^\infty (2^{\nu m} (1 + \nu)^{b + 1/q} \omega_{k+m}(f,2^{-\nu})_p)^q \\
		& \hspace{-7cm}\leq \|f\|_{\mathbf{B}^{m,b+1/q}_{p,q}(\mathbb{R}^d)}^q.
	\end{align*}
	Hence, we have shown that $|D^\beta f|_{\mathbf{B}^{0,b}_{p,q}(\mathbb{R}^d)} \lesssim  \|f\|_{\mathbf{B}^{m,b+1/q}_{p,q}(\mathbb{R}^d)}$.
\end{rem}

Next we establish the sharpness of Theorem \ref{Theorem 11.1}(i).

 \begin{prop}\label{Proposition 11.4}
 	Let $m \in \mathbb{N}, 1 < p < \infty, 0 < q \leq \infty$, and $b > -1/q$. Given any $\varepsilon > 0$, there exists $f \in \mathbf{B}^{m,b+1/\min\{2,p,q\}-\varepsilon}_{p,q}(\mathbb{T})$ such that $D^m f  \not \in \mathbf{B}^{0,b}_{p,q}(\mathbb{T})$.
 \end{prop}
 \begin{proof}
 If $p=\min\{2,p,q\}$ we consider the function $f$ with
 the Fourier series
 \begin{equation*}
 \sum_{n=1}^\infty n^{-m-1+1/p} (1 + \log n)^{-\beta} \cos nx, \quad x \in \mathbb{T},
 \end{equation*}
 where $\max\{b+1/q + 1/p - \varepsilon,1/p\} < \beta \leq b+1/q+1/p$. By Theorem \ref{Theorem 3.4}, we have
 \begin{equation*}
 	\|f\|_{\mathbf{B}^{m,b+1/p-\varepsilon}_{p,q}(\mathbb{T})}^q \asymp \sum_{n=1}^\infty (1 + \log n)^{(b + 1/p - \varepsilon -\beta) q}  \frac{1}{n} < \infty
 \end{equation*}
 and
 \begin{align*}
 	\|D^m f\|^q_{\mathbf{B}^{0,b}_{p,q}(\mathbb{T})} &\asymp \sum_{n=1}^\infty (1 + \log n)^{b q} \left(\sum_{\nu=n}^\infty (1 + \log \nu)^{-\beta p} \frac{1}{\nu}\right)^{q/p} \frac{1}{n} \\
	& \asymp \sum_{n=1}^\infty (1 + \log n)^{(b-\beta +1/p) q} \frac{1}{n} = \infty.
 \end{align*}
 Therefore, $f \in \mathbf{B}^{m,b+1/p-\varepsilon}_{p,q}(\mathbb{T})$ but $D^m f \not \in \mathbf{B}^{0,b}_{p,q}(\mathbb{T})$.

 Assume that $2 = \min\{2,p,q\}$. Let
 \begin{equation*}
 	f(x) \sim \sum_{j=0}^\infty 2^{-j m} (1 + j)^{-\delta} \cos 2^j x, \quad x \in \mathbb{T},
 \end{equation*}
 where $\max\{b+1/q + 1/2-\varepsilon, 1/2\} < \delta \leq b + 1/q + 1/2$. Applying Theorem \ref{Theorem 4.4} we derive
  \begin{equation*}
 	\|f\|_{\mathbf{B}^{m,b+1/2-\varepsilon}_{p,q}(\mathbb{T})}^q \asymp \sum_{j=0}^\infty (1 + j)^{(b+1/2 -\varepsilon - \delta) q} < \infty
 \end{equation*}
 and
 \begin{equation*}
 	\|D^m f\|^q_{\mathbf{B}^{0,b}_{p,q}(\mathbb{T})} \asymp \sum_{j=0}^\infty (1 + j)^{b q} \Big(\sum_{\nu=j}^\infty (1 + \nu)^{-2 \delta}\Big)^{q/2} \asymp\sum_{j=0}^\infty (1 + j)^{(b-\delta + 1/2) q} = \infty.
 \end{equation*}
 This shows that $f \in \mathbf{B}^{m,b+1/2-\varepsilon}_{p,q}(\mathbb{T})$ and $D^m f \not \in \mathbf{B}^{0,b}_{p,q}(\mathbb{T})$.

 Finally we deal with the case $q= \min\{2,p,q\}$. We proceed by contradiction. Let us suppose that
 \begin{equation}\label{15}
 	D^m : \mathbf{B}^{m,b+1/q-\varepsilon}_{p,q}(\mathbb{T}) \longrightarrow  \mathbf{B}^{0,b}_{p,q}(\mathbb{T}).
 \end{equation}
 Let $s$ be such that $m < s < m+1$. Since $D^m : W^{m+1}_p(\mathbb{T}) \longrightarrow W^1_p(\mathbb{T})$, it follows from (\ref{15}) and the interpolation property that
 \begin{equation}\label{16}
 	D^m : (\mathbf{B}^{m,b+1/q-\varepsilon}_{p,q}(\mathbb{T}), W^{m+1}_p(\mathbb{T}))_{s-m,q} \longrightarrow (\mathbf{B}^{0,b}_{p,q}(\mathbb{T}), W^1_p(\mathbb{T}))_{s-m,q}.
 \end{equation}
 According to (\ref{PrelimInterpolationBWnew}) and (\ref{BesovComparison}), $(\mathbf{B}^{m,b+1/q-\varepsilon}_{p,q}(\mathbb{T}), W^{m+1}_p(\mathbb{T}))_{s-m,q}  = B^{s, (1-s+m)(b+1/q-\varepsilon)}_{p,q}(\mathbb{T})$. Concerning the target space in (\ref{16}), by (\ref{PrelimInterpolationnew2.3}) and (\ref{BesovComparison}), we have
 \begin{equation*}
 	 (\mathbf{B}^{0,b}_{p,q}(\mathbb{T}), W^1_p(\mathbb{T}))_{s-m,q}
	  = B^{s-m, (1-s+m)(b+1/q)}_{p,q}(\mathbb{T}).
 \end{equation*}
 Hence, we have shown that
 \begin{equation}\label{17}
 	D^m : B^{s, (1-s+m)(b+1/q-\varepsilon)}_{p,q}(\mathbb{T}) \longrightarrow B^{s-m, (1-s+m)(b+1/q)}_{p,q}(\mathbb{T}).
 \end{equation}
 However, the operator $D^m$ does not fulfil (\ref{17}). Indeed, in virtue of Theorem \ref{TheoremBesovGMPer}, the function $f$ with  the Fourier series
 $$ \sum_{n=1}^\infty n^{-s - 1 + 1/p} (1 + \log n)^{-(1-s+m)(b+1/q-\varepsilon/2) - 1/q} \cos nx, \quad x \in \mathbb{T},$$ belongs to $B^{s, (1-s+m)(b+1/q-\varepsilon)}_{p,q}(\mathbb{T})$ because
 \begin{equation*}
 	\|f\|_{B^{s, (1-s+m)(b+1/q-\varepsilon)}_{p,q}(\mathbb{T})}^q \asymp \sum_{n=1}^\infty (1 + \log n)^{-\frac{\varepsilon}2 (1-s+m)q -1} \frac{1}{n} < \infty
 \end{equation*}
 but $D^m f \not \in B^{s-m, (1-s+m)(b+1/q)}_{p,q}(\mathbb{T})$ since
  \begin{equation*}
 	\|D^m f\|_{B^{s-m, (1-s+m)(b+1/q)}_{p,q}(\mathbb{T})}^q \asymp \sum_{n=1}^\infty (1 + \log n)^{\frac{\varepsilon}2 (1-s+m)q-1} \frac{1}{n} = \infty.
 \end{equation*}
 Therefore, (\ref{17}) and, therefore, (\ref{15}) are not valid.
 \end{proof}

 \begin{rem}
 In the particular case of $m=1$ and $p=2 \leq q < \infty$, a different approach to show Proposition \ref{Proposition 11.4} was given in \cite[Proposition 4.3]{CobosDominguez4}.
 \end{rem}

 In the rest of this section we study the converse to Theorem \ref{Theorem 11.1}.


   \begin{thm}\label{Theorem 11.1Conv}
 		Let $m \in \mathbb{N}, 0 < q \leq \infty$, and $b > -1/q$.
		
		\begin{enumerate}[\upshape(i)]
		\item Assume that $1 < p < \infty$. We have
		\begin{equation}\label{aux12.24}
			 \|f\|_{ \mathbf{B}^{m,b+1/\max\{2,p,q\}}_{p,q}(\mathbb{R}^d)} \lesssim \|f\|_{L_p(\mathbb{R}^d)} + \sum_{j = 1}^d \Big\|\frac{\partial^m f}{\partial x_j^m} \Big\|_{\mathbf{B}^{0,b}_{p,q}(\mathbb{R}^d)}.
		\end{equation}
		
		\item Let $p=1, \infty$. We have
		\begin{equation*}
			\|f\|_{ \mathbf{B}^{m,b}_{p,q}(\mathbb{R}^d)} \lesssim \|f\|_{L_p(\mathbb{R}^d)} + \sum_{|\beta| = m} \|D^\beta f\|_{\mathbf{B}^{0,b}_{p,q}(\mathbb{R}^d)}.
		\end{equation*}
		\end{enumerate}
		
		The same results also hold in the periodic setting.
 \end{thm}
 \begin{proof}
 	Let $k \in \mathbb{N}$. Applying elementary estimates and (\ref{inequalTrebels2}), we obtain
	\begin{align}
		\sum_{j = 1}^d \Big|\frac{\partial^m f}{\partial x_j^m} \Big|^q_{\mathbf{B}^{0,b}_{p,q}(\mathbb{R}^d)} & \asymp \sum_{i=0}^\infty (1 + i)^{b q} \sup_{j= 1, \ldots, d} \omega_k\Big(\frac{\partial^m f}{\partial x_j^m}, 2^{-i}\Big)_p^q \nonumber \\
		& \gtrsim  \sum_{i=0}^\infty (1 + i)^{b q} \left(\sum_{\nu=i}^\infty (2^{\nu m} \omega_{k+m}(f, 2^{-\nu})_p)^{\max\{p,2\}} \right)^{q/\max\{p,2\}}. \label{12.21new}
	\end{align}
	We distinguish two possible cases. First, suppose that $q \geq \max\{p,2\}$. Then, changing the order of summation we derive
	\begin{equation*}
			\sum_{j = 1}^d \Big|\frac{\partial^m f}{\partial x_j^m} \Big|^q_{\mathbf{B}^{0,b}_{p,q}(\mathbb{R}^d)} \gtrsim \sum_{\nu=0}^\infty (2^{\nu m} (1 + \nu)^{b+ 1/q}  \omega_{k+m}(f, 2^{-\nu})_p)^q \asymp |f|_{\mathbf{B}^{m,b+1/q}_{p,q}(\mathbb{R}^d)}^q.
	\end{equation*}
	This proves (i) for $q \geq \max\{p,2\}$.
	
	Next we deal with the case $q < \max\{p,2\}$. Let $\varepsilon > 1/q -1/\max\{p,2\}$. Applying H\"older's inequality,
	\begin{align*}
		\sum_{\nu =i}^\infty (2^{\nu m} (1 + \nu)^{-\varepsilon} \omega_{k+m}(f,2^{-\nu})_p)^q \\
		&\hspace{-4cm} \lesssim (1 + i)^{-\varepsilon q + 1 - q/\max\{p,2\}} \left(\sum_{\nu=i}^\infty (2^{\nu m} \omega_{k+m}(f, 2^{-\nu})_p)^{\max\{p,2\}} \right)^{q/\max\{p,2\}}.
	\end{align*}
	Using this estimate, we derive from (\ref{12.21new}) that
	\begin{align*}
			\sum_{j = 1}^d \Big|\frac{\partial^m f}{\partial x_j^m} \Big|^q_{\mathbf{B}^{0,b}_{p,q}(\mathbb{R}^d)} &\gtrsim \sum_{\nu=0}^\infty  (2^{\nu m} (1 + \nu)^{-\varepsilon} \omega_{k+m}(f,2^{-\nu})_p)^q \sum_{i=0}^\nu (1 + i)^{(b + \varepsilon -1/q + 1/\max\{p,2\})q} \\
		& \hspace{-1.5cm}\asymp  \sum_{\nu=0}^\infty  (2^{\nu m} (1 + \nu)^{b + 1/\max\{p,2\}} \omega_{k+m}(f,2^{-\nu})_p)^q \asymp |f|^q_{\mathbf{B}^{m,b+1/\max\{p,2\}}_{p,q}(\mathbb{R}^d)}.
	\end{align*}
This concludes the proof  of (i).
	
	The proof of (ii) is an immediate consequence of (\ref{convTrebels}).
 \end{proof}

 \begin{rem}
 Analogously to Remark \ref{Remark 12.8}(ii) we write down the extreme case $b=-1/q$ of Theorem \ref{Theorem 11.1Conv}: Let $m \in \mathbb{N}$ and $0 < q \leq \infty$. If $1 < p < \infty$, then
\begin{equation*}
		 \|f\|_{ \mathbf{B}^{m,-1/q+1/\max\{2,p,q\}, 1/\max\{2,p,q\}}_{p,q}(\mathbb{R}^d)} \lesssim \|f\|_{L_p(\mathbb{R}^d)} +  \sum_{j = 1}^d \Big\|\frac{\partial^m f}{\partial x_j^m} \Big\|_{\mathbf{B}^{0,b}_{p,q}(\mathbb{R}^d)}
\end{equation*}
and, if $p=1, \infty$, then
\begin{equation*}
	\|f\|_{ \mathbf{B}^{m,-1/q}_{p,q}(\mathbb{R}^d)} \lesssim \|f\|_{L_p(\mathbb{R}^d)} + \sum_{|\beta| = m} \|D^\beta f\|_{\mathbf{B}^{0,-1/q}_{p,q}(\mathbb{R}^d)}.
\end{equation*}
 \end{rem}

 Theorem \ref{Theorem 11.1Conv}(i) is also shown to be sharp in the following sense.

  \begin{prop}\label{Proposition 11.4new}
 	Let $m \in \mathbb{N}, 1 < p < \infty, 0 < q \leq \infty$, and $b > -1/q$.
 Given any $\varepsilon > 0$, there exists $f \in L_p(\mathbb{T})$ such that $D^m f \in \mathbf{B}^{0,b}_{p,q}(\mathbb{T})$ and $f  \not \in \mathbf{B}^{m,b+ 1/\max\{2,p,q\} + \varepsilon}_{p,q}(\mathbb{T})$.
 \end{prop}
 \begin{proof}
 	The proof follows essentially the same steps as the proof of Proposition \ref{Proposition 11.4}. We only show the case $q= \max\{2,p,q\}$. We argue by contradiction. We  assume that there is $\varepsilon > 0$ such that the following inequality
	\begin{equation}\label{12.23new}
	 \|f\|_{ \mathbf{B}^{m,b+1/q + \varepsilon}_{p,q}(\mathbb{T})} \lesssim \|f\|_{L_p(\mathbb{T})} +  \|D^m f\|_{\mathbf{B}^{0,b}_{p,q}(\mathbb{T})}
	\end{equation}
holds.
	Moreover, it is well known that
	\begin{equation}\label{12.24new}
		\|f\|_{W^{m+1}_p(\mathbb{T})} \asymp \|f\|_{L_p(\mathbb{T})} + \|D^m f\|_{W^1_p(\mathbb{T})}.
	\end{equation}
	Since $f \in L_p(\mathbb{T})$ with $D^m f \in L_p(\mathbb{T})$, there exists $g \in L_p(\mathbb{T})$ such that $(i k)^{-m} \widehat{g}(k) = \widehat{f}(k)$ for all $k \in \mathbb{Z}, k\neq 0$. Put $I^m g(x) \sim \sum_{k=-\infty}^\infty  \widehat{\Psi_m}(k) \widehat{g}(k) e^{ikx}$ where $\Psi_m \in L_1(\mathbb{T})$ with
	\begin{equation*}
		\widehat{\Psi_m}(k)  =  \left\{\begin{array}{lc} (i k)^{-m}, &k \in \mathbb{Z}, k \neq 0, \\
		&  \\
		0 , & k = 0.
		       \end{array}
                        \right.
	\end{equation*}
	Further, note that $D^m f = D^m(I^m g) = g-\widehat{g}(0)$ and $I^m$ acts boundedly on $L_p(\mathbb{T})$. Therefore, we derive from (\ref{12.23new}) and (\ref{12.24new}) that
	\begin{equation}\label{12.23new2}
	 \|I^m g\|_{ \mathbf{B}^{m,b+1/q + \varepsilon}_{p,q}(\mathbb{T})} \lesssim \| g\|_{\mathbf{B}^{0,b}_{p,q}(\mathbb{T})} \quad \text{and
} \quad \|I^m g\|_{W^{m+1}_p(\mathbb{T})} \asymp \|g\|_{W^1_p(\mathbb{T})}.
	\end{equation}
	Take any $m < s < m+1$. Thus, by (\ref{12.23new2}) we obtain
	\begin{equation}\label{12.27new}
		I^m : (\mathbf{B}^{0,b}_{p,q}(\mathbb{T}), W^1_p(\mathbb{T}))_{s-m,q} \longrightarrow (\mathbf{B}^{m,b+1/q + \varepsilon}_{p,q}(\mathbb{T}), W^{m+1}_p(\mathbb{T}))_{s-m,q}.
	\end{equation}
	By (\ref{PrelimInterpolationnew2.3}) and (\ref{BesovComparison}), we have $ (\mathbf{B}^{0,b}_{p,q}(\mathbb{T}), W^1_p(\mathbb{T}))_{s-m,q}  =  B^{s-m, (1-s+m)(b+1/q)}_{p,q}(\mathbb{T})$. On the other hand, by (\ref{PrelimInterpolationBWnew}) and (\ref{BesovComparison}) we have $(\mathbf{B}^{m,b+1/q+\varepsilon}_{p,q}(\mathbb{T}), W^{m+1}_p(\mathbb{T}))_{s-m,q}  = B^{s, (1-s+m)(b+1/q+\varepsilon)}_{p,q}(\mathbb{T})$. Hence, (\ref{12.27new}) claims that
	\begin{equation*}
		I^m : B^{s-m, (1-s+m)(b+1/q)}_{p,q}(\mathbb{T}) \longrightarrow B^{s, (1-s+m)(b+1/q+\varepsilon)}_{p,q}(\mathbb{T}).
	\end{equation*}
	However, the latter is not true as can be checked by taking the Fourier series
	\begin{equation*}
		\sum_{n=1}^\infty n^{-s+m-1+1/p} (1 + \log n)^{-(1-s+m)(b+1/q + \varepsilon/2) -1/q} \cos nx, \quad x \in \mathbb{T},
	\end{equation*}
	and applying Theorem \ref{TheoremBesovGMPer}.
 \end{proof}

\newpage

 \section{Lifting operators in Besov spaces}\label{Lifting Estimates}

Recall the definitions of the Riesz potential (cf. (\ref{RieszPotential}))
\begin{equation}\label{RieszPotential2}
	J_\sigma f = (|\xi|^\sigma \widehat{f})^\vee
\end{equation}
and the Bessel potential (cf. (\ref{liftdef}))
\begin{equation}\label{BesselPotential2}
	I_\sigma f = ((1 + |\xi|^2)^{\sigma/2} \widehat{f})^\vee.
\end{equation}
The study of these operators is a central issue in potential theory and they have important applications to PDE's, harmonic analysis and function spaces. See \cite{AdamsHedberg}, \cite{Landkof}, \cite{Stein}, and \cite{Triebel1}, and the references therein.

Before proceeding further, there are some technical issues that we would like to clarify. Namely, since $|\xi|^\sigma$ has a singularity at $\xi =0$ if $\sigma < 0$, the operator $J_\sigma$ (see (\ref{RieszPotential2})) is not well defined on $\mathcal{S}(\mathbb{R}^d)$. In fact, it is plain to check that $|\xi|^\sigma$ is not a tempered distribution for $\sigma \leq -d$.  To overcome this obstacle, following \cite{Peetre} and \cite{Triebel1}, we deal with the subclass $\dot{\mathcal{S}}(\mathbb{R}^d)$ of $\mathcal{S}(\mathbb{R}^d)$ given by\index{\bigskip\textbf{Spaces}!$\dot{\mathcal{S}}(\mathbb{R}^d)$}\label{SCLASS}
\begin{equation*}
	\dot{\mathcal{S}}(\mathbb{R}^d) = \{f \in \mathcal{S}(\mathbb{R}^d) : (D^\beta \widehat{f})(0) = 0 \text{ for } \beta \in \mathbb{N}_0^d\}.
\end{equation*}
Then $J_\sigma$ given by (\ref{RieszPotential2}) is well defined for any $f \in \dot{\mathcal{S}}(\mathbb{R}^d)$ and any $\sigma \in \mathbb{R}$. Let $\dot{\mathcal{S}}'(\mathbb{R}^d)$ be the dual space of $\dot{\mathcal{S}}(\mathbb{R}^d)$. It is easy to see that $\dot{\mathcal{S}}'(\mathbb{R}^d)$\index{\bigskip\textbf{Spaces}!$\dot{\mathcal{S}}'(\mathbb{R}^d)$}\label{S'CLASS} can be identified with the space $\mathcal{S}'(\mathbb{R}^d)$ modulo polynomials. By duality, one can extend $J_\sigma$ in the class $\dot{\mathcal{S}}'(\mathbb{R}^d)$ and $J_\sigma: \dot{\mathcal{S}}'(\mathbb{R}^d) \longrightarrow \dot{\mathcal{S}}'(\mathbb{R}^d)$.

We also mention that a different approach to deal with homogeneity has been recently proposed by Triebel \cite{Triebel15}. It is based on the so-called tempered homogeneous spaces within the classical framework of the dual pairing $(\mathcal{S}(\mathbb{R}^d), \mathcal{S}'(\mathbb{R}^d))$.

The goal of this section is twofold. First, we investigate the smoothness properties of Riesz and Bessel potentials by showing sharp estimates for their moduli of smoothness. Secondly, we will show a rather surprising result: contrary to the case of $B^{0,b}_{p,q}(\mathbb{R}^d)$ spaces, the operator $I_\sigma$ does not act as a lift in the scale of $\mathbf{B}^{0,b}_{p,q}(\mathbb{R}^d)$ spaces.

\subsection{Inequalities for moduli of smoothness of Riesz and Bessel potentials}

We establish the following sharp estimates for the moduli of smoothness of Riesz potential.

\begin{thm}\label{Thm Trebels2Lift}
 	Let $0 < q \leq \infty$ and $\sigma, \alpha > 0$.
	\begin{enumerate}
	\item[\upshape($i_a$)] Let $1 < p < \infty$ and $f \in L_p(\mathbb{R}^d)$. Then,
	\begin{equation}\label{QLift}
		 \omega_\alpha(J_\sigma f, t)_p \lesssim \left(\int_0^t (u^{-\sigma} \omega_{\alpha + \sigma}(f,u)_p)^{q} \frac{du}{u}\right)^{1/q} \iff q \leq \min\{p,2\}.
	\end{equation}
	
	\item[\upshape($i_b$)] Let $p=1, \infty$ and $f \in L_p(\mathbb{R}^d)$. Then,
	\begin{equation}\label{Qlift2}
		\|J_\sigma f - W^\alpha_t J_\sigma f\|_{L_p(\mathbb{R}^d)} \lesssim \left(\int_0^t (u^{-\sigma} \|f- W^{\alpha + \sigma}_u f\|_{L_p(\mathbb{R}^d)})^{q} \frac{du}{u}\right)^{1/q} \iff q \leq 1.
	\end{equation}

	\item[\upshape($ii_a$)] Let $1 < p < \infty$ and $f \in \dot{\mathscr{L}}^{\sigma}_p(\mathbb{R}^d)$. Then,
	\begin{equation}\label{Qlift3}
		  \left(\int_0^t (u^{-\sigma} \omega_{\alpha + \sigma}(f,u)_p)^q \frac{du}{u}\right)^{1/q} \lesssim  \omega_\alpha(J_\sigma f, t)_p \iff q \geq \max\{p,2\}.
	\end{equation}

\item[\upshape($ii_b$)]
	Let $p=1, \infty$ and $f \in \dot{\mathscr{L}}^{\sigma}_p(\mathbb{R}^d)$. Then,
	\begin{equation}\label{Qlift4}
		\left(\int_0^t (u^{-\sigma} \|f- W^{\alpha + \sigma}_u f\|_{L_p(\mathbb{R}^d)})^{q} \frac{du}{u}\right)^{1/q}  \lesssim  \|J_\sigma f - W^\alpha_t J_\sigma f\|_{L_p(\mathbb{R}^d)} \iff q = \infty.
	\end{equation}

	\end{enumerate}
 \end{thm}

 \begin{proof}
	{\it (i)} It is clear that
	\begin{equation}\label{12.7new2*}
	J_\sigma : \dot{\mathscr{L}}^{\alpha+\sigma}_p(\mathbb{R}^d) \longrightarrow \dot{\mathscr{L}}^{\alpha}_p(\mathbb{R}^d).
	\end{equation}
	 Furthermore, by Lemma \ref{LemmaLift2}(ii) below, we have $J_\sigma : \dot{B}^{\sigma}_{p,q}(\mathbb{R}^d) \longrightarrow \dot{B}^{0}_{p,q}(\mathbb{R}^d)$. Suppose that either $1 < p < \infty, q \leq \min\{p,2\}$ or $p=1, \infty, q \leq 1$. Then
	 \begin{equation}\label{12.8new*}
	  J_\sigma : \dot{B}^{\sigma}_{p,q}(\mathbb{R}^d) \longrightarrow L_{p}(\mathbb{R}^d)
	  \end{equation}
	   because $\dot{B}^{0}_{p,q}(\mathbb{R}^d) \hookrightarrow L_p(\mathbb{R}^d)$ (see (\ref{Prelim7.18}) with $s=b=0$). According to (\ref{12.7new2*}), and (\ref{12.8new*}), we obtain
	   \begin{equation}\label{12.9new2*}
	   	K(t^\alpha, J_\sigma f; L_p(\mathbb{R}^d),  \dot{\mathscr{L}}^{\alpha}_p(\mathbb{R}^d)) \lesssim K(t^\alpha, f; \dot{B}^{\sigma}_{p,q}(\mathbb{R}^d), \dot{\mathscr{L}}^{\alpha+\sigma}_p(\mathbb{R}^d)).
	   \end{equation}
	  Since (cf. \cite[Theorem 6.3.1, p. 147]{BerghLofstrom})
	\begin{equation*}
		 \dot{B}^\sigma_{p,q}(\mathbb{R}^d) = (L_p(\mathbb{R}^d), \dot{\mathscr{L}}^{\alpha+\sigma}_p(\mathbb{R}^d))_{\frac{\sigma}{\alpha+\sigma},q},
	\end{equation*}
 we can apply Lemma \ref{PrelimHolmstedt}(ii) to get
	\begin{align}
		K(t^\alpha, f; \dot{B}^{\sigma}_{p,q}(\mathbb{R}^d), \dot{\mathscr{L}}^{\alpha+\sigma}_p(\mathbb{R}^d))&\asymp \left(\int_0^t (u^{-\sigma} K(u^{\alpha+\sigma},f;L_p(\mathbb{R}^d), \dot{\mathscr{L}}^{\alpha+\sigma}_p(\mathbb{R}^d)))^q \frac{du}{u}\right)^{1/q} \label{KfunctBWHom2}.
	\end{align}
	Inserting this estimate into (\ref{12.9new2*}) we arrive at
	\begin{equation*}
	K(t^\alpha, J_\sigma f; L_p(\mathbb{R}^d),  \dot{\mathscr{L}}^{\alpha}_p(\mathbb{R}^d))  \lesssim  \left(\int_0^t (u^{-\sigma} K(u^{\alpha+\sigma},f;L_p(\mathbb{R}^d), \dot{\mathscr{L}}^{\alpha+\sigma}_p(\mathbb{R}^d)))^q \frac{du}{u}\right)^{1/q}.
	\end{equation*}
	Hence the estimate given in (\ref{QLift}) (respectively, (\ref{Qlift2})) follows from Lemma \ref{LemmaModuli} (respectively, (\ref{5.13new2})).
	
	Conversely, assume that $1 < p < \infty$ and the inequality stated in (\ref{QLift}) holds. Let $f \in \dot{B}^\sigma_{p,q}(\mathbb{R}^d)$. Then,
	\begin{equation}\label{12.12new2*}
		 \omega_\alpha(J_\sigma f, t)_p \lesssim \left(\int_0^\infty (u^{-\sigma} \omega_{\alpha+\sigma}(f,u)_p)^q \frac{du}{u}\right)^{1/q} \asymp \|f\|_{\dot{B}^\sigma_{p,q}(\mathbb{R}^d)}.
	\end{equation}
	Further, it follows from (\ref{AsympMod}) and the trivial inequalities $\omega_\alpha(f,t)_p \lesssim \omega_k(f,t)_p \lesssim \omega_\tau(f,t)_p, 0 < \tau < k < \alpha, k \in \mathbb{N}$, that
	\begin{equation}\label{12.12new2**}
	 \sup_{t > 0} \omega_\alpha(f, t)_p \asymp \|f\|_{L_p(\mathbb{R}^d)}, \quad 1 \leq p < \infty, \alpha > 0.
	 \end{equation}
	 Therefore, (\ref{12.12new2*}) implies that $\dot{B}^\sigma_{p,q}(\mathbb{R}^d) \hookrightarrow \dot{\mathscr{L}}^{\sigma}_p(\mathbb{R}^d)$, or equivalently, $\dot{B}^0_{p,q}(\mathbb{R}^d) \hookrightarrow L_p(\mathbb{R}^d)$ (see Lemma \ref{LemmaLift2}(ii) below). The latter embedding implies $q \leq \min\{p,2\}$ and the part $(i_a)$ is proved.

	 Next we show that if $p=1, \infty$ and the inequality given in (\ref{Qlift2}) holds then it is necessary that $q \leq 1$.
 Indeed, by the homogeneous counterparts of (\ref{BesovWeierstrass}) with $b=0$ and (\ref{BesovComparison}), we observe that
	 \begin{equation}\label{BesovWeierstrassHom}
	 	\|f\|_{\dot{B}^{s}_{p,q}(\mathbb{R}^d)}
\asymp \left(\int_0^\infty t^{-s q}
\|f-W^{\alpha}_{t} f\|_{L_p(\mathbb{R}^d)}^q \frac{dt}{t}\right)^{1/q}, \quad \alpha > s > 0.
	 \end{equation}
	  Therefore, we have
	 \begin{equation}\label{12.12new2***}
	 \|J_\sigma f - W^\alpha_t J_\sigma f\|_{L_p(\mathbb{R}^d)} \lesssim \left(\int_0^\infty (u^{-\sigma} \|f- W^{\alpha + \sigma}_u f\|_{L_p(\mathbb{R}^d)})^{q} \frac{du}{u}\right)^{1/q} \asymp \|f\|_{\dot{B}^\sigma_{p,q}(\mathbb{R}^d)}.
	 \end{equation}
	 On the other hand, combining (\ref{5.13new2}) and (\ref{aux4.12}) we have $ \|J_\sigma f - W^\alpha_t J_\sigma f\|_{L_p(\mathbb{R}^d)} \asymp K(t^\alpha,J_\sigma f;L_p(\mathbb{R}^d), \dot{\mathscr{L}}^{\alpha}_p(\mathbb{R}^d)) \gtrsim \omega_\beta (J_\sigma f, t)_p$ for $\beta > \alpha$ and thus, by (\ref{12.12new2***}), the estimate
$$\omega_\beta (J_\sigma f, t)_p \lesssim  \|f\|_{\dot{B}^\sigma_{p,q}(\mathbb{R}^d)}$$ holds. Letting $t \to \infty$ and using (\ref{12.12new2**}), we find $\dot{B}^\sigma_{p,q}(\mathbb{R}^d) \hookrightarrow \dot{\mathscr{L}}^{\sigma}_p(\mathbb{R}^d)$, or equivalently, $\dot{B}^0_{p,q}(\mathbb{R}^d) \hookrightarrow L_p(\mathbb{R}^d)$ which gives $q \leq 1$.

 	{\it (ii)} Obviously we have
	\begin{equation}\label{12.7new2**}
	J_{-\sigma} : \dot{\mathscr{L}}^{\alpha}_p(\mathbb{R}^d) \longrightarrow \dot{\mathscr{L}}^{\alpha + \sigma}_p(\mathbb{R}^d)
	\end{equation}
	and
	\begin{equation}\label{12.7new2***}
	J_{-\sigma} : L_{p}(\mathbb{R}^d) \longrightarrow \dot{\mathscr{L}}^{\sigma}_p(\mathbb{R}^d).
	\end{equation}
	Assume that one of the following conditions hold: $1 < p < \infty$ and $q \geq \max\{p,2\}$, or $p=1, \infty$ and $q=\infty$. Then, the estimate (\ref{12.7new2***}) implies
	 \begin{equation}\label{12.8new**}
	  J_{-\sigma} : L_{p}(\mathbb{R}^d) \longrightarrow \dot{B}^{\sigma}_{p,q}(\mathbb{R}^d)
	  \end{equation}
	 because $\dot{\mathscr{L}}^{\sigma}_p(\mathbb{R}^d) \hookrightarrow  \dot{B}^{\sigma}_{p,q}(\mathbb{R}^d)$ (see (\ref{Prelim7.18}) if $1 < p < \infty$). Hence, by (\ref{12.7new2**}), (\ref{12.8new**}) and (\ref{KfunctBWHom2}), one gets
	   \begin{align}
	    K(t^\alpha, J_\sigma f; L_p(\mathbb{R}^d),  \dot{\mathscr{L}}^{\alpha}_p(\mathbb{R}^d)) & \gtrsim K(t^\alpha, f; \dot{B}^{\sigma}_{p,q}(\mathbb{R}^d), \dot{\mathscr{L}}^{\alpha+\sigma}_p(\mathbb{R}^d)) \nonumber \\
	    & \asymp \left(\int_0^t (u^{-\sigma} K(u^{\alpha+\sigma},f;L_p(\mathbb{R}^d), \dot{\mathscr{L}}^{\alpha+\sigma}_p(\mathbb{R}^d)))^q \frac{du}{u}\right)^{1/q} \label{12.9new2**}
	   \end{align}
	   where we have also used the fact that $J_{-\sigma} J_{\sigma} f = f$. Note that (\ref{12.9new2**}) gives the desired estimate in (\ref{Qlift3}) by applying Lemma \ref{LemmaModuli} and the corresponding one in (\ref{Qlift4}) as a consequence of (\ref{5.13new2}).
	
	 Next we proceed with the converse statement, that is, we show that if the inequality given in (\ref{Qlift3}) (respectively, (\ref{Qlift4})) holds then $q \geq \max\{p,2\}$ (respectively, $q=\infty$). Assume first that $1< p < \infty$ and the estimate (\ref{Qlift3}) is true. Since $\omega_\alpha(f,t)_p \lesssim \|f\|_{L_p(\mathbb{R}^d)}$, we obtain
	 \begin{equation*}
	 	\left(\int_0^t (u^{-\sigma} \omega_{\alpha + \sigma}(f,u)_p)^q \frac{du}{u}\right)^{1/q} \lesssim \|J_\sigma f\|_{L_p(\mathbb{R}^d)} = \|f\|_{\dot{\mathscr{L}}^\sigma_p(\mathbb{R}^d)}, \quad t > 0,
	 \end{equation*}
	 and letting $t  \to \infty$ we derive $\|f\|_{\dot{B}^\sigma_{p,q}(\mathbb{R}^d)} \lesssim \|f\|_{\dot{\mathscr{L}}^{\sigma}_p(\mathbb{R}^d)}.$ The latter embedding yields that $q \geq \max\{p,2\}$.
	
	 Let $p=1, \infty$. We first remark that $\|f - W^\alpha_t f\|_{L_p(\mathbb{R}^d)} \lesssim \|f\|_{L_p(\mathbb{R}^d)}$ for all $t > 0$. Then, taking into account (\ref{BesovWeierstrassHom}), one can follow the approach given above in order to show that the validity of the inequality in (\ref{Qlift4}) implies $q=\infty$.
	
 \end{proof}

 In the special case $p=2$ in Theorem \ref{Thm Trebels2Lift} we derive the following equivalence result.

 \begin{cor}\label{cor EquivRiesz}
 	Let $\sigma, \alpha > 0$. Then,
	\begin{equation*}
		\omega_\alpha(J_\sigma f, t)_2 \asymp \left(\int_0^t (u^{-\sigma} \omega_{\alpha + \sigma}(f,u)_2)^{2} \frac{du}{u}\right)^{1/2}.
	\end{equation*}
 \end{cor}

 The corresponding result for Bessel potentials reads as follows.
  \begin{thm}\label{Thm Trebels2Lift*}
 	Let $0 < q \leq \infty$ and $\sigma, \alpha > 0$.
	\begin{enumerate}
	\item[\upshape($i_a$)] Let $1 < p < \infty$ and $f \in L_p(\mathbb{R}^d)$. Then,
	\begin{equation*}
		 \min\{1, t^\alpha\} \|I_\sigma f\|_{L_p(\mathbb{R}^d)} + \omega_\alpha(I_\sigma f, t)_p \lesssim  \min\{1, t^\alpha\} \|f\|_{L_p(\mathbb{R}^d)} + \left(\int_0^t (u^{-\sigma} \omega_{\alpha + \sigma}(f,u)_p)^{q} \frac{du}{u}\right)^{1/q}
	\end{equation*}
	if and only if $q \leq \min\{p,2\}$.
	
		\item[\upshape($i_b$)] Let $p=1, \infty$ and $f \in L_p(\mathbb{R}^d)$. Then,
	\begin{align*}
		\min\{1, t^\alpha\} \|I_\sigma f\|_{L_p(\mathbb{R}^d)} +  \|I_\sigma f - W^\alpha_t I_\sigma f\|_{L_p(\mathbb{R}^d)} \\
		& \hspace{-6cm}\lesssim \min\{1, t^\alpha\} \| f\|_{L_p(\mathbb{R}^d)} +  \left(\int_0^t (u^{-\sigma} \|f- W^{\alpha + \sigma}_u f\|_{L_p(\mathbb{R}^d)})^{q} \frac{du}{u}\right)^{1/q}
	\end{align*}
	if and only if $q \leq 1$.

	\item[\upshape($ii_a$)] Let $1 < p < \infty$ and $f \in H^{\sigma}_p(\mathbb{R}^d)$. Then,
	\begin{equation*}
		 \min\{1, t^\alpha\} \| f\|_{L_p(\mathbb{R}^d)}+ \left(\int_0^t (u^{-\sigma} \omega_{\alpha + \sigma}(f,u)_p)^q \frac{du}{u}\right)^{1/q} \lesssim \min\{1, t^\alpha\} \|I_\sigma f\|_{L_p(\mathbb{R}^d)} + \omega_\alpha(I_\sigma f, t)_p
	\end{equation*}
	if and only if $q \geq \max\{p,2\}.$
	
	\item[\upshape($ii_b$)] Let $p=1, \infty$ and $f \in H^{\sigma}_p(\mathbb{R}^d)$. Then,
	\begin{align*}
		 \min\{1, t^\alpha\} \| f\|_{L_p(\mathbb{R}^d)} + \left(\int_0^t (u^{-\sigma} \|f- W^{\alpha + \sigma}_u f\|_{L_p(\mathbb{R}^d)})^{q} \frac{du}{u}\right)^{1/q}  \\
		 & \hspace{-8.5cm} \lesssim   \min\{1, t^\alpha\} \|I_\sigma f\|_{L_p(\mathbb{R}^d)} +  \|I_\sigma f - W^\alpha_t I_\sigma f\|_{L_p(\mathbb{R}^d)}
	\end{align*}
	 if and only if $q = \infty.$
	\end{enumerate}
 \end{thm}

Our method of proof of Theorem \ref{Thm Trebels2Lift} can be carried over from Riesz potentials to Bessel potentials.
 We leave the details of the proof of Theorem \ref{Thm Trebels2Lift*} to the reader.

We write down the analogue of Corollary \ref{cor EquivRiesz} for Bessel potentials.

 \begin{cor}\label{cor EquivBessel}
 	Let $\sigma, \alpha > 0$. Then,
	\begin{equation*}
		\min\{1, t^\alpha\}\|I_\sigma f\|_{L_2(\mathbb{R}^d)} +  \omega_\alpha(I_\sigma f, t)_2 \asymp \min\{1, t^\alpha\}\|f\|_{L_2(\mathbb{R}^d)} + \left(\int_0^t (u^{-\sigma} \omega_{\alpha + \sigma}(f,u)_2)^{2} \frac{du}{u}\right)^{1/2}.
	\end{equation*}
 \end{cor}

\subsection{Lifting property on $\mathbf{B}^{0,b}_{p,q}(\mathbb{R}^d)$}

The study of the boundedness properties of the operators $J_\sigma$ and $I_\sigma$  plays a key role in the theory of function spaces. For the convenience of the reader we recall how the operators $I_\sigma$ and $J_\sigma$ interact with Besov-norms and Triebel-Lizorkin-norms.

\begin{lem}\label{LemmaLift2}
	Let $0 < q \leq \infty$, and $s, \sigma \in \mathbb{R}$.
	
	\begin{enumerate}[\upshape(i)]
	\item Let $1 \leq p \leq \infty$. Then, $I_\sigma$ maps $B^{s}_{p,q}(\mathbb{R}^d)$ isomorphically onto $B^{s-\sigma}_{p,q}(\mathbb{R}^d)$, and
	\begin{equation}\label{18}
		\|I_\sigma f\|_{B^{s-\sigma}_{p,q}(\mathbb{R}^d)} \asymp \|f\|_{B^{s}_{p,q}(\mathbb{R}^d)}.
	\end{equation}
	\item  Let $1 \leq p < \infty$. Then, $I_\sigma$ maps $F^{s}_{p,q}(\mathbb{R}^d)$ isomorphically onto $F^{s-\sigma}_{p,q}(\mathbb{R}^d)$, and
	\begin{equation}\label{18.2}
		\|I_\sigma f\|_{F^{s-\sigma}_{p,q}(\mathbb{R}^d)} \asymp \|f\|_{F^{s}_{p,q}(\mathbb{R}^d)}.
	\end{equation}
	
	\item	The corresponding estimates also hold for the operator $J_\sigma$ when we replace the spaces $B^{s}_{p,q}(\mathbb{R}^d)$ and $F^{s}_{p,q}(\mathbb{R}^d)$ by their homogeneous counterparts $\dot{B}^{s}_{p,q}(\mathbb{R}^d)$ and $\dot{F}^{s}_{p,q}(\mathbb{R}^d)$, respectively.
	\end{enumerate}
\end{lem}

For the proof of Lemma \ref{LemmaLift2} we refer to \cite[Sections 2.3.8 and 5.2.3]{Triebel1}. The extensions of (\ref{18}) and (\ref{18.2}) to the function spaces of generalized smoothness $B^{s,b}_{p,q}(\mathbb{R}^d)$ and $F^{s,b}_{p,q}(\mathbb{R}^d)$ were obtained in \cite[Proposition 1.8]{Moura} (cf. Lemma \ref{LemmaLift}).

In this section we focus on the smoothness properties of the Bessel potential $I_\sigma$ for functions in $\mathbf{B}^{0,b}_{p,q}(\mathbb{R}^d)$. In sharp contrast with the case of functions in $B^{0,b}_{p,q}(\mathbb{R}^d)$ (see Lemmas \ref{LemmaLift2} and \ref{LemmaLift}), we will show that $I_\sigma$ does not act as a lift in the setting of the $\mathbf{B}^{0,b}_{p,q}(\mathbb{R}^d)$ spaces.

 \begin{thm}\label{Theorem 12.1}
 	Let $\sigma \in \mathbb{R}, 0 < q \leq \infty$ and $b > -1/q$.
	\begin{enumerate}[\upshape(i)]
	\item Assume that $1 < p < \infty$. Then
	\begin{equation}\label{19}
		I_\sigma: \mathbf{B}^{0,b}_{p,q}(\mathbb{R}^d) \longrightarrow B^{-\sigma, b+1/\max\{2,p,q\}}_{p,q}(\mathbb{R}^d)
	\end{equation}
	and
	\begin{equation}\label{20}
		I_\sigma: B^{\sigma,b+1/\min\{2,p,q\}}_{p,q}(\mathbb{R}^d) \longrightarrow \mathbf{B}^{0, b}_{p,q}(\mathbb{R}^d).
	\end{equation}
	\item Let $p=1, \infty$. Then
	\begin{equation}\label{19new}
		I_\sigma: \mathbf{B}^{0,b}_{p,q}(\mathbb{R}^d) \longrightarrow B^{-\sigma, b}_{p,q}(\mathbb{R}^d)
	\end{equation}
	and
	\begin{equation}\label{20new}
		I_\sigma: B^{\sigma,b+1/\min\{1,q\}}_{p,q}(\mathbb{R}^d) \longrightarrow \mathbf{B}^{0, b}_{p,q}(\mathbb{R}^d).
	\end{equation}
	\end{enumerate}
 \end{thm}
 \begin{proof}
 We start by showing (\ref{19}). Since $L_p(\mathbb{R}^d) \hookrightarrow B^0_{p, \max\{p,2\}}(\mathbb{R}^d)$ (see (\ref{Prelim7.18}) with $s=b=0$), we get
 \begin{equation*}
 	I_\sigma : L_p(\mathbb{R}^d) \longrightarrow B^{-\sigma}_{p, \max\{p,2\}}(\mathbb{R}^d)
 \end{equation*}
 where we have used (\ref{18}). On the other hand, by (\ref{18}), we know that, in particular, $I_\sigma : B^s_{p,q}(\mathbb{R}^d) \longrightarrow  B^{s-\sigma}_{p,q}(\mathbb{R}^d)$ for $s > 0$. Applying the $((0,b),q)$-interpolation method, we arrive at
 \begin{equation}\label{12.6i}
 	I_\sigma : (L_p(\mathbb{R}^d), B^s_{p,q}(\mathbb{R}^d) )_{(0,b),q} \longrightarrow (B^{-\sigma}_{p, \max\{p,2\}}(\mathbb{R}^d),B^{s-\sigma}_{p,q}(\mathbb{R}^d) )_{(0,b),q}.
 \end{equation}
Concerning the source space in (\ref{12.6i}), we make use of (\ref{BesovComparison}) and the interpolation formula (\ref{PrelimInterpolationnew2}) to get $(L_p(\mathbb{R}^d), B^s_{p,q}(\mathbb{R}^d) )_{(0,b),q} = \mathbf{B}^{0,b}_{p,q}(\mathbb{R}^d)$. As far as the target space in (\ref{12.6i}), it follows from Lemma \ref{PrelimLemmaCF} and Lemma \ref{PrelimLemma7.2}(vi),
\begin{align*}
	(B^{-\sigma}_{p, \max\{p,2\}}(\mathbb{R}^d),B^{s-\sigma}_{p,q}(\mathbb{R}^d) )_{(0,b),q} & = ((B^{-s-\sigma}_{p,q}(\mathbb{R}^d), B^{s-\sigma}_{p,q}(\mathbb{R}^d))_{\frac{1}{2},\max\{p,2\}}, B^{s-\sigma}_{p,q}(\mathbb{R}^d))_{(0,b),q} \\
	& \hspace{-3.5cm}\hookrightarrow (B^{-s-\sigma}_{p,q}(\mathbb{R}^d), B^{s-\sigma}_{p,q}(\mathbb{R}^d))_{1/2,q; b + 1/\max\{p,2,q\}} = B^{-\sigma, b+1/\max\{p,2,q\}}_{p,q}(\mathbb{R}^d).
\end{align*}
This implies that
 \begin{equation*}
 	I_\sigma : \mathbf{B}^{0,b}_{p,q}(\mathbb{R}^d) \longrightarrow B^{-\sigma, b+1/\max\{p,2,q\}}_{p,q}(\mathbb{R}^d).
 \end{equation*}

 The proofs of (\ref{20}), (\ref{19new}) and (\ref{20new}) follow the same steps as above, but now using the embeddings $B^0_{p,\min\{p,2\}}(\mathbb{R}^d) \hookrightarrow L_p(\mathbb{R}^d)$ (see (\ref{Prelim7.18}) with $s=b=0$), and $B^0_{p,1}(\mathbb{R}^d) \hookrightarrow L_p(\mathbb{R}^d) \hookrightarrow B^0_{p,\infty}(\mathbb{R}^d) , p=1, \infty$ (see, e.g., \cite[(3.1.3) and (3.1.6)]{SickelTriebel}).

 \end{proof}

 \begin{rem}
(i) Note that the estimates (\ref{19}) and (\ref{20}) can be also obtained via the relationships (\ref{1}) and (\ref{18}).

(ii) A similar approach to the one given in the proof of Theorem \ref{Theorem 12.1} but using now the reiteration (\ref{limiting interpolation reiteration formula}) allows us to derive the extreme case $b=-1/q$ in Theorem \ref{Theorem 12.1}. Namely, if $1 < p < \infty$, then
\begin{equation*}
		I_\sigma: \mathbf{B}^{0,-1/q}_{p,q}(\mathbb{R}^d) \longrightarrow B^{-\sigma, -1/q+1/\max\{2,p,q\}, 1/\max\{2,p,q\}}_{p,q}(\mathbb{R}^d)
	\end{equation*}
	and
		\begin{equation*}
		I_\sigma: B^{\sigma,-1/q+1/\min\{2,p,q\}, 1/\min\{2,p,q\}}_{p,q}(\mathbb{R}^d) \longrightarrow \mathbf{B}^{0, -1/q}_{p,q}(\mathbb{R}^d).
	\end{equation*}
	If $p=1, \infty$ then
		\begin{equation*}
		I_\sigma: \mathbf{B}^{0,-1/q}_{p,q}(\mathbb{R}^d) \longrightarrow B^{-\sigma, -1/q}_{p,q}(\mathbb{R}^d)
	\end{equation*}
	and
	\begin{equation*}
		I_\sigma: B^{\sigma,-1/q+1/\min\{1,q\}, 1/\min\{1,q\}}_{p,q}(\mathbb{R}^d) \longrightarrow \mathbf{B}^{0, -1/q}_{p,q}(\mathbb{R}^d).
	\end{equation*}
\end{rem}

 The boundedness results given in Theorem \ref{Theorem 12.1} are sharp. More precisely, we have the following

 \begin{prop}\label{Proposition 12.2}
	Let $\sigma \in \mathbb{R}, 1 < p < \infty, 0 < q \leq \infty,$ and $b > -1/q$. Given any $\varepsilon > 0$, there exists $f \in \mathbf{B}^{0,b}_{p,q}(\mathbb{R}^d)$ such that $I_\sigma f \not \in B^{-\sigma, b+1/\max\{2,p,q\}+\varepsilon}_{p,q}(\mathbb{R}^d)$.
 \end{prop}

  \begin{prop}\label{Proposition 12.3}
	Given any $\varepsilon > 0$, there exists $f \in  B^{\sigma, b+1/\min\{2,p,q\}-\varepsilon}_{p,q}(\mathbb{R}^d)$ such that $I_\sigma f \not \in \mathbf{B}^{0,b}_{p,q}(\mathbb{R}^d)$ in any of the following cases:
		\begin{enumerate}[\upshape (a)]
		\item $\sigma \in \mathbb{R}, \frac{2d}{d + 1} < p < \infty, 0 < q \leq \infty, b > -1/q$, and $p= \min \{2,p,q\}$.
		\item $\sigma \in \mathbb{R}, 1 < p < \infty, 0 < q \leq \infty, b > -1/q$, and $2 = \min \{2,p,q\}$.
		\item $\sigma \in \mathbb{R}, 1 < p < \infty, 0 < q \leq \infty, b > -1/q$, and $q = \min \{2,p,q\}$.
	\end{enumerate}
 \end{prop}

 \begin{prop}\label{Proposition 12.4}
 	Let $ p=1, \infty, \sigma \in \mathbb{R}, 0 < q \leq \infty,$ and $b > -1/q$. For any $\varepsilon > 0$, there is a function $f \in  B^{\sigma, b+1/\min\{1,q\}-\varepsilon}_{p,q}(\mathbb{R}^d)$ such that $I_\sigma f \not \in \mathbf{B}^{0,b}_{p,q}(\mathbb{R}^d)$.
 \end{prop}

 \begin{proof}[Proofs of Propositions \ref{Proposition 12.2} and \ref{Proposition 12.3}]

 Let us show that
  there is $f \in \mathbf{B}^{0,b}_{p,q}(\mathbb{R}^d)$ such that $I_\sigma f \not \in B^{-\sigma, b+1/\max\{2,p,q\}+\varepsilon}_{p,q}(\mathbb{R}^d)$. Assume first that $p=\max\{2,p,q\}$ (and, so $p > \frac{2d}{d+1}$). Let
 \begin{equation*}
 	F_0(t) = t^{-d+d/p} (1 + |\log t|)^{-\beta}, \quad t > 0,
 \end{equation*}
 where $b+1/q+1/p < \beta < b + 1/q+1/p + \varepsilon$. Define $f(x) = f_0(|x|)$, where $f_0$ is given by (\ref{3.4new+}). According to Theorem \ref{Theorem 3.2}, we have
 \begin{align*}
 	\|f\|_{\mathbf{B}^{0,b}_{p,q}(\mathbb{R}^d)} & \asymp \left(\int_0^1 (1-\log t)^{-\beta p} \frac{dt}{t}\right)^{1/p}  \\
	& \hspace{1cm}+ \left(\int_1^\infty (1 + \log t)^{b q} \Big(\int_t^\infty (1 + \log u)^{-\beta p} \frac{du}{u}\Big)^{q/p} \frac{dt}{t}\right)^{1/q} \\
	& \asymp \left(\int_0^1 (1-\log t)^{-\beta p} \frac{dt}{t}\right)^{1/p} + \left(\int_1^\infty (1 + \log t)^{(b - \beta + 1/p)q} \frac{dt}{t}\right)^{1/q} < \infty.
 \end{align*}
 Hence, $f \in \mathbf{B}^{0,b}_{p,q}(\mathbb{R}^d)$. On the other hand, taking into consideration (\ref{18}), Theorem \ref{Theorem 3.6} yields that
 \begin{equation*}
 	\|I_\sigma f\|_{B^{-\sigma,b+1/p+\varepsilon}_{p,q}(\mathbb{R}^d)}^q \gtrsim \int_1^\infty (1 + \log t)^{(b+1/p+\varepsilon-\beta)q} \frac{dt}{t} = \infty.
 \end{equation*}
 Then, $I_\sigma f \not \in B^{-\sigma,b+1/p+\varepsilon}_{p,q}(\mathbb{R}^d)$.

 Suppose now that $2 = \max\{2,p,q\}$. Let
\begin{equation*}
	\{b_j\}_{j \in \mathbb{N}_0} = \{(1 + j)^{- \delta}\}_{j \in \mathbb{N}_0}, \quad  b + 1/q +1/2 < \delta < b + 1/q +1/2+ \varepsilon,
\end{equation*}
and
\begin{equation*}
	W(x) \sim \sum_{j=3}^\infty b_j e^{i (2^j - 2) x_1}, \quad x \in \mathbb{R}^d.
\end{equation*}
Let $f = \psi W$ with $\psi \in \mathcal{S}(\mathbb{R}^d) \backslash \{0\}$ satisfying (\ref{4.2}). Applying (\ref{4.6}) we get
\begin{align*}
	\|f\|^q_{\mathbf{B}^{0,b}_{p,q}(\mathbb{R}^d)} & \asymp \sum_{j=3}^\infty (1 + j)^{b q} \left(\sum_{k=j}^\infty |b_k|^2\right)^{q/2}  = \sum_{j=3}^\infty (1 + j)^{b q} \left(\sum_{k=j}^\infty (1 + k)^{- 2 \delta} \right)^{q/2}\\
	& \asymp \sum_{j=3}^\infty (1 + j)^{(b - \delta +1/2) q} < \infty.
\end{align*}
On the other hand, by (\ref{18}) and (\ref{4.3}), we obtain
\begin{equation*}
	\|I_\sigma f\|^q_{B^{-\sigma,b+1/2+\varepsilon}_{p,q}(\mathbb{R}^d)}  \asymp \sum_{j=3}^\infty (1 + j)^{(b+1/2 + \varepsilon) q} |b_j|^q = \sum_{j=3}^\infty (1 + j)^{(b + 1/2 + \varepsilon - \delta) q} = \infty.
\end{equation*}
Thus, $f \in \mathbf{B}^{0,b}_{p,q}(\mathbb{R}^d)$ but $I_\sigma f \not \in B^{-\sigma, b+1/2+\varepsilon}_{p,q}(\mathbb{R}^d)$.

In order to prove the sharpness assertion under the assumption $q = \max\{2,p,q\}$, we will proceed by contradiction. Assume that
 \begin{equation}\label{12.4new2}	
 I_\sigma: \mathbf{B}^{0,b}_{p,q}(\mathbb{R}^d) \longrightarrow B^{-\sigma, b+1/q+\varepsilon}_{p,q} (\mathbb{R}^d) \quad \text{for some} \quad \varepsilon > 0.
 \end{equation}
Moreover, by (\ref{18}) and (\ref{BesovComparison}), we have $I_\sigma : \mathbf{B}^s_{p,q}(\mathbb{R}^d) \longrightarrow B^{s-\sigma}_{p,q}(\mathbb{R}^d)$ for any $s > 0$. Given $\theta \in (0,1)$, applying the interpolation property we derive the boundedness of the operator
  \begin{equation}\label{12.4}	
 I_\sigma: (\mathbf{B}^{0,b}_{p,q}(\mathbb{R}^d), \mathbf{B}^s_{p,q}(\mathbb{R}^d))_{\theta,q} \longrightarrow (B^{-\sigma, b+1/q+\varepsilon}_{p,q} (\mathbb{R}^d),B^{s-\sigma}_{p,q}(\mathbb{R}^d))_{\theta,q}.
 \end{equation}
 Note that, by Lemma \ref{PrelimLemmaCF}, we have
 \begin{equation}\label{12.5}
  (B^{-\sigma, b+1/q+\varepsilon}_{p,q} (\mathbb{R}^d),B^{s-\sigma}_{p,q}(\mathbb{R}^d))_{\theta,q} = B^{\theta s - \sigma, (1-\theta)(b+1/q+\varepsilon)}_{p,q}(\mathbb{R}^d).
  \end{equation}
On the other hand, by (\ref{PrelimInterpolationnew3}) and (\ref{BesovComparison}), we have
\begin{equation}\label{12.6}
	(\mathbf{B}^{0,b}_{p,q}(\mathbb{R}^d), \mathbf{B}^s_{p,q}(\mathbb{R}^d))_{\theta,q}  = B^{\theta s, (1-\theta)(b+1/q)}_{p,q}(\mathbb{R}^d).
\end{equation}
It follows from (\ref{12.4}),  (\ref{12.5}) and  (\ref{12.6}) that
\begin{equation*}
	I_\sigma : B^{\theta s, (1-\theta)(b+1/q)}_{p,q}(\mathbb{R}^d) \longrightarrow B^{\theta s - \sigma, (1-\theta)(b+1/q+\varepsilon)}_{p,q}(\mathbb{R}^d),
\end{equation*}
and then, by (\ref{18}),  $B^{\theta s - \sigma, (1-\theta)(b+1/q)}_{p,q}(\mathbb{R}^d) \hookrightarrow B^{\theta s - \sigma, (1-\theta)(b+1/q+\varepsilon)}_{p,q}(\mathbb{R}^d)$, which is clearly not true because $\varepsilon > 0$ (see Proposition \ref{RecallEmb*optim}). Hence, (\ref{12.4new2}) does not hold. This finishes the proof of Proposition \ref{Proposition 12.2}.

Similar arguments can be carried out to prove Proposition \ref{Proposition 12.3}. The details are left to the reader.

\end{proof}

The proof of Proposition \ref{Proposition 12.4} given below requires some decomposition methods in Besov spaces. We need some preparation. For $\nu \in \mathbb{N}_0$ and $m \in \mathbb{Z}^d$, we denote by $Q_{\nu m}$\index{\bigskip\textbf{Sets}!$Q_{\nu m}$}\label{CUBE} the cube in $\mathbb{R}^d$ with sides parallel to the coordinate axes, centred at $2^{-\nu} m$, and with side length $2^{-\nu}$. For a cube $Q$ in $\mathbb{R}^d$ and $r > 0$ we shall mean by $r Q$ the cube in $\mathbb{R}^d$ concentric with $Q$ and with side length $r$ times the side length of $Q$.

Let $K \in \mathbb{N}_0$ and $c > 1$. A $K$ times differentiable complex-valued function $a$ on $\mathbb{R}^d$ (continuous if $K=0$) is called an \emph{$1_K$-atom} if
\begin{equation*}
	\text{supp } a \subset c \, Q_{0 m} \text{ for some } m \in \mathbb{Z}^d
\end{equation*}
and
\begin{equation*}
	|D^\alpha a (x)| \leq 1 \text{ for } |\alpha| \leq K.
\end{equation*}
Let $1 \leq p \leq \infty, -\infty < s, b < \infty,$ and $L+ 1 \in \mathbb{N}_0$. A $K$ times differentiable complex-valued function $a$ on $\mathbb{R}^d$ (continuous if $K=0$) is called an \emph{$(s,p,b)_{K,L}$-atom} if for some $\nu \in \mathbb{N}_0$
\begin{equation*}
	\text{supp } a \subset c \, Q_{\nu m} \text{ for some } m \in \mathbb{Z}^d,
\end{equation*}
\begin{equation*}
	|D^\alpha a (x)| \leq 2^{-\nu (s-d/p) + |\alpha| \nu} (1 + \nu)^{-b} \text{ for } |\alpha| \leq K
\end{equation*}
and
\begin{equation}\label{moment}
	\int_{\mathbb{R}^d} x^\beta a(x) \, d x = 0 \text{ if } |\beta| \leq L.
\end{equation}
If the atom $a$ is located on $Q_{\nu m}$, that means
\begin{equation*}
	\text{supp } a \subset c \,Q_{\nu m},
\end{equation*}
then we write $a_{\nu m}$\index{\bigskip\textbf{Functionals and functions}!$a_{\nu m}$}\label{ATOM}. If $L=-1$ in (\ref{moment}) then there are no moment conditions. The atomic characterizations of classical Besov and Triebel-Lizorkin spaces can be found in \cite{FrazierJawerth} and \cite[Section 13]{Triebel97}. For distribution spaces of generalized smoothness, we refer to \cite{Moura, FarkasLeopold}.

Let $0 < q \leq \infty$. Let $K \in \mathbb{N}_0$ and $L+1 \in \mathbb{N}_0$ with
\begin{equation*}
	K \geq \max\{0, 1 + [s]\} \text{ and } L \geq \max\{-1, [- s]\}
\end{equation*}
be fixed. Then $f \in \mathcal{S}'(\mathbb{R}^d)$ belongs to $B^{s,b}_{p,q}(\mathbb{R}^d)$ if, and only if, there is a representation
\begin{equation}\label{atom1}
	f = \sum_{\nu=0}^\infty \sum_{m \in \mathbb{Z}^d} \lambda_{\nu m} a_{\nu m} \text{ (convergence in $\mathcal{S}'(\mathbb{R}^d)$)},
\end{equation}
where $a_{\nu m}$ are $1_K$-atoms ($\nu = 0$) or $(s,p,b)_{K,L}$-atoms ($\nu \in \mathbb{N}$) such that they are located at $Q_{\nu m}$ and
$\sum_{\nu =0}^\infty \left(\sum_{m \in \mathbb{Z}^d} |\lambda_{\nu m}|^p\right)^{q/p}<\infty$. Furthermore, we have
\begin{equation}\label{atom2}
	\|f\|_{B^{s,b}_{p,q}(\mathbb{R}^d)} \asymp \inf \left(\sum_{\nu =0}^\infty \left(\sum_{m \in \mathbb{Z}^d} |\lambda_{\nu m}|^p\right)^{q/p}\right)^{1/q}
\end{equation}
where the infimum is taken over all admissible representations (\ref{atom1}).

\begin{proof}[Proof of Proposition \ref{Proposition 12.4}]
	Let $p=1$. In this case the proof is split into two possible cases. First, let us assume that $q \geq 1$ and
	\begin{equation}\label{contradiction}
		I_\sigma: B^{\sigma, b+1-\varepsilon}_{1,q}(\mathbb{R}^d) \longrightarrow \mathbf{B}^{0,b}_{1,q}(\mathbb{R}^d) \quad \text{for some } \quad \varepsilon > 0.
	\end{equation}
	Let $\Omega$ be a bounded domain in $\mathbb{R}^d$. We denote by $J_\Omega$ the corresponding restriction operator given by $J_\Omega f = f|\Omega$. Note that, by (\ref{PrelimInterpolationnew2}), (\ref{BesovComparison}), and Lemma \ref{LemBianchiniContinuous},
	\begin{equation*}
		J_\Omega : \mathbf{B}^{0,b}_{1,q}(\mathbb{R}^d) = (L_1(\mathbb{R}^d), B^{d}_{1,1}(\mathbb{R}^d))_{(0,b),q} \longrightarrow (L_1(\Omega), L_\infty(\Omega))_{(0,b),q}
	\end{equation*}
	and, by (\ref{18}),
	\begin{equation*}
		I_{-\sigma} : B^{0, b+1-\varepsilon}_{1,q}(\mathbb{R}^d)\longrightarrow B^{\sigma, b+1-\varepsilon}_{1,q}(\mathbb{R}^d).
	\end{equation*}
	Consequently, we have
	\begin{equation}\label{contradiction2}
		 J_\Omega : B^{0, b+1-\varepsilon}_{1,q}(\mathbb{R}^d) \longrightarrow (L_1(\Omega), L_\infty(\Omega))_{(0,b),q}.
	\end{equation}
	Applying the well-known formula (cf. \cite[Theorem 5.2.1, p. 109]{BerghLofstrom})
	\begin{equation*}
		K(t,f; L_1(\Omega), L_\infty(\Omega)) = \int_0^t f^\ast (s) \, ds
	\end{equation*}
	where $f^\ast$ denotes the non-increasing rearrangement of $f$ given by \index{\bigskip\textbf{Functionals and functions}!$f^\ast$}\label{REARRANGEMENT}
	(cf. \cite{ChongRice})
\begin{equation}\label{rearrangement2}
	f^\ast (t) = \sup_{\substack{E \subset \Omega \\ |E|_d = t}} \inf_{x \in E} |f(x)|, \quad 0 < t < 1,
\end{equation}
	we may rewrite (\ref{contradiction2}) as
	\begin{equation}\label{contradiction2*}
		\left(\int_0^1 (1 - \log t)^{b q} \Big(\int_0^t (J_\Omega f)^\ast (s) \, ds\Big)^q \frac{dt}{t}\right)^{1/q} \lesssim \|f\|_{B^{0, b+1-\varepsilon}_{1,q}(\mathbb{R}^d)}.
	\end{equation}
	However, the latter inequality is not true. Indeed, let
	\begin{equation}\label{contradiction3}
		f(x) = \sum_{j=1}^\infty a_j 2^{j d} (1 + j)^{-(b+1 - \varepsilon)} (\Phi(2^j x) - \Phi (2^j x - x^0)), \quad x \in \mathbb{R}^d,
	\end{equation}
	where the function $\Phi$ is given by
		\begin{equation*}
    \Phi(x) = \left\{\begin{array}{lcl}
                            e^{- \frac{1}{1 - |x|^2}} & ,  & |x| < 1, \\
                            & & \\
                            0 & , & |x| \geq 1,
            \end{array}
            \right.
	\end{equation*}
	 and $x^0 \in \mathbb{R}^d$ with $|x^0| > 3$. By construction, $2^{j d} (1 + j)^{-(b + 1 - \varepsilon)} (\Phi(2^j x) - \Phi (2^j x - x^0))$ is an $(0, 1, b + 1 -\varepsilon)_{K,0}$-atom in $B^{0,b+1 -\varepsilon}_{1,q}(\mathbb{R}^d)$, where $K \geq 1$. Since (\ref{contradiction3}) is an atomic decomposition of $f$ in $B^{0,b+1-\varepsilon}_{1,q}(\mathbb{R}^d)$, applying (\ref{atom2}) we find
	\begin{equation*}
		\|f\|_{B^{0,b+1-\varepsilon}_{1,q}(\mathbb{R}^d)}^q \lesssim \sum_{j=1}^\infty |a_j|^q.
	\end{equation*}
	This yields that if
	\begin{equation*}
		a_j = (1 + j)^{-\beta}, \quad \max\{1/q, -b + \varepsilon\} < \beta < 1/q + \varepsilon,
	\end{equation*}
	then the function $f$ belongs to $B^{0,b+1-\varepsilon}_{1,q}(\mathbb{R}^d)$. Moreover, there holds that (see (\ref{rearrangement2}))
	\begin{equation*}
		f^\ast(2^{-k d})  \geq \inf_{x \in Q_{k, (0, \ldots, 0)}} |f(x)| = \inf_{x \in Q_{k, (0, \ldots, 0)} \backslash  Q_{k+1, (0, \ldots, 0)}} |f(x)| \asymp  2^{k d} (1 + k)^{-(b + 1 - \varepsilon)-\beta}
	\end{equation*}
	for $k \in \mathbb{N}.$ Let $\Omega = 2^{-1} [0,1]^d$. By construction, the support of $f$ is contained in $\Omega$, and then $J_\Omega f = f$. Using monotonicity properties of $f^\ast$, we can estimate the left-hand side in (\ref{contradiction2*}) by
	\begin{align*}
		\int_0^1 (1 - \log t)^{b q} \Big(\int_0^t f^\ast (s) \, ds\Big)^q \frac{dt}{t} & \gtrsim \sum_{j=0}^\infty (1 + j)^{b q} \left(\sum_{k=j+1}^\infty f^\ast( 2^{-k d}) 2^{-k d}\right)^q \\
		&\hspace{-4cm} \gtrsim \sum_{j=0}^\infty (1 + j)^{b  q} \left(\sum_{k=j+1}^\infty (1 + k)^{-(b+1-\varepsilon) - \beta}\right)^q \asymp \sum_{j=0}^\infty (1 + j)^{(\varepsilon - \beta)q} = \infty.
	\end{align*}
	This contradicts (\ref{contradiction2*}). Hence, (\ref{contradiction}) is not true.
	
	Secondly, if $q < 1$ then we will proceed by contradiction, that is, the boundedness result
	\begin{equation}\label{contradiction4}
		I_\sigma: B^{\sigma, b+1/q-\varepsilon}_{1,q}(\mathbb{R}^d) \longrightarrow \mathbf{B}^{0,b}_{1,q}(\mathbb{R}^d) \quad \text{for some } \quad \varepsilon > 0
	\end{equation}
	will yield a contradiction. Indeed, by (\ref{18}) and (\ref{BesovComparison}), we have $I_\sigma : B^{s+\sigma}_{1,q}(\mathbb{R}^d) \longrightarrow \mathbf{B}^s_{1,q}(\mathbb{R}^d) $ for any $s > 0$. Given $\theta \in (0,1)$, applying the interpolation property we derive the boundedness of the operator
  \begin{equation}\label{12.4new}	
 I_\sigma: (B^{\sigma, b+1/q-\varepsilon}_{1,q} (\mathbb{R}^d),B^{s+\sigma}_{1,q}(\mathbb{R}^d))_{\theta,q}  \longrightarrow (\mathbf{B}^{0,b}_{1,q}(\mathbb{R}^d), \mathbf{B}^s_{1,q}(\mathbb{R}^d))_{\theta,q} .
 \end{equation}
 By Lemma \ref{PrelimLemmaCF}, we have
 \begin{equation}\label{12.5new}
  (B^{\sigma, b+1/q-\varepsilon}_{1,q} (\mathbb{R}^d),B^{s+\sigma}_{1,q}(\mathbb{R}^d))_{\theta,q} = B^{\theta s + \sigma, (1-\theta)(b+1/q-\varepsilon)}_{1,q}(\mathbb{R}^d).
  \end{equation}
On the other hand, using (\ref{PrelimInterpolationnew3}) and (\ref{BesovComparison}), we have
\begin{equation} \label{12.6new}
	(\mathbf{B}^{0,b}_{1,q}(\mathbb{R}^d), \mathbf{B}^s_{1,q}(\mathbb{R}^d))_{\theta,q}  = B^{\theta s, (1-\theta)(b+1/q)}_{1,q}(\mathbb{R}^d).
\end{equation}
Therefore, (\ref{12.4new}),  (\ref{12.5new}) and  (\ref{12.6new}) imply that
\begin{equation*}
	I_\sigma : B^{\theta s+\sigma, (1-\theta)(b+1/q-\varepsilon)}_{1,q}(\mathbb{R}^d) \longrightarrow B^{\theta s , (1-\theta)(b+1/q)}_{1,q}(\mathbb{R}^d),
\end{equation*}
and then, by (\ref{18}),  $B^{\theta s + \sigma, (1-\theta)(b+1/q - \varepsilon)}_{1,q}(\mathbb{R}^d) \hookrightarrow B^{\theta s + \sigma, (1-\theta)(b+1/q)}_{1,q}(\mathbb{R}^d)$. It is clear that the latter embedding is no longer true because $\varepsilon > 0$ (see Proposition \ref{RecallEmb*optim}). Hence, (\ref{contradiction4}) cannot hold.

It remains to show the sharpness assertion for $p=\infty$. We use again the method of contradiction. Suppose that
		\begin{equation}\label{contradiction5}
		I_\sigma: B^{\sigma, b+1/\min\{1,q\}-\varepsilon}_{\infty,q}(\mathbb{R}^d) \longrightarrow \mathbf{B}^{0,b}_{\infty,q}(\mathbb{R}^d) \quad \text{for some } \quad \varepsilon > 0.
	\end{equation}
	Since $I_{-\sigma} : B^{0, b+1/\min\{1,q\}-\varepsilon}_{\infty,q}(\mathbb{R}^d)  \longrightarrow B^{\sigma, b+1/\min\{1,q\}-\varepsilon}_{\infty,q}(\mathbb{R}^d) $ (see (\ref{18})), we derive
	\begin{equation}\label{contradiction6}
		B^{0, b+1/\min\{1,q\}-\varepsilon}_{\infty,q}(\mathbb{R}^d) \hookrightarrow  \mathbf{B}^{0,b}_{\infty,q}(\mathbb{R}^d) .
	\end{equation}
	Note that the space $ \mathbf{B}^{0,b}_{\infty,q}(\mathbb{R}^d)$ fits into the scale of the so-called generalized H\"older spaces investigated in \cite{MouraNevesSchneider}. In particular, embeddings of Besov spaces of generalized smoothness into generalized H\"older spaces are fully characterized in  \cite[Theorem 3.2]{MouraNevesSchneider}. For sake of completeness, we write down this characterization in the special case of the spaces $B^{0,b}_{\infty,q}(\mathbb{R}^d)$ and $ \mathbf{B}^{0,b}_{\infty,q}(\mathbb{R}^d)$.
	
	\begin{lem}[\cite{MouraNevesSchneider}]\label{MNS}
		Let $0 < q \leq \infty, -\infty < \xi < \infty$ and $b > -1/q$. Then
		\begin{equation*}
			B^{0,\xi}_{\infty,q}(\mathbb{R}^d) \hookrightarrow \mathbf{B}^{0,b}_{\infty,q}(\mathbb{R}^d) \iff   \xi \geq b + \frac{1}{\min\{1,q\}}.
		\end{equation*}
	\end{lem}
	
	According to Lemma \ref{MNS}, the embedding (\ref{contradiction6}) does not hold because $\varepsilon > 0$. Consequently, (\ref{contradiction5}) is not valid.

\end{proof}

\newpage
\section{Regularity estimates of the fractional Laplace operator}\label{section-fr-lap}

The fractional Laplace operator $(-\Delta)^s$ for $s \in (0,1)$ is defined by\index{\bigskip\textbf{Operators}!$(-\Delta)^{s}$}\label{FRACTLAPLACE}
\begin{equation}\label{LaplaceRiesz}
	(-\Delta)^{s} f(x)= c_{d,s} \, \text{P.V.} \int_{\mathbb{R}^d} \frac{f(x)-f(y)}{|x-y|^{d + 2 s}} dy, \quad f \in  \mathcal{S}(\mathbb{R}^d),
\end{equation}
where
 $c_{d,s} > 0$ is a normalization constant depending only on $d$ and $s$.


There are several equivalent approaches in the literature
to introduce this operator.  
  For instance, it can also be defined as the pseudodifferential operator (see (\ref{RieszPotential2}))
 $$
 \widehat{(-\Delta)^{s} f}(\xi) = \widehat{J_{2s} f}(\xi) = |\xi|^{2 s} \widehat{f}(\xi) , \quad  f \in  \mathcal{S}(\mathbb{R}^d),
 $$
 see, e.g., \cite{Landkof}. This formula  is useful to derive properties of the fractional Laplacian  on the whole space $\mathbb{R}^d$. However, in order to obtain local estimates for $(-\Delta)^{s}$, it is more convenient to work with the characterization in terms of the heat semi-group $W_t$ (cf. (\ref{WeierstrassSemiGroup})). Namely,
\begin{equation}\label{LaplaceRiesz2}
	(-\Delta)^{s} f(x)= \frac{1}{\Gamma(-s)} \int_0^\infty (W_t f(x)-f(x)) \frac{dt}{t^{1+s}},
\end{equation}
see, e.g., \cite{Stinga, StingaTorrea}.

There are several possible extensions of the fractional Laplace operator to spaces formed by locally integrable functions. Throughout this section we shall deal with the extension of (\ref{LaplaceRiesz}) to the weighted $L_1$-space\index{\bigskip\textbf{Spaces}!$L_1(s)$}\label{WEIGHTEDL1}
\begin{equation}\label{weightedL1}
	L_1(s) =\left\{f : \mathbb{R}^d \longrightarrow \mathbb{R}: \|f\|_{L_1(s)} = \int_{\mathbb{R}^d} \frac{|f(x)|}{1 + |x|^{d+2 s}} d x < \infty \right\}.
\end{equation}
 See the works of Silvestre \cite{Silvestre1, Silvestre} for further details.


\subsection{Function spaces in PDE's}

The goal of this section is to clarify the relationships between the function spaces introduced in Section \ref{Section 2.2} and some function spaces that   arise naturally in the study of regularity problems for the fractional Laplacian. 

Let $-\infty <s,  b < \infty$. The space $\Lambda^{s,b}(\mathbb{R}^d)$ is formed by all $f \in \mathcal{S}'(\mathbb{R}^d)$ such that\index{\bigskip\textbf{Spaces}!$\Lambda^{s,b}(\mathbb{R}^d)$}\label{LIPSCHITZHEAT}
\begin{equation}\label{def:LambdaSpace}
	|f|_{\Lambda^{s,b}(\mathbb{R}^d)} =  \sup_{0 < t < \infty} t^{-s/2} (1+|\log t|)^b \Big\|t^k \frac{\partial^k W_t f}{\partial t^k} \Big\|_{L_\infty(\mathbb{R}^d)} < \infty.
\end{equation}
The space $\Lambda^{s,b}(\mathbb{R}^d)$ is equipped with the norm
\begin{equation}\label{def:LambdaSpace2}
	\|f\|_{\Lambda^{s,b}(\mathbb{R}^d)} = \|f_0\|_{L_\infty(\mathbb{R}^d)} + |f|_{\Lambda^{s,b}(\mathbb{R}^d)}
\end{equation}
where
\begin{equation*}
	f_0 = \left\{\begin{array}{cl}  f & \text{if} \quad  s > 0, \\
	& \\
		I_{-2m}f  & \text{if} \quad s \leq 0.
		       \end{array}
                        \right.
\end{equation*}
Here, $k \in \mathbb{N}_0$ with $k > s/2$ and $m \in \mathbb{N}$ with $m > -s/2$. It will be shown in Proposition \ref{lem:LambdaBesov}(ii) below that different values of $k$ and $m$ give equivalent norms on $\Lambda^{s,b}(\mathbb{R}^d)$.
Note that if $b=0$ in $\Lambda^{s,b}(\mathbb{R}^d)$ then we recover the classical spaces $\Lambda^s(\mathbb{R}^d)$.

Let $0 \leq s \leq 1, -\infty < b < \infty$, and $k \in \mathbb{N}_0$. The H\"older space $C^{k,s,b}(\mathbb{R}^d)$\index{\bigskip\textbf{Spaces}!$C^{k,s,b}(\mathbb{R}^d)$}\label{HOLDER} is the collection of all $f \in C^k (\mathbb{R}^d)$ such that
\begin{equation}\label{def:HolderSpace}
	\sup_{\substack{x, y \in \mathbb{R}^d \\ 0 < |x-y| < 1}} \frac{(1 - \log |x-y|)^b |D^\beta f(x)- D^\beta f(y)|}{|x-y|^s} < \infty
\end{equation}
for all $|\beta| = k$. If $b=0$ we obtain the classical spaces $C^{k,s}(\mathbb{R}^d)$. One can easily check that $C^{k,0,b}(\mathbb{R}^d) = C^k(\mathbb{R}^d)$ if $b \leq 0$ and $C^{k,1,b}(\mathbb{R}^d)=\{0\}$ if $b > 0$.

Let $-\infty < b < \infty$. The Zygmund space $\mathcal{Z}^{1,b}(\mathbb{R}^d)$\index{\bigskip\textbf{Spaces}!$\mathcal{Z}^{1,b}(\mathbb{R}^d)$}\label{ZYGMUND} is defined as the set of all $f \in C(\mathbb{R}^d)$ such that
\begin{equation*}
	\|f\|_{\mathcal{Z}^{1,b}(\mathbb{R}^d)} = \|f\|_{C(\mathbb{R}^d)} + \sup_{\substack{x, h \in \mathbb{R}^d \\ 0 < |h| < 1}} \frac{(1 -\log |h|)^{b}|\Delta^2_h f (x)|}{|h|} < \infty.
\end{equation*}
See \cite[Definition 2.18, p. 23]{HaroskeBook}. Obviously, if $b=0$ we recover the classical space $\mathcal{Z}^{1}(\mathbb{R}^d)$.

Now we study the relationships between Lipschitz, Besov, H\"older, and Zygmund spaces involving logarithmic smoothness. In particular, we complement and extend the assertions stated in \cite[Proposition 4.3]{RoncalStinga} by using the properties of moduli of smoothness studied in previous sections and avoiding the semi-group techniques applied in \cite{RoncalStinga}.

We start with investigating basic  properties of $\Lambda$ spaces.
\begin{prop}\label{lem:LambdaBesov}

	\begin{enumerate}[\upshape(i)]
		\item
	Let $-\infty < s, b < \infty$. Then, we have
	\begin{equation}\label{lem:LambdaBesov*}
		\Lambda^{s,b}(\mathbb{R}^d) = B^{s,b}_{\infty,\infty}(\mathbb{R}^d)
	\end{equation}
	with equivalence of norms.

		\item 	Let $-\infty < s, b < \infty$. The norm  $\|f\|_{\Lambda^{s,b}(\mathbb{R}^d)}$ given by
  (\ref{def:LambdaSpace2})
does not depend on $k \in \mathbb{N}_0$ with $k > s/2$ and $m \in \mathbb{N}$ with $m > -s/2$.
		\item Let $-\infty < s_0, s_1, b_0, b_1 < \infty$. Then,
		\begin{equation*}
			\Lambda^{s_0,b_0}(\mathbb{R}^d) \hookrightarrow \Lambda^{s_1,b_1}(\mathbb{R}^d)
		\end{equation*}
		if and only if one of the following conditions holds:
		\begin{enumerate}[\upshape(a)]
			\item $s_0 > s_1$,
			\item $s_0 = s_1$ and $b_0 \geq b_1$.

		\end{enumerate}
\item The following Sobolev-type characterization for $\Lambda^{s,b}(\mathbb{R}^d)$ holds: Let $-\infty <s, b < \infty$ and $m \in \mathbb{N}$. Then,
			\begin{equation}\label{LambdaDifferences}
				\|f\|_{\Lambda^{s,b}(\mathbb{R}^d)} \asymp \sum_{|\beta| \leq m} \|D^\beta f\|_{\Lambda^{s-m,b}(\mathbb{R}^d)}.
			\end{equation}
	\end{enumerate}

\end{prop}
\begin{rem}
 If $b=0$, then  (\ref{lem:LambdaBesov*}) 
  is  well known; see, e.g., \cite[Theorem 2.12.2, p. 184]{Triebel1} and \cite[2.6.4, pp. 151--155]{Triebel92}. However the related arguments given in \cite{Triebel1} and \cite{Triebel92} have no immediate log-counterparts. We give an elementary proof of part (i) 
   based on the heat kernel representation for the spaces $\mathbf{B}^{s,b}_{\infty,\infty}(\mathbb{R}^d)$ given by (\ref{BesovHeat}) and lifting properties of Besov spaces.

\end{rem}

\begin{proof}[Proof of Proposition \ref{lem:LambdaBesov}]
(i), (ii):
	Let $k \in \mathbb{N}$. We will make use of the following well-known formulas
	\begin{equation}\label{lem:LambdaBesov1}
		\Big\|t^k \frac{\partial^k W_t f}{\partial t^k} \Big\|_{L_\infty(\mathbb{R}^d)} \lesssim \|[W_t - \text{id}_{\mathbb{R}^d}]^k f\|_{L_\infty(\mathbb{R}^d)}
	\end{equation}
	and
	\begin{equation}\label{lem:LambdaBesov2}
		 \|[W_t - \text{id}_{\mathbb{R}^d}]^k f\|_{L_\infty(\mathbb{R}^d)} \lesssim \int_0^t \Big\|u^k \frac{\partial^k W_u f}{\partial u^k} \Big\|_{L_\infty(\mathbb{R}^d)} \frac{du}{u}.
	\end{equation}
	See \cite[Lemmas 3.5.4 and 3.5.5]{ButzerBerens}.
	
	Assume first that $s > 0$. Let $k \in \mathbb{N}$ with $k >s/2$. By (\ref{def:LambdaSpace}), (\ref{lem:LambdaBesov1}), (\ref{BesovHeat}) and (\ref{BesovComparison}), we get
	\begin{align*}
		\|f\|_{\Lambda^{s,b}(\mathbb{R}^d)} &\lesssim  \|f\|_{ L_\infty(\mathbb{R}^d)} + \sup_{0 < t < \infty} t^{-s/2} (1+|\log t|)^b \|[W_t - \text{id}_{\mathbb{R}^d}]^k f\|_{L_\infty(\mathbb{R}^d)} \\
		& \hspace{-1.5cm}\lesssim  \|f\|_{ L_\infty(\mathbb{R}^d)} + \sup_{0 < t < 1} t^{-s/2} (1-\log t)^b \|[W_t - \text{id}_{\mathbb{R}^d}]^k f\|_{L_\infty(\mathbb{R}^d)} \\
		&  \hspace{-1.5cm}  \asymp \|f\|_{\mathbf{B}^{s,b}_{\infty,\infty}(\mathbb{R}^d)} \asymp \|f\|_{B^{s,b}_{\infty,\infty}(\mathbb{R}^d)}.
	\end{align*}
	
	To show the converse inequality, we apply (\ref{BesovComparison}), (\ref{BesovHeat}), (\ref{lem:LambdaBesov2}), 
  and (\ref{def:LambdaSpace}),
	\begin{align*}
		\|f\|_{B^{s,b}_{\infty,\infty}} & \asymp  \|f\|_{ L_\infty(\mathbb{R}^d)} + \sup_{0 < t < 1} t^{-s/2} (1-\log t)^b \|[W_t - \text{id}_{\mathbb{R}^d}]^k f\|_{L_\infty(\mathbb{R}^d)} \\
		& \lesssim \|f\|_{ L_\infty(\mathbb{R}^d)} + \sup_{0 < t < 1} t^{-s/2} (1-\log t)^b \int_0^t \Big\|u^k \frac{\partial^k W_u f}{\partial u^k} \Big\|_{L_\infty(\mathbb{R}^d)} \frac{du}{u} \\
		& \lesssim \|f\|_{ L_\infty(\mathbb{R}^d)} + \sup_{0 < t < 1} t^{-s/2} (1-\log t)^b \Big\|t^k \frac{\partial^k W_t f}{\partial t^k} \Big\|_{L_\infty(\mathbb{R}^d)}  \leq \|f\|_{\Lambda^{s,b}(\mathbb{R}^d)}.
	\end{align*}
	Furthermore, this also showed (ii) for $\Lambda^{s,b}(\mathbb{R}^d)$ with $s > 0$.
	
	Next we will prove (\ref{lem:LambdaBesov*}) with $s \leq 0$. Let $m \in \mathbb{N}$ with $m > -s/2$ and $k \in \mathbb{N}_0$ with $k > s/2$. For $s + 2 m > 0$  we have already shown that $B^{s+2m,b}_{\infty,\infty}(\mathbb{R}^d) = \Lambda^{s+2m,b}(\mathbb{R}^d)$. Therefore, in virtue of lifting property of Besov spaces (see Lemma \ref{LemmaLift}) and (\ref{def:LambdaSpace}), we have
	\begin{align}
		\|f\|_{B^{s,b}_{\infty,\infty}(\mathbb{R}^d)} & \asymp \|I_{-2 m} f\|_{B^{s+2m,b}_{\infty,\infty}(\mathbb{R}^d)} \asymp \|I_{-2 m} f\|_{\Lambda^{s+2m,b}(\mathbb{R}^d)} \nonumber\\
		&  = \|I_{-2 m}f\|_{ L_\infty(\mathbb{R}^d)} + \sup_{0 < t < \infty} t^{-\frac{s}{2}} (1+|\log t|)^b \Big\|t^{k} \frac{\partial^{m+k} W_t I_{-2 m} f}{\partial t^{m+k}} \Big\|_{L_\infty(\mathbb{R}^d)} \label{Heat*}
	\end{align}
	because $m+k >  m + s/2$. On the other hand, since $W_t f$ is a solution of the heat equation, that is, $\frac{\partial W_t f}{\partial t} = \Delta W_t f$ (see Section \ref{balls3}), using the formula $\Delta W_t f = W_t \Delta f$, it is plain to see that $\frac{\partial^{m+k} W_t f}{\partial t^{m+k}} = \Delta^{m+k} W_t f$, which allows us to derive
	\begin{equation}\label{Heat**}
		 \frac{\partial^{m+k} W_t I_{-2 m} f}{\partial t^{m+k}} = \Delta^{m+k} W_t I_{-2m}f = \Delta^{m+k} I_{-2m} W_t f = \Delta^{m} ( \text{id}_{\mathbb{R}^d} + \Delta)^{-m}  \frac{\partial^{k} W_t f}{\partial t^{k}}.
	\end{equation}
	Finally, making use of the fact that $ \Delta^{m} ( \text{id}_{\mathbb{R}^d} + \Delta)^{-m}$ is a bounded operator on $L_\infty(\mathbb{R}^d)$ (see \cite[Chapter V, Section 3.2, page 133]{Stein}), (\ref{Heat*}) and (\ref{Heat**}), we obtain
	\begin{equation*}
		\|f\|_{B^{s,b}_{\infty,\infty}(\mathbb{R}^d)}  \asymp \|I_{-2 m}f\|_{ L_\infty(\mathbb{R}^d)} + \sup_{0 < t < \infty} t^{-\frac{s}{2}} (1+|\log t|)^b \Big\|t^{k}  \frac{\partial^{k} W_t f}{\partial t^{k}} \Big\|_{L_\infty(\mathbb{R}^d)}.
	\end{equation*}
	Furthermore, the argument given above also showed that $\Lambda^{s,b}(\mathbb{R}^d)$ with $s \leq 0$ does not depend on $m \in \mathbb{N}$ with $m > -s/2$ and $k \in \mathbb{N}_0$ with $k > s/2$.
	

(iii): We first note that, by (\ref{lem:LambdaBesov*}), the embedding $\Lambda^{s_0,b_0}(\mathbb{R}^d) \hookrightarrow \Lambda^{s_1,b_1}(\mathbb{R}^d)$ is equivalent to $B^{s_0,b_0}_{\infty,\infty}(\mathbb{R}^d) \hookrightarrow  B^{s_1,b_1}_{\infty,\infty}(\mathbb{R}^d)$.

Assume that (a) or (b) holds. Then, it is plain to check that $B^{s_0,b_0}_{\infty,\infty}(\mathbb{R}^d) \hookrightarrow B^{s_1,b_1}_{\infty,\infty}(\mathbb{R}^d)$. Conversely, suppose that $B^{s_0,b_0}_{\infty,\infty}(\mathbb{R}^d) \hookrightarrow B^{s_1,b_1}_{\infty,\infty}(\mathbb{R}^d)$ holds. This implies that one of the conditions $(a)$ or $(b)$ is satisfied. Indeed, we first note that applying the lifting property given in Lemma \ref{LemmaLift}, we may assume without loss of generality that $s_0 > 0$.
Suppose that $s_0 < s_1$. Let $k > s_1$ and $s_0 < \varepsilon < s_1$. Then there exists $f \in L_\infty(\mathbb{R}^d)$ such that $\omega_k(f,t)_\infty \asymp t^{\varepsilon}$, see \cite{Tikhonov-real}.
Clearly, we have $f \in \mathbf{B}^{s_0,b_0}_{\infty,\infty}(\mathbb{R}^d) = B^{s_0,b_0}_{\infty,\infty}(\mathbb{R}^d) $ but $f \not \in \mathbf{B}^{s_1,b_1}_{\infty,\infty}(\mathbb{R}^d) = B^{s_1,b_1}_{\infty,\infty}(\mathbb{R}^d)$. If $s_0 = s_1$ and $b_0 < b_1$, we let $f \in L_\infty(\mathbb{R}^d)$ with $\omega_k(f,t)_\infty \asymp t^{s_0} (1-\log t)^{-b_0}$. This gives the desired counterexample.

(iv): The characterization (\ref{LambdaDifferences}) follows from (\ref{lem:LambdaBesov*}) and Lemma \ref{CharSobolevNorm2}.
\end{proof}

Our next result gives some characterizations of H\"{o}lder and Zygmund spaces.
\begin{prop}\label{PropRonStin+}
\begin{enumerate}[\upshape(i)]
\item Let 
 $-\infty < b < \infty$. Then, 
\begin{equation}\label{BesovDiff****}
\mathcal{Z}^{1,b}(\mathbb{R}^d)= \Lambda^{1,b}(\mathbb{R}^d) = \mathbf{B}^{1,b}_{\infty,\infty}(\mathbb{R}^d).
 \end{equation}

			\item Let $k \in \mathbb{N}_0$. Then,
$$
C^{k,s,b}(\mathbb{R}^d)=
\left\{
  \begin{array}{ll}
    \mathbf{B}^{k + s,b}_{\infty,\infty}(\mathbb{R}^d)=\Lambda^{k + s,b}(\mathbb{R}^d), & \hbox{$0<s<1$ and $-\infty < b < \infty$;} \\
    \mathbf{B}^{0,b}_{\infty,\infty}(\mathbb{R}^d), & \hbox{$k=s=0$ and $b > 0$;}\\
    \emph{Lip}^{(k+1,b)}_{\infty,\infty}(\mathbb{R}^d), & \hbox{$s=1$ and $b\le0$;}
  \end{array}
\right.
$$
and, for $k \in \mathbb{N}$ and $b > 0$,
\begin{equation}\label{CGagliardo2}
	\mathbf{B}^{k,b+1}_{\infty,\infty}(\mathbb{R}^d)
\hookrightarrow C^{k,0,b}(\mathbb{R}^d) \hookrightarrow \mathbf{B}^{k,b}_{\infty,\infty}(\mathbb{R}^d)
\end{equation}
or, equivalently, 
$$
\Lambda^{k ,b+1}(\mathbb{R}^d) \hookrightarrow C^{k,0,b}(\mathbb{R}^d) \hookrightarrow 
\Lambda^{k,b}(\mathbb{R}^d).
$$
More precisely, if $s=0$ then
\begin{equation}\label{CGagliardo}
\|f\|_{C^{k,0,b}(\mathbb{R}^d)} \asymp \sum_{|\beta| \leq k} \|D^\beta f\|_{ \mathbf{B}^{0,b}_{\infty,\infty}(\mathbb{R}^d)}, \quad k \in \mathbb{N}_0, \quad b > 0.
\end{equation}
			\item Let $k \in \mathbb{N}_0, b \geq 0$, and $-\infty < d < \infty$. Then,
			\begin{equation}\label{EmbCLambda}
			C^{k,1,-b}(\mathbb{R}^d) = \emph{Lip}^{(k+1,-b)}_{\infty,\infty}(\mathbb{R}^d) \hookrightarrow \Lambda^{k+1,-d}(\mathbb{R}^d) \iff d \geq b.
			\end{equation}
						\item
Let $m \in \mathbb{N}$ and $-\infty < b < \infty$. Then,
			\begin{equation}\label{LambdaDifferences2}
				\|f\|_{\Lambda^{m+1,b}(\mathbb{R}^d)} \asymp \sum_{|\beta| \leq m} \|D^\beta f\|_{ \mathcal{Z}^{1,b}(\mathbb{R}^d)}.
			\end{equation}
	\end{enumerate}
\end{prop}

\begin{rem}
(i) Note that (\ref{CGagliardo}) also holds if $b \leq 0$ because both sides are equivalent to $\|f\|_{C^k(\mathbb{R}^d)}$.


(ii) If $b < 0$ then the embedding given in (\ref{EmbCLambda}) trivially holds for all $d \in \mathbb{R}$ because $C^{k,1,-b}(\mathbb{R}^d)= \{0\}$.

\end{rem}

\begin{proof}[Proof of Proposition \ref{PropRonStin+}]
(i): Formula (\ref{BesovDiff****}) follows from (\ref{BesovDiff**}) (with $s=1$ and $k=2$).

(ii): By (\ref{lem:LambdaBesov*}), (\ref{BesovComparison}), (\ref{11.1i}), and (\ref{BesovDiff**}) we have for $0<s<1$
\begin{equation*}
	\|f\|_{\Lambda^{k + s,b}(\mathbb{R}^d)} \asymp \|f\|_{\mathbf{B}^{k+s,b}_{\infty,\infty}(\mathbb{R}^d)} \asymp \sum_{|\beta| \leq k} \|D^\beta f\|_{\mathbf{B}^{s,b}_{\infty,\infty}(\mathbb{R}^d)} \asymp \|f\|_{C^{k,s,b}(\mathbb{R}^d)}.
\end{equation*}
To show (\ref{CGagliardo}), compare (\ref{BesovDiff**}) with $s=0$ and (\ref{def:HolderSpace}). In particular, putting $k=0$ in (\ref{CGagliardo}) we derive $C^{0,0,b}(\mathbb{R}^d) = \mathbf{B}^{0,b}_{\infty,\infty}(\mathbb{R}^d)$. Suppose now that $k \in \mathbb{N}$ and $b > 0$. According to Theorem \ref{Theorem 11.1Conv}(ii) and (\ref{CGagliardo}), we have
\begin{equation*}
	\|f\|_{\mathbf{B}^{k,b}_{\infty,\infty}(\mathbb{R}^d)} \lesssim \sum_{|\beta| \leq k} \|D^\beta f\|_{\mathbf{B}^{0,b}_{\infty,\infty}(\mathbb{R}^d)} \asymp \|f\|_{C^{k,0,b}(\mathbb{R}^d)},
\end{equation*}
that is, $C^{k,0,b}(\mathbb{R}^d) \hookrightarrow \mathbf{B}^{k,b}_{\infty,\infty}(\mathbb{R}^d)$. Conversely, applying (\ref{11.1new}), we get
\begin{equation}\label{DerEmb*}
	\sum_{|\beta| < k} \|D^\beta f\|_{\mathbf{B}^{0,b}_{\infty,\infty}(\mathbb{R}^d)} \lesssim \sum_{|\beta| < k} \|D^\beta f\|_{\mathbf{B}^{k-|\beta|,b}_{\infty,\infty}(\mathbb{R}^d)} \lesssim \|f\|_{\mathbf{B}^{k,b}_{\infty,\infty}(\mathbb{R}^d)}
\end{equation}
where the first estimate follows from the trivial embedding $\mathbf{B}^{k-|\beta|,b}_{\infty,\infty}(\mathbb{R}^d) \hookrightarrow \mathbf{B}^{0,b}_{\infty,\infty}(\mathbb{R}^d)$ with $|\beta| < k$. Furthermore, by Theorem \ref{Theorem 11.1}(ii), we have
\begin{equation}\label{DerEmb**}
	\sum_{|\beta| = k} \|D^\beta f\|_{\mathbf{B}^{0,b}_{\infty,\infty}(\mathbb{R}^d)} \lesssim \|f\|_{\mathbf{B}^{k,b+1}_{\infty,\infty}(\mathbb{R}^d)}.
\end{equation}
Finally, by (\ref{CGagliardo}), (\ref{DerEmb*}) and (\ref{DerEmb**}), we arrive at
\begin{equation*}
	 \|f\|_{C^{k,0,b}(\mathbb{R}^d)} \asymp \sum_{|\beta| \leq k} \|D^\beta f\|_{\mathbf{B}^{0,b}_{\infty,\infty}(\mathbb{R}^d)} \lesssim  \|f\|_{\mathbf{B}^{k,b}_{\infty,\infty}(\mathbb{R}^d)} +  \|f\|_{\mathbf{B}^{k,b+1}_{\infty,\infty}(\mathbb{R}^d)} \lesssim \|f\|_{\mathbf{B}^{k,b+1}_{\infty,\infty}(\mathbb{R}^d)}
\end{equation*}
which gives the left-hand side embedding in (\ref{CGagliardo2}).

In order to verify that
			\begin{equation}\label{LipC}
			\text{Lip}^{(k+1,-b)}_{\infty,\infty}(\mathbb{R}^d) = C^{k,1,-b}(\mathbb{R}^d),
\qquad b \geq 0,
			\end{equation}
by elementary computations we arrive at 
\begin{align*}
	\sup_{0 < t < 1} t^{-1} (1+ |\log t|)^{-b} \omega_1(f,t)_\infty & \asymp \sup_{|h| < 1} \|\Delta_h f\|_{L_\infty(\mathbb{R}^d)} \sup_{ |h| < t < 1} t^{-1} (1 -\log t)^{-b} \\
	& \asymp \sup_{0 < |h| < 1} \frac{(1 -\log |h|)^{-b} \|\Delta_h f\|_{L_\infty(\mathbb{R}^d)}}{|h| } \\
	& \asymp \sup_{\substack{x, y \in \mathbb{R}^d \\ 0 < |x-y| < 1}} \frac{(1 -\log |x-y|)^{-b} |f(x)- f(y)|}{|x-y|}.
\end{align*}
Hence, by (\ref{Lipschitz}) and (\ref{def:HolderSpace}), we have shown that
\begin{equation}\label{14.13**}
	f \in C^{k,1,-b}(\mathbb{R}^d) \quad \text{if and only if} \quad D^\beta f \in \text{Lip}^{(1,-b)}_{\infty,\infty}(\mathbb{R}^d) \quad \text{for all} \quad  |\beta| = k.
\end{equation}
This finishes the proof of (\ref{LipC}) if $k=0$. If $k \geq 1$ the proof follows from (\ref{14.13**}) and the following lemma, which is interesting for its own sake.

\begin{lem}\label{lem:derivativesLipschitz}
	Let $k \in \mathbb{N}, 1 \leq p \leq \infty, 0 < q \leq \infty$ and $-\infty < b < \infty$. Then,
	\begin{equation}\label{lem:derivativesLipschitz1}
		\|f\|_{\emph{Lip}^{(k+1,-b)}_{p,q}(\mathbb{R}^d)} \asymp \sum_{|\beta| \leq k} \|D^\beta f\|_{\emph{Lip}^{(1,-b)}_{p,q}(\mathbb{R}^d)}.
	\end{equation}
\end{lem}


\begin{rem}
 Lemma \ref{lem:derivativesLipschitz} provides the Sobolev-type description of Lipschitz spaces, see also
 (\ref{11.1i}).
\end{rem}

\begin{proof}[Proof of Lemma \ref{lem:derivativesLipschitz}]
	According to (\ref{convTrebels}), we have
	\begin{align*}
		\sum_{|\beta| \leq k} \|D^\beta f\|_{\text{Lip}^{(1,-b)}_{p,q}(\mathbb{R}^d)}^q & \asymp 	 \| f\|_{W^k_p(\mathbb{R}^d)}^q +\int_0^1t^{-q} (1 - \log t)^{-b q} \Big( \sum_{|\beta| \leq k} \omega_1(D^\beta f,t)_p\Big)^q \frac{dt}{t} \\
		& \hspace{-2cm}\gtrsim \|f\|^q_{L_p(\mathbb{R}^d)} + \int_0^1 t^{-(k+1) q} (1-\log t)^{b q} \omega_{k+1}(f,t)_p^q \frac{dt}{t} \asymp \|f\|_{\text{Lip}^{(k+1,-b)}_{p,q}(\mathbb{R}^d)}^q.
	\end{align*}
	
	Conversely, applying (\ref{MarchaudClassic}) and Hardy's inequality (\ref{HardyInequal1}) (noting that the function $\omega_{k+1}(f,t)_p/t^{k+1}$ is equivalent to a decreasing function), one gets
	\begin{align*}
		\|D^\beta f\|_{\text{Lip}^{(1,-b)}_{p,q}(\mathbb{R}^d)} & \lesssim \|D^\beta f\|_{L_p(\mathbb{R}^d)}\\
		&  \hspace{1cm}+  \left(\int_0^1\left(t^{-1} (1 - \log t)^{-b}  \int_0^t u \frac{\omega_{k+1}(f,u)_p}{u^{k+1}} \frac{du}{u}\right)^q \frac{dt}{t}\right)^{1/q} \\
		& \lesssim \|D^\beta f\|_{L_p(\mathbb{R}^d)} + \left(\int_0^1 (t^{-(k+1)} (1-\log t)^b \omega_{k+1}(f,t)_p)^q \frac{dt}{t}\right)^{1/q}
	\end{align*}
	for every $|\beta| \leq k$. Consequently,
	\begin{equation*}
		\sum_{|\beta| \leq k} \|D^\beta f\|_{\text{Lip}^{(1,-b)}_{p,q}(\mathbb{R}^d)} \lesssim \|f\|_{W^k_p(\mathbb{R}^d)} + \|f\|_{\text{Lip}^{(k+1,-b)}_{p,q}(\mathbb{R}^d)} \lesssim  \|f\|_{\text{Lip}^{(k+1,-b)}_{p,q}(\mathbb{R}^d)},
	\end{equation*}
	where the last estimate follows from the trivial embedding $\text{Lip}^{(k+1,-b)}_{p,\infty}(\mathbb{R}^d) \hookrightarrow W^k_p(\mathbb{R}^d)$.
\end{proof}

(iii): Next we treat (\ref{EmbCLambda}). Assume that $d \geq b$. Using the fact that $\omega_{k+2}(f,t)_\infty \lesssim \omega_{k+1}(f,t)_\infty$, we obtain
\begin{equation*}
	\sup_{0 < t < 1} t^{-(k+1)} (1 - \log t)^{-d} \omega_{k+2}(f,t)_\infty \lesssim \sup_{0 < t < 1} t^{-(k+1)} (1 - \log t)^{-b} \omega_{k+1}(f,t)_\infty,
\end{equation*}
which implies, by (\ref{lem:LambdaBesov*}), (\ref{BesovComparison}) and (\ref{LipC}), that $C^{k,1,-b}(\mathbb{R}^d) = \text{Lip}^{(k+1,-b)}_{\infty,\infty}(\mathbb{R}^d)\hookrightarrow \mathbf{B}_{\infty, \infty}^{k+1,-d}(\mathbb{R}^d) = \Lambda^{k+1,-d}(\mathbb{R}^d) $.

Conversely, suppose that
\begin{equation}\label{14.17new*}
C^{k,1,-b}(\mathbb{R}^d) \hookrightarrow \Lambda^{k+1,-d}(\mathbb{R}^d)
\end{equation}
 holds. Let $\theta \in (0,1)$. By (\ref{PrelimInterpolation}), Lemma \ref{PrelimLemma7.2}(ii), (\ref{interpolationLipschitz}) and (\ref{LipC}), we have
\begin{align}
	\mathbf{B}^{\theta(k+1),-\theta b}_{\infty,\infty}(\mathbb{R}^d) & = (L_\infty(\mathbb{R}^d), W^{k+1}_\infty(\mathbb{R}^d))_{\theta,\infty;-\theta b} \nonumber\\
	& = (L_\infty(\mathbb{R}^d), (L_\infty(\mathbb{R}^d), W^{k+1}_\infty(\mathbb{R}^d))_{(1,-b),\infty})_{\theta,\infty} \nonumber\\
	& = (L_\infty(\mathbb{R}^d),\text{Lip}^{(k+1,-b)}_{\infty,\infty}(\mathbb{R}^d) )_{\theta,\infty} \nonumber\\
	 & = (L_\infty(\mathbb{R}^d),C^{k,1,-b}(\mathbb{R}^d) )_{\theta,\infty}. \label{14.18new*}
\end{align}
On the other hand, by (\ref{lem:LambdaBesov*}) and (\ref{PrelimInterpolationnew2.2}), we derive
\begin{equation}\label{14.19new*}
	(L_\infty(\mathbb{R}^d), \Lambda^{k+1,-d}(\mathbb{R}^d) )_{\theta,\infty} = (L_\infty(\mathbb{R}^d), \mathbf{B}^{k+1,-d}_{\infty,\infty}(\mathbb{R}^d) )_{\theta,\infty} = \mathbf{B}^{\theta (k+1),-\theta d}_{\infty,\infty}(\mathbb{R}^d).
\end{equation}
It follows from (\ref{14.17new*}), (\ref{14.18new*}) and (\ref{14.19new*}) that $\mathbf{B}^{\theta(k+1),-\theta b}_{\infty,\infty}(\mathbb{R}^d) \hookrightarrow \mathbf{B}^{\theta (k+1),-\theta d}_{\infty,\infty}(\mathbb{R}^d)$ which implies $b \leq d$.

(iv): Equivalence (\ref{LambdaDifferences2}) is a combination of (\ref{LambdaDifferences}) together with (\ref{BesovDiff****}).
\end{proof}

\begin{rem}
 Let $b \geq 0$. The optimal relationships between $C^{k,1,-b}(\mathbb{R}^d)$ and $\Lambda^{k+1,-b}(\mathbb{R}^d)$ given in (\ref{EmbCLambda}) is complemented by the assertion that these spaces do not coincide. For technical reasons we switch temporarily to the periodic case. Assume, e.g., that $k$ is odd and let
\begin{equation*}
	f (x) \sim \sum_{j=1}^\infty j^{-k-2} (1 + \log j)^{-\varepsilon} \cos jx, \quad x \in \mathbb{T},
\end{equation*}
where $-b < \varepsilon < 1-b$. Note that $f$ is a continuous function (see, e.g., \cite[Theorem 1.1]{Tikhonov-jat}). For $n \in \mathbb{N}$, applying the realization result obtained in \cite[Corollary 4.4.1]{Tikhonov-jat}, we get
\begin{align*}
	\omega_{k+1}\Big(f,\frac{1}{n}\Big)_\infty &\asymp n^{-k-1} \sum_{j=1}^n j^{k+1}  j^{-k-2} (1 + \log j)^{-\varepsilon} + \sum_{j=n}^\infty  j^{-k-2} (1 + \log j)^{-\varepsilon} \\
	& \asymp n^{-k-1} (1+\log n)^{1-\varepsilon}
\end{align*}
 (note that $\varepsilon < 1$ since $b \geq 0$), and
\begin{equation*}
	\omega_{k+3}\Big(f,\frac{1}{n}\Big)_\infty  \asymp n^{-k-1} (1+\log n)^{-\varepsilon}.
\end{equation*}
Therefore, 
\begin{equation}\label{14.26**}
	\sup_{0 < t < 1} t^{-k-1} (1-\log t)^{-b} \omega_{k+3}(f,t)_\infty \asymp \sup_{0 < t < 1}  (1-\log t)^{-b-\varepsilon} < \infty
\end{equation}
and
\begin{equation}\label{14.27**}
	\sup_{0 < t < 1} t^{-k-1} (1-\log t)^{-b} \omega_{k+1}(f,t)_\infty \asymp \sup_{0 < t < 1}  (1-\log t)^{-b-\varepsilon + 1} = \infty.
\end{equation}
According to (\ref{14.26**}), (\ref{14.27**}), (\ref{lem:LambdaBesov*}) and (\ref{LipC}) we derive that $f \in \Lambda^{k+1,-b}(\mathbb{T})$ but $f \not \in C^{k,1,-b}(\mathbb{T})$. Analogously, one can proceed with the case $k$ is even by taking the Fourier series
\begin{equation*}
	g (x) \sim \sum_{j=1}^\infty j^{-k-2} (1 + \log j)^{-\varepsilon} \sin jx, \quad x \in \mathbb{T},
\end{equation*}
where $-b < \varepsilon < 1-b$.
\end{rem}

\begin{rem}
	The homogeneous counterparts of Propositions \ref{lem:LambdaBesov} and \ref{PropRonStin+} (with appropriate modifications) also hold true. Further details are left to the interested reader.
\end{rem}


\subsection{Local regularity for the Dirichlet fractional Laplacian}
We also introduce some smoothness spaces on domains which will be useful to formulate the Dirichlet problem involving $(-\Delta)^s$.
Let $\Omega$\index{\bigskip\textbf{Sets}!$\Omega$}\label{DOMAIN} be a bounded open set of $\mathbb{R}^d$. The notations $\partial \Omega$\index{\bigskip\textbf{Sets}!$\partial \Omega$}\label{BOUNDARY} and $\overline{\Omega}$\index{\bigskip\textbf{Sets}!$\overline{\Omega}$}\label{CLOSURE} stand for the boundary and the closure of $\Omega$, respectively. The space $\dot{\mathscr{L}}^{s}_{p,\text{loc}}(\Omega)$\index{\bigskip\textbf{Spaces}!$\dot{\mathscr{L}}^{s}_{p,\text{loc}}(\Omega)$}\label{LOCALRIESZ} is defined as
\begin{equation*}
	\dot{\mathscr{L}}^{s}_{p,\text{loc}}(\Omega) = \{f \in L_p(\Omega) : \zeta f \in \dot{\mathscr{L}}^{s}_p(\mathbb{R}^d) \quad \text{for any test function}  \quad \zeta \in \mathcal{D}(\Omega)\}.
\end{equation*}
Analogously, we introduce\index{\bigskip\textbf{Spaces}!$\dot{\mathbf{B}}^{s,b}_{p,q,\text{loc}}(\Omega)$}\label{LOCALBESOV}
\begin{equation*}
	\dot{\mathbf{B}}^{s,b}_{p,q,\text{loc}}(\Omega) = \{f \in L_p(\Omega) : \zeta f \in \dot{\mathbf{B}}^{s,b}_{p,q}(\mathbb{R}^d) \quad \text{for any test function}  \quad \zeta \in \mathcal{D}(\Omega)\}.
\end{equation*}
By (\ref{BesovDiff2}), $\dot{\mathscr{L}}^{s}_{2,\text{loc}}(\Omega) =\dot{\mathbf{B}}^{s}_{2,2,\text{loc}}(\Omega), s > 0$.

Let\index{\bigskip\textbf{Spaces}!$\dot{\mathscr{L}}^{s}_2(\overline{\Omega})$}\label{SOBOLEVDOMAIN}
\begin{equation*}
	\dot{\mathscr{L}}^{s}_2(\overline{\Omega}) = \{f \in \dot{\mathscr{L}}^{s}_2(\mathbb{R}^d) : f = 0 \quad \text{on} \quad \mathbb{R}^d \setminus \Omega\}, \quad 0 < s < 1,
\end{equation*}
endowed with the semi-norm
\begin{equation*}
	\|f\|_{\dot{\mathscr{L}}^{s}_2(\overline{\Omega}) } =  \left(\int_{\mathbb{R}^d} \int_{\mathbb{R}^d} \frac{|f(x) - f(y)|^2}{|x-y|^{d + s 2}} d x d y \right)^{1/2}
\end{equation*}
(see (\ref{BesovDiff}) and (\ref{BesovDiff*3})).
Define $\dot{\mathscr{L}}^{-s}_2(\overline{\Omega}) = (\dot{\mathscr{L}}^{s}_2(\overline{\Omega}))'$\index{\bigskip\textbf{Spaces}!$\dot{\mathscr{L}}^{-s}_2(\overline{\Omega})$}\label{SOBOLEVDUAL}. The following embeddings hold
\begin{equation}\label{14.7n}
	\dot{\mathscr{L}}^{s}_2(\overline{\Omega}) \hookrightarrow L_2(\Omega) \hookrightarrow \dot{\mathscr{L}}^{-s}_2(\overline{\Omega}).
\end{equation}

Let $0 \leq s < 1$ and $1 \leq p \leq \infty$. Following (\ref{BesovDiff*3}), one can define the Besov space $W^s_p(\Omega)$ as the collection\index{\bigskip\textbf{Spaces}!$W^s_p(\Omega)$}\label{BESOVDOMAIN}
\begin{equation*}
	W^s_p(\Omega) = \left\{f \in L_p(\Omega) : |f|_{W^s_p(\Omega)} = \left(\int_{\Omega} \int_{\Omega}\frac{|f(x) - f(y)|^p}{|x-y|^{d + s p}} d x d y  \right)^{1/p} < \infty \right\}
\end{equation*}
endowed with the norm
\begin{equation*}
	\|f\|_{W^s_p(\Omega)} = \|f\|_{L_p(\Omega)} + |f|_{W^s_p(\Omega)}.
\end{equation*}
Further properties and characterizations of the spaces considered above can be found in the books \cite{Burenkov, Mazya, Triebel1, Triebel3, Triebel08}.

Next we collect some known results on the local regularity of solutions to the Dirichlet elliptic problem related to the fractional Laplacian given by
\begin{equation}\label{DirProblem}
		 \left\{\begin{array}{cl}  (-\Delta)^s u = f & \text{in} \quad  \Omega, \\
		u= 0  & \text{on} \quad \mathbb{R}^d \setminus \Omega.
		       \end{array}
                        \right.
	\end{equation}
It is important to notice that, due to the non-local nature of the operator $(-\Delta)^s$, the Dirichlet condition in (\ref{DirProblem}) is not involving only the boundary of $\Omega$ but its complement $\mathbb{R}^d \setminus \Omega$.

The existence and uniqueness of weak solutions to the Dirichlet problem (\ref{DirProblem}) is guaranteed by the following result.

\begin{lem}\label{lem:DirProblem}
	Let $\Omega$ be a bounded open set in $\mathbb{R}^d$ and $0 < s < 1$. If $f \in \dot{\mathscr{L}}^{-s}_2(\overline{\Omega})$ then there exists a unique weak solution $u \in \dot{\mathscr{L}}^{s}_2(\overline{\Omega}) $ of the Dirichlet problem (\ref{DirProblem}). Furthermore, we have
	\begin{equation*}
		\|u\|_{\dot{\mathscr{L}}^{s}_2(\overline{\Omega}) } \lesssim \|f\|_{ \dot{\mathscr{L}}^{-s}_2(\overline{\Omega})}.
	\end{equation*}
\end{lem}

The proof of Lemma \ref{lem:DirProblem} is a straightforward consequence of the Lax-Milgram Theorem. For further details, see \cite[Theorem 12]{LeonoriPeralPrimoSoria} and \cite[Proposition 2.1]{BiccariWarmaZuazua}.

The local $L_2$-regularity for weak solutions of the Dirichlet problem (\ref{DirProblem}) has been recently obtained by Biccari, Warma and Zuazua \cite[Theorem 1.3]{BiccariWarmaZuazua} (with \cite{Grubb} as a  forerunner) and reads as follows.

\begin{thm}\label{thm:DirProblem2}
	Let $\Omega$ be a bounded open set in $\mathbb{R}^d$ and let $0 < s < 1$. Let $f \in \dot{\mathscr{L}}^{-s}_2(\overline{\Omega})$ and let $u \in \dot{\mathscr{L}}^{s}_2(\overline{\Omega}) $  be the unique weak solution to the Dirichlet problem (\ref{DirProblem}). If $f \in L_2(\Omega)$ then $u \in \dot{\mathscr{L}}^{2s}_{2,\emph{loc}}(\Omega)$.
\end{thm}

The corresponding extension to $L_p(\Omega), 1 < p < \infty,$ was investigated in \cite[Theorem 3]{BiccariWarmaZuazua2}.

\begin{thm}\label{thm:DirProblemP}
	Let $\Omega$ be a bounded open set in $\mathbb{R}^d$ and let $0 < s < 1$. Let $f \in \dot{\mathscr{L}}^{-s}_2(\overline{\Omega})$ and let $u \in \dot{\mathscr{L}}^{s}_2(\overline{\Omega}) $  be the unique weak solution to the Dirichlet problem (\ref{DirProblem}). Assume that $1 < p < \infty$. If $f \in L_p(\Omega)$ and $u \in W^s_p(\Omega)$ then $u \in \dot{\mathbf{B}}^{2 s}_{p,\max\{2,p\}, \emph{loc}}(\Omega)$.

\end{thm}

 We refer to  \cite{LeonoriPeralPrimoSoria} (cf. also \cite{BiccariWarmaZuazua}) for precise definitions of the concepts of weak solutions to the problem (\ref{DirProblem}) for $L_p$-data, $1 < p < \infty$.

\subsection{$L_2$-Regularity estimates}

The goal of this section is to show a quantitative version of the local $L_2$-regularity result for the Dirichlet problem (\ref{DirProblem}) given in Theorem \ref{thm:DirProblem2} above. In order to get this, we will make use of the following classical result for the fractional Poisson-type equation
\begin{equation*}
	(-\Delta)^s u = f \quad \text{in} \quad \mathbb{R}^d.
\end{equation*}

\begin{thm}\label{thm:FractLaplace}
	Let $s \in (0,1)$. Let $f \in \dot{\mathscr{L}}^{-s}_2(\mathbb{R}^d)$ and let $u \in \dot{\mathscr{L}}^{s}_2(\mathbb{R}^d)$ be the weak solution to the fractional Poisson-type equation
\begin{equation*}
	(-\Delta)^s u = f \quad \text{in} \quad \mathbb{R}^d.
\end{equation*}
Let $\lambda \geq 0$. Then
\begin{equation*}
	u \in  \dot{\mathscr{L}}^{\lambda + 2 s}_2(\mathbb{R}^d) \quad \text{if and only if} \quad f \in \dot{\mathscr{L}}^{\lambda}_2(\mathbb{R}^d),
\end{equation*}
and
\begin{equation}\label{EquivFractLaplace}
\|u\|_{\dot{\mathscr{L}}^{\lambda + 2 s}_2(\mathbb{R}^d)} \asymp \|f\|_{\dot{\mathscr{L}}^{\lambda}_2(\mathbb{R}^d)}.
\end{equation}
In particular, we have
\begin{equation}\label{EquivFractLaplace*}
\|u\|_{\dot{\mathscr{L}}^{2 s}_2(\mathbb{R}^d)} \asymp \|f\|_{L_2(\mathbb{R}^d)}.
\end{equation}
\end{thm}

An elementary proof of Theorem \ref{thm:FractLaplace} can be found in \cite[Propositions 3.4 and 3.6]{DiNPalVald}. See also \cite[Remark 10.5, p. 59]{LionsMagenes}.

\begin{thm}\label{thm 14.2}
	Let $s \in (0,1)$. Let $f \in \dot{\mathscr{L}}^{-s}_2(\mathbb{R}^d)$ and let $u \in \dot{\mathscr{L}}^{s}_2(\mathbb{R}^d)$ be the weak solution to the fractional Poisson-type equation
\begin{equation*}
	(-\Delta)^s u = f \quad \text{in} \quad \mathbb{R}^d.
\end{equation*}
Let $\lambda > 0$. Then
	\begin{equation}\label{EquivFractLaplace0}
		\left(\int_0^t (\xi^{-2 s} \omega_{ \lambda + 2 s} (u, \xi)_2)^2 \frac{d \xi}{\xi}\right)^{1/2} \asymp \omega_\lambda(f,t)_2.
	\end{equation}
\end{thm}

\begin{rem}
	As application of (\ref{EquivFractLaplace0}) one immediately gets the estimates (\ref{EquivFractLaplace}) and (\ref{EquivFractLaplace*}). For example, taking the supremum over all $t > 0$ in (\ref{EquivFractLaplace0}) and using (\ref{12.12new2**}) and (\ref{BesovDiff2}), we derive (\ref{EquivFractLaplace*}). Similarly, one can deduce (\ref{EquivFractLaplace}) from (\ref{EquivFractLaplace0}).
\end{rem}

\begin{proof}[Proof of Theorem \ref{thm 14.2}]
	It follows from (\ref{EquivFractLaplace*}) and (\ref{EquivFractLaplace}) that
	\begin{equation}\label{EquivFractLaplace1}
		K(t^\lambda, u; \dot{\mathscr{L}}^{2 s}_2(\mathbb{R}^d), \dot{\mathscr{L}}^{\lambda + 2 s}_2(\mathbb{R}^d)) \asymp K(t^\lambda,f ; L_2(\mathbb{R}^d), \dot{\mathscr{L}}^{\lambda}_2(\mathbb{R}^d)).
	\end{equation}
	By Lemma \ref{LemmaModuli} we have
	\begin{equation}\label{EquivFractLaplace2}
		 K(t^\lambda,f ; L_2(\mathbb{R}^d), \dot{\mathscr{L}}^{\lambda}_2(\mathbb{R}^d)) \asymp \omega_{\lambda}(f,t)_2, \quad \lambda > 0.
	\end{equation}
	On the other hand, using the well-known interpolation formula for Sobolev spaces $\dot{\mathscr{L}}^{2 s}_2(\mathbb{R}^d) = (L_2(\mathbb{R}^d), \dot{\mathscr{L}}^{\lambda + 2 s}_2(\mathbb{R}^d))_{\frac{2 s}{ \lambda + 2 s}, 2}$ (cf. \cite[Lemma 23.1, p. 115]{Tartar}), Lemma \ref{PrelimHolmstedt}(ii) (with $\alpha = 0$) and (\ref{EquivFractLaplace2}), we get
	\begin{align*}
		K(t^\lambda, u; \dot{\mathscr{L}}^{2 s}_2(\mathbb{R}^d), \dot{\mathscr{L}}^{\lambda + 2 s}_2(\mathbb{R}^d)) & \asymp \left(\int_0^{t} (\xi^{- 2 s} K(\xi^{\lambda + 2 s}, u; L_2(\mathbb{R}^d), \dot{\mathscr{L}}^{\lambda + 2 s}_2(\mathbb{R}^d)))^2 \frac{d \xi}{\xi}\right)^{1/2} \\
		& \asymp 	\left(\int_0^t (\xi^{-2 s} \omega_{ \lambda + 2 s} (u, \xi)_2)^2 \frac{d \xi}{\xi}\right)^{1/2}.
	\end{align*}
	Combined with (\ref{EquivFractLaplace1}) and (\ref{EquivFractLaplace2}) one obtains (\ref{EquivFractLaplace0}).
\end{proof}

As a consequence of Theorem \ref{thm 14.2}, we show that Theorem \ref{thm:FractLaplace} can be extended from the Hilbert setting $\dot{\mathscr{L}}^{\lambda}_2(\mathbb{R}^d) = \dot{\mathbf{B}}^{\lambda}_{2,2}(\mathbb{R}^d), \lambda > 0,$ (cf. (\ref{BesovDiff2})) to the non-Hilbert setting $\dot{\mathbf{B}}^{\lambda}_{2,q}(\mathbb{R}^d), 0 < q \leq \infty$.

\begin{thm}\label{CorFractLapl2}
	Let $s \in (0,1)$. Let $f \in \dot{\mathscr{L}}^{-s}_2(\mathbb{R}^d)$ and let $u \in \dot{\mathscr{L}}^{s}_2(\mathbb{R}^d)$ be the weak solution to the fractional Poisson-type equation
\begin{equation*}
	(-\Delta)^s u = f \quad \text{in} \quad \mathbb{R}^d.
\end{equation*}
Let $\lambda > 0, b \in \mathbb{R}$, and $0 < q \leq \infty$. Then
\begin{equation*}
	u \in  \dot{\mathbf{B}}^{2 s + \lambda, b}_{2, q}(\mathbb{R}^d) \quad \text{if and only if} \quad f \in \dot{\mathbf{B}}^{\lambda, b}_{2,q}(\mathbb{R}^d),
\end{equation*}
and
\begin{equation*}
\|u\|_{\dot{\mathbf{B}}^{2 s + \lambda, b}_{2, q}(\mathbb{R}^d)} \asymp \|f\|_{\dot{\mathbf{B}}^{\lambda, b}_{2,q}(\mathbb{R}^d)}.
\end{equation*}
\end{thm}
\begin{proof}
	Let $\lambda_0 > \lambda$. Applying (\ref{EquivFractLaplace0}) and elementary monotonicity estimates, we derive
	\begin{align}
		\|f\|_{\dot{\mathbf{B}}^{\lambda, b}_{2,q}(\mathbb{R}^d)}
		& \asymp \left(\int_0^\infty t^{-\lambda q} (1 + |\log t|)^{b q} \left(\int_0^t (\xi^{-2 s} \omega_{\lambda_0 + 2s}(u,\xi))^2 \frac{d \xi}{\xi}\right)^{q/2} \frac{dt}{t}\right)^{1/q} \nonumber \\
		& \asymp \left(\sum_{j=-\infty}^\infty 2^{j \lambda q} (1 + |j|)^{b q} \left(\sum_{\nu = j}^\infty (2^{\nu 2 s} \omega_{\lambda_0 + 2 s}(u,2^{-\nu})_2)^2 \right)^{q/2}\right)^{1/q} \label{14.8}.
	\end{align}
	In particular, this implies
	\begin{equation*}
		\|f\|_{\dot{\mathbf{B}}^{\lambda, b}_{2,q}(\mathbb{R}^d)} \gtrsim \left(\sum_{j=-\infty}^\infty 2^{j (\lambda + 2 s) q} (1 + |j|)^{b q} \omega_{\lambda_0 + 2 s}(u,2^{-j})_2^q\right)^{1/q} \asymp \|u\|_{\dot{\mathbf{B}}^{\lambda + 2 s, b}_{2,q}(\mathbb{R}^d)}.
	\end{equation*}
	
	Conversely, assume first that $q \geq 2$. Then, it follows from (\ref{14.8}) and Hardy's inequality (see (\ref{HardyInequal2**})) that
	\begin{equation*}
		\|f\|_{\dot{\mathbf{B}}^{\lambda, b}_{2,q}(\mathbb{R}^d)} \lesssim  \left(\sum_{j=-\infty}^\infty 2^{j (\lambda + 2 s) q} (1 + |j|)^{b q} \omega_{\lambda_0 + 2 s}(u,2^{-j})_2^q\right)^{1/q} \asymp \|u\|_{\dot{\mathbf{B}}^{\lambda + 2 s, b}_{2,q}(\mathbb{R}^d)}.	
		\end{equation*}
		On the other hand, if $q < 2$ then, by (\ref{14.8}), we have
		\begin{align*}
				\|f\|_{\dot{\mathbf{B}}^{\lambda, b}_{2,q}(\mathbb{R}^d)}&  \lesssim  \left(\sum_{j=-\infty}^\infty 2^{j \lambda q} (1 + |j|)^{b q} \sum_{\nu = j}^\infty (2^{\nu 2 s} \omega_{\lambda_0 + 2 s}(u, 2^{-\nu})_2)^q\right)^{1/q} \\
				& = \left(\sum_{\nu=-\infty}^\infty (2^{\nu 2 s} \omega_{\lambda_0 + 2 s}(u, 2^{-\nu})_2)^q \sum_{j=-\infty}^\nu 2^{j \lambda q} (1 + |j|)^{b q} \right)^{1/q} \\
				& \asymp \left(\sum_{\nu=-\infty}^\infty 2^{\nu(\lambda + 2 s) q} (1 + |\nu|)^{b q} \omega_{\lambda_0 + 2 s}(u, 2^{-\nu})_2^q\right)^{1/q}  \asymp \|u\|_{\dot{\mathbf{B}}^{\lambda + 2 s, b}_{2,q}(\mathbb{R}^d)}.
		\end{align*}
		The proof is finished.
\end{proof}

We can now establish the quantitative version of Theorem \ref{thm:DirProblem2}.

\begin{thm}
	Let $\Omega$ be a bounded open set in $\mathbb{R}^d$ and let $0 < s < 1$. Let $f \in \dot{\mathscr{L}}^{-s}_2(\overline{\Omega})$ and let $u \in \dot{\mathscr{L}}^{s}_2(\overline{\Omega}) $  be the weak solution to the Dirichlet problem
	\begin{equation*}
		 \left\{\begin{array}{cl}  (-\Delta)^s u = f & \text{in} \quad  \Omega, \\
		u= 0  & \text{on} \quad \mathbb{R}^d \setminus \Omega.
		       \end{array}
                        \right.
	\end{equation*}
	Assume that $f \in L_2(\Omega)$ and $\lambda > 0$. Let $\omega$ and $\widetilde{\omega}$ be two open subsets of $\Omega$ such that $\widetilde{\omega}  \subset \omega \subset \Omega$, and let $\zeta \in C^\infty(\mathbb{R}^d)$ satisfying
		\begin{equation}\label{cutoff}
		 \left\{\begin{array}{ll} \zeta (x) = 1 & \text{if} \quad x \in \widetilde{\omega} ,\\
		 0 \leq \zeta(x) \leq 1 & \text{if} \quad x \in \omega \setminus \widetilde{\omega}, \\
		 \zeta(x) =0 & \text{if} \quad x \in \mathbb{R}^d \setminus \omega.
		 \end{array}
                        \right.
	\end{equation}
Then
	\begin{equation*}
		\left(\int_0^t (\xi^{-2 s} \omega_{ \lambda + 2 s} (\zeta u, \xi)_2)^2 \frac{d \xi}{\xi}\right)^{1/2} \lesssim \omega_\lambda(\zeta f,t)_2 + \|f\|_{L_2(\Omega)}, \quad t > T > 0.
	\end{equation*}
\end{thm}
\begin{proof}
	We shall apply the following technical lemma which has been shown in \cite[Proposition 1.5, (3.2) and (3.3)]{BiccariWarmaZuazua} (see also \cite{RosOtonSerra}).
	
	\begin{lem}
		The following pointwise formula holds
		\begin{equation}\label{PointwiseLaplacian}
			(-\Delta)^s (\zeta u) = \zeta (-\Delta)^s u + g \quad \text{in} \quad \mathbb{R}^d
		\end{equation}
		where $g \in L_2(\mathbb{R}^d)$ with
		\begin{equation}\label{PointwiseLaplacianRemainder}
		\|g\|_{L_2(\mathbb{R}^d)} \lesssim \|u\|_{\dot{\mathscr{L}}^{s}_2(\overline{\Omega})}.
		\end{equation}
	\end{lem}
	
	Since $f \in L_2(\Omega)$, we derive that $\zeta (-\Delta)^s u + g \in L_2(\mathbb{R}^d)$. By (\ref{PointwiseLaplacian}) and Theorem \ref{thm 14.2}, we get
	\begin{align*}
		\left(\int_0^t (\xi^{-2 s} \omega_{ \lambda + 2 s} (\zeta u, \xi)_2)^2 \frac{d \xi}{\xi}\right)^{1/2} & \asymp \omega_\lambda(\zeta (-\Delta)^s u + g, t)_2 \\
		& \lesssim  \omega_\lambda(\zeta (-\Delta)^s u, t)_2 + \|g\|_{L_2(\mathbb{R}^d)} \\
		& \lesssim  \omega_\lambda(\zeta f, t)_2 + \|f\|_{L_2(\Omega)}
	\end{align*}
	where we have used (\ref{PointwiseLaplacianRemainder}), Lemma \ref{lem:DirProblem} and (\ref{14.7n}) in order to obtain the last estimate.
\end{proof}

\subsection{$L_p$-Regularity estimates} In this section we focus on quantitative estimates for the fractional Laplacian in the general $L_p$ setting. First, we recall the $L_p$ regularity result for the Poisson-type equation involving the fractional Laplacian. See, e.g., \cite{Stein}.

\begin{thm}\label{thm:FractLaplaceP}
	Let $s \in (0,1)$ and $1 < p < \infty$. Let $f \in \dot{\mathscr{L}}^{-s}_2(\mathbb{R}^d)$ and let $u \in \dot{\mathscr{L}}^{s}_2(\mathbb{R}^d)$ be the weak solution to the fractional Poisson-type equation
\begin{equation*}
	(-\Delta)^s u = f \quad \text{in} \quad \mathbb{R}^d.
\end{equation*}
Let $\lambda \geq 0$. Then
\begin{equation*}
	u \in  \dot{\mathscr{L}}^{\lambda + 2 s}_p(\mathbb{R}^d) \quad \text{if and only if} \quad f \in \dot{\mathscr{L}}^{\lambda}_p(\mathbb{R}^d),
\end{equation*}
and
\begin{equation}\label{EquivFractLaplaceP}
\|u\|_{\dot{\mathscr{L}}^{\lambda + 2 s}_p(\mathbb{R}^d)} \asymp \|f\|_{\dot{\mathscr{L}}^{\lambda}_p(\mathbb{R}^d)}.
\end{equation}
In particular, we have
\begin{equation}\label{FractLaplaceP}
\|u\|_{ \dot{\mathbf{B}}^{2 s}_{p, q}(\mathbb{R}^d)} \lesssim \|f\|_{L_p(\mathbb{R}^d)}, \quad q \geq \max\{p,2\}
\end{equation}
and
\begin{equation}\label{FractLaplaceConvP}
\|u\|_{ \dot{\mathbf{B}}^{2 s}_{p, q}(\mathbb{R}^d)} \gtrsim \|f\|_{L_p(\mathbb{R}^d)}, \quad q \leq \min\{p,2\}.
\end{equation}
\end{thm}

Now we are in a position to show the $L_p$ version of Theorem \ref{thm 14.2}.

\begin{thm}\label{thm 14.2 p}
	Let $s \in (0,1)$. Let $f \in \dot{\mathscr{L}}^{-s}_2(\mathbb{R}^d)$ and let $u \in \dot{\mathscr{L}}^{s}_2(\mathbb{R}^d)$ be the weak solution to the fractional Poisson-type equation
\begin{equation*}
	(-\Delta)^s u = f \quad \text{in} \quad \mathbb{R}^d.
\end{equation*}
	Let $ \lambda > 0, 1 < p < \infty$, and $0 < q \leq \infty$. Then
	\begin{equation}\label{EquivFractLaplace0 p}
		\left(\int_0^t (\xi^{-2 s} \omega_{ \lambda + 2 s} (u, \xi)_p)^q \frac{d \xi}{\xi}\right)^{1/q} \lesssim \omega_\lambda(f,t)_p \iff q \geq \max\{p,2\}
	\end{equation}
	and
	\begin{equation}\label{EquivFractLaplace0 p*}
		\left(\int_0^t (\xi^{-2 s} \omega_{ \lambda + 2 s} (u, \xi)_p)^q \frac{d \xi}{\xi}\right)^{1/q} \gtrsim \omega_\lambda(f,t)_p \iff q \leq \min\{p,2\}.
	\end{equation}
\end{thm}
\begin{proof}
Let $q \geq \max\{p,2\}$. By (\ref{EquivFractLaplaceP}), (\ref{FractLaplaceP}) and Lemma \ref{LemmaModuli}, we have
	\begin{equation}\label{EquivFractLaplace1p}
		K(t^\lambda, u; \dot{\mathbf{B}}^{2 s}_{p, q}(\mathbb{R}^d), \dot{\mathscr{L}}^{2 s + \lambda}_p(\mathbb{R}^d)) \lesssim K(t^\lambda,f ; L_p(\mathbb{R}^d), \dot{\mathscr{L}}^{\lambda}_p(\mathbb{R}^d)) \asymp \omega_{\lambda}(f,t)_p.
	\end{equation}
	On the other hand, since $\dot{\mathbf{B}}^{2 s}_{p, q}(\mathbb{R}^d) = (L_p(\mathbb{R}^d), \dot{\mathscr{L}}^{2 s + \lambda}_p(\mathbb{R}^d))_{\frac{2 s}{ 2 s+\lambda}, q}$ (cf. \cite[Theorem 6.4.5]{BerghLofstrom}. This is a special case of the homogeneous counterpart of (\ref{PrelimInterpolationW})), one can apply Lemma \ref{PrelimHolmstedt}(ii) (with $\alpha = 0$) and (\ref{8+}) in order to derive
	\begin{align}
		K(t^\lambda, u; \dot{\mathbf{B}}^{2 s}_{p, q}(\mathbb{R}^d), \dot{\mathscr{L}}^{2 s + \lambda}_p(\mathbb{R}^d)) & \nonumber\\
		& \hspace{-4cm} \asymp \left(\int_0^{t} (\xi^{- 2 s} K(\xi^{\lambda + 2 s}, u; L_p(\mathbb{R}^d),  \dot{\mathscr{L}}^{2 s + \lambda}_p(\mathbb{R}^d)))^{q} \frac{d \xi}{\xi}\right)^{1/q} \nonumber\\
		& \hspace{-4cm} \asymp 	\left(\int_0^t (\xi^{-2 s} \omega_{ \lambda + 2 s} (u, \xi)_p)^{q} \frac{d \xi}{\xi}\right)^{1/q}. \label{14.22}
	\end{align}
	Inserting this estimate into (\ref{EquivFractLaplace1p}) we get the desired estimate given in (\ref{EquivFractLaplace0 p}).
	
	Conversely, assume that there exists $q$ such that the following inequality holds
	\begin{equation*}
		\left(\int_0^t (\xi^{-2 s} \omega_{ \lambda + 2 s} (u, \xi)_p)^q \frac{d \xi}{\xi}\right)^{1/q} \lesssim \omega_\lambda(f,t)_p.
	\end{equation*}
	In particular, we have
		\begin{equation*}
		\left(\int_0^t (\xi^{-2 s} \omega_{ \lambda + 2 s} (u, \xi)_p)^q \frac{d \xi}{\xi}\right)^{1/q} \lesssim \|f\|_{L_p(\mathbb{R}^d)}
	\end{equation*}
	which yields that $\|u\|_{\dot{\mathbf{B}}^{2 s}_{p,q}(\mathbb{R}^d)} \lesssim \|f\|_{L_p(\mathbb{R}^d)} = \|(-\Delta)^s u\|_{L_p(\mathbb{R}^d)}$. This implies that $q \geq \max\{p,2\}$ (see \cite[Chapter V, Sections 6.8 and 6.9]{Stein}). The proof of (\ref{EquivFractLaplace0 p}) is completed.
	
	Next we will show (\ref{EquivFractLaplace0 p*}). Assume that $q \leq \min\{p,2\}$. It follows from (\ref{14.22}), (\ref{EquivFractLaplaceP}), (\ref{FractLaplaceConvP}) and Lemma \ref{LemmaModuli} that
	\begin{align*}
		 	\left(\int_0^t (\xi^{-2 s} \omega_{ \lambda + 2 s} (u, \xi)_p)^{q} \frac{d \xi}{\xi}\right)^{1/q} & \asymp K(t^\lambda, u; \dot{\mathbf{B}}^{2 s}_{p, q}(\mathbb{R}^d), \dot{\mathscr{L}}^{2 s + \lambda}_p(\mathbb{R}^d)) \\
			& \gtrsim K(t^\lambda,f ; L_p(\mathbb{R}^d), \dot{\mathscr{L}}^{\lambda}_p(\mathbb{R}^d)) \asymp \omega_{\lambda}(f,t)_p.
	\end{align*}
	This gives the desired estimate in (\ref{EquivFractLaplace0 p*}).
	
	Conversely, assume that there is $q$ such that
	\begin{equation*}
		\left(\int_0^t (\xi^{-2 s} \omega_{ \lambda + 2 s} (u, \xi)_p)^q \frac{d \xi}{\xi}\right)^{1/q} \gtrsim \omega_\lambda(f,t)_p, \quad t > 0.
	\end{equation*}
	Therefore, by (\ref{12.12new2**}), we derive
		\begin{equation*}
		\|u\|_{\dot{\mathbf{B}}^{2 s}_{p,q}(\mathbb{R}^d)} \gtrsim \|f\|_{L_p(\mathbb{R}^d)}=  \|(-\Delta)^s u\|_{L_p(\mathbb{R}^d)}
	\end{equation*}
	which implies $q \leq \min\{p,2\}$ (see \cite[Chapter V, Sections 6.8 and 6.9]{Stein}).
\end{proof}

As application of Theorem \ref{thm 14.2 p} we obtain the next result which deals with the smoothness properties of the operator $(-\Delta)^s$ and complements the $L_2$-result of Theorem \ref{CorFractLapl2}.

\begin{thm}
	Let $s \in (0,1)$. Let $f \in \dot{\mathscr{L}}^{-s}_2(\mathbb{R}^d)$ and let $u \in \dot{\mathscr{L}}^{s}_2(\mathbb{R}^d)$ be the weak solution to the fractional Poisson-type equation
\begin{equation*}
	(-\Delta)^s u = f \quad \text{in} \quad \mathbb{R}^d.
\end{equation*}
Let $\lambda > 0, b \in \mathbb{R}, 1 < p < \infty$, and $0 < q \leq \infty$. Then
\begin{equation*}
	u \in  \dot{\mathbf{B}}^{2 s + \lambda, b}_{p, q}(\mathbb{R}^d) \quad \text{if and only if} \quad f \in \dot{\mathbf{B}}^{\lambda, b}_{p,q}(\mathbb{R}^d),
\end{equation*}
and
\begin{equation*}
\|u\|_{\dot{\mathbf{B}}^{2 s + \lambda, b}_{p, q}(\mathbb{R}^d)} \asymp \|f\|_{\dot{\mathbf{B}}^{\lambda, b}_{p,q}(\mathbb{R}^d)}.
\end{equation*}
\end{thm}
\begin{proof}
	The proof is similar to that given in Theorem \ref{CorFractLapl2} and is left to the reader.
\end{proof}

We now turn to the quantitative version of the regularity result for the Dirichlet problem (\ref{DirProblem}) with $L_p$-data given in Theorem \ref{thm:DirProblemP}.

\begin{thm}
	Let $\Omega$ be a bounded open set in $\mathbb{R}^d$ and let $0 < s < 1$. Let $f \in \dot{\mathscr{L}}^{-s}_2(\overline{\Omega})$ and let $u \in \dot{\mathscr{L}}^{s}_2(\overline{\Omega}) $  be the weak solution to the Dirichlet problem
	\begin{equation*}
		 \left\{\begin{array}{cl}  (-\Delta)^s u = f & \text{in} \quad  \Omega, \\
		u= 0  & \text{on} \quad \mathbb{R}^d \setminus \Omega.
		       \end{array}
                        \right.
	\end{equation*}
	Let $1 < p < \infty, q \geq \max\{p,2\}$ and $\lambda > 0$. Let $\zeta \in C^\infty(\mathbb{R}^d)$ be a cut-off function satisfying (\ref{cutoff}).
Assume that $f \in L_p(\Omega)$ and $u \in W^s_p(\Omega)$. Then
	\begin{equation*}
		\left(\int_0^t (\xi^{-2 s} \omega_{ \lambda + 2 s} (\zeta u, \xi)_p)^q \frac{d \xi}{\xi}\right)^{1/q} \lesssim \omega_\lambda(\zeta f,t)_p + \|u\|_{W^s_p(\Omega)}, \quad t > T > 0.
	\end{equation*}
\end{thm}
\begin{proof}
	The proof relies on the following result obtained in \cite[Section 3.2, (3.10) and (3.18)]{BiccariWarmaZuazua}.
	
	\begin{lem}
		The following pointwise formula holds
		\begin{equation}\label{PointwiseLaplacian*}
			(-\Delta)^s (\zeta u) = \zeta (-\Delta)^s u + g \quad \text{in} \quad \mathbb{R}^d
		\end{equation}
		where $g \in L_p(\mathbb{R}^d)$ with
		\begin{equation}\label{PointwiseLaplacianRemainderP*}
		\|g\|_{L_p(\mathbb{R}^d)} \lesssim \|u\|_{W^{s}_{p}(\Omega)}.
		\end{equation}
	\end{lem}
	
	Since $f \in L_p(\Omega)$, it follows that $\zeta (-\Delta)^s u + g \in L_p(\mathbb{R}^d)$. According to (\ref{PointwiseLaplacian*}), (\ref{EquivFractLaplace0 p}) and (\ref{PointwiseLaplacianRemainderP*}), we have
	\begin{align*}
		\left(\int_0^t (\xi^{-2 s} \omega_{ \lambda + 2 s} (\zeta u, \xi)_p)^q \frac{d \xi}{\xi}\right)^{1/q} & \lesssim \omega_\lambda(\zeta (-\Delta)^s u + g, t)_p \\
		& \lesssim  \omega_\lambda(\zeta (-\Delta)^s u, t)_p + \|g\|_{L_p(\mathbb{R}^d)} \\
		& \lesssim  \omega_\lambda(\zeta f, t)_p + \|u\|_{W^{s}_{p}(\Omega)}.
	\end{align*}
\end{proof}

\newpage
\appendix
\section{List of symbols}\label{A}

\bigskip

  \textbf{Sets}

\medskip

    $\mathbb{R}^d$, $d$-dimensional real Euclidean space,  \pageref{SETR}

    $\mathbb{T}$, unit circle, \pageref{SETT}

    $\mathbb{Z}^d$, integer lattice in $\mathbb{R}^d$, \pageref{SETZ}

    $\mathbb{N}$, natural numbers, \pageref{SETN}

    $\mathbb{N}_0$, natural numbers with $0$, \pageref{SETN0}

    $GM$, set of general monotone functions, \pageref{GM}

    $\widehat{GM}^d$, set of radial functions whose Fourier transform is general monotone, \pageref{GMF}

    $\mathfrak{L}$, set of lacunary Fourier series, \pageref{LAC}

    $\mathfrak{L}^{\text{wm}}$, set of lacunary Fouries series with weak monotone Fourier coefficients, \pageref{LACMON}

     $\mathfrak{l}$, set of lacunary sequences, \pageref{LACSEQ}

     $SV$, class of slowly varying functions, \pageref{SV}

     $B_t(x)$, ball of radius $t$ centered at $x$, \pageref{BALL}

     $\mathbb{R}^{d+1}_+$, upper half-space of $\mathbb{R}^{d+1}$, \pageref{HALFSPACE}

     $Q_{\nu m}$, cube with sides parallel to the coordinate axes, centred at $2^{-\nu} m$, and with side length $2^{-\nu}$, \pageref{CUBE}

     $\Omega$, bounded open set, \pageref{DOMAIN}

     $\partial \Omega$, boundary of $\Omega$, \pageref{BOUNDARY}

     $\overline{\Omega}$, closure of $\Omega$, \pageref{CLOSURE}
%
%
%
%

\bigskip

\textbf{Numbers and relations}

\medskip

    $X \hookrightarrow Y$, continuous embedding, \pageref{XY}

    $[a]$,  the largest integer not greater than $a$, \pageref{[a]}

    $p'$, dual exponent of $p$, \pageref{p'}

    $A \lesssim B$, the estimate $A \leq C B,$ where $C$ is a positive constant, \pageref{AB}

    $A \asymp B$, the estimates $C^{-1} B \leq A \leq C B$, where $C > 1$ is a constant, \pageref{ASYMP}

    $m_t$, surface measure of the sphere of radius $t$, \pageref{MEAS}

%
%

\bigskip

 \textbf{Spaces}

\medskip

$\mathbf{B}^{s,b}_{p,q}(\mathbb{R}^d)$, Besov space defined by differences, \pageref{BESOVDIFF}

$\text{Lip}^{(k,-b)}_{p,q}(\mathbb{R}^d)$, logarithmic Lipschitz space, \pageref{LOGLIPSCHITZ}

$\text{Lip}(\mathbb{R}^d)$, Lipschitz space, \pageref{LIPSCHITZ}

$\text{BV}(\mathbb{R}^d)$,  the space of functions of bounded variation, \pageref{BV}

$\mathbf{B}^{s,b}_{p,q}(\mathbb{T})$, periodic Besov space defined by differences, \pageref{BESOVDIFFPER}

$\text{Lip}^{(k,-b)}_{p,q}(\mathbb{T})$, periodic Lipschitz space, \pageref{LOGLIPSCHITZPER}

$\mathcal{D}(\mathbb{T})$, the space of infinite differentiable periodic functions, \pageref{D}

$\mathcal{D}'(\mathbb{T})$, dual space of $\mathcal{D}(\mathbb{T})$, \pageref{D'}

$\mathcal{S}(\mathbb{R}^d)$, Schwartz space, \pageref{S}

$\mathcal{S}'(\mathbb{R}^d)$, space of tempered distributions, \pageref{S'}

$B^{s,b}_{p,q}(\mathbb{R}^d)$, Besov space defined by Fourier analytic tools, \pageref{BESOVF}

$B^{s,b}_{p,q}(\mathbb{T})$, periodic Besov space defined by Fourier analytic tools, \pageref{BESOVFPER}

$F^{s,b}_{p,q}(\mathbb{R}^d)$, Triebel-Lizorkin space, \pageref{TL}

$H^{s,b}_p(\mathbb{R}^d)$, Sobolev space, \pageref{SOB}

$F^{s,b}_{p,q}(\mathbb{T})$, periodic Triebel-Lizorkin space, \pageref{TLPER}

$H^{s,b}_p(\mathbb{T})$, periodic Sobolev space, \pageref{SOBPER}

$W^k_p(\mathbb{R}^d)$, classical Sobolev space, \pageref{SOBCLAS}

$W^k_p(\mathbb{T})$, classical periodic Sobolev space, \pageref{SOBPERCLAS}

$(A_0,A_1)_{\theta,q}$, real interpolation space, \pageref{REAL}

$(A_0,A_1)_{\theta,q;\mathbb{A}}, (A_0,A_1)_{\theta,q;\alpha}, (A_0,A_1)_{\theta,q;\alpha,\xi}$, logarithmic interpolation spaces, \pageref{REALLOG}, \pageref{REALLOGIT}

$(A_0,A_1)_{(\theta,\alpha),q}$, limiting interpolation space, \pageref{REALLIM}

$B^{s,b,\xi}_{p,q}(\mathbb{R}^d)$, Besov space of iterated logarithmic smoothness defined by Fourier analytic tools, \pageref{BESOVFLOG}

$H^{s,b,\Psi}_p(\mathbb{R}^d)$, Sobolev space of generalized smoothness, \pageref{SOBGENERAL}

$B^{s,b,\Psi}_{p,q}(\mathbb{R}^d)$, Besov space of generalized smoothness, \pageref{BESOVGENERAL}

$A'$, dual space of $A$, \pageref{DUAL}

$\dot{\mathscr{L}}^{s}_p(\mathbb{R}^d)$, Riesz potential space, \pageref{RIESZ}

$\dot{W}^k_p(\mathbb{R}^d)$, homogeneous Sobolev space, \pageref{SOBHOM}

$\dot{B}^s_{p,q}(\mathbb{R}^d)$, homogeneous Besov space defined by Fourier analytic tools, \pageref{BESOVHOM}

$D(\Lambda)$, domain of the infinitesimal generator $\Lambda$ of a semi-group, \pageref{DOMGEN}

$V_p(\mathbb{R})$, bounded $p$-variation space, \pageref{pVAR}

$V_p(\mathbb{T})$, periodic bounded $p$-variation space, \pageref{pVARPER}

$\dot{V}_p(\mathbb{R})$, homogeneous bounded $p$-variation space, \pageref{pVARHOM}

$\dot{V}_p(\mathbb{T})$, homogeneous periodic bounded $p$-variation space, \pageref{pVARPERHOM}

$\mathbf{B}^{s,b,\xi}_{p,q}(\mathbb{R}^d)$, Besov space of iterated logarithmic smoothness defined by differences, \pageref{BESOVDLOG}

$\dot{\mathcal{S}}(\mathbb{R}^d)$, space formed by $f \in \mathcal{S}(\mathbb{R}^d)$ with $(D^\beta \widehat{f})(0) = 0$ for $\beta \in \mathbb{N}_0^d$, \pageref{SCLASS}

$\dot{\mathcal{S}}'(\mathbb{R}^d)$, $\mathcal{S}'(\mathbb{R}^d)$ modulo polynomials, \pageref{S'CLASS}

$L_1(s)$, weighted $L_1$-space, \pageref{WEIGHTEDL1}

$\Lambda^{s,b}(\mathbb{R}^d)$, Lipschitz space defined in terms of the heat semi-group, \pageref{LIPSCHITZHEAT}

$C^{k,s,b}(\mathbb{R}^d)$, H\"older space, \pageref{HOLDER}

$\mathcal{Z}^{1,b}(\mathbb{R}^d)$, Zygmund space, \pageref{ZYGMUND}

$\dot{\mathscr{L}}^{s}_{p,\text{loc}}(\Omega)$, local Riesz potential space, \pageref{LOCALRIESZ}

$\dot{\mathbf{B}}^{s,b}_{p,q,\text{loc}}(\Omega)$, local (homogeneous) Besov space, \pageref{LOCALBESOV}

$\dot{\mathscr{L}}^{s}_2(\overline{\Omega})$, (homogeneous) Sobolev space on $\Omega$, \pageref{SOBOLEVDOMAIN}

$\dot{\mathscr{L}}^{-s}_2(\overline{\Omega})$, dual space of $\dot{\mathscr{L}}^{s}_2(\overline{\Omega})$, \pageref{SOBOLEVDUAL}

$W^s_p(\Omega)$, Besov space on $\Omega$, \pageref{BESOVDOMAIN}

%

\bigskip

\textbf{Functionals and functions}

\medskip

$\chi_E$, characteristic function of $E$, \pageref{INDICATOR}

     $\omega_{k}(f,t)_p$, modulus of smoothness of integer order $k$, \pageref{MODK}

     $\Delta^k_h$, difference of integer order $k$ with step $h$, \pageref{DELTA}

     $K(t,a)$, Peetre's $K$-functional, \pageref{K}

     $\ell(t)$, $\ell^{\mathbb{A}}(t)$, logarithmic functions, \pageref{LOG}

    $E^\ast_n(f)_{L_p(\mathbb{T})}$, $L_p$-best approximation of the periodic function $f$ by trigonometric polynomials of degree at most $n-1$, \pageref{ERRORPER}

 $E_n(f)_{L_p(\mathbb{R}^d)}$,  $L_p$-best approximation of the function $f$ by entire functions of spherical exponential type $n$, \pageref{ERROR}

     $\omega_{\alpha}(f,t)_{p}$, modulus of smoothness of fractional order $\alpha$, \pageref{MODA}

     $\Delta^\kappa_h$, difference of fractional order $\kappa$ with step $h$, \pageref{DELTAFRAC}

     $a_{\nu m}$, atom located on $Q_{\nu m}$, \pageref{ATOM}

     $f^\ast$, non-increasing rearrangement of $f$, \pageref{REARRANGEMENT}

\bigskip

\textbf{Operators}

\medskip

$\text{id}_X$, identity map on $X$, \pageref{ID}

    $\widehat{f}$, Fourier transform of $f$, \pageref{FT}

    $f^\vee$, inverse Fourier transform of $f$, \pageref{IFT}

    $D^\alpha$, partial derivative, \pageref{DERIVATIVE}

    $\mathfrak{J}_{s,b}$, periodic lifting operator, \pageref{LIFTPER}

    $I_\sigma$, lifting operator, \pageref{LIFT}

    $S_n(f)$, $n$-th partial sum of the Fourier series of $f$, \pageref{PARTIALSUM}

    $\eta_R$, de la Vall\'ee-Poussin operator, \pageref{VALLEEPOUSSIN}

    $\mathcal{J}_b$, logarithmic lifting operator, \pageref{LIFTLOG}

    $J_s$, Riesz potential operator, \pageref{RIESZOPERATOR}

    $B_t, B_{l,t}$, ball average operators, \pageref{AVERAGE}

    $\Delta^l$, $l$th order Laplacian, \pageref{LAPLACEP}

    $V_t, V_{l,t}$, sphere average operators, \pageref{AVERAGESPH}

    $S^{\lambda,\alpha}_t$, generalized Bochner-Riesz operators, \pageref{BOCHNERRIESZ}

    $W^\alpha_t$, Weierstrass operators, \pageref{WEIERSTRASS}

    $W_t$, Gauss-Weierstrass semi-group, \pageref{GAUSSWEIERSTRASS}

    $P_t$, Cauchy-Poisson semi-group, \pageref{POISSON}

    $\Lambda$, infinitesimal generator of a semi-group, \pageref{GEN}

    $(-\Delta)^{s}$, fractional Laplace operator, \pageref{FRACTLAPLACE}
%

%
%
\end{document}